\documentclass[11pt,a4paper,reqno]{amsart}
\pdfoutput=1
\usepackage{amsmath,amsthm,latexsym,amssymb,wasysym}
\usepackage{numprint}
\usepackage[margin=1.0in,tmargin=0.9in,bmargin=0.80in]{geometry}
\usepackage{hyperref}
\usepackage{bbm}
\usepackage{xargs}
\usepackage{array}
\usepackage{enumerate}
\usepackage{mathrsfs}
\usepackage{enumitem}
\usepackage{subcaption}
\setlist[enumerate]{label={\upshape(\roman*)}}
\usepackage[textsize=footnotesize]{todonotes}
\usepackage[flushleft]{threeparttable}
\usepackage{multirow}
\usepackage{longtable}
\usepackage{times}
\usepackage{microtype}
\usepackage[numbers]{natbib}
\usepackage[normalem]{ulem}
\usepackage{amssymb}
\usepackage{booktabs} 
\usepackage{lscape} 
\usepackage[figuresright]{rotating}
\usepackage{graphicx} 
\usepackage{verbatim}
\newcommand\iid{i.i.d.}
\allowdisplaybreaks[3]

\npthousandsep{,}
\definecolor{bred}{rgb}{0.8,0,0}
\hypersetup{colorlinks,linkcolor={blue},citecolor={bred},urlcolor={blue}}

\newtheorem{theorem}{Theorem}[section]
\newtheorem{proposition}[theorem]{Proposition}
\newtheorem{lemma}[theorem]{Lemma}
\newtheorem{corollary}[theorem]{Corollary}
\newtheorem{remark}[theorem]{Remark}
\newtheorem{definition}[theorem]{Definition}

\newtheorem{assumption}{Assumption}

\def\rmd{\mathrm{d}}

\newcommand{\N}{\mathbb{N}}
\newcommand{\R}{\mathbb{R}}

\allowdisplaybreaks

\begin{document}

\title[]{Non-asymptotic convergence analysis of\\ the stochastic gradient Hamiltonian Monte Carlo algorithm\\ with discontinuous stochastic gradient\\ with applications to training of ReLU neural networks}

\author[L. Liang]{Luxu Liang}
\author[A. Neufeld]{Ariel Neufeld}
\author[Y. Zhang]{Ying Zhang}

\address{School of Mathematics, Renmin University of China, Beijing, China}
\email{lianglux@ruc.edu.cn}

\address{Division of Mathematical Sciences, Nanyang Technological University, 21 Nanyang Link, 637371 Singapore}
\email{ariel.neufeld@ntu.edu.sg}

\address{Financial Technology Thrust, Society Hub, The Hong Kong University of Science and Technology (Guangzhou), Guangzhou, China}
\email{yingzhang@hkust-gz.edu.cn}

\date{}
\thanks{Financial supports by the MOE AcRF Tier 2 Grant MOE-T2EP20222-0013 and the Guangzhou-HKUST(GZ) Joint Funding 
Programs (No. 2024A03J0630 and No. 2025A03J3322) are gratefully acknowledged.}

\begin{abstract}
    In this paper, we provide a non-asymptotic analysis of the convergence of the stochastic gradient Hamiltonian Monte Carlo (SGHMC) algorithm to a target measure in Wasserstein-1 and Wasserstein-2 distance. Crucially, compared to the existing literature on SGHMC, we allow its stochastic gradient to be discontinuous.  This allows us to provide explicit upper bounds, which can be controlled to be arbitrarily small, for the expected excess risk of non-convex stochastic optimization problems with discontinuous stochastic gradients, including, among others, the training of neural networks with ReLU activation function. To illustrate the applicability of our main results, we consider numerical experiments on quantile estimation and on several optimization problems involving ReLU neural networks relevant in finance and artificial intelligence.
\end{abstract}
\maketitle

\section{Introduction}\label{sec:1}
\noindent In this paper, we consider the nonconvex stochastic optimization problem
\begin{equation}\label{opt}
\operatorname{minimize}\quad\mathbb{R}^d \ni \theta \mapsto u(\theta):=\mathbb{E}_{X \sim \mathcal{D}}[U(\theta, X)],
\end{equation}
where $U: \mathbb{R}^{d} \times \mathbb{R}^{m}\rightarrow \mathbb{R}$ is a measurable function and $X$ is an $m$-dimensional random variable with some known probability distribution $\mathcal{D}$. Given a set of observations, our goal is to construct an estimator $\hat{\theta}$ such that the expected excess risk given by
$$
\mathbb{E}[u(\hat{\theta})]-\inf _{\theta \in \mathbb{R}^d} u(\theta)
$$
is minimized. It is well known that computing an approximate global minimizer of $u$ is closely linked to sampling from the distribution $\pi_\beta(\mathrm{d} \theta) \propto \exp (-\beta u(\theta)) \mathrm{d} \theta$ which concentrates around the global minimizers of $u$ for sufficiently large $ \beta$, see~\cite{hwang1980laplace}. Thus, the optimization problem~\eqref{opt} can be converted to the problem of sampling from $\pi_{\beta}$. 

As $\pi_{\beta}$ is the unique invariant measure of the Langevin stochastic differential equation (SDE) given by
\begin{align}\label{SDE}
Z_0 = \theta_0, \quad \mathrm{d} Z_t=-h\left(Z_t\right) \mathrm{d} t+\sqrt{2\beta^{-1}} \mathrm{d} B_t,\quad t\in \mathbb{R}_{+},
\end{align} 
where $\theta_0$ is an $\mathbb{R}^d$-valued random variable, $h:=\nabla u$, $\beta>0$ is the so-called inverse temperature parameter, and $\left\{B_t\right\}_{t\geq 0}$ is a $d$-dimensional Brownian motion, sampling from $\pi_{\beta}$ can be achieved by employing numerical algorithms that track~\eqref{SDE}, see~\cite{hwang1980laplace}. One of the widely used method is the stochastic gradient Langevin dynamics (SGLD) algorithm given by
$$
\begin{aligned}
\theta_0^{\mathrm{SGLD}}:=\theta_0, \quad \theta_{n+1}^{\mathrm{SGLD}}=\theta_n^{\mathrm{SGLD}}-\eta H\left(\theta_n^{\mathrm{SGLD}}, X_{n+1}\right)+\sqrt{2 \eta \beta^{-1}} \xi_{n+1},
\end{aligned}
$$
where $\eta > 0$ is the step size, $\left(X_n\right)_{n \in \mathbb{N}_0}$ is the data sequence being \iid\ copies of $X$, $H$ is an unbiased estimator of $h$, i.e., $h(\theta) = \mathbb{E}\left[H(\theta, X_0)\right]$ for any $\theta \in \mathbb{R}^d$, called stochastic gradient of $u$, and $\left\{\xi_n\right\}_{n \in \mathbb{N}_0}$ is a sequence of i.i.d.\ standard Gaussian random variables in $\mathbb{R}^{d}$. Recently, this method has been popular in nonconvex optimization and Bayesian inference with well-established convergence results, see, e.g.,~\cite{barkhagen2021stochastic, chau2021stochastic, chu2024nonasymptotic, farghly2021time, hu2020non, li2016high, neufeld2022non, neufeld2024robust, raginsky2017non, welling2011bayesian, xu2018global, zhang2023nonasymptotic} and references therein.


In this paper, we focus on an alternative method to tackle the problem of sampling from $\pi_{\beta}$. We consider the second-order (underdamped) Langevin SDE given by 
\begin{equation}\label{SDE2}
\begin{aligned}
& \mathrm{d} V_t=-\gamma V_t \mathrm{~d} t-h\left(\theta_t\right) \mathrm{d} t+\sqrt{2 \gamma\beta^{-1}} \mathrm{d} B_t, \\
& \mathrm{~d} \theta_t=V_t \mathrm{~d} t,
\end{aligned}
\end{equation}
where $\left(\theta_t, V_t\right)_{t \in \mathbb{R}_{+}}$ are called the position and momentum process, respectively, and $\gamma > 0$ is the friction coefficient. Under suitable assumptions, SDE~\eqref{SDE2} admits a unique invariant measure (see, e.g.~\cite{pavliotis2014stochastic})
\begin{equation}\label{pi_intro}
\bar{\pi}_\beta( \theta,  v) \propto \exp \left(-\beta\left(u(\theta) + \frac{1}{2}\lvert v\rvert^2\right)\right).
\end{equation}
We note that the marginal distribution of~\eqref{pi_intro} in  $\theta$ is exactly the target measure $\pi_{\beta}$, thus, sampling from~\eqref{pi_intro} in the extended space while only considering the samples from the first coordinate (i.e., from $\theta$) is equivalent to sampling from $\pi_{\beta}$. We consider the stochastic gradient Hamiltonian Monte Carlo (SGHMC) algorithm given by
\begin{equation}\label{SGHMC_intro}
\begin{aligned}
V_0^{\eta}&:=V_0, \quad V_{n+1}^{\eta}  = V_{n}^{\eta} - \eta\left[\gamma  V_{n}^{\eta}+H(\theta_{n}^{\eta}, X_{n+1})\right]+\sqrt{2 \gamma \eta \beta^{-1}} \xi_{n+1}, \\ 
\theta_0^{\eta}&:=\theta_0, \quad \theta_{n+1}^{\eta}  = \theta_{n}^{\eta} + \eta V_{n}^{\eta},
\end{aligned}
\end{equation}
where $\eta > 0$ is the step size. Note that while SGLD can be viewed as the analogue of the stochastic gradient in Markov Chain Monte Carlo (MCMC) literatures, SGHMC can be viewed as the analogue of stochastic gradient methods augmented with momentum. Several numerical studies have demonstrated that SGHMC algorithm is superior over SGLD algorithm in term of robust updates w.r.t.\ gradient estimation noise, see, e.g.,~\cite{chen2014stochastic, liu2020improved, wangenhancing}. Theoretical results for the SGHMC algorithm~\eqref{SGHMC_intro} have been established under the condition that the stochastic gradient is assumed to be (locally) Lipschitz continuous, see, e.g.~\cite{akyildiz2020nonasymptotic, chau2022stochastic, gao2022}. However, these assumptions fail to accommodate discontinuous stochastic gradients used in practice, such as the signed regressor, signed error, sign-sign algorithms, and optimization problems involving neural networks with ReLU activation function, see, e.g.,~\cite{eweda1990analysis, eweda1995convergence, lim2024non}. To tackle these problems, we consider the case where $H$ is discontinuous and only satisfies a certain continuity \textit{in average} condition (see Assumption~\ref{asm:A4} below). We then provide non-asymptotic results in Wasserstein-1 and Wasserstein-2 distance between the law of the $n$-th iterate of the SGHMC algorithm and the target distribution $\bar{\pi}_\beta$. This further allows us to provide a non-asymptotic upper bound for the expected excess risk of the associated optimization problem~\eqref{opt}. 

To illustrate the applicability of our results, we present three key examples in statistical machine learning, namely, quantile estimation, optimization involving ReLU neural networks, and hedging under asymmetric risk, where each of the corresponding functions $H$ are discontinuous. Numerical results support our theoretical findings and demonstrate the superiority of the SGHMC algorithm over its SGLD counterpart, which align with the findings of~\cite{chen2015convergence, chen2014stochastic, wangenhancing}. 


We conclude this section by introducing some notations. For $1 \leq p<\infty$, $L^p$ is used to denote the usual space of $p$-integrable real-valued random variables. For an integer $d \geq 1$, the Borel sigma-algebra of $\mathbb{R}^d$ is denoted by $\mathcal{B}\left(\mathbb{R}^d\right)$. We also denote all non-negative real numbers by $\mathbb{R}_{+}$.
For a positive real number $a$, we denote by $\lfloor a\rfloor$ its integer part, and define $\lceil a\rceil:=\lfloor a\rfloor+1$. For any matrix $M \in \mathbb{R}^{p \times p}$, denote by $\operatorname{diag}(M):=\left(M_{11}, \ldots, M_{p p}\right)$ the vector of the diagonal elements of $M$, and denote by $M^{\top}$ its transpose. Denote by $I_p$ the $p \times p$ identity matrix. We denote the dot product by $\langle\cdot, \cdot\rangle$ while $|\cdot|$ denotes the associated norm. The set of probability measures defined on space $\left(\mathbb{R}^d, \mathcal{B}\left(\mathbb{R}^d\right)\right)$ is denoted by $\mathcal{P}\left(\mathbb{R}^d\right)$. For an $\mathbb{R}^d$-valued random variable $X$, $\mathcal{L}(X)$ and $\mathbb{E}[X]$ are used to denote its law and its expectation respectively. We denote by $\mathbb{N}_0$ the non-negative integers including zero. For $\mu, \nu \in \mathcal{P}\left(\mathbb{R}^d\right)$, let $\mathcal{C}(\mu, \nu)$ denote the set of probability measures $\Gamma$ on $\mathcal{B}\left(\mathbb{R}^{2 d}\right)$ such that its marginals are $\mu, \nu$. Finally, for any $p \geq 1$, we denote by $\mathcal{P}_{p}\left(\mathbb{R}^d\right)$ the probability measure in $\mathcal{P}\left(\mathbb{R}^d\right)$ with $p$-th finite moments. For any $p \geq 1$, $\mu, \nu \in \mathcal{P}_{p}\left(\mathbb{R}^d\right)$, we denote the Wasserstein distance of order $p \geq 1$
$$
W_p(\mu, \nu):=\inf _{\Gamma \in \mathcal{C}(\mu, \nu)}\left(\int_{\mathbb{R}^d} \int_{\mathbb{R}^d}\left|\theta-\theta^{\prime}\right|^p \Gamma\left(\mathrm{d} \theta, \mathrm{d} \theta^{\prime}\right)\right)^{1 / p}.
$$

\section{Setting and Assumptions}\label{sec:2}
\subsection{Setting}
Let $U: \mathbb{R}^d \times \mathbb{R}^m \rightarrow \mathbb{R}_{+}$ be a Borel measurable function, and let $X$ be an $\mathbb{R}^m$-valued random variable defined on the probability space $(\Omega, \mathcal{F}, \mathbb{P})$ with probability law $\mathcal{L}(X)$ satisfying $\mathbb{E}[|U(\theta, X)|]<~\infty$ for all $\theta \in \mathbb{R}^d$. We assume that $u: \mathbb{R}^d \rightarrow \mathbb{R}$ defined by $u(\theta):=\mathbb{E}[U(\theta, X)], \theta \in \mathbb{R}^d$, is a continuously differentiable function, and denote by $h:=\nabla u$ its gradient. In addition, for any $\beta>0$, we define
\begin{equation}\label{eq:pibetaexp}
\pi_{\beta}(\mathrm{d} \theta) := \frac{e^{-\beta u(\theta)}}{\int_{\R^d} e^{-\beta u(\theta)} \, \rmd \theta}\mathrm{d} \theta,
\end{equation}
where we assume $\int_{\R^d} e^{-\beta u(\theta)} \, \rmd \theta <\infty$. Denote by $(\mathcal{G}_n)_{n\in\N_0}$ a given filtration representing the flow of past information, and denote by $\mathcal{G}_{\infty} := \sigma(\bigcup_{n \in \N_0} \mathcal{G}_n)$. Moreover, let $(X_n)_{n\in\N_0}$ be a $(\mathcal{G}_n)$-adapted process such that $(X_n)_{n\in\N_0}$ is a sequence of i.i.d. $\R^m$-valued random variables with probability law $\mathcal{L}(X)$ and let $\left(\xi_n\right)_{n \in \mathbb{N}_0}$ be a sequence of independent standard d-dimensional Gaussian random variables. 

In this paper, we consider the SGHMC algorithm given by 

\begin{align}\label{SGHMC 1}
V_0^{\eta}&:=V_0, \quad V_{n+1}^{\eta}  = V_{n}^{\eta} - \eta\left[\gamma  V_{n}^{\eta}+H(\theta_{n}^{\eta}, X_{n+1})\right]+\sqrt{2 \gamma \eta \beta^{-1}} \xi_{n+1}, \\ \label{SGHMC 2}
\theta_0^{\eta}&:=\theta_0, \quad \theta_{n+1}^{\eta}  = \theta_{n}^{\eta} + \eta V_{n}^{\eta},
\end{align}
where $\eta>0$ is the step-size, $\gamma>0$ is the friction coefficient, $\beta > 0$ is the inverse temperature parameter, and $\mathbb{E}\left[H\left(\theta, X_0\right)\right]=h(\theta)$ for every $\theta \in \mathbb{R}^d$. We assume throughout the paper that the $\R^d$-valued random variables $\theta_0$, $V_0$ (initial condition), $\mathcal{G}_{\infty}$, and $(\xi_{n})_{n\in\N_0}$ are independent. Moreover, we assume that $\theta_0$ and $V_0$ satisfy $\mathbb{E}\left[\left|\theta_0\right|^4\right] + \mathbb{E}\left[\left|V_0\right|^4\right]<\infty$.

\subsection{Assumptions}\label{sec:asm}

\begin{assumption} \label{asm:A3}
    The function $H: \mathbb{R}^d \times \mathbb{R}^m \rightarrow \mathbb{R}^d$ takes the form of 
    \begin{equation*}
    H(\theta, x) = F(\theta, x) + G(\theta, x), 
    \end{equation*}
    where 
    \begin{enumerate}
	\item $F: \mathbb{R}^d \times \mathbb{R}^m \rightarrow \mathbb{R}^d$ satisfies that there exist $L_1, L_2 > 0$ and $\rho > 0$ such that for any $\left(\theta, x\right), \left(\theta^{\prime}, x^{\prime}\right) \in \mathbb{R}^d \times \mathbb{R}^m$,
        $$\lvert F(\theta, x) - F(\theta^{\prime}, x^{\prime})\rvert \leq (1 + \lvert x \rvert + \lvert x^{\prime}\rvert)^{\rho}(L_1\lvert\theta - \theta^{\prime}\rvert + L_2\lvert x - x^{\prime}\rvert).$$
	\item $G: \mathbb{R}^d \times \mathbb{R}^m \rightarrow \mathbb{R}^d$ is measurable and bounded in $\theta$, i.e., there exist a measurable function $\bar{K}_1: \mathbb{R}^m \rightarrow (0, \infty)$ such that for any $\left(\theta, x\right) \in \mathbb{R}^d \times \mathbb{R}^m$, $$\lvert G(\theta, x)\rvert \leq \bar{K}_1(x).$$
    \end{enumerate}
    
\end{assumption}

\begin{remark}\label{rmk:1}
By Assumption~\ref{asm:A3}, we have $$\lvert H(\theta, x)\rvert \leq (1 + \lvert x\rvert) ^{\rho + 1} (L_1\lvert\theta\rvert + L_2) + F_{\ast}(x), $$ where $F_{\ast}(x) :=  2 \lvert \bar{K}_1(0)\rvert + \lvert \bar{K}_1(x)\rvert + \lvert F(0, 0)\rvert$.
\end{remark}

\begin{assumption} \label{asm:A2}
    The process $\left\{X_n\right\}_{n \in \N_0}$ is i.i.d. with $\mathbb{E}\left[\lvert X_0 \rvert^{4(\rho + 1)}\right] < \infty$ and $\mathbb{E}\left[\lvert \bar{K}_1(X_0)\rvert^2\right] < \infty$, where $\rho > 0$ and $\bar{K}_1: \mathbb{R}^m \rightarrow (0, \infty)$ is introduced in Assumption~\ref{asm:A3}.
    
\end{assumption}

\begin{assumption} \label{asm:A4}
    There exists a positive constant $L>0$ such that, for all $\theta, \theta^{\prime} \in \mathbb{R}^d$, $$\mathbb{E}\left[\lvert H(\theta, X_0) - H(\theta^{\prime}, X_0)\rvert\right] \leq L\lvert\theta - \theta^{\prime}\rvert.$$
\end{assumption}

\begin{remark}\label{rmk:2}
By Assumption~\ref{asm:A4}, we have $$\lvert h(\theta) - h(\theta^{\prime})\rvert \leq L\lvert \theta - \theta^{\prime}\rvert, $$ which implies, by leting  $\theta^{\prime} = 0$, that $$\lvert h(\theta)\rvert \leq L\lvert\theta\rvert + \lvert h(0)\rvert.$$
\end{remark}

\begin{assumption} \label{asm:A5}
    There exist measurable functions $A:\mathbb{R}^m \rightarrow \mathbb{R}^{d\times d}$ and $B:\mathbb{R}^m \rightarrow \mathbb{R}$ such that the following hold: for any $\left(\theta, x\right) \in \mathbb{R}^d \times \mathbb{R}^m$,
        \begin{enumerate}[itemsep=5pt,topsep=5pt,parsep=2pt]
	\item $\left\langle \theta, A(x)\theta \right\rangle \geq 0$ and $\left\langle F(\theta, x),\theta \right\rangle \geq \left\langle\theta, A(x)\theta \right\rangle- B\left(x\right)$,
        \item The smallest eigenvalue of $\mathbb{E}\left[\lvert A(X_0)\rvert\right]$ is a positive real number $a > 0$ and $\mathbb{E}\left[\lvert B(X_0)\rvert\right] = b \geq 0$.
    \end{enumerate}
\end{assumption}

\begin{remark}\label{rmk:3}
By Assumption~\ref{asm:A3},~\ref{asm:A2}, and~\ref{asm:A5}, we have $$\langle h(\theta), \theta\rangle \geq a^{\prime}\lvert\theta\rvert^2 - b^{\prime},$$ where $a^{\prime} := \frac{a}{2} , b^{\prime} := b + \frac{1}{2a}\mathbb{E}\left[\lvert \bar{K}_1(X_0)\rvert^2\right]$.
\end{remark}

\begin{remark}\label{rmk:4}
By Remark~\ref{rmk:2} and~\ref{rmk:3}, for any $\theta \in \mathbb{R}^d$, we obtain, 
    $$\langle h(\theta), \theta\rangle \geq 2\lambda (u(\theta) + \frac{\gamma^2}{4} \lvert\theta\rvert^2) - 2A_c / \beta,$$
where $A_c := \frac{\beta}{2} (2\lambda u(0) + 2\lambda L\lvert h(0)\rvert + b^{\prime})$ and $\lambda := \min\left\{\frac{1}{4}, \frac{a^{\prime}}{L + 2L\lvert h(0)\rvert + \frac{\gamma^2}{2}}\right\}$.
\end{remark}

\noindent Detailed proofs of the remarks in this section can be found in Appendix~\ref{Appendix: B}.

\section{Main results and overview}\label{sec:3}
\noindent We first define 
\begin{equation}\label{eta_max}
\eta_{\max} := \min\left\{1,\frac{2}{\gamma}, \frac{\gamma\lambda}{2K_1}, \frac{K_3}{K_2}, \frac{\lambda\gamma}{2\tilde{K}}\right\}> 0,
\end{equation}
where $\gamma$ is introduced in the SGHMC algorithm~\eqref{SGHMC 1}-\eqref{SGHMC 2}, $\lambda$ is defined in Remark~\ref{rmk:4}, the constants $K_1$, $K_2$, and $K_3$ are explicitly given in~\eqref{K_1},~\eqref{K_2}, and~\eqref{K_3} whereas $\tilde{K}$ is explicitly given in~\eqref{tilde_K}. 

Then, under the assumptions presented in Section~\ref{sec:asm}, we obtain the following non-asymptotic upper bound in Wasserstein-2 distance.
\begin{theorem}\label{thm:2.1}
    Let Assumption~\ref{asm:A3}-\ref{asm:A5} hold. Then, for any $\beta > 0$, there exist constants $C_1^{\star}, C_2^{\star}, C_3^{\star}, C_4^{\star} > 0$ such that, for every $n \in \mathbb{N}_0$,  $0 < \eta \leq \eta_{\max}$ with $\eta_{\max}$ defined in~\eqref{eta_max}, we obtain 
$$
W_2\Big(\mathcal{L}\left(\theta_n^\eta, V_n^\eta\right), \bar{\pi}_\beta\Big) \leq C_1^{\star} \eta^{1 / 2}+C_2^{\star} \eta^{1 / 4}+C_3^{\star} e^{-C_4^{\star} \eta n}, 
$$
where $C_1^{\star}, C_2^{\star}, C_3^{\star}$, and $C_4^{\star}$ are made explicit and are summarized in Table~\ref{table 1}. 

In particular, for any $\epsilon > 0$, if we choose $\eta \leq \min\left\{\frac{\epsilon^2}{9 C_1^{\star 2}}, \frac{\epsilon^4}{81 C_2^{\star 4}}, \eta_{\max}\right\}$ and $n \geq \frac{\ln \left(3 C_3^{\star} / \epsilon \right)}{C_4^{\star} \min\left\{\frac{\epsilon^2}{9 C_1^{\star 2}}, \frac{\epsilon^4}{81 C_2^{\star 4}}, \eta_{\max}\right\}}$, then we obtain ${W_2\left(\mathcal{L}\left(\theta_n^\eta, V_n^\eta\right), \bar{\pi}_\beta\right) \leq \epsilon}$.
\end{theorem}
Also, an analogous result in Wasserstein-1 distance can be obtained.
\begin{corollary}\label{cor:2.2}
    Let Assumption~\ref{asm:A3}-\ref{asm:A5} hold. Then, for any $\beta > 0$, there exist constants $C_1^{\ast}, C_2^{\ast}$, $C_3^{\ast} > 0$ such that, for every $n \in \mathbb{N}_0$,  $0 < \eta \leq \eta_{\max}$ with $\eta_{\max}$ defined in~\eqref{eta_max}, we obtain 
$$
W_1\Big(\mathcal{L}\left(\theta_n^\eta, V_n^\eta\right), \bar{\pi}_\beta\Big) \leq C_1^{\ast} \eta^{1 / 2} + C_2^{\ast} e^{-C_3^{\ast} \eta n}, 
$$
where $C_1^{\ast}, C_2^{\ast}$, and $C_3^{\ast}$ are made explicit and are summarized in Table~\ref{table 1}. 

In particular, for any $\epsilon > 0$, if we choose $\eta \leq \min\left\{\frac{\epsilon^2}{4 C_1^{*2}}, \eta_{\max}\right\}$ and $n \geq \max\left\{\frac{4 C_1^{* 2} \ln\left(2 C_2^{\ast} / \epsilon\right)}{\epsilon^2 C_3^*}, \frac{ \ln\left(2 C_2^{\ast} / \epsilon\right)}{C_3^* \eta_{\max}}\right\}$, then we obtain ${W_1\big(\mathcal{L}\left(\bar{\theta}_{n}^\eta, \bar{V}_{n}^\eta\right), \bar{\pi}_\beta\big) \leq \epsilon}$.
\end{corollary}
By using the convergence result in Wasserstein-2 distance as presented in Theorem~\ref{thm:2.1}, we can obtain an upper bound for the expected excess risk $\mathbb{E}[u(\theta_n^\eta)]-\inf _{\theta \in \mathbb{R}^d} u(\theta)$.
\begin{theorem}\label{thm:2.2}
Let Assumption~\ref{asm:A3}-\ref{asm:A5} hold. Then, for any $\beta > 0$, there exist constants $\bar{C}_1^{\star}, \bar{C}_2^{\star}, \bar{C}_3^{\star}, \bar{C}_4^{\star} > 0$ such that, for every $n \in \mathbb{N}_0$, $0 < \eta \leq \eta_{\max}$ with $\eta_{\max}$ defined in~\eqref{eta_max}, we obtain 
$$
\mathbb{E}\left[u\left(\theta_n^\eta\right)\right] - \inf _{\theta \in \mathbb{R}^d} u(\theta) \leq \bar{C}_1^{\star} \eta^{1 / 2}+\bar{C}_2^{\star} \eta^{1 / 4}+\bar{C}_3^{\star} e^{-\bar{C}_4^{\star} \eta n}+\frac{d}{2 \beta} \log \left(\frac{8 e L}{\gamma^2 \lambda(1 - 2\lambda)}\left(\frac{A_c}{d} + 1\right)\right),
$$
where $\bar{C}_1^{\star}, \bar{C}_2^{\star}, \bar{C}_3^{\star}$, and $\bar{C}_4^{\star}$ are given in Table~\ref{table 1}. 

In particular, for any $\epsilon > 0$, if we first choose $\beta \geq \max\left\{\frac{16 d^2}{\epsilon^2}, \frac{4d}{\epsilon} \log\left(\frac{8eL}{\gamma^2 \lambda(1 - 2\lambda)}\left(\frac{\lambda u(0) + \lambda L\lvert h(0)\rvert + b^{\prime} / 2}{d} + 1\right) \right)\right\}$, then choose $\eta \leq \min\left\{\frac{\epsilon^2}{16\bar{C}_1^{\star 2}}, \frac{\epsilon^4}{256\bar{C}_2^{\star 4}}, \eta_{\max}\right\}$ and $n \geq \frac{\ln(4 \bar{C}_3^{\star} / \epsilon)}{\bar{C}_4^{\star}\min\left\{\frac{\epsilon^2}{16\bar{C}_1^{\star 2}}, \frac{\epsilon^4}{256\bar{C}_2^{\star 4}}, \eta_{\max}\right\}}$, then we obtain ${\mathbb{E}\left[u\left(\theta_n^\eta\right)\right] - \inf _{\theta \in \mathbb{R}^d} u(\theta) \leq \epsilon}$.
\end{theorem}
\subsection{Related papers and discussion}\label{sec:3.1}

In this section, we compare our work with the most related ones~\cite{akyildiz2020nonasymptotic, chau2022stochastic, gao2022}.

In~\cite{gao2022}, the authors provide a non-asymptotic upper bound in Wasserstein-2 distance between the law of the SGHMC algorithm~\eqref{SGHMC 1}-\eqref{SGHMC 2} and of the underdamped Langevin SDE~\eqref{underdamped Langevin SDE} being $\mathcal{O}\big((\delta^{1 / 4}+\eta^{1 / 4})\left(n \eta\right)^{3/2}\sqrt{\log (n \eta)} + \left(n \eta\right)\sqrt{\eta}\big)$, where $\delta > 0$ is the gradient noise level independent of the stepsize and $n$ is the number of iterations.~\cite{chau2022stochastic} improved the upper bound for SGHMC in~\cite{gao2022} significantly to an order of $\mathcal{O}\left(\delta^{1 / 4}+\eta^{1 / 4}\right)$ in the sense that it does not depend on the number of iterations. The results in~\cite{chau2022stochastic} and~\cite{gao2022} are obtained under the same set of assumptions~\cite[Assumption i-v]{gao2022}.~\cite[Assumption ii]{gao2022} assumes that $H$ is (global) Lipschitz continuous in $\theta$ uniformly with respect to $x$. In our case, in Assumption~\ref{asm:A3}, we split $H$ into two parts, i.e., $H = F + G$, where $F$ is jointly locally Lipschitz continuous and $G$ is only required to be measurable, without assuming any continuity assumption. In particular, we do not impose $H$ to be continuous. Instead, we only assume $H$ to satisfy a continuity \textit{in average} condition so as to allow discontinuous $H$ in Assumption~\ref{asm:A4}. By allowing $H$ to be discontinuous, we can cover important optimization problems involving neural networks with ReLU activation function as well as quantile estimation. 
\cite[Assumption iii]{gao2022} assume the dissipativity assumption (which comes from the theory of dynamical systems, see~\cite{stuart1998dynamical}) whereas we impose a weaker local dissipativity condition in Assumption~\ref{asm:A5} to allow the dependence on the data stream. In addition,~\cite[Assumption v]{gao2022} requires the finiteness of an exponential moment of the initial value, while we require only polynomial moments of the initial value in Assumption~\ref{asm:A2}. 

We highlight that, under~\cite[Assumption i-v]{chau2022stochastic, gao2022}, the authors obtain non-asymptotic upper bounds that cannot be made to vanish by selecting an arbitrarily small step size $\eta > 0$, since $\delta > 0$ is predefined and independent of $\eta$. However, under our Assumptions~\ref{asm:A3}-\ref{asm:A5}, our non-asymptotic convergence bounds can be made arbitrarily small by choosing appropriate parameters.

We now compare our work with~\cite{akyildiz2020nonasymptotic}.~\cite{akyildiz2020nonasymptotic} provides convergence results in Wasserstein-2 distance under the assumptions~\cite[Assumption 2.1-2.4]{akyildiz2020nonasymptotic}. In~\cite[Assumption 2.1, 2.2]{akyildiz2020nonasymptotic}, the authors assume that the stochastic gradients are unbiased with finite moments which are the same as our Assumption~\ref{asm:A2}.~\cite[Assumption 2.3]{akyildiz2020nonasymptotic} states a local Lipschitz condition of $H$ in $x$ and a global Lipschitz condition in $\theta$, while we split $H$ into two parts $F$ and $G$ in Assumption~\ref{asm:A3}. More precisely, we assume $F$ being jointly local Lipschitz continuous and $G$ only being bounded in $\theta$ without requiring it to be continuous. Furthermore, we only assume a continuity \textit{in average} condition of $H$ in Assumption~\ref{asm:A4}, which allows for discontinuous $H$. In~\cite[Assumption 2.4]{akyildiz2020nonasymptotic}, authors assume a local dissipativity condition whereas we have similar assumption in Assumption~\ref{asm:A5}. Under our Assumptions~\ref{asm:A3}-\ref{asm:A5}, we obtain similar convergence results as those of~\cite{akyildiz2020nonasymptotic} with the same rates of convergence in Wasserstein distance. However, we highlight that our results are applicable to optimization problems with discontinuous stochastic gradients, which can not be covered by the results in~\cite{akyildiz2020nonasymptotic}.

Moreover, we comment on our Assumption~\ref{asm:A4}. Assumption~\ref{asm:A4} is similar to~\cite[Eq.\ (6)]{chau2019fixed} and~\cite{fort2016convergence, sabanis2020fully}, which allows for discontinuous $H$ for the corresponding stochastic gradient descent (SGD) and SGLD algorithms. However, we note that the analysis of the SGHMC algorithm involves studying the underdamped Langevin diffusion and its associated discrete time scheme on $\mathbb{R}^{2d}$, which necessitates the use of different techniques compared to those employed in~\cite{fort2016convergence, sabanis2020fully}. 

Furthermore, from a technical point of view, we adopt a splitting procedure similar to that of \cite{akyildiz2020nonasymptotic, chau2022stochastic} in order to upper bound the Wasserstein distance between the distribution of the SGHMC algorithm and the unique invariant measure $\bar{\pi}_\beta$ of SDE \eqref{SDE2}, see \eqref{sec4:1}. However, due to the different setting and assumptions compared to \cite{akyildiz2020nonasymptotic, chau2022stochastic}, we develop different techniques to upper bound the terms in the splitting \eqref{sec4:1}.

\begin{landscape}
\begin{table}[htbp]
\renewcommand{\arraystretch}{2}
\centering
\begin{threeparttable}
\scriptsize
\begin{tabular}{@{}ccll@{}} 
\toprule 
\multicolumn{2}{c}{Constant}  &\multicolumn{1}{c}{Full expression} &\multicolumn{1}{c}{Dependence on dimension}\\ 
\midrule
	\multirow{7}{*}{Theorem~\ref{thm:2.1}}			&\multirow{3}{*}{$C_1^{\star}$}		& $\bigg(4 e^{4 \gamma^2} \left(3 \gamma^2 C_v+6 \mathbb{E}\left[\left(1+\left|X_0\right|\right)^{2(\rho+1)}\right]\left(C_\theta L_1^2+2 L_2^2\right)+12 \mathbb{E}\left[F_*^2\left(X_0\right)\right]+6 \gamma \beta^{-1} d\right)$										& \multirow{3}{*}{$\mathcal{O}\left(d^{1 / 2}\right)$}\\
      																	&											&$ +\left(4 + 2 4 e^{4 \gamma^2}\right) \bigg(8 L_2^2\left(\mathbb{E}\left[\left(1+\left|X_0\right|+\left|\mathbb{E}\left[X_0\right]\right|\right)^{2 \rho}\left|X_0-\mathbb{E}\left[X_0\right]\right|^2\right]\right)$		& \\& &$ + 16\left(\mathbb{E}\left[\bar{K}_1^2\left(X_0\right)+\bar{K}_1^2\left(\mathbb{E}\left[X_0\right]\right)\right]\right)\bigg) \bigg)^{1 / 2} e^{4 e^{4 \gamma^2} L^2 / 2+\gamma^2}$		& \\ \cline{2-4}
                       
                                                        &\multirow{2}{*}{$C_2^{\star}$}		& $ \bigg(2^{1 / 2} e^{1+\Lambda_c / 2} \frac{1+\gamma}{\min \left\{1, \alpha_c\right\}}\left(\max \left\{1,4 \frac{1}{\min \left\{1, R_1\right\}}\left(1+2 \alpha_c+2 \alpha_c^2\right)\left(d+A_c\right) \beta^{-1} \gamma^{-1} c_2^{-1}\right\}\right)^{1 / 2}\bigg)$										& \multirow{2}{*}{$\mathcal{O}\left(d^{5 / 4} e^{\mathcal{O}(d)}\right)$}\\
      																	&											&$ \times \sqrt{3 \max \left\{1+\alpha_c, \gamma^{-1}\right\}}\left(1 - e^{ - C_4^{\star}\eta_{\max} / 2}\right)^{-1} \sqrt{C_1^{\star}\left(1+\varepsilon_c C_{\mathcal{V}} + \varepsilon_c C_{\mathcal{V}}^{\prime}\right)}$		& \\ \cline{2-4}
                       
                                                        &$C_3^{\star}$		& $ \frac{\sqrt{2}(1+\gamma) e^{1+\Lambda_c / 2}}{\min \left\{1, \alpha_c\right\}}\left(\max \left\{1,4 \frac{\left(1+2 \alpha+2 \alpha_c^2\right)\left(d+A_c\right)}{\min \left\{1, R_1\right\} \beta \gamma \dot{c}}\right\}\right)^{1 / 2} \sqrt{\mathcal{W}_\rho\left(\mu_0, \bar{\pi}_\beta\right)}$		& $\mathcal{O}\left(e^{\mathcal{O}(d)}\right)$  \\ \cline{2-4}
                                                        
                                                        &$C_4^{\star}$		& $ \frac{\gamma}{768} \min \left\{\lambda L \gamma^{-2}, \Lambda_c^{1 / 2} e^{-\Lambda_c} L \gamma^{-2}, \Lambda_c^{1 / 2} e^{-\Lambda_c}\right\}$  & $\mathcal{O}\left(d^{1 / 2} e^{-\mathcal{O}(d)}\right)$  \\ \hline
                       
	\multirow{4}{*}{Corollary~\ref{cor:2.2}}		 &\multirow{2}{*}{$C_1^{\ast}$}		& $ C_1^{\star} + 6 C_1^{\star} e^{2+\Lambda_c} \frac{1+\gamma}{\min \left\{1, \alpha_c\right\}} \max \left\{1,4 \frac{\max \left\{1, R_1^{-1}\right\}}{\min \left\{1, R_1\right\}}\left(1+2 \alpha_c+2 \alpha_c^2\right)\left(d+A_c\right) \beta^{-1} \gamma^{-1} c_1^{-1}\right\}$										& \multirow{2}{*}{$\mathcal{O}\left(d^{1 / 2}e^{\mathcal{O}(d)}\right)$}\\										&											&$ \times \max \left\{1+\alpha_c, \gamma^{-1}\right\}\left(1-e^{- \left(\frac{\gamma}{384} \min \left\{\lambda_c L \gamma^{-2}, \Lambda_c^{1 / 2} e^{-\Lambda_c} L \gamma^{-2}, \Lambda_c^{1 / 2} e^{-\Lambda_c}\right\}\right) \eta_{\max } / 2}\right)^{-1} \left(1+\varepsilon_c\left(C_{\mathcal{V}}+C_{\mathcal{V}}^{\prime}\right)\right)$		& \\ \cline{2-4}

																&$ C_2^{\ast}$									&$\bigg(2 e^{2+\Lambda_c} \frac{1+\gamma}{\min \left\{1, \alpha_c\right\}} \max \left\{1,4 \frac{\max \left\{1, R_1^{-1}\right\}}{\min \left\{1, R_1\right\}}\left(1+2 \alpha_c+2 \alpha_c^2\right)\left(d+A_c\right) \beta^{-1} \gamma^{-1} c_1^{-1}\right\}\bigg) \mathcal{W}_\rho\left(\mu_0, \bar{\pi}_\beta\right)$&
																	$\mathcal{O}\left(e^{\mathcal{O}(d)}\right)$\\\cline{2-4}
                 
																&$ C_3^{\ast}$									&$\frac{\gamma}{384} \min \left\{\lambda_c L \gamma^{-2}, \Lambda_c^{1 / 2} e^{-\Lambda_c} L \gamma^{-2}, \Lambda_c^{1 / 2} e^{-\Lambda_c}\right\}$				&
																	$\mathcal{O}\left(d^{1 / 2} e^{-\mathcal{O}(d)}\right)$\\\hline
                 
	\multirow{4}{*}{Theorem~\ref{thm:2.2}}			        &$\bar{C}_1^{\star}$		& $ \left(L \max \left\{C_\theta^c, C_\theta\right\}+|h(0)|\right) C_1^{\star}$		& $\mathcal{O}\left(d^{3 / 2}\right)$  \\ \cline{2-4}
                                                                &$\bar{C}_2^{\star}$		& $ \left(L \max \left\{C_\theta^c, C_\theta\right\}+|h(0)|\right)C_2^{\star}$		& $\mathcal{O}\left(d^{9 / 4} e^{\mathcal{O}(d)}\right)$  \\ \cline{2-4}
                                                                &$\bar{C}_3^{\star}$		& $ \left(L \max \left\{C_\theta^c, C_\theta\right\}+|h(0)|\right)C_3^{\star}$		& $\mathcal{O}\left(e^{\mathcal{O}(d)}\right)$  \\ \cline{2-4}
                                                                &$\bar{C}_4^{\star}$		& $ \frac{\gamma}{768} \min \left\{\lambda L \gamma^{-2}, \Lambda_c^{1 / 2} e^{-\Lambda_c} L \gamma^{-2}, \Lambda_c^{1 / 2} e^{-\Lambda_c}\right\}$		& $\mathcal{O}\left(d^{1 / 2} e^{-\mathcal{O}(d)}\right)$  \\ \bottomrule

    \end{tabular}
    \end{threeparttable}
\caption{Explicit expressions of the constants in Theorem~\ref{thm:2.1}, Corollary~\ref{cor:2.2}, and Theorem~\ref{thm:2.2}}
\label{table 1}
\end{table}
\end{landscape}

\section{Applications}\label{sec:6}
\noindent In this section, we apply the SGHMC algorithm together with our theoretical results to several problems relevant in practice, including quantile estimation and regularized optimization problems involving neural networks in order to illustrate the convergence properties and wide applicability of the SGHMC algorithm for optimization problems with discontinuous stochastic gradient. In Section~\ref{sec:6.1}, we apply the SGHMC algorithm to quantile estimation, covering the Gaussian distribution, logistic distribution, and Gumbel distribution. Then, in Section~\ref{sec:6.2}, we apply the SGHMC algorithm to train ReLU neural networks in order to solve certain regularized optimization problems. First, in Section~\ref{transfer_learning}, we provide an example involving the use of a two-hidden-layer feed forward neural network in the transfer learning setting, where (part of) its weight and bias parameters are sourced from a pre-trained model. Transfer learning is the improvement of learning in a new task through the transfer of knowledge from a related task that has already been learned, which is widely employed especially in the network-based tasks, see, e.g.,~\cite{goodfellow2016deep} and references therein. Moreover, in Section~\ref{sec:6.3}, we apply the SGHMC algorithm to solve the problem of hedging under asymmetric risk, where the SGHMC is used to train the neural networks. Datasets are generated from the discrete-time version of the multidimensional Black-Scholes-Merton (BSM) model. We further conduct numerical experiments on regression and classification problems on real-world datasets, see Section~\ref{sec:6.2.2}. Crucially, we show that the quantile estimation and optimization problem in the transfer learning setting satisfy our Assumptions~\ref{asm:A3}-\ref{asm:A5}, hence Theorem~\ref{thm:2.2} provides a theoretical guarantee for SGHMC to solve the optimization problems under consideration. Furthermore, by providing comparisons of SGHMC with other optimizers including SGLD, TUSLA, ADAM, AMSGrad, and RMSProp, we demonstrate numerically the superior performance of the SGHMC algorithm. The Python code for all the experiments in this paper is available at~\href{https://github.com/LuxLiang/SGHMC\_discontinuous}{https://github.com/LuxLiang/SGHMC\_discontinuous}.
\subsection{Quantile Estimation with \texorpdfstring{$L^2$}~-Regularization}\label{sec:6.1}
Quantiles are frequently used to assess risk in a wide range of application areas, such as finance, nuclear engineering, and service industries (see~\cite{austin2003quantile, dong2020tutorial, jorion2007value, wu2005bahadur} and references therein). In this example, our goal is to identify the quantile $q\in\left(0, 1\right)$ of a given random variable under certain distributions, which is also studied in~\cite{bardou2009computing, egloff2010quantile, fort2016convergence}. To this end, we consider the following $L^2$-regularized optimization problem:
\begin{equation}\label{op:qe}
\operatorname{minimize} \quad\mathbb{R}^d \ni \theta \mapsto u(\theta):=\mathbb{E}\left[l_{q}(X-\theta)\right]+ \lambda_r \lvert\theta\rvert^{2},
\end{equation}
where $X$ is the input random variable, $\theta \in \mathbb{R}^{d}$ is the parameter to be optimized, $\lambda_r$ is a positive regularization constant, and 
$$
l_{q}(s) := (q - \mathbbm{1}_{\{s < 0\}}) s, \quad s \in \mathbb{R},
$$
with $0 < q < 1$. Then, for any $\left(\theta, x \right) \in \mathbb{R}\times\mathbb{R}$, we define $H(\theta, x):=G(\theta, x)+F(\theta, x)$ with
\begin{equation}\label{SG_exp1}
F(\theta, x):= 2\lambda_r \theta, \quad G(\theta, x):= -q + \mathbbm{1}_{\{x < \theta\}}.
\end{equation}
\begin{proposition}\label{prop:6.1}
Assume that $X$ has bounded density with respect to the Lebesgue measure on $\mathbb{R}$ and has finite fourth moment. Then, the optimization problem~\eqref{op:qe} with its associated stochastic gradient defined in~\eqref{SG_exp1} satisfies Assumptions~\ref{asm:A3}-\ref{asm:A5}.
\end{proposition}
\begin{proof}
See Appendix~\ref{pf_prop:7.1}.
\end{proof}
\noindent Therefore, by using Theorem~\ref{thm:2.2}, we can, for any given precision level $\epsilon > 0$, construct an estimator using the SGHMC algorithm, which minimizes the corresponding expected excess risk.

For the numerical experiments, we estimate quantiles for several target distributions, including Guassian distribution $\mathcal{N}\left(\mu, \sigma\right)$, Logistic distribution $L \left(\alpha, \beta\right)$, and Gumbel distribution $\mathcal{G}\left(\mu, \beta\right)$ with respect to corresponding parameters $\left(\mu, \sigma\right)$, $\left(\alpha, \beta\right)$, and $\left(\mu, \beta\right)$. We note that the true $q$-quantile values for Logistic and Gumbel distributions are $\alpha + \beta \log \frac{q}{(1-q)}$ and $\mu-\beta \ln (-\ln q)$, respectively, which we use to obtain the optimal values for the comparison with our numerical results. These true values are recorded as $Q^{\ast}$ in Table~\ref{table2}. For initialization, we set $\theta_0\sim \mathcal{N}\left(0, 1\right), v_0 \sim \mathcal{N}\left(0, 1\right)$, $\beta=10^{10}$, $\gamma=0.5$, $\lambda_r=10^{-5}$, $\eta=10^{-3}$ (satisfying the step restriction given in~\eqref{eta_max}). The experiments are run using $5$ different seeds that control the randomness of initialization and sampling. In Figure~\ref{excess_risk}, by taking Logistic distribution ($\alpha=0, \beta=1$) as an example, solid lines in the left graph show the mean of the path of $\mathbb{E}\left[u(\theta_n)\right] - \inf _{\theta \in \mathbb{R}^d} u(\theta)$ with $q = 0.95$, while the shaded regions indicate the range over the random seeds. Moreover, the reference line in the right graph has a slope of 0.5 and the dashed line shows that the rate of convergence of the SGHMC algorithm is 0.594, which aligns with our theoretical results.  One notes that the samples from $\pi_{\beta}$ is generated by running the SGHMC algorithm with $\eta = 10^{-5}$. In Table~\ref{table2}, to compare the performance between SGHMC and SGLD, we denote by $Q_{\text{SGHMC}}$ and $T_{\text{SGHMC}}$ the numerical approximations of the values of the corresponding optimization problem~\eqref{op:qe} using the SGHMC algorithm~\eqref{SGHMC 1} and~\eqref{SGHMC 2} and their running time\footnote{The running time is computed as follows. We first select the random seed and calculate $\mathbb{E}\left[u(\theta_n)\right] - \inf _{\theta \in \mathbb{R}^d} u(\theta)$. Once $\mathbb{E}\left[u(\theta_n)\right] - \inf _{\theta \in \mathbb{R}^d} u(\theta)$ is less than $\epsilon=0.0001$, we record the total time that the algorithm runs. Finally, we average all the results over different seeds.}, respectively, whereas $Q_{\text{SGLD}}$ and $T_{\text{SGLD}}$ are the corresponding values and running time for the SGLD algorithm. Each approximation and running time in the table are obtained based on 5 different seeds and at most $10^{7}$ iterations followed by its mean squared error shown in brackets. We notice that, compared to the SGLD algorithm, it requires less time for the SGHMC algorithm to converge. In addition, the results obtained by SGHMC have lower mean square error.

\begin{figure}[t]
	\centering
	\begin{minipage}{0.49\linewidth}
		\centering
		\includegraphics[width=1\linewidth]{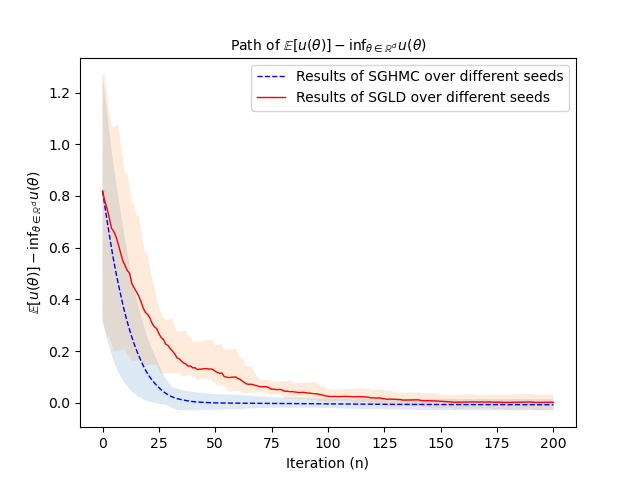}
	\end{minipage} 
	\begin{minipage}{0.49\linewidth}
		\centering
		\includegraphics[width=1\linewidth]{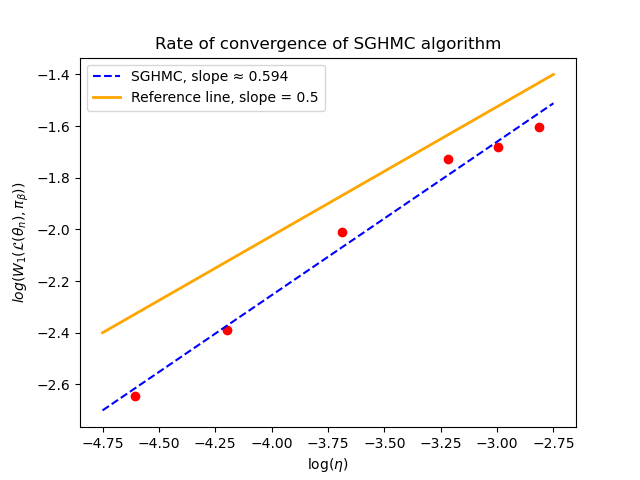}
	\end{minipage}
 \caption{Expected excess risk of the optimization problem~\eqref{op:qe} by using SGHMC and SGLD algorithm (left) and the rate of convergence of the SGHMC algorithm based on 100000 samples (right).}
 \label{excess_risk}
\end{figure}

\begin{table}[t]
\footnotesize
\centering
\scalebox{0.93}{
\begin{tabular}{c|ccccc|ccccc}
 \hline
 \hline
& & &$q = 0.95$ & &  && $q = 0.99$&&\\ 
 \cline{2-11}
  & $Q^*$ &$Q_{\text{SGHMC}}$ &$Q_{\text{SGLD}}$& $T_{\text{SGHMC}}$&$T_{\text{SGLD}}$  & $Q^*$  &${Q}_{\text{SGHMC}}$ &${Q}_{\text{SGLD}} $&${T}_{\text{SGHMC}}$&${T}_{\text{SGLD}}$  \\
 \hline
 \multirow{2}{*}{ $ \mathcal{N}\left(\mu = -1, \sigma = 1\right)$ }& \multirow{2}{*}{0.645} & 
 0.654&0.639&\multirow{2}{*}{{0.079s}} & \multirow{2}{*}{0.111s} & \multirow{2}{*}{1.326}&1.311 &1.265&\multirow{2}{*}{{0.078s}}&\multirow{2}{*}{0.372s}\\
 &  &(0.007) &(0.010)& & & &(0.010) &(0.018)&  \\

 \multirow{2}{*}{ $ \mathcal{N}\left(\mu = 1, \sigma = 2\right)$ }& \multirow{2}{*}{4.290} & 
 4.329&4.235&\multirow{2}{*}{{0.695s}} & \multirow{2}{*}{1.848s} & \multirow{2}{*}{5.653}&5.529&5.525&\multirow{2}{*}{{1.830s}}&\multirow{2}{*}{5.011s}\\
 & &(0.050)&(0.060)& &&&(0.125) &(0.240)&  \\
 \multirow{2}{*}{ $ \mathcal{N}\left(\mu =3, \sigma = 5\right)$ }& \multirow{2}{*}{11.224} & 
 11.137&11.127&\multirow{2}{*}{{2.209s}} & \multirow{2}{*}{5.316s} & \multirow{2}{*}{14.632}&14.379&14.379&\multirow{2}{*}{{6.828s}}&\multirow{2}{*}{15.769s}\\
 & &(0.730)&(1.300)& &&&(1.960) &(3.000)&  \\
 \hline 
 \multirow{2}{*}{ $ L\left(\alpha = 0, \beta = 1\right)$ }& \multirow{2}{*}{2.944} & 
 2.969&2.883&\multirow{2}{*}{{0.584s}} & \multirow{2}{*}{1.817s} & \multirow{2}{*}{4.595}&4.454&4.455&\multirow{2}{*}{{2.497s}}&\multirow{2}{*}{6.134s}\\
  &  &(0.026)&(0.040)& &&&(0.107) &(0.236)&  \\

 \multirow{2}{*}{ $ L\left(\alpha = -1, \beta = 1\right)$ }& \multirow{2}{*}{1.944} & 
 1.907&1.766&\multirow{2}{*}{{0.432s}} & \multirow{2}{*}{1.766s} & \multirow{2}{*}{3.595}&3.454&3.455&\multirow{2}{*}{{2.615s}}&\multirow{2}{*}{6.160s}\\
  &  &(0.016)&(0.024)& &&&(0.085) &(0.176)&  \\
 \multirow{2}{*}{ $ L\left(\alpha =-3, \beta = 3\right)$ }& \multirow{2}{*}{5.833} & 
 5.742&5.717&\multirow{2}{*}{{26.656s}} & \multirow{2}{*}{56.744s} & \multirow{2}{*}{10.785}&10.463&10.466&\multirow{2}{*}{{86.582s}}&\multirow{2}{*}{210.637s}\\
  &  &(0.314)&(0.590)& &&&(1.890) &(5.240)&  \\
 \hline 
  \multirow{2}{*}{ $ \mathcal{G}\left(\mu = 0, \beta = 1\right)$ }& \multirow{2}{*}{2.970} & 
 3.005&2.912&\multirow{2}{*}{{0.815s}} & \multirow{2}{*}{2.048s} & \multirow{2}{*}{4.600}&4.456&4.451&\multirow{2}{*}{{2.690s}}&\multirow{2}{*}{6.234s}\\
  &  &(0.023)&(0.030)& &&&(0.100) &(0.210)&  \\

 \multirow{2}{*}{ $ \mathcal{G}\left(\mu = 0, \beta = 2\right)$ }& \multirow{2}{*}{5.940} & 
 5.859&5.850&\multirow{2}{*}{{18.831s}} & \multirow{2}{*}{37.833s} & \multirow{2}{*}{9.200}&8.969&8.971&\multirow{2}{*}{{53.861s}}&\multirow{2}{*}{130.092s}\\
  &  &(0.120)&(2.100)& &&&(0.827) &(1.800)&  \\
 \multirow{2}{*}{ $ \mathcal{G}\left(\mu =1, \beta = 2\right)$ }& \multirow{2}{*}{6.940} & 
 6.859&6.851&\multirow{2}{*}{{16.806s}} & \multirow{2}{*}{36.883s} & \multirow{2}{*}{10.200}&9.966&9.966&\multirow{2}{*}{{55.363s}}&\multirow{2}{*}{131.361s}\\
  &  &(0.005)&(0.07)& &&&(0.542) &(0.960)&  \\
 \hline
 \hline 
\end{tabular}}
\caption{Quantile Estimation for several distributions.}
\label{table2}
\end{table}

\subsection{Solving regularized optimization problems using neural network}\label{sec:6.2}
In this section, we consider optimization problems involving neural networks, which includes transfer learning, hedging under asymmetric risk as well as regression and classification problems on real world datasets. 
\subsubsection{Transfer Learning}\label{transfer_learning}
In transfer learning and multi-task learning, neural networks with pre-trained parameters are employed to reduce computational costs (see, e.g.,~\cite{minar2018recent} and references therein). 

In this subsection, we analyze the following optimization problem using neural networks in a transfer learning setup. Let $d_1, d_2, m_1, m_2 \in \mathbb{N}$, and let $\mathfrak{N}=\left(\mathfrak{N}^1, \ldots, \mathfrak{N}^{m_2}\right): \mathbb{R}^d \times \mathbb{R}^{m_1} \rightarrow \mathbb{R}^{m_2}$ be the two-hidden-layer feed forward neural network (TLFN) with its $j$-th element given by
\begin{equation}\label{2LFN}
\mathfrak{N}^{j}(\theta, z):=\sum_{n=1}^{d_2} W_2^{j n} \sigma_2\big(\sum_{k=1}^{d_1}\big[f\big(W_1^{n k}\big) \sigma_1\big(\big\langle W_0^{k \cdot}, z\big\rangle+f(b_0^k)\big)\big]+b_1^n\big),\ j \in \lbrace 1, \cdots, m_2\rbrace,
\end{equation}
where $z=\left(z^1, \ldots, z^{m_1}\right) \in \mathbb{R}^{m_1}$ is the input vector, $\sigma_1: \mathbb{R} \rightarrow \mathbb{R}$ is the ReLU activation function, i.e., $\sigma_1(v):= \max\{0, v\}$, and $\sigma_2: \mathbb{R} \rightarrow \mathbb{R}$ is the sigmoid activation function, i.e., $\sigma_2(v):= 1 / \left(1 + e^{-v} \right)$, $
W_0 \in \mathbb{R}^{d_1 \times m_1}$ and $W_2 \in \mathbb{R}^{m_2 \times d_2}
$ are \textit{fixed} weight matrices\footnote{Later in the application resolved in section~\ref{Simulated_data}, $W_0$ and $W_2$ will correspond to the fixed parameters which were transferred from a previously solved optimization problem.}, $b_0 \in \mathbb{R}^{d_1}$ and $b_1 \in \mathbb{R}^{d_2}$, and $ W_1 \in \mathbb{R}^{d_2 \times d_1}$ are parameters over which we optimize, $f: \mathbb{R} \rightarrow \mathbb{R}$ is given by $f(v) := c \tanh \left(v / c\right)$ with some\footnote{We note that it makes sense to apply $f$ to $b_0$ and $W_1$ from a practical perspective, see~\cite{meyer2023reachability, rao1998function}, and as a consequence, the optimization problem~\eqref{op:nn} can be shown to satisfy our Assumptions~\ref{asm:A3}-\ref{asm:A5}, see Proposition~\ref{prop:6.2}.} $c > 0$. Then, the parameter of the neural network~\eqref{2LFN} which needs to be trained is given by
\begin{equation}\label{theta}
\theta=\big(\big[W_1\big], b_0, b_1\big) \in \mathbb{R}^d
\end{equation}
with $d := d_1 + d_2 + d_1 d_2$, where we denote by $\left[ W_1 \right]$ the vector of all elements in $W_1$. Moreover, we denote by $c_{W_0}:=\max\limits_{i, j} \left\{W_0^{i j} \right\}$ and $c_{W_2}:=\max\limits_{i, j}\left\{W_2^{i j}\right\}$, and assume that at least one element in each row of $W_0 \in \mathbb{R}^{d_1 \times m_1}$ is nonzero, i.e., for each $K=1, \ldots, d_1$, there exists $I=1, \ldots, m_1$ such that $W_0^{K I} \neq 0$.

Our goal is to address the following optimization problem:
\begin{equation}\label{op:nn}
\operatorname{minimize} \quad\mathbb{R}^d \ni \theta \mapsto u(\theta):=\mathbb{E}\left[\left\lvert Y-\mathfrak{N}(\theta, Z)\right\rvert^2\right] + \lambda_r \left|\theta\right|^{2},
\end{equation}
where $Z$ is an $\mathbb{R}^{m_1}$-valued input random variable, $Y$ is the corresponding $\mathbb{R}^{m_2}$-valued target random variable, $\theta \in \mathbb{R}^{d}$ is the parameter defined in~\eqref{theta}, and $\lambda_r > 0$ is the regularization constant.
\begin{proposition}\label{prop:6.2}
Let $X = (Y, Z) \in \mathbb{R}^{m_2} \times \mathbb{R}^{m_1}$ with $m := m_1 + m_2$ be a continuously distributed random variable. Assume that, for any $I=1, \ldots, m_1, J=1, \ldots, m_2$, the density function of $Z^I$ given $Z^1, \ldots, Z^{I-1}, Z^{I+1}, \ldots, Z^{m_1}, Y^J$ denoted by
$$
f_{Z^I \mid Z^1, \ldots, Z^{I-1}, Z^{I+1}, \ldots, Z^{m_1}, Y^J}: \mathbb{R} \rightarrow[0, \infty)
$$
satisfies, for any $I = 1, \ldots, m_1, J=1, \ldots, m_2$, that there exist constants $C_{Z^I}, \bar{C}_{Z^I}>0$, such that, for any $z=\left(z^1, \ldots, z^{m_1}\right) \in$ $\mathbb{R}^{m_1}, y^J \in \mathbb{R}$,
\begin{equation}\label{prop:6.2:c}
\begin{aligned}
\left|z^I\right|^2 f_{Z^I \mid Z^1, \ldots, Z^{I-1}, Z^{I+1}, \ldots, Z^{m_1, Y^J}}\left(z^I \mid z^1, \ldots, z^{I-1}, z^{I+1}, \ldots, z^{m_1}, y^J\right) &\leq \bar{C}_{Z^I},\\
f_{Z^I \mid Z^1, \ldots, Z^{I-1}, Z^{I+1}, \ldots, Z^{m_1}, Y^J}\left(z^I \mid z^1, \ldots, z^{I-1}, z^{I+1}, \ldots, z^{m_1}, y^J\right) &\leq C_{Z^I}.
\end{aligned}
\end{equation}
Let $F, G:\mathbb{R}^d \times \mathbb{R}^m \rightarrow \mathbb{R}^d$ be defined, for all $\left(\theta, x\right) \in \mathbb{R}^d \times \mathbb{R}^m$ with $x=(y, z)$, $y \in \mathbb{R}^{m_2}, z \in \mathbb{R}^{m_1}$, by
\begin{equation}\label{h=f+g}
F(\theta, x):= 2 \lambda_r \theta, \quad G(\theta, x):=\left(G_{W_1^{1 1}}, \cdots, G_{W_1^{d_2 d_1}}, G_{b_0^1}(\theta, x), \cdots, G_{b_0^{d_1}}(\theta, x), G_{b_1^1}(\theta, x), \ldots, G_{b_1^{d_2}}(\theta, x)\right),
\end{equation}
where for every $K=1, \ldots, d_1, N=1, \ldots, d_2$,
\begin{equation}\label{application_1_f_g}
\begin{aligned}
G_{W_1^{N K}}(\theta, x) &:=-2 \sum_{j=1}^{m_2}\big(y^j-\mathfrak{N}^j(\theta, z)\big) W_2^{j N} \sigma_2\left(\sum_{k=1}^{d_1}\left[f(W_1^{N k}) \sigma_1\left(\left\langle W_0^{k\cdot}, z\right\rangle+f(b_0^k)\right)\right]+b_1^N\right) f^{\prime}(W_1^{N K})
\\ & \quad \times \left(1 - \sigma_2\left(\sum_{k=1}^{d_1}\left[f(W_1^{N k})\sigma_1\left(\left\langle W_0^{k \cdot}, z\right\rangle+f(b_0^k)\right)\right]+b_1^N\right)\right)\sigma_1\left(\left\langle W_0^{K\cdot}, z\right\rangle+f(b_0^K)\right),
\end{aligned}
\end{equation}
$$
\begin{aligned}
G_{b_0^{K}}(\theta, x) & := -2 \sum_{j=1}^{m_2}\big(y^j-\mathfrak{N}^j(\theta, z)\big)\sum_{n=1}^{d_2} W_2^{j n} \sigma_2\left(\sum_{k=1}^{d_1}\left[f(W_1^{n k}) \sigma_1\left(\left\langle W_0^{k \cdot}, z\right\rangle+f(b_0^k)\right)\right]+b_1^n\right)\\ & \quad \times \left(1 - \sigma_2\left(\sum_{k=1}^{d_1}\left[f(W_1^{n k}) \sigma_1\left(\left\langle W_0^{k \cdot}, z\right\rangle+f(b_0^k)\right)\right]+b_1^n\right)\right) f(W_1^{n K}) f^{\prime}(b_0^K)\mathbbm{1}_{\{\langle W_0^{K \cdot }, z\rangle + f(b_0^{K}) > 0\}},
\end{aligned}
$$
$$
\begin{aligned}
G_{b_1^N}(\theta, x) &:=-2 \sum_{j=1}^{m_2}\left(y^j-\mathfrak{N}^j(\theta, z)\right) W_2^{j N} \sigma_2\left(\sum_{k=1}^{d_1}\left[f(W_1^{N k}) \sigma_1\left(\left\langle W_0^{K\cdot}, z\right\rangle+f(b_0^k)\right)\right]+b_1^N\right)
\\ & \quad \times \left(1 - \sigma_2\left(\sum_{k=1}^{d_1}\left[f(W_1^{N k})\sigma_1\left(\left\langle W_0^{k \cdot}, z\right\rangle+f(b_0^k)\right)\right]+b_1^N\right)\right).
\end{aligned}
$$
Then, the optimization problem~\eqref{op:nn} satisfies Assumptions~\ref{asm:A3}-\ref{asm:A5} with corresponding stochastic gradient $H:=F + G$.
\end{proposition}
\begin{proof}
    See Appendix~\ref{pf_prop:7.2}.
\end{proof}
\begin{figure}[t]
	\centering
	\begin{minipage}{0.49\linewidth}
		\centering
		\includegraphics[width=1.15\linewidth]{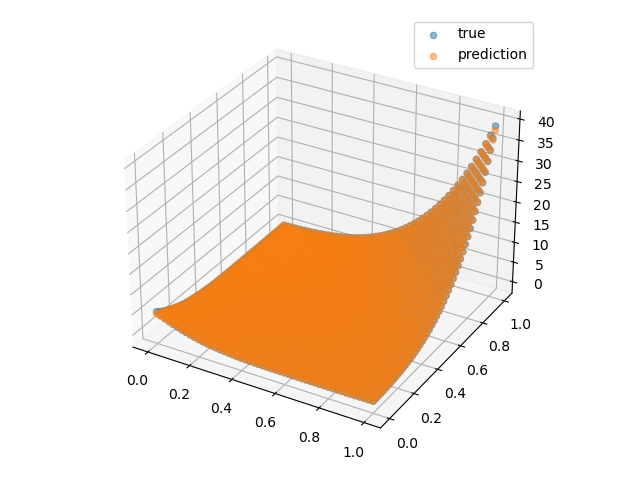}
	\end{minipage} 
	\begin{minipage}{0.49\linewidth}
		\centering
		\includegraphics[width=1.15\linewidth]{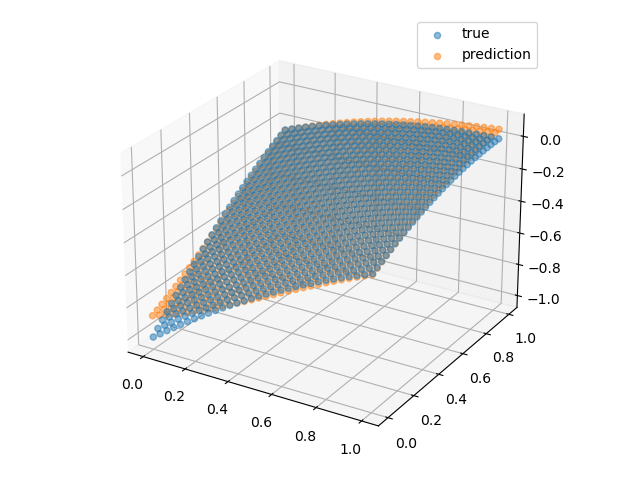}
	\end{minipage}
 \caption{True and prediction values for original learning task involving~\eqref{3LFN_eg} (left) and new learning task involving~\eqref{2LFN} (right).}
 \label{fig.3}
\end{figure}
\begin{corollary}\label{corollary 7.3}
Let $Y$ be an $m_2$-dimensional random variable defined by $Y=y(Z)$, where $y: \mathbb{R}^{m_1} \rightarrow \mathbb{R}^{m_2}$ is a Borel measurable function satisfying $|y(z)| \leq c_y\left(1+|z|^{\rho}\right)$ with $c_y \geq 0, \rho \geq 1$ for any $z \in \mathbb{R}^{m_1}$. For any $I=1, \ldots, m_1$, let $f_{Z^I \mid Z^1, \ldots, Z^{I-1}, Z^{I+1}, \ldots, Z^{m_1}}$ be the density function of $Z^I$ given $Z^1, \ldots, Z^{I-1}, Z^{I+1}, \ldots, Z^{m_1}$ and assume that there exist constants $C_{Z^I}, \bar{C}_{Z^I}>0$, such that for any $z=\left(z^1, \ldots, z^{m_1}\right) \in \mathbb{R}^{m_1}$,
$$
\begin{aligned}
\max\left\{\left|z^I\right|^{2}, \left|z^I\right|^{\rho}, \left|z^I\right|^{\rho + 2} \right\} f_{Z^I \mid Z^1, \ldots, Z^{I-1}, Z^{I+1}, \ldots, Z^{m_1}}\left(z^I \mid z^1, \ldots, z^{I-1}, z^{I+1}, \ldots, z^{m_1}\right) \leq \bar{C}_{Z^I}, \\
f_{Z^K \mid Z^1, \ldots, Z^{I-1}, Z^{I+1}, \ldots, Z^{m_1}}\left(z^I \mid z^1, \ldots, z^{I-1}, z^{I+1}, \ldots, z^{m_1}\right) \leq C_{Z^I}.
\end{aligned}
$$
Then,the optimization problem~\eqref{op:nn} satisfies Assumption~\ref{asm:A3}-\ref{asm:A5} with corresponding stochastic gradient $H:=F + G$ defined in~\eqref{h=f+g}-\eqref{application_1_f_g}.
\end{corollary}
\begin{proof}
    See Appendix~\ref{pf_prop:7.3}.
\end{proof}
\paragraph{\textbf{Simulated data}}\label{Simulated_data}
Set $d_1=30, d_2 = 30, m_1=2, m_2=1$. We denote by $z=\left(z^1, z^2\right) \in \mathbb{R}^{2}$ with $z^1, z^2 \in \mathbb{R}$. We aim to approximate the function $y(z)=-\left|1.2 z^1+0.9 z^2-0.8\right|^2$ on $[0,1] \times[0,1]$ using the TLFN defined in~\eqref{2LFN}. To this end, we consider the following method to obtain the fixed parameters $\left(\left[W_0\right], \left[W_2\right]\right)$: 
\begin{enumerate}
    \item First, we define the three-hidden-layer feed forward neural network (ThreeLFN) $\dot{\mathfrak{N}}: \mathbb{R}^{\dot{d}} \times \mathbb{R}^{m_1} \rightarrow \mathbb{R}^{m_2}$ with its $j$-th element given by given by
\begin{equation}\label{3LFN_eg}
\dot{\mathfrak{N}}^j(\dot{\theta}, \dot{z})=\sum_{m=1}^{d_3}\dot{W}_3^{j m}\sigma_3 \bigg( \sum_{n=1}^{d_2} \dot{W}_2^{m n} \sigma_2\left(\sum_{k=1}^{d_1}\left[\dot{W}_1^{n k} \sigma_1\left(\left\langle\dot{W}_0^{k \cdot}, \dot{z}\right\rangle+\dot{b}_0^k\right)\right]+\dot{b}_1^n\right) + \dot{b}_2^{m}\bigg),
\end{equation}
where $\dot{z}=\left(\dot{z}^1, \dot{z}^2\right) \in \mathbb{R}^{m_1}, \dot{W}_0 \in \mathbb{R}^{d_1 \times m_1}, \dot{W}_1 \in \mathbb{R}^{d_2 \times d_1}, \dot{W}_2 \in \mathbb{R}^{d_3 \times d_2}, \dot{W}_3 \in \mathbb{R}^{m_2 \times d_3}, \dot{b}_0 \in \mathbb{R}^{d_1}, \dot{b}_1 \in \mathbb{R}^{d_2}, \dot{b}_2 \in \mathbb{R}^{d_3}$ with $d_1 = d_2 = d_3=30, \sigma_1$ is the ReLU activation function, $\sigma_2$ and $\sigma_3$ are tanh activation functions, and where the parameter is
$$\dot{\theta}=\left(\left[\dot{W}_0\right], \left[\dot{W}_1\right], \left[\dot{W}_2\right], \left[\dot{W}_3\right], \dot{b}_0, \dot{b}_1, \dot{b}_2\right) \in \mathbb{R}^{\dot{d}}$$ with $\dot{d}=d_1\left(m_1+1\right)+d_2\left(d_1+1\right) + d_3\left(d_2 + m_2 + 1\right)$.
    \item Then, we train the ThreeLFN to approximate the function $\dot{y}(\dot{z})=\left|2\dot{z}^1+2 \dot{z}^2-1.5\right|^3$ on $[0,1] \times[0,1]$, that is, we aim to solve the following optimization problem: 
\begin{equation}\label{og_task}
\operatorname{minimize} \quad \mathbb{R}^d \ni \dot{\theta} \mapsto u(\dot{\theta}):=\mathbb{E}\left[|\dot{y}(\dot{Z})-\mathfrak{N}(\dot{\theta}, \dot{Z})|^2\right]+\lambda_r|\dot{\theta}|^2,
\end{equation}
where $\dot{Z}=\left(\dot{Z}^1, \dot{Z}^2\right) \in \mathbb{R}^2$. Note that~\eqref{og_task} corresponds to a similar task (but not the same) as our goal is to approximate the Borel function $y$. For the simulation, we generate 10000 independent samples $\dot{y}_n=\dot{y}\left(\dot{z}_n\right)$ with $\dot{z}_n=\left(\dot{z}_n^1, \dot{z}_n^2\right)$ for $\dot{z}^1, \dot{z}^2 \sim \mathcal{U}(0,1)$ being independent and split the dataset into subsets including training set (80 $\%$) and validation set (20 $\%$). Moreover, we set $\eta=10^{-2}, \lambda_r=10^{-6}, \beta=10^8, \gamma=0.5$, and obtain the initial values $\dot{\theta}_0$ using Xavier initialization (i.e., the default setting of PyTorch, see~\cite{glorot2010understanding});
    \item Once the ThreeLFN~\eqref{3LFN_eg} is fully trained, we obtain the trained parameters ${\dot{\theta}^*=\left([\dot{W}_0^*],[\dot{W}_1^*],[\dot{W}_2^*],\right.} \allowbreak {\left. [\dot{W}_3^*], \dot{b}_0^*, \dot{b}_1^*, \dot{b}_2^*\right)}$. Finally, we let $W_0:=\dot{W}_0^{\ast}$, $W_2 := \dot{W}_3^{\ast}$ be the fixed weight parameters in TLFN~\eqref{2LFN}.
\end{enumerate}
Now we proceed to approximate $y$ using TLFN~\eqref{2LFN}. The initial value of $\theta_0$ is obtained by using Xavier initialization (see~\cite{glorot2010understanding}).  Let $X=(y(Z), Z) \in \mathbb{R}^3$ with $Y \in \mathbb{R}$ and $Z=\left(Z^1, Z^2\right) \in \mathbb{R}^2$ where $Z^1, Z^2 \sim \mathcal{U}(0,1)$ are independent input random variables. We generate $n = 10000$ independent samples $\left(x_i\right)_{i=1}^{n}=\left(y_i, z_i\right)_{i=1}^{n}$ with $z_i=\left(z_i^1, z_i^2\right)$ and $y_i=y\left(z_i\right)=-\left|1.2 z_i^1+0.9 z_i^2-0.8\right|^2$. One notes that Assumptions~\ref{asm:A3}-\ref{asm:A5} hold in this setting due to Corollary~\ref{prop:6.2}.  Figure~\ref{fig.3} shows the true functions (in blue) and the approximated functions (in orange). These results indicate that SGHMC can be used for solving minimization problems involving ReLU neural networks in the transfer learning setting.

\subsubsection{Hedging under asymmetric risk}\label{sec:6.3}
Hedging is a major concern in finance, both from a theoretical and a practical point of view. Theoretical foundations for continuous-time hedging are well-established (see~\cite{karatzas1998methods}, for instance). However, in practice, hedging can be performed only at discrete time points, which yields a residual risk, namely, the tracking error. The major problem is to obtain hedging policies that minimize this error. To this end, this subsection adapts a deep learning approach proposed in~\cite{tsang2020deep} to solve hedging under asymmetric risk. We view the given hedging problem as a Markov Decision Process (MDP) and then approximate the optimal policy function of the MDP using neural networks. We use SGHMC to train the neural networks and compare its performance with other popular optimization algorithms.

We assume that the financial market is defined on a filtered probability space $\left(\Omega,\left\{\mathcal{F}_k\right\}_{k=0}^K, \mathbb{P}\right)$ with $\mathcal{F}_0=\{\emptyset, \Omega\}$ where $p \in \mathbb{N}$ assets $\left(\mathcal{S}_k^i\right)_{k=0, \cdots, K}^{i=1, \cdots, p}$ can be traded at discrete time points, i.e., $0=t_0 < t_1 < t_2 < \ldots < t_{K-1}< t_K=T$ for some fixed finite time horizon $0 < T < \infty$. We are interested in hedging the claim $h_K \in \mathcal{F}_K$ at time $T$ using the $p$ different hedging instruments. We denote by $R_k \in \mathbb{R}^p$, $k=0, \cdots, K-1$, the excess return vector of the $p$ risky assets between the period $\left(t_k, t_{k+1}\right]$ which are $\mathcal{F}_{k+1}$-measurable, whereas the constant risk-free return is denoted by $R_f \in \left[0, \infty\right)$.  For any $k=0, \ldots, K-1$, we consider the following (discrete-time) multidimensional Black-Scholes-Merton model for the excess return $R_k$:
$$
R_k=\exp \left(\left(\tilde{r} \mathbbm{1}+\Sigma \tilde{\lambda}-\frac{1}{2} \operatorname{diag}\left(\Sigma \Sigma^{\top}\right)\right) \Delta+\sqrt{\Delta} \Sigma \epsilon_k\right)-R_f \mathbbm{1},
$$
where $\tilde{r} \in \left[0, \infty\right), \mathbbm{1}=(1, \ldots, 1) \in \mathbb{R}^p, \Sigma \in \mathbb{R}^{p \times m}, m \geq p, \tilde{\lambda} \in \mathbb{R}^m, \Delta>0$ is a constant rebalancing time period, $\epsilon_k, k=0, \ldots, K-1$, are i.i.d. $m$-dimensional Gaussian vectors with mean $\mathbf{0}$ and covariance $I_m$, i.e., $\epsilon_k \sim \mathcal{N}_m\left(\mathbf{0}, I_m\right)$, and $R_f:=\exp (\tilde{r} \Delta)$ denotes the risk free return. In this setting, the excess returns $\left\{R_k\right\}_{k=0}^{K-1}$ are i.i.d.. Note that the market is complete when $m = p$ and $\Sigma \Sigma^{\top}$ is invertible and incomplete when $m > p$. Moreover, denote by $\mathcal{W}_k \in \mathbb{R}$ the wealth of the portfolio and $\mathcal{S}_k=\left(\mathcal{S}_k^{1}, \cdots, \mathcal{S}_k^{p}\right) \in \mathbb{R}^p$ the equity prices at time point $t_k$. To model this as an MDP, we define the state process $\left(s_k\right)_{k=0, \cdots, K}$ by $s_k:=\left(\mathcal{W}_k, \mathcal{S}_k\right) \in \mathbb{R}^{1 + p}$ and denote by $D \subseteq \mathbb{R}^p$ the set of possible actions representing the proportion of current wealth invested in each risky asset. Then, for any $k=0,1, \ldots, K-1$, the evolution of the equity prices between the time points $k$ and $k+1$ is $\mathcal{S}_{k+1} := \mathcal{S}_k \left(1 + R_k\right)$ and the evolution of the wealth is given by
$$
\mathcal{W}_{k+1}:=\mathcal{W}_k\left(\left\langle g_k\left(s_k\right), R_k\right\rangle+R_f\right),
$$
where $g_k(\cdot): \mathbb{R}^{1+p} \rightarrow D$ (being measurable) is the investment control policy function on $p$ risky assets at time point $t_k$. 

In this setting, we are interested in finding the hedging instrument selection of $p$ risky assets which minimizes the loss function of the residual risk. The loss function minimization problem can be written as an MDP problem as follows:
\begin{equation}\label{BS_MDP}
\begin{aligned}
\min _{g_0, \ldots, g_{K-1}} & \mathscr{E}_K \left(s_0\right) := 
\min _{g_0, \ldots, g_{K-1}} \mathbb{E}\left[\frac{\Big(1 - \gamma \operatorname{Sign}(\mathcal{W}_K - h_K)\Big)^2 \left(\mathcal{W}_K - h_K\right)^2}{2} \right],\\
\text { s.t. } s_{k+1}=&\big(\mathcal{W}_{k+1}, \mathcal{S}_{k+1}\big) =\big(\mathcal{W}_k\left(\left\langle g_k\left(s_k\right), R_k\right\rangle+R_f\right), \mathcal{S}_k \left(1 + R_k\right)\big), \quad k=0,1, \ldots, K-1 ,
\end{aligned}
\end{equation}
where $\operatorname{Sign}(y):=\mathbbm{1}_{y>0}-\mathbbm{1}_{y<0}$, $h_K:=h(\mathcal{S}_K)$ is the claim for some measurable function $h: \mathbb{R}^{p} \rightarrow \left[0, \infty\right)$, and $\gamma \in\left[0,1\right)$ is the penalization parameter. Note that the loss function $\ell_\gamma(y):=\frac{(1 - \gamma \operatorname{Sign}(y))^2 y^2}{2}$, $y\in\mathbb{R}$, is asymmetric when $\gamma > 0$, which penalizes losses~(i.e., when $\mathcal{W}_K - h_K < 0$) more than profits.
\begin{table}[t]
\centering
\begin{tabular}{c | c  c | c  c}
 \hline
$p, m$ & $5$, $5$ & $50$, $50$& $5$, $10$ & $50$, $60$ \\ \hline \hline
$\tilde{K}$ & $5 $& $50 $& $5$ & $60$\\ \hline
$\mathcal{W}_0$ & $1$ &$1$ &$1$  &$1$ \\ \hline
$\mathcal{S}_0$ & $\mathbbm{1}$ &$\mathbbm{1}$ &$\mathbbm{1}$  &$\mathbbm{1}$ \\ \hline
$\gamma$ & $0.5$& $0.5$&$0.5$ &$0.5$ \\  \hline
$\widetilde{r}$ &$ 0.03$ & $0.03$ &$ 0.03 $ & $0.03$\\ \hline
$\Delta$ & $1/40 $& $1/40 $& $1/40$ & $1/40$\\ \hline
$D$ & $[0, 1]^p$ & $[0, 1]^p$ & $[0, 1]^p$ & $[0, 1]^p$  \\ \hline
\multirow{4}{*}{$\widetilde{\lambda}$} & $\widetilde{\lambda}_i=0.1$  & $\widetilde{\lambda}_i=0.01$ & $\widetilde{\lambda}_i=0.01$  & $\widetilde{\lambda}_i=0.01$ \\

& for $i=1,2$ & for $i=1,\ldots, 25$ & for $i=1, 2$ & for $i=1,\ldots, 25$ \\ 

& $\widetilde{\lambda}_i=0.2$ & $\widetilde{\lambda}_i=0.05$ & $\widetilde{\lambda}_i=0.05$ & $\widetilde{\lambda}_i=0.05$ \\ 

& for $i=3,4,5$ & for $i=26,\ldots, 50$ & for $i=3, 4, 5$ & for $i=26,\ldots, 50$ \\  \hline

\multirow{2}{*}{$\Sigma$} & $\Sigma_{ii}=0.15$ & $\Sigma_{ii}=0.15$  & $\Sigma_{ii}=0.15$ 
& $\Sigma_{ii}=0.15$  \\
& $\Sigma_{ij}=0.01$  for $i\neq j$ & $\Sigma_{ij}=0.001$  for $i\neq j$& $\Sigma_{ij}=0.01$  for $i\neq j$ & $\Sigma_{ij}=0.001$  for $i\neq j$ \\  \hline
\end{tabular}
\caption{Parameters for optimization problem~\eqref{deep_learning_method_mdp}.}
\label{tab:iid}
\end{table}

Following~\cite{tsang2020deep}, we solve the MDP problem~\eqref{BS_MDP} by employing standard feed forward neural networks to approximate each of the investment control policy functions $g_k(\cdot)$. We denote by $\mathcal{G}_\nu$ the set of standard feed forward neural networks with two hidden layers, which is given explicitly by
\begin{equation}\label{nn_mdp}
\begin{aligned}
\mathcal{G}_\nu = &\big\{f: \mathbb{R}^{1+p} \rightarrow \mathbb{R}^p \mid f(x)=\tanh \left(W_3 z+b_3\right), z=\sigma\left(W_2 y+b_2\right), \\
&y=\sigma\left(W_1 x+b_1\right), W_1 \in \mathbb{R}^{\nu \times (1+p)}, W_2 \in \mathbb{R}^{\nu \times \nu}, W_3 \in \mathbb{R}^{p \times \nu}, b_1, b_2 \in \mathbb{R}^\nu, b_3 \in \mathbb{R}^p\big\},
\end{aligned}
\end{equation}
where $\nu$ denotes the number of neurons on each layer of the neural network, $\tanh (x)$, for any $x \in \mathbb{R}^p$, is the hyperbolic tangent function at $x$ applied componentwise, and $\sigma(y):=\max \{0, y\}, y \in \mathbb{R}^\nu$, is the ReLU activation function applied componentwise.
\begin{figure}[h]
    \centering
    \begin{subfigure}[b]{0.32\textwidth}
        \includegraphics[width=\textwidth]{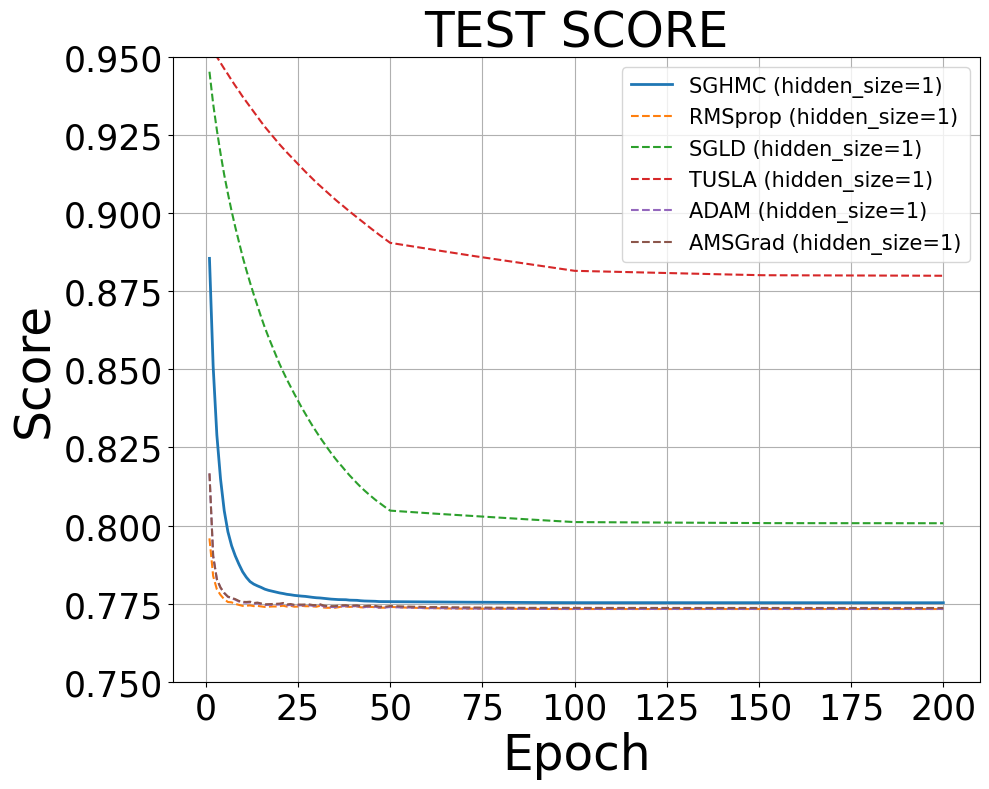}
        \caption{$p=m=5$ and $\nu=1$}
    \end{subfigure}
    \begin{subfigure}[b]{0.32\textwidth}
        \includegraphics[width=\textwidth]{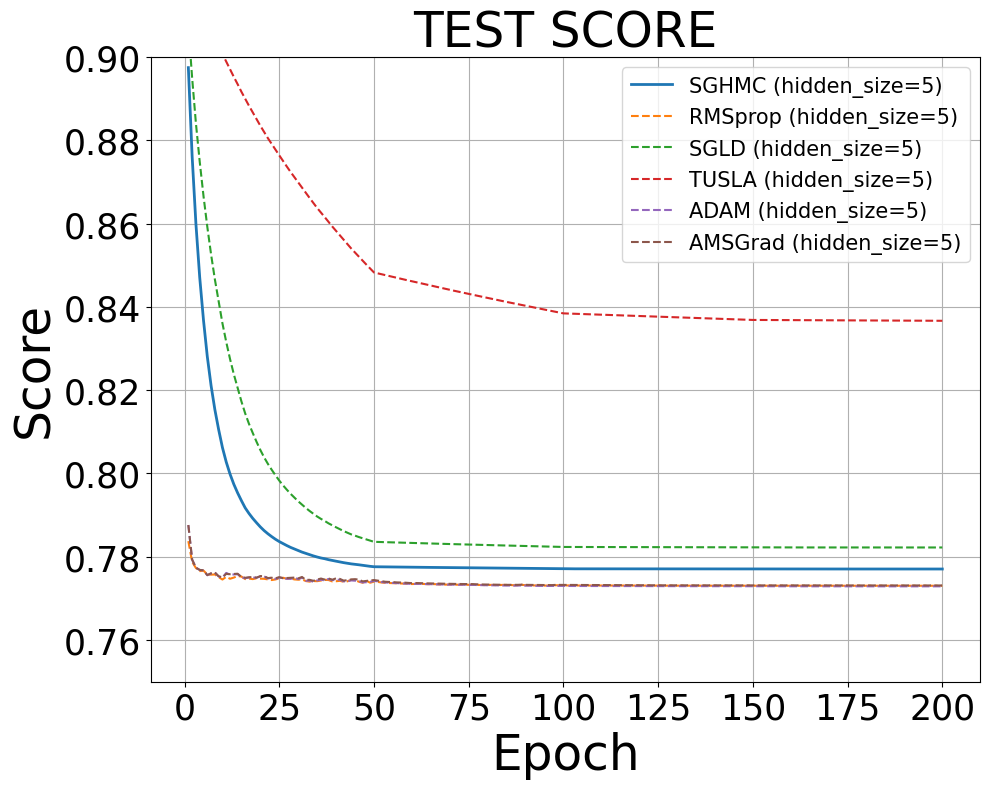}
        \caption{$p=m=5$ and $\nu=5$}
    \end{subfigure}
    \begin{subfigure}[b]{0.32\textwidth}
        \includegraphics[width=\textwidth]{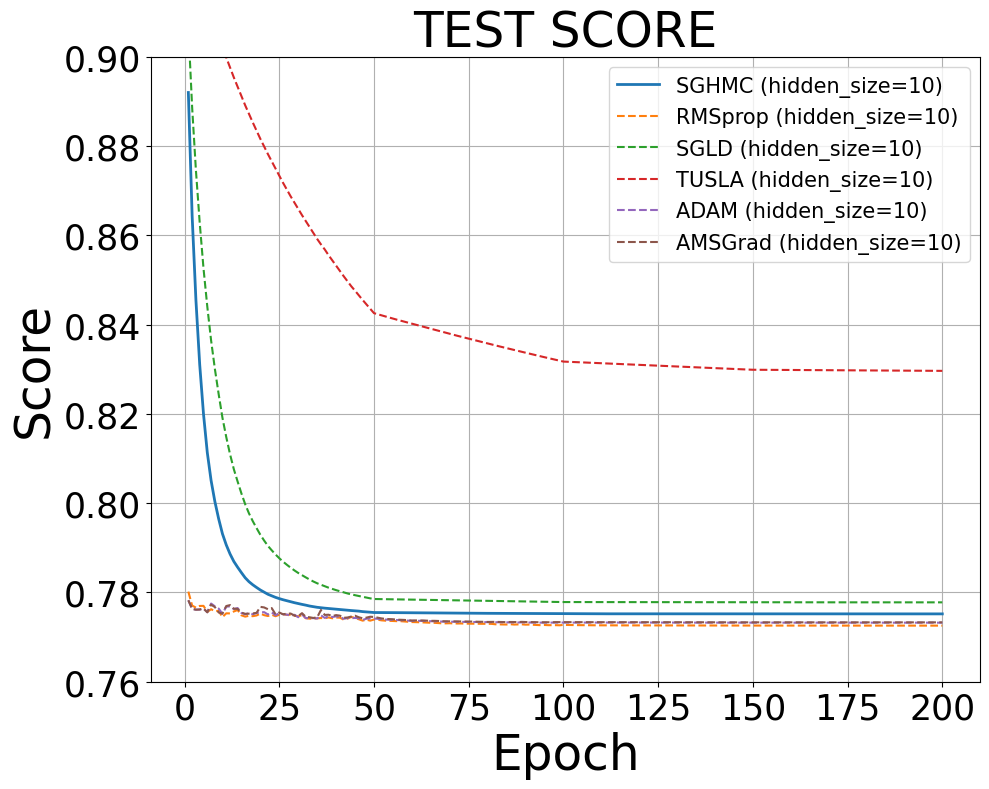}
        \caption{$p=m=5$ and $\nu=10$}
    \end{subfigure}

    \begin{subfigure}[b]{0.32\textwidth}
        \includegraphics[width=\textwidth]{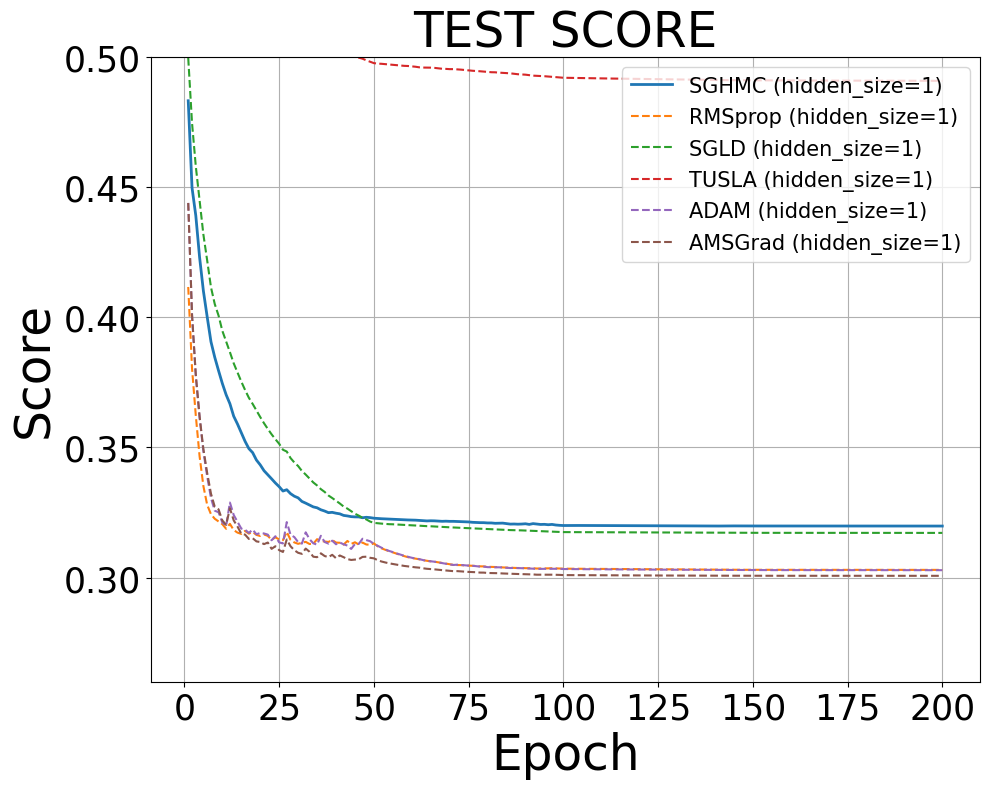}
        \caption{$p=m=50$ and $\nu=1$}
    \end{subfigure}
    \begin{subfigure}[b]{0.32\textwidth}
        \includegraphics[width=\textwidth]{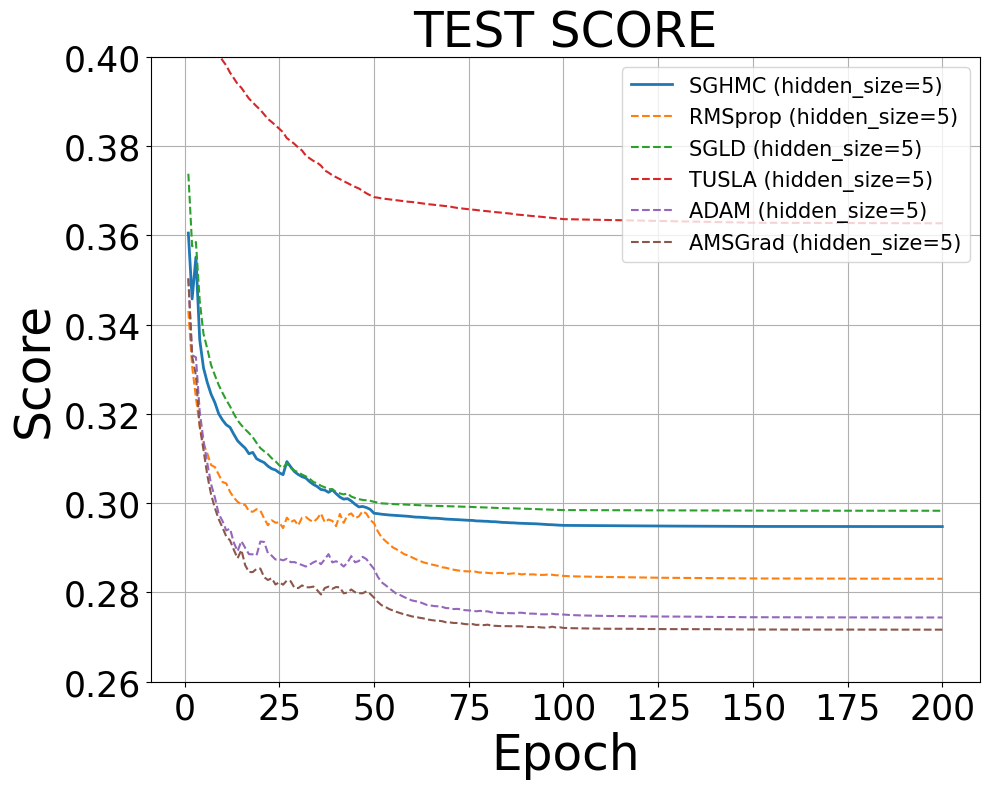}
        \caption{$p=m=50$ and $\nu=5$}
    \end{subfigure}
    \begin{subfigure}[b]{0.32\textwidth}
        \includegraphics[width=\textwidth]{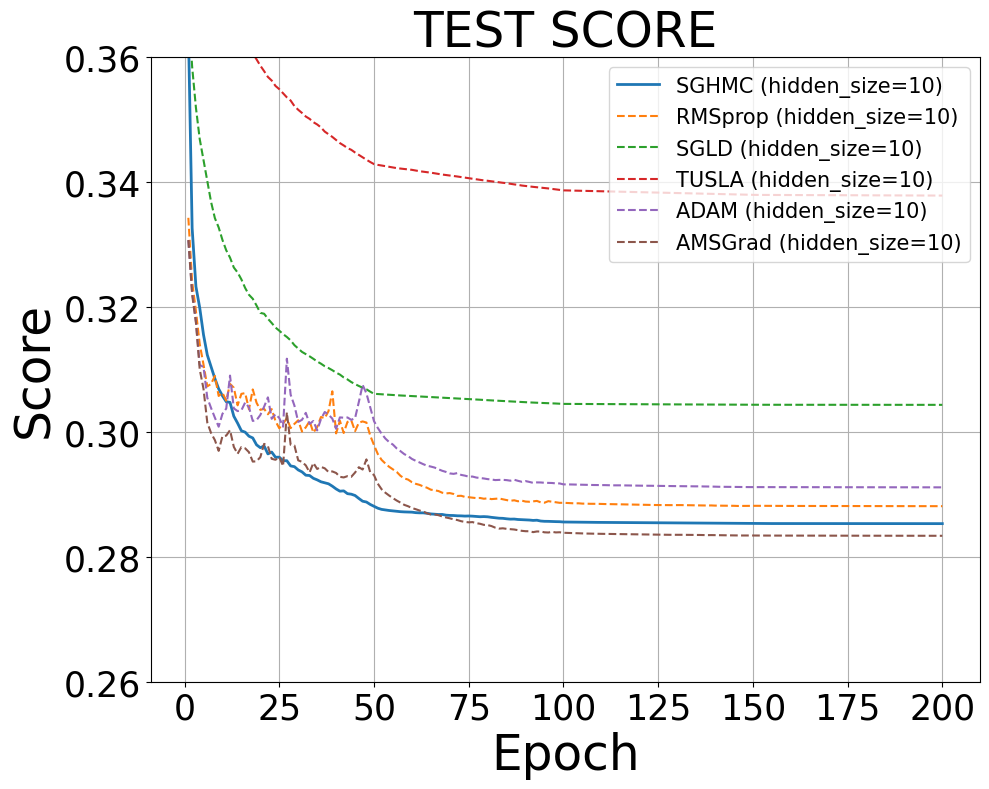}
        \caption{$p=m=50$ and $\nu=10$}
    \end{subfigure}

    \caption{Plots of test scores $\mathscr{E}_K^{\mathcal{N} \mathcal{N}}\left(s_0\right)$ for different numbers of assets and hidden sizes under the Black-Scholes-Merton model in the complete market case. The parameter settings are summarized in Table~\ref{tab:iid}.}
    \label{fig:iid_com}
\end{figure}

\begin{figure}[h]
    \centering
    \begin{subfigure}[b]{0.32\textwidth}
        \includegraphics[width=\textwidth]{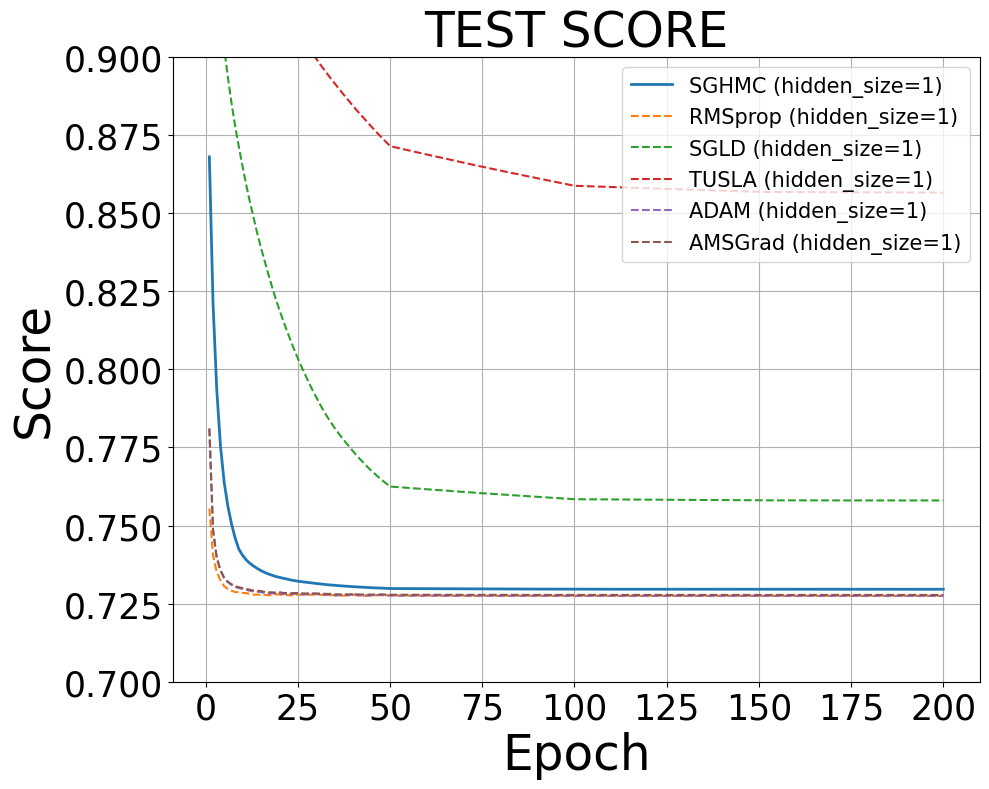}
        \caption{$p=5$, $m=10$, and $\nu=1$}
    \end{subfigure}
    \begin{subfigure}[b]{0.32\textwidth}
        \includegraphics[width=\textwidth]{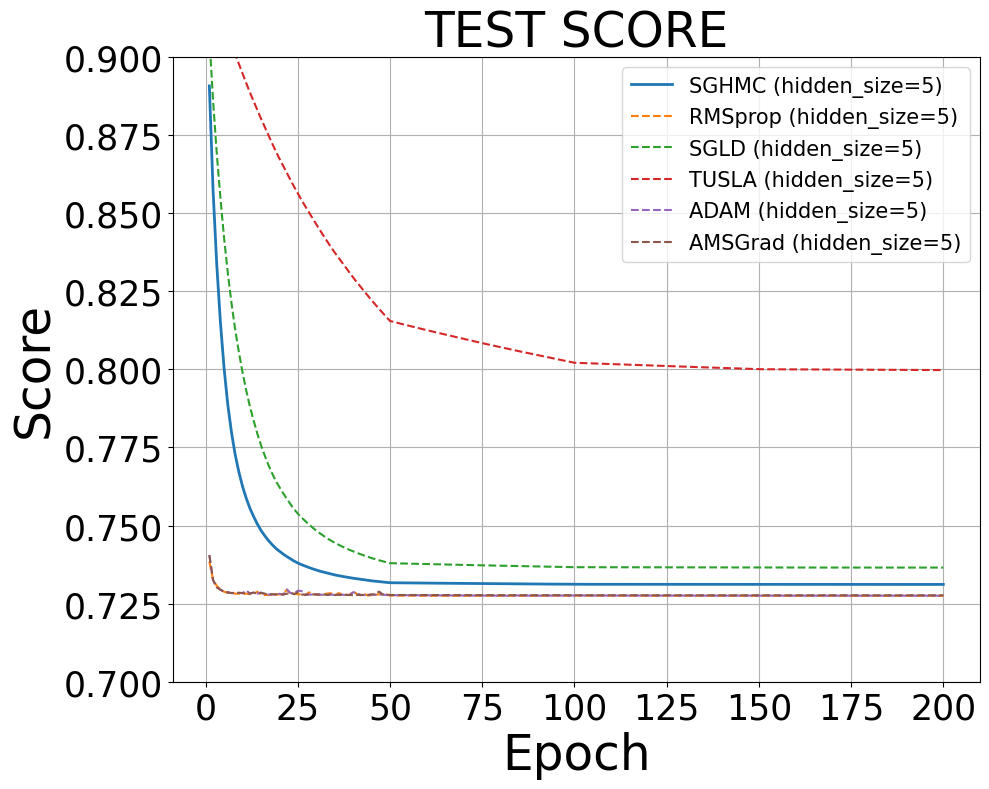}
        \caption{$p=5$, $m=10$, and $\nu=5$}
    \end{subfigure}
    \begin{subfigure}[b]{0.32\textwidth}
        \includegraphics[width=\textwidth]{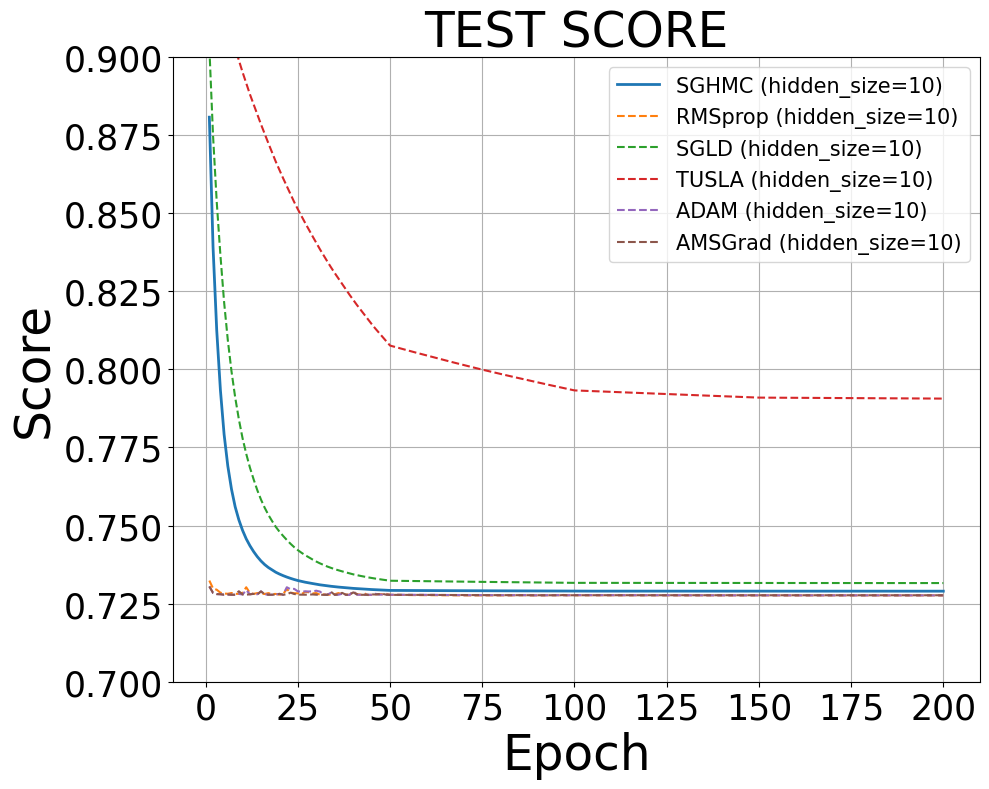}
        \caption{$p=5$, $m=10$, and $\nu=10$}
    \end{subfigure}

    \begin{subfigure}[b]{0.32\textwidth}
        \includegraphics[width=\textwidth]{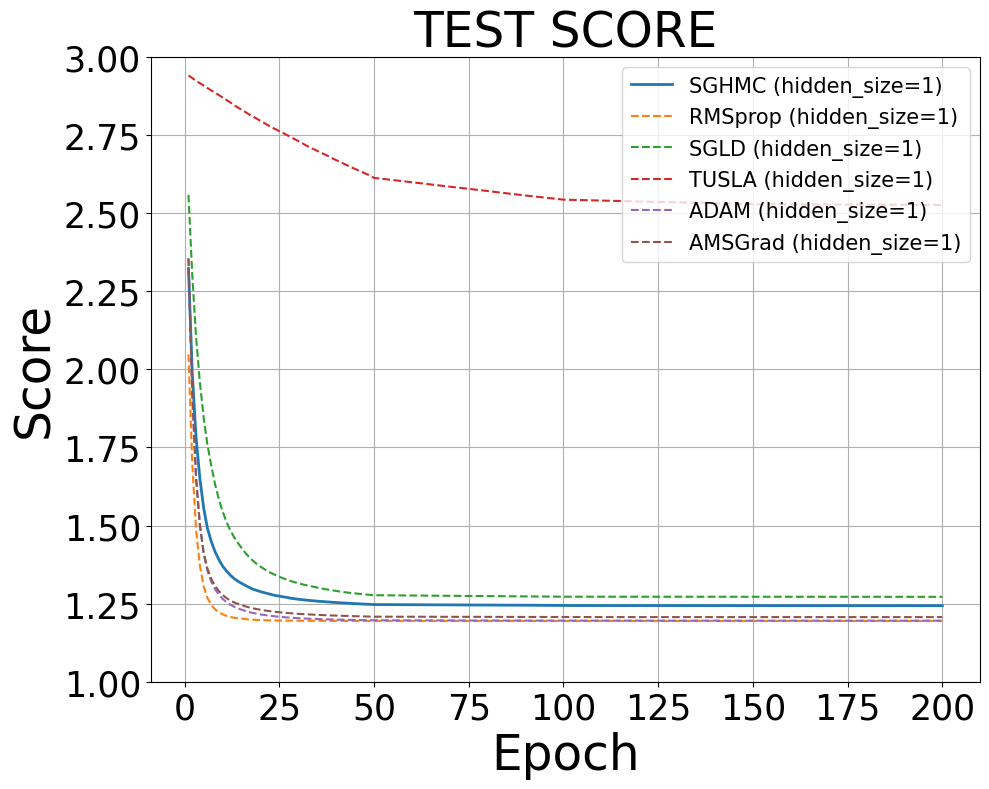}
        \caption{$p=50$, $m=60$, and $\nu=1$}
    \end{subfigure}
    \begin{subfigure}[b]{0.32\textwidth}
        \includegraphics[width=\textwidth]{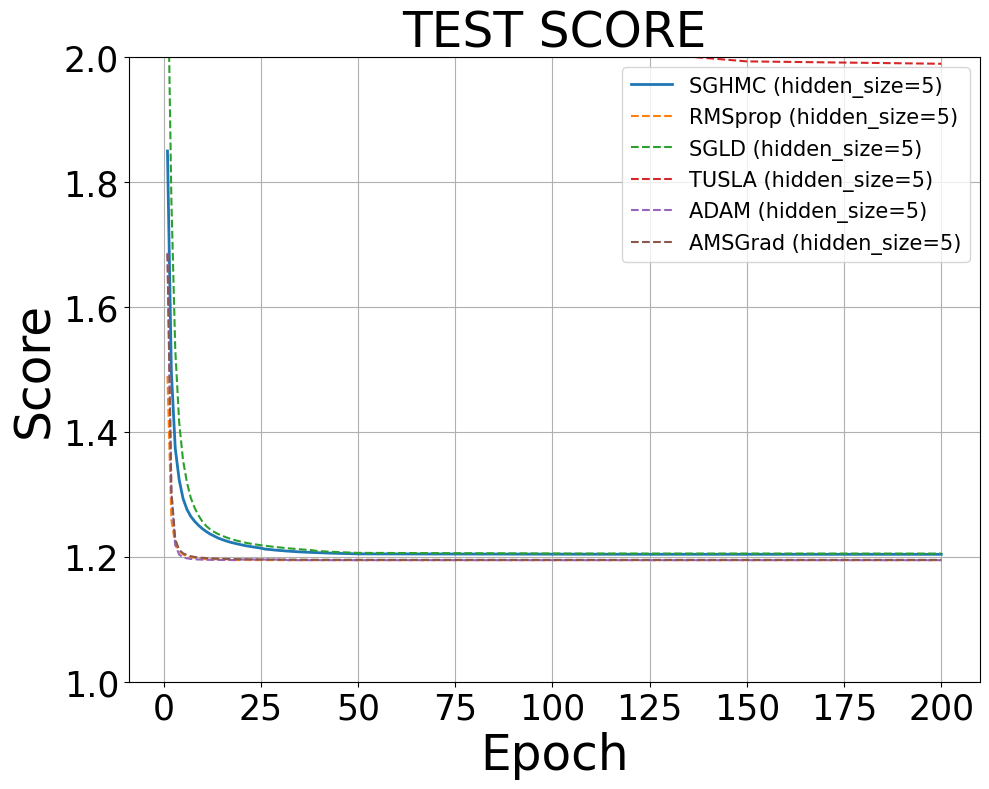}
        \caption{$p=50$, $m=60$, and $\nu=5$}
    \end{subfigure}
    \begin{subfigure}[b]{0.32\textwidth}
        \includegraphics[width=\textwidth]{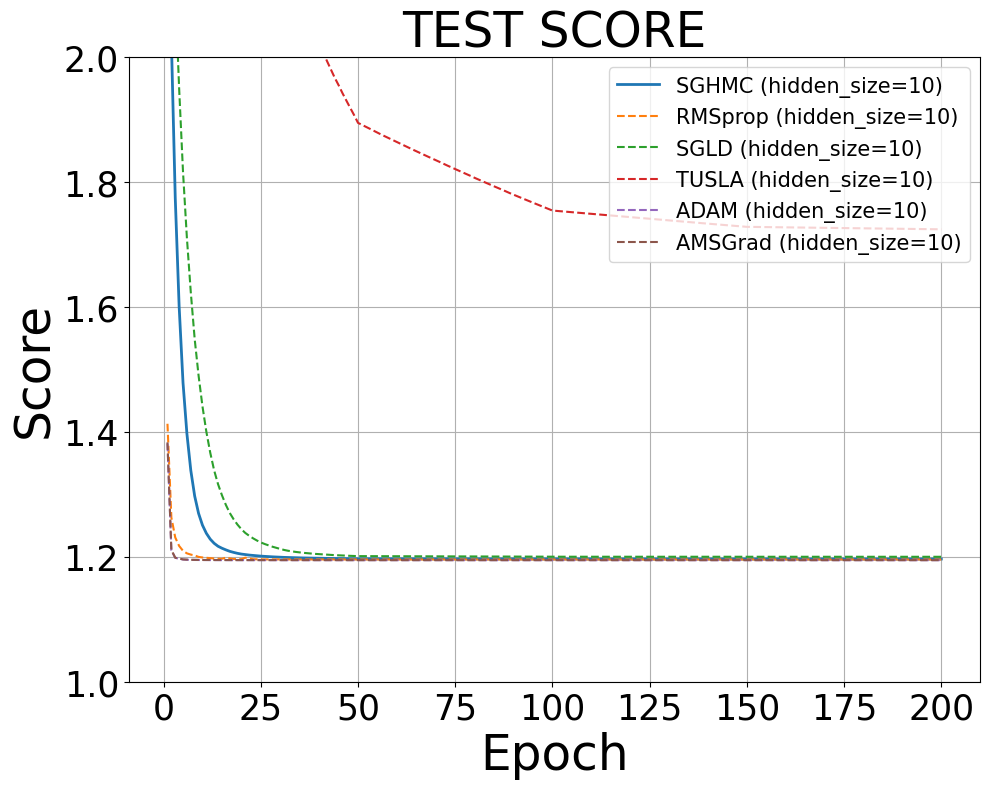}
        \caption{$p=50$, $m=60$, and $\nu=10$}
    \end{subfigure}

    \caption{Plots of test scores $\mathscr{E}_K^{\mathcal{N} \mathcal{N}}\left(s_0\right)$ for different numbers of assets and hidden sizes under the Black-Scholes-Merton model in the incomplete market case. The parameter settings are summarized in Table~\ref{tab:iid}.}
    \label{fig:iid_incom}
\end{figure}
For any $k=0,1, \ldots, K-1$, denote by $g_k\left(\cdot;\theta_k\right):\mathbb{R}^{1+p}\rightarrow\mathbb{R}^p$ the approximated policy function at time $t_k$ using a neural network with its structure defined in with~\eqref{nn_mdp}, where $\theta_k := $ $\left(b_1, b_2, b_3,\left[W_1\right],\left[W_2\right],\left[W_3\right]\right) \in \mathbb{R}^{\nu\left(1+p+\nu+p+2\right)+p}$ denotes the parameter of the neural network. Then, the MDP problem~\eqref{BS_MDP} can be approximated by restricting $g_k\left(\cdot ; \theta_k\right) \in \mathcal{G}_\nu$ :
\begin{equation}\label{deep_learning_method_mdp}
\min _{g_0, \ldots, g_{K-1}} \mathscr{E}_K\left(s_0\right) \approx \min _\theta \mathbb{E}\left[\frac{\Big(1 - \gamma \operatorname{Sign}(\mathcal{W}_K^{\mathcal{N} \mathcal{N}}\left(\zeta ; \theta\right) - h_K)\Big)^2 \left(\mathcal{W}_K^{\mathcal{N} \mathcal{N}}\left(\zeta ; \theta\right) - h_K\right)^2}{2} \right] =: \min _\theta \mathscr{E}_K^{\mathcal{N} \mathcal{N}}\left(s_0\right),
\end{equation}
where $\zeta:=\left(s_0, R_0, \ldots, R_{K-1}\right)$ denotes the vector of the initial state variable and all the returns throughout the trading time horizon $[0, T]$, and where $s_K^{\mathcal{N} \mathcal{N}}(\zeta ; \theta)$ is recursively defined, for $k=0,1, \ldots, K-1$, by
\begin{equation}\label{deep_learning_method_2_mdp}
s_{k+1}^{\mathcal{N} \mathcal{N}}\left(\zeta ; \theta\right)=\left(\mathcal{W}_{k+1}^{\mathcal{N} \mathcal{N}}\left(\zeta ; \theta\right), \mathcal{S}_{k+1}\right) =\bigg(\mathcal{W}_{k}^{\mathcal{N} \mathcal{N}}\left(\zeta ; \theta\right)\big(\left\langle g_k\left(s_k^{\mathcal{N} \mathcal{N}}; \theta_k\right), R_k\right\rangle+R_f\big), \mathcal{S}_k \left(1 + R_k\right)\bigg),
\end{equation}
where $g_k\left(\cdot ; \theta_k\right) \in \mathcal{G}_\nu$ is the approximated policy function with $\mathcal{G}_\nu$ given in~\eqref{nn_mdp} and ${\theta=\left(\theta_0, \ldots, \theta_{K-1}\right) \in \mathbb{R}^d}$ is the parameter for the neural networks with $d:=$ $K\left(\nu\left(1+p+\nu+p+2\right)+p\right)$.

We set the claim $h_K := \max\left\{\sum_{i=1}^{p}\mathcal{S}_K^{i} - \tilde{K}, 0\right\}$, where $\tilde{K}$ is the strike price. Different settings of the Black-Scholes-Merton model for simulations are summarized in Table~\ref{tab:iid}. Similar to~\cite{tsang2020deep}, the models are trained for 200 steps with a batch size of 128. Each step involves generating 20000 training samples, requiring $\lceil 20000 / 128\rceil = 157$ iterations per step for model training. Subsequently, 100000 test samples are used to compute the test residual error. Furthermore, three different hidden sizes $\nu$ are tested for each experimental setting. For $p=5:$~$\nu$ values are $\{1,5,10\}$ with $m$ values $\{5, 10\}$; for $p=50:$~$\nu$ values are $\{1,5,10\}$ with $m$ values $\{50, 60\}$. For SGHMC, optimal hyperparameters are selected from $\eta=\{0.5, 0.1, 0.05, 0.01\}$, $\gamma=\left\{0.5, 1, 5\right\}$, and $\beta=10^{12}$. For SGLD and TUSLA (see~\cite{lim2024non, lovas2023taming}), we choose hyperparameters among $\eta=\{0.5, 0.1, 0.05, 0.01\}$ and $\beta=10^{12}$. For ADAM, AMSGrad, and RMSProp, the best learning rate is chosen among $\eta=\{0.1, 0.05, 0.01\}$ with other hyperparameters set as default.

Figure~\ref{fig:iid_com} and~\ref{fig:iid_incom} plots learning curves for all optimization algorithms across various configurations in complete and incomplete markets. As depicted, the SGHMC algorithm generally outperforms SGLD and TUSLA in this task. Furthermore, the SGHMC algorithm achieves comparable performances to ADAM, AMSGrad, and RMSProp.


\subsubsection{Real-world datasets}\label{sec:6.2.2}
In this subsection, we apply the SGHMC algorithm to solve optimization problems using real-world datasets. More precisely, we consider a regression problem using concrete compressive strength dataset (see~\cite{yeh1998modeling}) as well as an image classification problem using FashionMNIST (see~\cite{xiao2017fashion}). We then compare the performance of SGHMC with other popular optimizers including ADAM, AMSGrad, and RMSProp.

We consider to solve the following optimization problem:
\begin{equation}\label{opt:real-world}
\operatorname{minimize}\quad \mathbb{R}^d \ni \theta \mapsto u(\theta):=\mathbb{E}[\ell(Y, \bar{\mathfrak{N}}(\theta, Z))] + \lambda_r \left\lvert\theta\right\rvert^{2},
\end{equation}
where $\ell:\mathbb{R}^{m_2} \times \mathbb{R}^{m_2} \rightarrow \mathbb{R}$ is a loss function with $m_2 \in \mathbb{N}$, $\theta \in \mathbb{R}^{d}$ is the parameter over which we optimize, $Z \in \mathbb{R}^{m_1}$ is the input random variable, $Y \in \mathbb{R}^{m_2}$ is the target random variable, and $\bar{\mathfrak{N}}$ is the neural network given by:
\begin{equation}\label{nn:real-world}
\bar{\mathfrak{N}}(\theta, z):=W_2 \sigma_1\left(W_1 z+b_1\right)+b_2,
\end{equation}
where $W_1 \in \mathbb{R}^{d_1 \times m_1}$, $W_2 \in \mathbb{R}^{m_2 \times d_1}$, $b_1 \in \mathbb{R}^{d_1}$, and $b_2 \in \mathbb{R}^{m_2}$ are the parameters, i.e., $\theta = ([W_1], [W_2], b_1, b_2) \in \mathbb{R}^{d}$ with $d = d_1(m_1 + m_2 + 1) + m_2$. 
\\

\paragraph{\textbf{Regression}} In the regression example, we explore the performance of SGHMC in solving~\eqref{opt:real-world} using concrete compressive strength dataset\footnote{Dataset is available in \href{https://archive.ics.uci.edu/dataset/165}{https://archive.ics.uci.edu/dataset/165}}, which includes 1030 instances and each sample has 9 attributes (consisting of 8 quantitative input variables and 1 quantitative output variable) with no missing values.  For the numerical experiments, we set $m_1=8, m_2=1$, $d_1 = 50$ and thus $d = 501$. Our goal is to find the best estimator that predicts the concrete compressive strength $Y \in \mathbb{R}$ given the input variable $Z \in \mathbb{R}^8$ by solving the optimization problem~\eqref{opt:real-world} with the squared loss function $\ell(u, v)=|u-v|^2$ for $u, v \in \mathbb{R}$. In our experiments, we divided $10\%$ of the dataset as a test set and employed the SGHMC, ADAM, AMSGrad, and RMSprop algorithms (see, e.g.,~\cite{kingma2014adam, tan2019convergence, xu2021convergence}) with three different seeds. We trained the models for 5000 epochs using a batch size of 256. For ADAM, AMSGrad, and RMSProp, we searched for the optimal learning rate within the range of $\{0.01, 0.001\}$ whereas the other parameters are set as default.
\\

\paragraph{\textbf{Classification}} In the classification example, we set $m_1=784, m_2=10, d_1=50$ and thus $d=39760$. We consider to solve~\eqref{opt:real-world} with the neural network defined in~\eqref{nn:real-world} and the cross entropy loss given by $\ell \left(u, v\right) := -\sum_{j=1}^{m_2} u_j \log v_j$ for $u, v \in \mathbb{R}^{m_2}$. We use the Fashion MNIST dataset\footnote{Dataset is available in \href{https://github.com/zalandoresearch/fashion-mnist}{https://github.com/zalandoresearch/fashion-mnist}} comprising of $28 \times 28 $ images of 70000 fashion products from 10 categories with 7000 images per category. The training set has 60000 samples and the test set has 10000 samples. Each sample is assigned to one of ten different labels. Then, the label variables are defined as $y_i = [y_{i, 0}, \cdots, y_{i, 9}]^\top \in \mathbb{R}^{10}$, where $y_{i, j} := \mathbbm{1}_{j = label_{i}}$, $i = 1, \cdots, 60000$. We train the models for 200 epochs with a batch size of 128. Also, we decay the initial learning rate by 10 after 150 epochs.
\begin{figure}[!t]
	\centering
	\begin{minipage}{0.49\linewidth}
		\centering
		\includegraphics[width=0.9\linewidth]{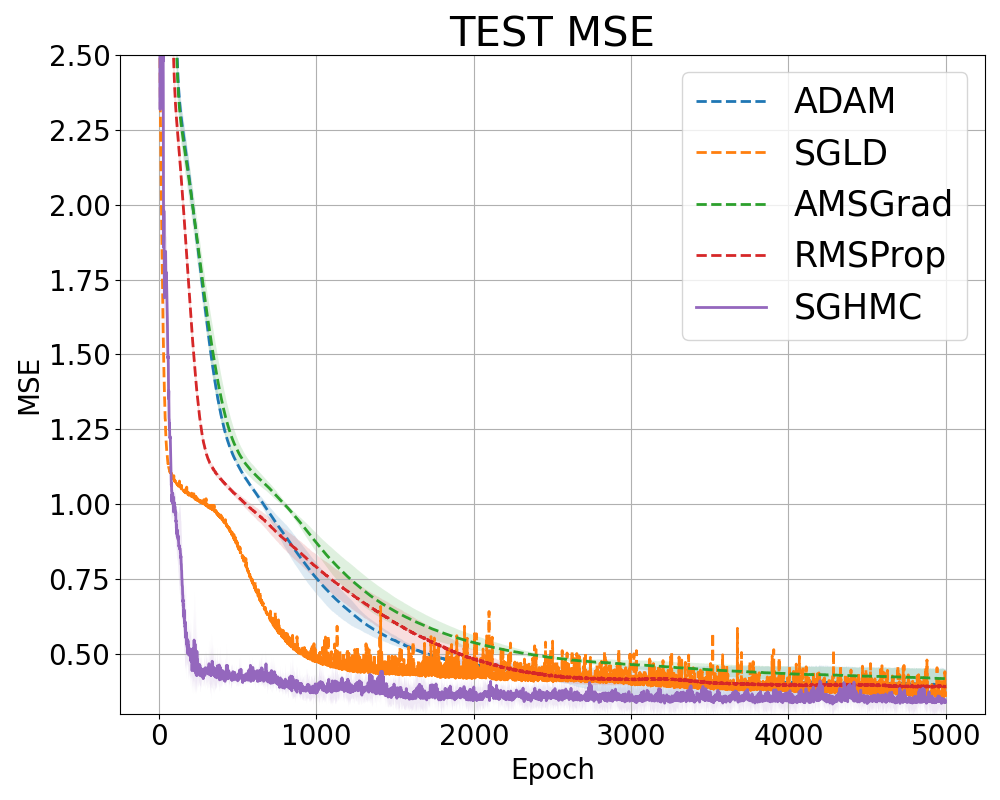}
	\end{minipage} 
	\begin{minipage}{0.49\linewidth}
		\centering
		\includegraphics[width=0.9\linewidth]{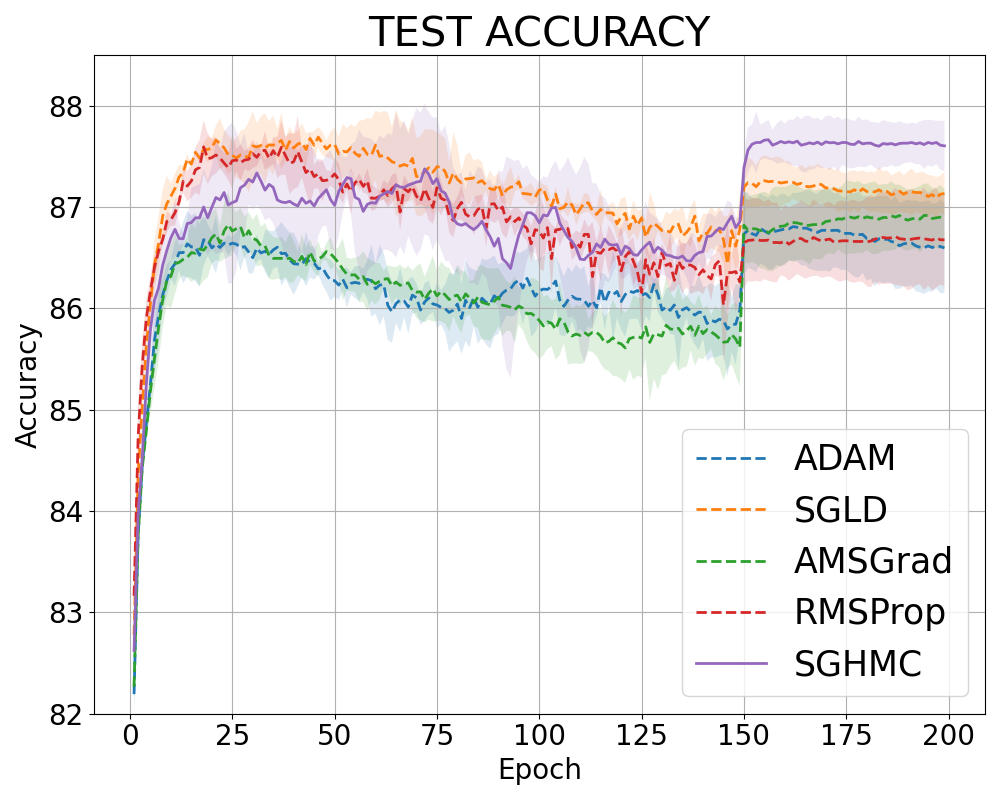}
	\end{minipage}
 \caption{Mean squared error for concrete compressive strength dataset (left) and Test accuracy curve for Fashion MNIST (right).}
 \label{fig.4}
\end{figure}

As shown in the left plot in Figure~\ref{fig.4}, the SGHMC algorithm achieves the lowest mean squared error compared to other optimizers. The right plot in Figure~\ref{fig.4} shows that the performance of the SGHMC is slightly better than ADAM, Amsgrad, and RMSProp in terms of test accuracy. Our results show that the SGHMC algorithm can outperform the SGLD algorithm in many applications and achieve comparable performance to popular optimizers on learning tasks involving real-world datasets.
\\

\subsection{Conclusion of numerical experiments}
We considered applying the SGHMC algorithm to various applications to justify our theoretical results. For the empirical experiments on quantile estimation in Section~\ref{sec:6.1},  SGHMC outperforms SGLD by having a smaller expected excess risk and a shorter training time under different settings. In Section~\ref{transfer_learning},  the numerical results show that SGHMC can be used for the training of ReLU neural networks in the transfer learning setting. Moreover, for the hedging problem under asymmetric risk described in Section~\ref{sec:6.3},
SGHMC outperforms SGLD and TUSLA in different scenarios while it achieves comparable test scores to ADAM, AMSGrad, and RMSProp. In addition, as discussed in Section~\ref{sec:6.2.2}, for regression and classification problems on real-world datasets, SGHMC demonstrates better performance compared to other optimizers in term of the test mean-squared error and test accuracy.

In summary, SGHMC achieves (at least) comparable results to existing algorithms including, e.g., ADAM, AMSGrad, RMSProp, SGLD, and TUSLA, while its performance is backed by theoretical guarantees as described in Theorem~\ref{thm:2.1}, Corollary~\ref{cor:2.2}, and Theorem~\ref{thm:2.2}.


\newpage

\section{Proof of Main Theorems}
\subsection{Introduction of auxiliary processes}
\noindent We first define the underdamped Langevin SDE $(\theta_t, V_t)_{t\in \mathbb{R}_{+}}$ given by
\begin{equation}\label{underdamped Langevin SDE}
\begin{aligned}
\mathrm{d} V_t & =-\left(\gamma V_t+h\left(\theta_t\right)\right) \mathrm{d} t+\sqrt{2 \gamma \beta^{-1}} \mathrm{d} B_t , \\ 
\mathrm{d} \theta_t & = V_t \mathrm{d} t,
\end{aligned}
\end{equation}
where $h:=\nabla u$, $\gamma>0$ is the friction coefficient, $\beta > 0$ is the inverse temperature parameter, and $\left(B_t\right)_{t \in \mathbb{R}_{+}}$ is a standard $d$-dimensional Brownian motion adapted to its completed natural filtration denoted by $\left(\mathcal{F}_t\right)_{t \in \mathbb{R}_{+}}$. We assume that $\left(\mathcal{F}_t\right)_{t \in \mathbb{R}_{+}}$ is independent of $\mathcal{G}_{\infty} \lor \sigma(\theta_0, V_0)$.

For each $\eta > 0$, we consider the scaled process $(\zeta_t^\eta, Z_t^\eta)$ := $(\theta_{\eta t}, V_{\eta t})$ with $(\theta_t, V_t)_{t\in \mathbb{R_{+}}}$ defined in~\eqref{underdamped Langevin SDE}. It can be written as follows: 
\begin{equation}\label{SKLD}
\begin{aligned}
\mathrm{d} Z_t^\eta & =-\eta\left(\gamma Z_t^\eta+h\left(\zeta_t^\eta\right)\right) \mathrm{d} t+\sqrt{2 \gamma \eta \beta^{-1}} \mathrm{~d} B_t^\eta, \\
\mathrm{d} \zeta_t^\eta & =\eta Z_t^\eta \mathrm{d} t,
\end{aligned}
\end{equation}
with the initial condition $Z_0^\eta:= V_0$ and $\zeta_0^\eta:= \theta_0$, where $B_t^{\eta} := \frac{1}{\sqrt\eta}B_{\eta t}$. We denote the filtration of $\left(B_t^{\eta}\right)_{t \in \mathbb{R}_{+}}$ by $\left(\mathcal{F}_t^{\eta}\right)_{t \in \mathbb{R}_{+}}:=\left(\mathcal{F}_{\eta t}\right)_{t \in \mathbb{R}_{+}}$, and we note that  $\left(\mathcal{F}_t^{\eta}\right)_{t \in \mathbb{R}_{+}}$ is independent of $\mathcal{G}_{\infty} \lor \sigma(\theta_0, V_0)$.

Next, we define the continuous-time interpolation of SGHMC as
\begin{equation}\label{cont.SGHMC}
\begin{aligned}
\mathrm{d} \bar{V}_t^\eta & =-\eta\left(\gamma \bar{V}_{\lfloor t\rfloor}^\eta+H\left(\bar{\theta}_{\lfloor t\rfloor}^\eta, X_{\lceil t\rceil}\right)\right) \mathrm{d} t+\sqrt{2 \gamma \eta \beta^{-1}} \mathrm{d} B_t^\eta, \\
\mathrm{d} \bar{\theta}_t^\eta & =\eta \bar{V}_{\lfloor t\rfloor}^\eta \mathrm{d} t,
\end{aligned}
\end{equation}
with initial value $\bar{V}_0^\eta:= V_0$ and $\bar{\theta}_0^\eta:=\theta_0$. The process~\eqref{cont.SGHMC} mimics the recursion~\eqref{SGHMC 1} and~\eqref{SGHMC 2} at grid points in the sense that $\mathcal{L}\left({\theta}_n^\eta, {V}_n^\eta\right) = \mathcal{L}\left(\bar{\theta}_n^\eta, \bar{V}_n^\eta\right)$, for each $n \in \mathbb{N}_0$. 

Furthermore, we define the auxiliary process $\left(\hat{\zeta}_t^{s,u,v,\eta}, \hat{Z}_t^{s,u,v,\eta}\right)_{t \geq s}$ as
$$
\begin{aligned}
\mathrm{d} \hat{Z}_t^{s,u,v,\eta} & =-\eta\left(\gamma \hat{Z}_t^{s,u,v,\eta}+h\left(\hat{\zeta}_t^{s,u,v,\eta}\right)\right) \mathrm{d} t+\sqrt{2 \gamma \eta \beta^{-1}} \mathrm{~d} B_t^\eta, \\
\mathrm{d} \hat{\zeta}_t^{s,u,v,\eta} & =\eta \hat{Z}_t^{s,u,v,\eta} \mathrm{d} t,
\end{aligned}
$$
with initial conditions $\hat{\zeta}_s^{s,u,v,\eta} = u$, $\hat{Z}_s^{s,u,v,\eta} = v$.

\begin{definition}\label{def:1}
    For any $0 < \eta \leq \eta_{\max}$ with $\eta_{\max}$ given in~\eqref{eta_max} and $n \in \N_0$, set $T := \lfloor 1/\eta\rfloor$ and define 
    \begin{equation}\label{auxiliary_process}
        \bar{\zeta}_t^{\eta, n} = \hat{\zeta}_t^{nT,\bar{\theta}_{nT}^{\eta},\bar{V}_{nT}^{\eta},\eta}, \qquad \bar{Z}_t^{\eta, n} = \hat{Z}_t^{nT,\bar{\theta}_{nT}^{\eta},\bar{V}_{nT}^{\eta},\eta}.
    \end{equation}
\end{definition}
We note that, by above Definition~\ref{def:1}, the process $(\bar{\zeta}_t^{\eta, n}, \bar{Z}_t^{\eta, n})_{t \geq n T}$ is the underdamped Langevin process~\eqref{SKLD} starting at time $nT$ with $(\bar{\theta}_{nT}^{\eta} , \bar{V}_{nT}^{\eta})$.

\subsection{Preliminary moment estimates}\label{subsec_5.2}
In this subsection, we provide preliminary results which are essential for proving the main results. To this end, consider the following Lyapunov function defined for any $(\theta, v) \in \mathbb{R}^d \times \mathbb{R}^d$ by 
\begin{equation}\label{Lyapunov}
    \mathcal{V}(\theta, v) = \beta u(\theta) + \frac{\beta}{4}\gamma^2 (\lvert \theta + \gamma^{-1}v \rvert^2 + \lvert\gamma^{-1}v \rvert^2 - \lambda\lvert \theta\rvert^2).
\end{equation}
Denote by $\mu_0$ the probability law  of the initial state $\left(\theta_0,V_0\right)$. Note that ${\mathbb{E}\left[\mathcal{V}(\theta_0, V_0)\right]< \infty}$, since we assume $\mathbb{E}\left[\left|\theta_0\right|^4\right] + \mathbb{E}\left[\left|V_0\right|^4\right]<\infty$.
\begin{lemma}
\label{lemma:3.2}
    Let Assumptions~\ref{asm:A3}-\ref{asm:A5} hold. Then, we have
\begin{equation}
\begin{aligned}
\sup _{t \geq 0} \mathbb{E}\left[\left|\theta_t\right|^2 \right]\leq C_\theta^c:=\frac{\int_{\mathbb{R}^{2 d}} \mathcal{V}(\theta, v) \mu_0(\mathrm{d}\theta, \mathrm{d}v)+\frac{d+A_c}{\lambda}}{\frac{1}{8}(1-2 \lambda) \beta \gamma^2},
\end{aligned}
\end{equation}
\begin{equation}
\begin{aligned}
\sup _{t \geq 0} \mathbb{E}\left[\left|V_t\right|^2\right] \leq C_v^c:=\frac{\int_{\mathbb{R}^{2 d}} \mathcal{V}(\theta, v) \mu_0(\mathrm{d}\theta, \mathrm{d}v)+\frac{d+A_c}{\lambda}}{\frac{1}{4}(1-2 \lambda) \beta}.
\end{aligned}
\end{equation}
\end{lemma}
\begin{proof}
    See Appendix~\ref{prf of lemma:3.2}.
\end{proof}

\begin{lemma}
\label{lemma:3.3}
    Let Assumptions~\ref{asm:A3}-\ref{asm:A5} hold. Then, we have, for any $0 < \eta \leq \eta_{\max}$ with $\eta_{\max}$ given in~\eqref{eta_max}, that
\begin{equation*}
\begin{aligned}
\sup _{t \geq 0} \mathbb{E}\left[\left|\theta_k^{\eta}\right|^2\right] \leq & C_\theta:=\frac{\int_{\mathbb{R}^{2 d}} \mathcal{V}(\theta, v) \mu_0(\mathrm{d}\theta, \mathrm{d} v)+\frac{4(d+A_c)}{\lambda}}{\frac{1}{8}(1-2 \lambda) \beta \gamma^2},\\
\sup _{t \geq 0} \mathbb{E}\left[\left|V_k^{\eta}\right|^2\right] \leq & C_v:=\frac{\int_{\mathbb{R}^{2 d}} \mathcal{V}(\theta, v) \mu_0(\mathrm{d} \theta, \mathrm{d} v)+\frac{4(d+A_c)}{\lambda}}{\frac{1}{4}(1-2 \lambda) \beta}.
\end{aligned}
\end{equation*}
\end{lemma}
\begin{proof}
    See Appendix~\ref{prf of lemma:3.3}.
\end{proof}

\begin{lemma}\label{lemma:3.4}
    Let Assumptions~\ref{asm:A3}-\ref{asm:A5} hold. Moreover, for any $0 < \eta \leq \eta_{\max}$ with $\eta_{\max}$ given in~\eqref{eta_max}, set $T := \lfloor 1/\eta\rfloor$. Then, we have, for any $0 < \eta \leq \eta_{\max}$ with $\eta_{\max}$ defined in~\eqref{eta_max}, that
\begin{equation}\label{lemma3.4.12}
\begin{aligned}
\sup _{n \in \mathbb{N}_0} \sup _{t \in\left[n T,(n+1) T\right]} \mathbb{E}\left[\left|\bar{\zeta}_t^{\eta, n}\right|^2\right] \leq C_\zeta:=\frac{\int_{\mathbb{R}^{2 d}} \mathcal{V}(\theta, v) \mu_0(\mathrm{d}\theta, \mathrm{d} v)+\frac{5\left(d+A_c\right)}{\lambda}}{\frac{1}{8}(1-2 \lambda) \beta \gamma^2}, \\
\sup _{n \in \mathbb{N}_0} \sup _{t \in[n T,(n+1) T]} \mathbb{E}\left[\left|\bar{Z}_t^{\eta, n}\right|^2\right] \leq C_Z:=\frac{\int_{\mathbb{R}^{2 d}} \mathcal{V}(\theta, v) \mu_0(\mathrm{d} \theta, \mathrm{d} v)+\frac{5\left(d+A_c\right)}{\lambda}}{\frac{1}{4}(1-2 \lambda) \beta}.
\end{aligned}
\end{equation}
\end{lemma}
\begin{proof}
    See Appendix~\ref{prf of lemma:3.4}.
\end{proof}

\begin{lemma}\label{lemma:3.5}
Let Assumptions~\ref{asm:A3}-\ref{asm:A5} hold. Then, we obtain, for any $0 \leq \eta \leq \eta_{\max}$ with $\eta_{\max}$ given in~\eqref{eta_max}, that
\begin{equation*}
\begin{aligned}
    \sup _{k \in \mathbb{N}_0} \mathbb{E}\left[\mathcal{V}^2\left(\theta_k^\eta, V_k^\eta\right)\right] \leq \mathbb{E}\left[\mathcal{V}^2\left(\theta_0, v_0\right)\right]+\frac{2 D}{\gamma \lambda}:= C_{\mathcal{V}},
\end{aligned}
\end{equation*}
\begin{equation*}
\begin{aligned}
    \sup _{k \in \mathbb{N}_0} \mathbb{E}\left[\mathcal{V}^2\left(\bar{\zeta}_t^{\eta, k}, \bar{Z}_t^{\eta, k}\right)\right] \leq  C_{\mathcal{V}}^{\prime},
\end{aligned}
\end{equation*}
where $C_{\mathcal{V}}^{\prime}$ is given by~\eqref{c_v_prime}.
\end{lemma}
\begin{proof}
    See Appendix~\ref{prf of lemma:3.5}.
\end{proof}

Next, following~\cite{eberle2019couplings, raginsky2017non}, we provide the contraction property of the underdamped SDE in the semi-metric defined below. For any $(\theta, v),\left(\theta^{\prime}, v^{\prime}\right) \in \mathbb{R}^{d} \times \mathbb{R}^{d} $, we denote by:
$$
\begin{aligned}
& r\left((\theta, v),\left(\theta^{\prime}, v^{\prime}\right)\right)=\alpha_c\left\lvert \theta-\theta^{\prime}\right\rvert + \left\lvert\theta-\theta^{\prime}+\gamma^{-1}\left(v-v^{\prime}\right)\right\rvert, \\
& \rho\left((\theta, v),\left(\theta^{\prime}, v^{\prime}\right)\right)=g\left(r\left((\theta, v),(\theta^{\prime}, v^{\prime})\right)\right) \cdot\left(1+\varepsilon_c \mathcal{V}(\theta, v)+\varepsilon_c \mathcal{V}\left(\theta^{\prime}, v^{\prime}\right)\right),
\end{aligned}
$$
where $\alpha_c$ and $\varepsilon_c$ are appropriately chosen positive constants (see, e.g.,~\cite[Theorem 2.3]{eberle2019couplings}) and $g:\left[0, \infty\right) \rightarrow \left[0, \infty\right)$ is a continuous, non-decreasing concave function such that $g(0)=0, g$ is twice continuously differentiable on $\left(0, R_1\right)$ for some constant $R_1>0$ with right-sided derivative $g_{+}^{\prime}(0)=1$ and left-sided derivative $g_{-}^{\prime}\left(R_1\right)>0$, and $g$ is constant on $\left[R_1, \infty\right)$. For any two probability measures $\mu, \nu$ on $\mathbb{R}^{2 d}$, we define
$$
\mathcal{W}_\rho(\mu, \nu):=\inf _{\left(\theta_1, V_1\right) \sim \mu,\left(\theta_2, V_2\right) \sim \nu} \mathbb{E}\left[\rho\left(\left(\theta_1, V_1\right),\left(\theta_2, V_2\right)\right)\right].
$$
We note that $\rho$ and $\mathcal{W}_\rho$ are semi-metrics but not necessarily metrics.
\begin{proposition}\label{prop:3.6}
Let Assumptions~\ref{asm:A3}-\ref{asm:A5} hold and let the initial value of the scaled underdamped Langevin SDEs $\left(\theta_t, V_t\right)$ and $\left(\theta_t^{\prime}, V_t^{\prime}\right)$ be $\left(\theta_0, V_0\right) \sim \mu$ and $\left(\theta_0^{\prime}, V_0^{\prime}\right) \sim \nu$, respectively. Then, there exist constants $\dot{c}, \dot{C} > 0$ such that, for any $t > 0$ and $1 \leq p \leq 2$, 
$$
W_p\left(\mathcal{L}\left(\theta_t, V_t\right), \mathcal{L}\left(\theta_t^{\prime}, V_t^{\prime}\right)\right) \leq \dot{C}\mathcal{W}_\rho(\mu, \nu)^{1/p} e^{-\dot{c} t}, 
$$
where the constants $\dot{c}$ and $\dot{C}$ are given in Table~\ref{table 3}.
\end{proposition}
\begin{proof}
By using Remark~\ref{rmk:2} and Remark~\ref{rmk:4}, the Assumption 2.1 of~\cite{eberle2019couplings} is satisfied. Hence, the result follows by applying the same argument as that in the proof of~\cite[Theorem 4.1]{chau2022stochastic}.
\end{proof}
\begin{table*}[t]
\renewcommand{\arraystretch}{2}
\centering
\begin{threeparttable}
\scriptsize
\begin{tabular}{@{}ccl@{}} 
\toprule 
\toprule 
\multicolumn{1}{c}{} & \multicolumn{1}{c}{Constant} & \multicolumn{1}{c}{Full expression} \\ 
\midrule
	\multirow{2}{*}{Lemma~\ref{lemma:3.2}}			        &$C_\theta^c$		& $\int_{\mathbb{R}^{2 d}} \mathcal{V}(\theta, v) \mu_0(\mathrm{~d} \theta, \mathrm{d} v)+\frac{d+A_c}{\lambda}\left(\frac{1}{8}(1-2 \lambda) \beta \gamma^2\right)^{ - 1} $		 \\ \cline{2-3}
                                                                &$C_v^c$		& $\int_{\mathbb{R}^{2 d}} \mathcal{V}(\theta, v) \mu_0(\mathrm{~d} \theta, \mathrm{d} v)+\frac{d+A_c}{\lambda}\left(\frac{1}{4}(1-2 \lambda) \beta\right)^{-1}$	  \\ \hline
                                                                
	\multirow{2}{*}{Lemma~\ref{lemma:3.3}}			        &$C_\theta$		& $\int_{\mathbb{R}^{2 d}} \mathcal{V}(\theta, v) \mu_0(\mathrm{~d} \theta, \mathrm{d} v)+\frac{4\left(d+A_c\right)}{\lambda}(\frac{1}{8}(1-2 \lambda) \beta \gamma^2)^{-1}$		  \\ \cline{2-3}
                                                                &$C_v$		& $\int_{\mathbb{R}^{2 d}} \mathcal{V}(\theta, v) \mu_0(\mathrm{~d} \theta, \mathrm{d} v)+\frac{4\left(d+A_c\right)}{\lambda}(\frac{1}{4}(1-2 \lambda) \beta)^{-1}$		  \\ \hline
                                                                
	\multirow{2}{*}{Lemma~\ref{lemma:3.4}}			        &$C_\zeta$		& $\int_{\mathbb{R}^{2 d}} \mathcal{V}(\theta, v) \mu_0(d \theta, d v)+\frac{5\left(d+A_c\right)}{\lambda}(\frac{1}{8}(1-2 \lambda) \beta \gamma^2)^{-1}$  \\ \cline{2-3}
                                                                &$C_Z$		& $\int_{\mathbb{R}^{2 d}} \mathcal{V}(\theta, v) \mu_0(d \theta, d v)+\frac{5\left(d+A_c\right)}{\lambda}(\frac{1}{4}(1-2 \lambda) \beta)^{-1}$		  \\ \hline
                                                                
	\multirow{6}{*}{Lemma~\ref{lemma:3.5}}			        &$C_{\mathcal{V}}$		& $\int_{\mathbb{R}^{2 d}} \mathcal{V}^2(\theta, v) \mu_0(d \theta, d v)+\frac{2 D}{\gamma \lambda}$  \\ \cline{2-3}
                                                                &$C_{\mathcal{V}}^{\prime}$		& $\bigg(C_{\mathcal{V}} + 2\beta\gamma\left(d+A_c\right)\left( u_0+\frac{|h(0)|^2}{2 L}\right) + \gamma \beta\left(8+\frac{3\left(d+A_c\right)}{2}\right) C_{Z} + 2\gamma \beta\left(\gamma^2+\left(d+A_c\right)\left(L+\frac{3\gamma^2}{4}\right)\right)C_\zeta \bigg)e^{2\lambda\gamma}$		 \\ \cline{2-3}
                                                                &\multirow{4}{*}{$D$}		& $(2\gamma A_c + 2\gamma d+6 \max \left\{\frac{5 \gamma d}{\frac{1}{8}(1-2 \lambda)}, \frac{3 \gamma d\left(\gamma^2+5 L_1^2 \mathbb{E}\left[\left(1+\left|X_0\right|\right)^{2(\rho+1)}\right]\right)}{\frac{1}{16}(1-2 \lambda) \gamma^2}\right\})(\mathbb{E}\left[\mathcal{V}\left(\theta_0, V_0\right)\right]+\frac{4\left(A_c+d\right)}{\lambda})$		\\
      														& &$+2 (2+\gamma)\left(\mathbb{E}\left[\left(1+\left|X_0\right|\right)^{2(\rho+1)}\right] L_2^2+\mathbb{E}\left[F_*^2\left(X_0\right)\right]\right)(\mathbb{E}\left[\mathcal{V}\left(\theta_0, V_0\right)\right]+\frac{4\left(A_c+d\right)}{\lambda}) \beta +16 \gamma^2 A_c^2+48 \gamma^2 d^2$ \\& &$  +\bigg(96\left(1+\frac{\gamma}{2}\right)^2\left(|h(0)|^4+4 L_2^4 \mathbb{E}\left[\left(1+\left|X_0\right|\right)^{4(\rho+1)}\right]+4 \mathbb{E}\left[F_*^4\left(X_0\right)\right]\right)+64\left(1+\frac{\gamma}{2}\right)^2$ \\ & &$ \times \bigg(L_2^4 \mathbb{E}\left[\left(1+\left|X_0\right|\right)^{4(\rho+1)}\right]+\mathbb{E}\left[F_*^4\left(X_0\right)\right]\bigg)\bigg) \beta^2+\left(90 \gamma d L_2^2 \mathbb{E}\left[\left(1+\left|X_0\right|\right)^{2(\rho+1)}\right]+90 \gamma d \mathbb{E}\left[F_*^2\left(X_0\right)\right]\right) \beta$ \\ \hline

	\multirow{6}{*}{Proposition~\ref{prop:3.6}}			     &$\dot{c}$		& $\frac{\gamma}{384 p} \min \left\{\lambda_c L \gamma^{-2}, \Lambda_c^{1 / 2} e^{-\Lambda_c} L \gamma^{-2}, \Lambda_c^{1 / 2} e^{-\Lambda_c}\right\}$  \\ \cline{2-3}			                 &$\dot{C}$		& $2^{1 / p} e^{2 / p+\Lambda_c / p} \frac{1+\gamma}{\min \left\{1, \alpha_c\right\}}\left(\max \left\{1,4 \frac{\max \left\{1, R_1^{p-2}\right\}}{\min \left\{1, R_1\right\}}\left(1+2 \alpha_c+2 \alpha_c^2\right)\left(d+A_c\right) \beta^{-1} \gamma^{-1} \dot{c}^{-1}\right\}\right)^{1 / p}$  \\ \cline{2-3}
 
                                                                &$\Lambda_c$		& $\frac{12}{5}\left(1+2 \alpha_c+2 \alpha_c^2\right)\left(d+A_c\right) L \gamma^{-2} \lambda^{-1}(1-2 \lambda)^{-1}$  \\ \cline{2-3}
                                                                
                                                                &$\varepsilon_c$		& $\gamma^{-1} \dot{c} /\left(d+A_c\right)$  \\ \cline{2-3}

                                                                &$R_1$		& $4 \cdot(6 / 5)^{1 / 2}\left(1+2 \alpha_c+2 \alpha_c^2\right)^{1 / 2}\left(d+A_c\right)^{1 / 2} \beta^{-1 / 2} \gamma^{-1}\left(\lambda-2 \lambda^2\right)^{-1 / 2}$  \\ \cline{2-3}
                                                                
                                                                &$\alpha_c$		& $\left(1+\Lambda_c^{-1}\right) L \gamma^{-2}>0$		  \\ 
                                                                \bottomrule
                                                                 \bottomrule
    \end{tabular}
    \end{threeparttable}
\caption{Explicit expressions of the constants in Section~\ref{subsec_5.2}.}
\label{table 3}
\end{table*}

\subsection{Proof of main results}\label{sec:5}
\noindent To provide a non-asymptotic upper bound for $W_p\left(\mathcal{L}\left(\bar{\theta}_t^\eta, \bar{V}_t^\eta\right), \bar{\pi}_\beta\right)$ for $p = 1, 2$, we first consider the following splitting using the scaled process given in~\eqref{SKLD}: for any $n \in \mathbb{N}_0$ and $t \in(n T,(n+1) T]$, 
\begin{equation}\label{sec4:1}
\begin{aligned}
W_p\left(\mathcal{L}\left(\bar{\theta}_t^\eta, \bar{V}_t^\eta\right), \bar{\pi}_\beta\right) & \leq  W_p\left(\mathcal{L}\left(\bar{\theta}_t, \bar{V}_t^\eta\right), \mathcal{L}\left(\bar{\zeta}_t^{\eta, n}, \bar{Z}_t^{\eta, n}\right)\right)+W_p\left(\mathcal{L}\left(\bar{\zeta}_t^{\eta, n}, \bar{Z}_t^{\eta, n}\right), \mathcal{L}\left(\zeta_t^\eta, Z_t^\eta\right)\right) \\ & \quad + W_p\left(\mathcal{L}\left(\zeta_t^\eta, Z_t^\eta\right), \bar{\pi}_\beta\right).
\end{aligned}
\end{equation}
Note that $\left(\bar{\theta}_t^\eta, \bar{V}_t^\eta\right)$ defined in~\eqref{cont.SGHMC} can be viewed as a continuous version of the first order Euler scheme of $\left(\bar{\zeta}_t^{\eta, n}, \bar{Z}_t^{\eta, n}\right)$ with stochastic gradient. We consider the corresponding $L^2$-distance and employ the synchronous coupling to obtain an estimate for the first term on the RHS of~\eqref{sec4:1}. To upper bound the second term on the RHS of~\eqref{sec4:1}, we view $\mathcal{L}\left(\bar{\zeta}_t^{\eta, n}, \bar{Z}_t^{\eta, n}\right)$ and $\mathcal{L}\left(\zeta_t^\eta, Z_t^\eta\right)$ as the laws of the time-changed process~\eqref{SKLD} starting from different initial points. Then, applying the contraction result in~\cite[Lemma 5.4]{chau2022stochastic} and letting $p=1, 2$ yield the result. We note that, as $\bar{\pi}_\beta$ is the invariant measure of the second-order (underdamped) Langevin SDE in~\eqref{underdamped Langevin SDE}, $\mathcal{L}\left(\zeta_t^\eta, Z_t^\eta\right)$ and $\bar{\pi}_\beta$ can also be viewed as the laws of the process~\eqref{underdamped Langevin SDE} from different initial points, i.e., from $(\theta_0, V_0)$ and a random variable distributed according to $\bar{\pi}_\beta$, respectively. Therefore, the last term on the RHS of~\eqref{sec4:1} can be upper bounded by the same approach as that for bounding the second term.

We establish upper bounds for each of the terms on the RHS of~\eqref{sec4:1}. The results are presented below.

\begin{proposition}\label{lemma:4.1}
Let Assumptions~\ref{asm:A3}-~\ref{asm:A5} hold. Then, for any $0 < \eta \leq \eta_{\max}$ with $\eta_{\max}$ given in~\eqref{eta_max}, $n \in \N_0$, and  $t \in\left[n T,(n+1) T \right)$ with $T := \lfloor 1/\eta\rfloor$, we obtain
$$
W_2\left(\mathcal{L}\left(\bar{\theta}_t^\eta, \bar{V}_t^\eta\right), \mathcal{L}\left(\bar{\zeta}_t^{\eta, n}, \bar{Z}_t^{\eta, n}\right)\right) \leq C_1^{\star} \eta^{1 / 2},
$$
where $C_1^{\star}$ is given explicitly in Table~\ref{table 1}.
\end{proposition}
\begin{proof}
    See Appendix~\ref{prf of lemma 4.1}.
\end{proof}

\begin{proposition}\label{lemma:4.2}
Let Assumptions~\ref{asm:A3}-~\ref{asm:A5} hold. Then, for any $0 < \eta \leq \eta_{\max}$ with $\eta_{\max}$ given in~\eqref{eta_max}, $n \in \N_0$, and  $t \in\left[n T,(n+1) T \right)$ with $T := \lfloor 1/\eta\rfloor$, we obtain
$$
W_1\left(\mathcal{L}\left(\bar{\zeta}_t^{\eta, n}, \bar{Z}_t^{\eta, n}\right), \mathcal{L}\left(\zeta_t^\eta, Z_t^\eta\right)\right) \leq C^{\ast} \eta^{1 / 2},\quad
W_2\left(\mathcal{L}\left(\bar{\zeta}_t^{\eta, n}, \bar{Z}_t^{\eta, n}\right), \mathcal{L}\left(\zeta_t^\eta, Z_t^\eta\right)\right) \leq C_2^{\star} \eta^{1 / 4},
$$
where $C^{\ast}$ is given in~\eqref{C ast} and $C_2^{\star}$ is given explicitly in Table~\ref{table 1}.
\end{proposition}
\begin{proof}
    See Appendix~\ref{prf of lemma 4.2}.
\end{proof}
\begin{proposition}\label{lemma:4.3}
Let Assumptions~\ref{asm:A3}-~\ref{asm:A5} hold. Then, for any $0 < \eta \leq \eta_{\max}$ with $\eta_{\max}$ given in~\eqref{eta_max}, $n \in \N_0$, and  $t \in\left[n T,(n+1) T \right)$ with $T := \lfloor 1/\eta\rfloor$, we have
$$
W_1\left(\mathcal{L}\left(\zeta_t^\eta, Z_t^\eta\right), \bar{\pi}_\beta\right) \leq C_2^{\ast} e^{-C_3^{\ast} \eta t}, \quad
W_2\left(\mathcal{L}\left(\zeta_t^\eta, Z_t^\eta\right), \bar{\pi}_\beta\right) \leq C_3^{\star} e^{-C_4^{\star} \eta t},
$$
where $C_2^{\ast}$, $C_3^{\ast}$, $C_3^{\star}$, and $C_4^{\star}$ are explicitly given in Table~\ref{table 1}. 
\end{proposition}
\begin{proof}
    See Appendix~\ref{prf of lemma 4.3}.
\end{proof}
\begin{proof}[\textbf{\textit{Proof of Theorem~\ref{thm:2.1} and Corollary~\ref{cor:2.2}}}]
First, by using Proposition~\ref{lemma:4.1},~\ref{lemma:4.2}, and~\ref{lemma:4.3}, we obtain, for any $n \in \N_0$ and $t \in[n T,(n+1) T)$ with $T := \lfloor 1/\eta\rfloor$, that
$$
W_2\left(\mathcal{L}\left(\bar{\theta}_{t}^\eta, \bar{V}_{t}^\eta\right), \bar{\pi}_\beta\right) \leq C_1^{\star} \eta^{1 / 2} + C_2^{\star} \eta^{1 / 4}+ C_3^{\star} e^{-C_4^{\star} n \eta T},
$$
which implies
$$
W_2\left(\mathcal{L}\left(\bar{\theta}_{n T}^\eta, \bar{V}_{n T}^\eta\right), \bar{\pi}_\beta\right) \leq C_1^{\star} \eta^{1 / 2} + C_2^{\star} \eta^{1 / 4}+ C_3^{\star} e^{-C_4^{\star} n \eta T}.
$$
To obtain a non-asymptotic upper bound for the SGHMC algorithm, we set $n T$ to $n$ on the left-hand side of the above inequality, while $n$ on the RHS of the above inequality is set to $n / T$. Hence, for any $n \in \mathbb{N}_0, 0<\eta \leq \eta_{\max }$ with $\eta_{\max}$ given in~\eqref{eta_max}, one writes
\begin{equation}\label{W_2}
W_2\left(\mathcal{L}\left(\bar{\theta}_{n}^\eta, \bar{V}_{n}^\eta\right), \bar{\pi}_\beta\right) \leq C_1^{\star} \eta^{1 / 2} + C_2^{\star} \eta^{1 / 4}+ C_3^{\star} e^{-C_4^{\star} n \eta}.
\end{equation}
Similarly, for Wasserstein-1 distance, we obtain
\begin{equation}\label{W_1}
W_1\left(\mathcal{L}\left(\bar{\theta}_{n}^\eta, \bar{V}_{n}^\eta\right), \bar{\pi}_\beta\right) \leq \left( C^* + C_1^{\star} \right)\eta^{1 / 2} + C_2^{\ast} e^{-C_3^* \eta n}.
\end{equation}
Moreover, for any $\epsilon>0$, if we choose $\eta$ and $n$ such that $\eta\leq\eta_{\max}$, $ C_1^{\star} \eta^{1 / 2}\leq \epsilon / 3$, $ C_2^{\star} \eta^{1 / 4} \leq \epsilon / 3$, and $C_3^{\star} e^{-C_4^{\star} n \eta}\leq \epsilon / 3$, where $\eta_{\max}$ is given in~\eqref{eta_max}, then $W_2\left(\mathcal{L}\left(\bar{\theta}_{n}^\eta, \bar{V}_{n}^\eta\right), \bar{\pi}_\beta\right)\leq \epsilon$. This further implies that $\eta \leq \min\left\{\frac{\epsilon^2}{9 C_1^{\star 2}}, \frac{\epsilon^4}{81 C_2^{\star 4}}, \eta_{\max}\right\}$ and $ n \eta\geq \frac{\ln(3C_3^{\star} / \epsilon)}{C_4^{\star}}$. Therefore, one can write $n \geq \frac{\ln(3C_3^{\star} / \epsilon)}{C_4^{\star} \min\left\{\frac{\epsilon^2}{9 C_1^{\star 2}}, \frac{\epsilon^4}{81 C_2^{\star 4}}, \eta_{\max}\right\}}$. Similarly, for any $\epsilon>0$, if we choose $\eta$ and $n$ such that $\eta\leq \min\left\{\frac{\epsilon^2}{4\left(C^* + C_1^{\star}\right)^2}, \eta_{\max}\right\}$ and $n \geq \max\left\{\frac{4\left(C^* + C_1^{\star}\right)^2 \ln(2 C_2^{\ast} / \epsilon)}{C_3^* \epsilon^2 }, \frac{\ln(2 C_2^{\ast} / \epsilon)}{C_3^* \eta_{\max}}\right\}$, then $W_1\left(\mathcal{L}\left(\bar{\theta}_{n}^\eta, \bar{V}_{n}^\eta\right), \bar{\pi}_\beta\right) \leq \epsilon$.
\end{proof}

To provide a non-asymptotic error bound for the expected excess risk, we proceed with the following decomposition:
\begin{equation}\label{EER}
\mathbb{E}\left[u\left(\theta_n^\eta\right)\right]-\inf _{\theta \in \mathbb{R}^d} u(\theta)=\underbrace{\mathbb{E}\left[u\left(\theta_n^\eta\right)\right]-\mathbb{E}\left[u\left(\theta_{\infty}\right)\right]}_{\mathcal{T}_1}+\underbrace{\mathbb{E}\left[u\left(\theta_{\infty}\right)\right]-\inf _{\theta \in \mathbb{R}^d} u(\theta)}_{\mathcal{T}_2},
\end{equation}
where $\theta_{\infty}$ is distributed according to $\pi_\beta$. We note that $\mathcal{T}_1$ depends on the sampling behavior of the SGHMC algorithm, which relates to the error in Wasserstein-2 distance between the law of the SGHMC algorithm and $\bar{\pi}_\beta$. $\mathcal{T}_2$ is about the concentration property of $\bar{\pi}_\beta$ and becomes small when the temperature parameter $\beta > 0$ is large. The following proposition presents a bound for $\mathcal{T}_1$ under our assumptions.
\begin{proposition}\label{lemma:4.4}
Let Assumptions~\ref{asm:A3}-\ref{asm:A5} hold. Then, for every $\beta > 0$, there exist constants $\bar{C}_1^{\star}, \bar{C}_2^{\star}, \bar{C}_3^{\star}, \bar{C}_4^{\star} > 0$ such that, for any $n \in \mathbb{N}_0$ and $0 \leq \eta \leq \eta_{\max}$ with $\eta_{\max}$ given in~\eqref{eta_max}, we have
$$
\mathbb{E}\left[u\left(\theta_n^\eta\right)\right]-\mathbb{E}\left[u\left(\theta_{\infty}\right)\right] \leq \bar{C}_1^{\star} \eta^{1 / 2}+\bar{C}_2^{\star} \eta^{1 / 4}+\bar{C}_3^{\star} e^{-\bar{C}_4^{\star} \eta n},
$$
where $\theta_{\infty} \sim \pi_\beta$ and $\bar{C}_i^{\star}:=C_i^{\star}\left(L C_m +\lvert h(0)\rvert\right)$ with $C_m:=\max \left(C_\theta^c, C_\theta\right)$, for $i = 1, 2, 3$ and ${\bar{C}_4^{\star}:= C_4^{\star}}$, and where $C_1^{\star}$, $C_2^{\star}$, $C_3^{\star}$, $C_4^{\star}$ are given explicitly in Table~\ref{table 1} while $C_\theta^c$ and $C_\theta$ are given explicitly in Table~\ref{table 3}.
\end{proposition}
\begin{proof}
    See Appendix~\ref{prf of lemma 4.4}.
\end{proof}
Next, we bound the second term $\mathcal{T}_2$ as follows.
\begin{proposition}\label{lemma:4.5}
Let Assumptions~\ref{asm:A3}-\ref{asm:A5} hold. Then, for any $\beta > 0$, we have
$$
\mathbb{E}\left[u\left(\theta_{\infty}\right)\right]-\inf _{\theta \in \mathbb{R}^d} u(\theta) \leq \frac{d}{2 \beta} \log \left(\frac{8 e L}{\gamma^2 \lambda(1 - 2\lambda)}\left(\frac{A_c}{d}+1\right)\right),
$$
where $A_c$ is explicitly given in Remark~\ref{rmk:4}.
\end{proposition}
\begin{proof}
    See Appendix~\ref{prf of lemma 4.5}.
\end{proof}
\begin{proof}[\textbf{\textit{Proof of Theorem~\ref{thm:2.2}}}]
Substituting the results in Proposition~\ref{lemma:4.4} and Proposition~\ref{lemma:4.5} into~\eqref{EER} yields the desired non-asymptotic upper bound. Moreover, for any $\epsilon>0$, if we choose $\beta$, $\eta$, and $n$ such that $\eta\leq\eta_{\max}$, $ \bar{C}_1^{\star} \eta^{1 / 2}\leq \epsilon / 4$, $ \bar{C}_2^{\star} \eta^{1 / 4} \leq \epsilon / 4$, $\bar{C}_3^{\star} e^{-\bar{C}_4^{\star} n \eta}\leq \epsilon / 4$, and $\frac{d}{2 \beta} \log \left(\frac{8 e L}{\gamma^2 \lambda(1 - 2\lambda)}\left(\frac{A_c}{d} + 1\right)\right) \leq \epsilon / 4$, where $\eta_{\max}$ is given in~\eqref{eta_max}, then $\mathbb{E}\left[u\left(\theta_{\infty}\right)\right]-\inf _{\theta \in \mathbb{R}^d} u(\theta) \leq \epsilon$. 
Calculations yield that $\eta \leq \min\left\{\frac{\epsilon^2}{16\bar{C}_1^{\star 2}}, \frac{\epsilon^4}{256\bar{C}_2^{\star 4}}, \eta_{\max}\right\}$ and $\eta n \geq \frac{\ln(4 \bar{C}_3^{\star} / \epsilon)}{\bar{C}_4^{\star}}$, which implies that $n \geq \frac{\ln(4 \bar{C}_3^{\star} / \epsilon)}{\bar{C}_4^{\star}\min\left\{\frac{\epsilon^2}{16\bar{C}_1^{\star 2}}, \frac{\epsilon^4}{256\bar{C}_2^{\star 4}}, \eta_{\max}\right\}}$. Furthermore, we recall that $A_c:=\beta (\lambda u(0) + \lambda L\lvert h(0)\rvert + b^{\prime} / 2)$, then, ${\frac{d}{2 \beta} \log \left(\frac{8 e L}{\gamma^2 \lambda(1 - 2\lambda)}\left(\frac{A_c}{d} + 1\right)\right) \leq \epsilon / 4}$ is achieved if we choose $\beta \geq \max\left\{\frac{16 d^2}{\epsilon^2}, \frac{4d}{\epsilon} \log\left(\frac{8eL}{\gamma^2 \lambda(1 - 2\lambda)}\left(\frac{\lambda u(0) + \lambda L\lvert h(0)\rvert + b^{\prime} / 2}{d} + 1\right) \right)\right\}$. Indeed, for any $\beta > 0$, we have
$$
\begin{aligned}
&\frac{d}{2 \beta} \log \left(\frac{8 e L}{\gamma^2 \lambda(1 - 2\lambda)}\left(\frac{A_c}{d} + 1\right)\right) \\& \leq \frac{d}{2 \beta} \log \left(\frac{8 e L}{\gamma^2 \lambda(1 - 2\lambda)}\left(\frac{\lambda u(0) + \lambda L\lvert h(0)\rvert + b^{\prime} / 2}{d} + 1\right)\right) + \frac{d}{2 \beta} \log \left(\beta + 1\right)\\ & \leq 
\frac{1}{\beta} \left(\frac{d}{2} \log\left(\frac{8 e L}{\gamma^2 \lambda(1 - 2\lambda)}\left(\frac{\lambda u(0) + \lambda L\lvert h(0)\rvert + b^{\prime} / 2}{d} + 1\right)\right)\right) + \frac{d}{2 \sqrt{\beta}} \\ & \leq 
\epsilon / 8 + \epsilon / 8 = \epsilon / 4,
\end{aligned}
$$
where the second inequality holds due to $\log\left(\beta+ 1 \right) / \beta \leq 1 / \sqrt{\beta + 1} \leq 1 / \sqrt{\beta } $ for any $\beta > 0$.

\end{proof}

\appendix


\newpage

\section{Proofs of Section~\ref{sec:2}}\label{Appendix: B}
\begin{proof}[\textit{Proof of Remark~\ref{rmk:1}}]\label{prf of rmk:1}
By using Assumption~\ref{asm:A3}, we obtain, for any $\left(\theta, x\right) \in \mathbb{R}^d \times \mathbb{R}^m$, 
\begin{equation*}
\begin{split}
    \lvert H(\theta, x) - H(0, 0)\rvert\leq & \lvert F(\theta, x) - F(0, 0)\rvert + \lvert G(\theta, x) - G(0, 0)\rvert\\
    \leq & (1 + \lvert x\rvert)^{\rho} (L_1\lvert\theta\rvert + L_2\lvert x\rvert) + \bar{K}_1(0) + \bar{K}_1(x)\\
    \leq & (1 + \lvert x\rvert) ^{\rho + 1} (L_1\lvert\theta\rvert + L_2) + \bar{K}_1(0) + \bar{K}_1(x).\\
\end{split}
\end{equation*}
Hence, letting $F_{\ast}(x) :=  2  \bar{K}_1(0) + \bar{K}_1(x) + \lvert F(0, 0)\rvert$, we have
\begin{equation*}
\begin{split}
    \lvert H(\theta, x)\rvert\leq &(1 + \lvert x\rvert) ^{\rho + 1} (L_1\lvert\theta\rvert + L_2) + 2 \bar{K}_1(0) + \bar{K}_1(x) + \lvert F(0, 0)\rvert\\
    = & (1 + \lvert x\rvert) ^{\rho + 1} (L_1\lvert\theta\rvert + L_2) + F_{\ast}(x).\\
\end{split}
\end{equation*}
\end{proof}

\begin{proof}[\textit{Proof of Remark~\ref{rmk:3}}]
By Assumption~\ref{asm:A3} and~\ref{asm:A5}, for any $(\theta, x) \in \mathbb{R}^d \times \mathbb{R}^m$,  we have 
\begin{equation*}
\begin{split}
    \langle \theta, H(\theta, x)\rangle =& \langle\theta, F(\theta, x)\rangle + \langle\theta, G(\theta, x)\rangle \\
    \geq & \langle\theta, A(x)\theta\rangle - B(x) - \bar{K}_1(x)\lvert\theta\rvert.
\end{split}
\end{equation*}
Then, by using Assumption~\ref{asm:A2} and Cauchy-Schwarz inequality, we obtain
\begin{equation*}
\begin{split}
    \langle\theta, h(\theta)\rangle \geq & a\lvert\theta\rvert^2 - b- \mathbb{E}\left[\bar{K}_1(X_0)\right] \lvert\theta\rvert\\
    \geq & a\lvert\theta\rvert^2 - b - \frac{a}{2} \lvert\theta\rvert^2 - \frac{1}{2a}(\mathbb{E}\left[(\bar{K}_1(X_0))\right])^2\\
    := &  a^{\prime}\lvert\theta\rvert^2 - b^{\prime},
\end{split}
\end{equation*}
where $a^{\prime} := \frac{a}{2}$  and  $b^{\prime} := b + \frac{1}{2a}(\mathbb{E}\left[(\bar{K}_1(X_0))\right])^2$.
\end{proof}

\begin{proof}[\textit{Proof of Remark~\ref{rmk:4}}]
Denote by $\Bar{u}(t):=u(t\theta), t \in [0, 1]$. Note that for any $\theta \in \mathbb{R}^d$, $\Bar{u}^{\prime}(t):=\langle\theta, h(t\theta)\rangle$, and 
\begin{equation}\label{u_delta}
\begin{split}
u(\theta) -u(0) = &\Bar{u}(1) - \Bar{u}(0) \\= &\int_{0}^{1} \langle\theta, h(t\theta)\rangle \mathrm{d}t\\
 \leq & \int_{0}^{1} \lvert\theta\rvert \lvert h(t\theta)\rvert \mathrm{d}t\\
 \leq & \int_{0}^{1} \lvert\theta\rvert L(t\lvert \theta\rvert + \lvert h(0)\rvert) \mathrm{d}t\\
 = & \frac{L}{2}\lvert \theta\rvert^2 + L\lvert h(0)\rvert\lvert \theta\rvert,
\end{split}
\end{equation}
where the second inequality holds due to Remark~\ref{rmk:2}. Define $\lambda := \min\left\{\frac{1}{4}, \frac{a^{\prime}}{L + 2L\lvert h(0)\rvert + \frac{\gamma^2}{2}}\right\}$. By using~\eqref{u_delta} and Remark~\ref{rmk:3}, we have, for any $\theta \in \mathbb{R}^d$, that
\begin{equation*}
\begin{split}
\langle h(\theta), \theta\rangle & \geq  a^{\prime}\lvert\theta\rvert^2 - b^{\prime} \\ &  \geq2\lambda\left(\frac{L}{2} + L\lvert h(0)\rvert + \frac{\gamma^2}{4}\right)\lvert\theta\rvert^2 - b^{\prime} \\ & 
= 2\lambda\left(\frac{L}{2}\lvert\theta\rvert^2 + L\lvert h(0)\rvert\lvert\theta\rvert -L\lvert h(0)\rvert\lvert\theta\rvert + L\lvert h(0)\rvert\lvert\theta\rvert^2 + \frac{\gamma^2}{4}\lvert\theta\rvert^2\right) - b^{\prime} \\ &  
\geq 2\lambda\left(u(\theta) - u(0) - L\lvert h(0)\rvert\lvert\theta\rvert + L\lvert h(0)\rvert\lvert\theta\rvert^2 + \frac{\gamma^2}{4}\lvert\theta\rvert^2\right) - b^{\prime}\\  &  
\geq 2\lambda\left(u(\theta) - u(0) - L\lvert h(0)\rvert + \frac{\gamma^2}{4}\lvert\theta\rvert^2\right) - b^{\prime}\\ & 
= 2\lambda\left(u(\theta)+ \frac{\gamma^2}{4}\lvert\theta\rvert^2\right) - 2\frac{A_c}{\beta}
\end{split}
\end{equation*}
with $A_c := \frac{\beta}{2} (2\lambda u(0) + 2\lambda L\lvert h(0)\rvert + b^{\prime})$, where the last inequality holds due to $\lvert \theta \rvert \leq 1 +\lvert \theta \rvert^2$.  
\end{proof}
\section{Proofs of Section~\ref{sec:3}}
\begin{proof}[\textit{Proof of Lemma~\ref{lemma:3.2}}]\label{prf of lemma:3.2}
Recall the Lyapunov function defined in~\eqref{Lyapunov}, i.e., for any $\theta, v \in \R^d$,  
\begin{equation*}
    \mathcal{V}(\theta, v) = \beta u(\theta) + \frac{\beta}{4}\gamma^2 (\lvert \theta + \gamma^{-1}v \rvert^2 + \lvert\gamma^{-1}v \rvert^2 - \lambda\lvert \theta\rvert^2).
\end{equation*}
Now we define the operator 
\begin{equation}\label{operator}
\mathcal{A} := -\langle\gamma v + h(\theta), \nabla_v\rangle + \gamma\beta^{-1}\Delta_v + \langle v, \nabla_{\theta}\rangle. 
\end{equation} Following similar arguments as in~\cite[Lemma 2.2]{eberle2019couplings}, by using Assumption~\ref{asm:A5}, Remark~\ref{rmk:2}, and Remark~\ref{rmk:4}, we have
\begin{equation*}
    \mathcal{A}\mathcal{V} \leq \gamma(d + A_c - \lambda \mathcal{V}) .
\end{equation*}
Indeed, by using~\eqref{operator}, we have that $\mathcal{A} u(\theta) = \langle v, h(\theta) \rangle$, $\mathcal{A} \lvert\theta\rvert^2 = 2\langle\theta, v\rangle$, and $$\mathcal{A}\left( \frac{1}{2}\lvert\gamma^{-1}v\rvert^2\right) = -\frac{1}{2}\langle\gamma v + h(\theta), \nabla_v \lvert\gamma^{-1}v\rvert^2\rangle + \frac{\gamma}{2\beta}\Delta_v\lvert\gamma^{-1}v\rvert^2 = \frac{d}{\gamma\beta} - \frac{\lvert v\rvert^2}{\gamma} - \frac{\langle h(\theta), v\rangle}{\gamma^2},$$ $$\mathcal{A} \left(\frac{1}{2}\lvert\theta + \gamma^{-1}v\rvert^2\right) = \frac{d}{\gamma\beta} - \frac{\langle h(\theta), v\gamma^{-1} + \theta\rangle}{\gamma}.$$
This implies that
\begin{equation}\label{lm4.2:2}
\begin{aligned}
    \mathcal{A}\mathcal{V} (\theta, v)& = \beta \langle v,  h(\theta)\rangle + \frac{\beta\gamma^2}{2}\left(\frac{d}{\gamma\beta} - \langle h(\theta), (v\gamma^{-1} + \theta)\gamma^{-1}\rangle + \frac{d}{\gamma\beta}- \frac{\lvert v\rvert^2}{\gamma} - \frac{\langle h(\theta), v\rangle}{\gamma^2} - \lambda\langle \theta, v\rangle\right) \\&
    \leq  \gamma\left(d + A_c - \lambda \beta u(\theta) - \frac{1}{4}\lambda\beta\gamma^2\left(\lvert\theta\rvert^2 + 2\lambda^{-1}\lvert\gamma^{-1}v\rvert^2 + 2\langle\theta, \gamma^{-1}v\rangle\right)\right)\\&
    \leq  \gamma\left(d + A_c - \lambda \beta u(\theta) - \frac{1}{4}\lambda\beta\gamma^2\left((1 - \lambda)\lvert\theta\rvert^2 + 2\lvert\gamma^{-1}v\rvert^2 + 2\langle\theta, \gamma^{-1}v\rangle\right)\right)
    \\ &
    = \gamma(d + A_c - \lambda \mathcal{V}(\theta, v)),\\
\end{aligned}
\end{equation}
where the first inequality holds due to Remark~\ref{rmk:4}, and the second inequality holds due to $0 < \lambda \leq \frac{1}{4}$. Moreover, by applying Itô formula, we obtain
\begin{equation*}
\begin{aligned}
    \mathrm{d}\left(e^{\lambda\gamma t}\mathcal{V}(\theta_t, V_t)\right)= \lambda\gamma e^{\lambda\gamma t}\mathcal{V}(\theta_t, V_t)\mathrm{d}t +  e^{\lambda\gamma t}\mathcal{A}\mathcal{V}(\theta_t, V_t) \mathrm{d}t + e^{\lambda\gamma t}\left(\beta V_t + \frac{\beta\gamma}{2}\theta_t\right)\sqrt{2\gamma\beta^{-1}}\mathrm{d}B_t.
\end{aligned}
\end{equation*}
Hence,
\begin{equation}\label{Ineq:1}
\begin{aligned}
    e^{\lambda\gamma t}\mathcal{V}(\theta_t, V_t)& \leq \mathcal{V}(\theta_0, V_0) + \gamma(d + A_c)\int_{0}^{t} e^{\lambda\gamma s}\mathrm{d}t -\int_{0}^{t}e^{\lambda\gamma s}\left(\beta V_s + \frac{\beta\gamma}{2}\theta_s\right)\sqrt{2\gamma\beta^{-1}}\mathrm{d}B_s.
\end{aligned}
\end{equation}
Note that by Remark~\ref{rmk:2}, SDE~\eqref{underdamped Langevin SDE} has a unique strong solution. By~\cite[Theorem 5.2.1]{oksendal2013stochastic}, for every $T > 0$, we have that $$\mathbb{E}\left[\int_{0}^{T}\left(\lvert V_s\rvert^2 + \lvert \theta_s\rvert^2\right)\mathrm{d}s\right] < \infty.$$ Hence, 
\begin{equation*}
\begin{aligned}
    \mathbb{E}\left[\int_{0}^{T}e^{2\lambda\gamma s}\left\lvert\beta V_s + \frac{\beta\gamma}{2}\theta_s\right\rvert^2\frac{2\gamma}{\beta} \mathrm{d}s\right] < \infty, 
\end{aligned}
\end{equation*}
implying $$\mathbb{E}\left[\int_{0}^{T}e^{\lambda\gamma s}\left(\beta\gamma(s) + \frac{\beta\gamma}{2}\theta(s)\right)\sqrt{\frac{2\gamma}{\beta}}dB_s \right] = 0.$$ 
For each $t > 0$, denote $\mathbb{L}(t) := \mathbb{E}\left[\mathcal{V}(\theta_t, V_t)\right]$. By using~\eqref{Ineq:1}, we obtain 
\begin{equation}\label{Ineq:2}
\begin{aligned}
\mathbb{L}(t) \leq \mathbb{L}(0)e^{-\lambda\gamma t} + \frac{d + A}{\lambda}\left(1 - e^{-\lambda\gamma t}\right).
\end{aligned}
\end{equation}
In addition, note that
\begin{equation}\label{func:V}
\begin{aligned}
    \mathcal{V} (\theta, v) & = \beta u(\theta) + \frac{\beta}{4}\gamma^2 (\lvert \theta + \gamma^{-1}v \rvert^2 + \lvert\gamma^{-1}v \rvert^2 - \lambda\lvert \theta+\gamma^{-1}v - \gamma^{-1}v\rvert^2)\\&
    \geq \beta u(\theta) + \frac{\beta}{4}\gamma^2 (\lvert \theta + \gamma^{-1}v \rvert^2 + \lvert\gamma^{-1}v \rvert^2 - 2\lambda (\lvert\theta + \gamma^{-1}v\rvert^2 + \lvert\gamma^{-1}v\rvert^2)\\&
    = \beta u(\theta) + \frac{\beta}{4}\gamma^2(1 - 2\lambda)(\lvert \theta + \gamma^{-1}v \rvert^2 + \lvert\gamma^{-1}v \rvert^2)\\&
    \geq \max\left\{ \frac{\beta\gamma^2}{8}(1 - 2\lambda)\lvert\theta\rvert^2, \frac{\beta\gamma^2}{4}(1 - 2\lambda)\lvert\gamma^{-1}v\rvert^2\right\}.
\end{aligned}
\end{equation}
Thus, combining the above result~\eqref{func:V} with~\eqref{Ineq:2} yields
\begin{equation*}
\begin{aligned}
& \frac{ \beta}{8}(1-2 \lambda) \gamma^2 \mathbb{E}\| \theta_t\|^2 \leq \mathbb{E}\left[\mathcal{V}\left(\theta_0, V_0\right)\right]+\frac{d+A_c}{\lambda},\\
& \frac{\beta}{4}(1-2 \lambda) \mathbb{E}\|V_t\|^2 \leq \mathbb{E}\left[\mathcal{V}\left(\theta_0, V_0\right)\right]+\frac{d+A_c}{\lambda}.
\end{aligned}
\end{equation*}
\end{proof}

\begin{proof}[\textit{Proof of Lemma~\ref{lemma:3.3}}]\label{prf of lemma:3.3}
By using Remark~\ref{rmk:1} and the fact that $X_{k+1} $ is independent of $\theta_k^{\eta}$ and $V_k^{\eta}$, we have
\begin{equation}\label{Ineq:3}
\begin{aligned}    
\mathbb{E}\left[\lvert H(\theta_k^{\eta}, X_{k+1})\rvert^2\right] &\leq 2
\mathbb{E}\left[(1 + \lvert X_{k+1}\rvert)^{2(\rho + 1)}L_1^2\lvert\theta_k^{\eta}\rvert^2\right] + 4\mathbb{E}\left[(1 + \lvert X_{k+1}\rvert)^{2(\rho + 1)}L_2^2 + F_{\ast}^2(X_{k+1})\right]\\
&= 2
\mathbb{E}\left[\lvert\theta_k^{\eta}\rvert^2\mathbb{E}\left[(1 + \lvert X_{k+1}\rvert)^{2(\rho + 1)}L_1^2 \bigg| \theta_k^{\eta}\right]\right] + 4\mathbb{E}\left[(1 + \lvert X_{0}\rvert)^{2(\rho + 1)}L_2^2 + F_{\ast}^2(X_{0})\right]\\
&= 2L_1^2
\mathbb{E}\left[(1 + \lvert X_{0}\rvert)^{2(\rho + 1)}\right]\mathbb{E}\left[\lvert\theta_k^{\eta}\rvert^2\right] + 4L_2^2\mathbb{E}\left[(1 + \lvert X_{0}\rvert)^{2(\rho + 1)}\right] + 4\mathbb{E}\left[F_{\ast}^2(X_{0})\right]\\
&= \widetilde{L}_1\mathbb{E}\left[\lvert\theta_k^{\eta}\rvert^2\right] + \widetilde{C}_1,\\
\end{aligned}    
\end{equation}
where $\widetilde{L}_1 := 2L_1^2\mathbb{E}\left[(1 + \lvert X_0\rvert)^{2(\rho + 1)}]\right]$, $\widetilde{C}_1 := 4L_2^2\mathbb{E}\left[(1 + \lvert X_0\rvert)^{2(\rho + 1)}\right] + 4\mathbb{E}\left[F_{\ast}(X_0)^2\right]$. Recall the SGHMC algorithm given in~\eqref{SGHMC 1} and~\eqref{SGHMC 2}, and using~\eqref{Ineq:3} yields

\begin{equation}\label{thm_pf:1}
\begin{aligned}
    \mathbb{E}\left[\lvert V_{k+1}^{\eta}\rvert^2\right] &= \mathbb{E}\left[|(1-\gamma\eta)V_k^{\eta}- \eta H(\theta_k^{\eta}, X_{k+1})|^2\right] + 2\gamma\eta\beta^{-1}d\\&
    = (1 - \gamma\eta)^2\mathbb{E}\left[\lvert V_{k}^{\eta}\rvert^2\right] - 2\eta(1 - \gamma\eta)\mathbb{E}\left[\langle V_k^{\eta}, h(\theta_k^{\eta})\rangle\right] + \eta^2\mathbb{E}\left[\lvert H(\theta_k^{\eta}, X_{k+1})\rvert^2\right] + 2\gamma\eta\beta^{-1} d\\&
    \leq  (1 - \gamma\eta)^2\mathbb{E}\left[\lvert V_{k}^{\eta}\rvert^2\right] - 2\eta(1 - \gamma\eta)\mathbb{E}\left[\langle V_k^{\eta}, h(\theta_k^{\eta})\rangle\right] + \eta^2\left[\widetilde{L}_1\mathbb{E}\left[\lvert \theta_k^{\eta}\rvert^2\right]+ \widetilde{C}_1\right] + 2\gamma\eta\beta^{-1} d,\\
\end{aligned}    
\end{equation}
where the second equality holds due to the fact that $X_{k+1} $ is independent of $\theta_k^{\eta}$ and $V_k^{\eta}$ , and that 
$$\mathbb{E}\left[\langle V_{k}^{\eta}, H(\theta_k^{\eta}, X_{k+1})\rangle\right] = \mathbb{E}\left[\mathbb{E}\left[\langle V_{k}^{\eta}, H(\theta_k^{\eta}, X_{k+1})\rangle|\theta_k^{\eta}, V_{k}^{\eta} \right]\right] = \mathbb{E}\left[\langle V_{k}^{\eta}, h(\theta_k^{\eta})\rangle\right].$$
Moreover, denote $\hat{U}_t := U(\theta_k^{\eta} + t\eta V_k^{\eta}), t \in [0, 1]$. We observe that $\hat{U}^{\prime}(t) = \langle h(\theta_k^{\eta} + t\eta V_k^{\eta}), \eta V_k^{\eta}\rangle$, and that, by using~\eqref{SGHMC 2}, $u(\theta_{k+1}^{\eta}) - u(\theta_{k}^{\eta}) = \hat{u}(1) - \hat{u}(0) = \int_{0}^{1}\hat{u}^{\prime}(t)\mathrm{d}t$. 
Thus, by Remark~\ref{rmk:2},
\begin{equation}\label{thm_pf:U}
\begin{aligned}
u(\theta_{k+1}^{\eta}) - u(\theta_k^{\eta}) &\leq \left\lvert\int_{0}^{1}\langle h(\theta_{k}^{\eta} + \tau\eta V_k^{\eta}) - h(\theta_{k}^{\eta}), \eta V_k^{\eta}\rangle \mathrm{d}\tau \right\rvert \\ &
\leq  \int_{0}^{1} \left\lvert h(\theta_{k}^{\eta} + \tau\eta V_k^{\eta}) - h(\theta_{k}^{\eta})\right\rvert \left\lvert\eta V_k^{\eta}\right\rvert \mathrm{d}\tau \\&
\leq \int_{0}^{1} L\tau \lvert\eta V_k^{\eta}\rvert^2 \mathrm{d}\tau \\&= \frac{L}{2}\eta^2\lvert V_k^{\eta}\rvert^2.
\end{aligned}    
\end{equation}
The above result implies
\begin{equation}\label{thm_pf:2}
\begin{aligned}
\mathbb{E}\left[u(\theta_{k+1}^{\eta})\right] - \mathbb{E}\left[u(\theta_k^{\eta}) \right] \leq \eta \mathbb{E}\left[\langle h(\theta_k^{\eta}), V_k^{\eta}\rangle\right] + \frac{L}{2}\eta^2\lvert V_k^{\eta}\rvert^2.
\end{aligned}    
\end{equation}
Then, notice that 
\begin{equation}\label{thm_pf:3}
\begin{aligned}
\mathbb{E}\left[\lvert\theta_{k+1}^{\eta}\rvert^2\right] = \mathbb{E}\left[\lvert\theta_k^{\eta}\rvert^2\right] + 2\eta\mathbb{E}\left[(\theta_k^{\eta}, V_k^{\eta})\right] + \eta^2\mathbb{E}\left[\lvert V_k^{\eta}\rvert^2\right].\\
\end{aligned}    
\end{equation}
Furthermore, by~\eqref{SGHMC 1} and~\eqref{SGHMC 2}, we have that 
\begin{equation}\label{thm_pf:4}
\begin{aligned}
\mathbb{E}\left[\left|\theta_{k+1}^\eta+\gamma^{-1} V_{k+1}^\eta\right|^2\right] & =\mathbb{E}\left[\left|\theta_k^\eta+\gamma^{-1} V_k^\eta-\eta \gamma^{-1} H\left(\theta_k^\eta, X_{k+1}\right)\right|^2\right]+2 \gamma^{-1} \beta^{-1} \eta d \\
& =\mathbb{E}\left[\left|\theta_k^\eta+\gamma^{-1} V_k^\eta\right|^2\right]-2 \eta \gamma^{-1} \mathbb{E}\left[\left\langle\theta_k^\eta+\gamma^{-1} V_k^\eta, h\left(\theta_k^\eta\right)\right\rangle\right] \\
&\quad +\eta^2 \gamma^{-2} \mathbb{E}\left[\left|H\left(\theta_k^\eta, X_{k+1}\right)\right|^2\right] + 2 \gamma^{-1} \eta \beta^{-1} d \\
& \leq \mathbb{E}\left[\left|\theta_k^\eta+\gamma^{-1} V_k^\eta\right|^2\right] - 2 \eta \gamma^{-1} \mathbb{E}\left[\left\langle\theta_k^\eta+\gamma^{-1} V_k^\eta, h\left(\theta_k^\eta\right)\right\rangle\right] \\
& \quad +\eta^2 \gamma^{-2}\left(\widetilde{L}_1 \mathbb{E}\left[\left|\theta_k^\eta\right|^2\right]+\widetilde{C}_1\right)+2 \gamma^{-1} \eta \beta^{-1} d,
\end{aligned}    
\end{equation}
where the last inequality holds due to~\eqref{Ineq:3}. Next we define, for any $k \in \mathbb{N}_0$, 
\begin{equation}\label{M_2(k)}
M_2(k) := \frac{\mathbb{E}\left[\mathcal{V}(\theta_k^{\eta}, V_k^{\eta})\right]} {\beta} =\mathbb{E} \left[ U(\theta_k^{\eta}) + \frac{\gamma^2}{4}(\lvert\theta_k^{\eta} + \gamma^{-1}V_k^{\eta}\rvert^2 + \lvert\gamma^{-1}V_k^{\eta}\rvert^2 - \lambda\lvert\theta_k^{\eta}\rvert^2) \right].
\end{equation}
By using~\eqref{thm_pf:1},~\eqref{thm_pf:2},~\eqref{thm_pf:3} and~\eqref{thm_pf:4}, we have
\begin{equation*}
\begin{aligned}
M_2(k+1) - M_2(k) &= \mathbb{E}\left[U(\theta_{k+1}^{\eta})\right] - \mathbb{E}\left[U(\theta_k^{\eta}) \right] + \frac{\gamma^2}{4}\left(\mathbb{E}\left[\lvert\theta_{k+1}^{\eta} + \gamma^{-1}V_{k+1}^{\eta}\rvert^2\right] - \mathbb{E}\left[\lvert\theta_{k}^{\eta} + \gamma^{-1}V_{k}^{\eta}\rvert^2\right]\right) \\ &\quad + \frac{1}{4}\left(\mathbb{E}\left[\lvert V_{k+1}^{\eta}\rvert^2\right] - \mathbb{E}\left[\lvert V_{k}^{\eta}\rvert^2\right]\right) - \frac{\gamma^2\lambda}{4}\left(\mathbb{E}\left[\lvert\theta_{k+1}^{\eta}\rvert^2\right] - \mathbb{E}\left[\lvert\theta_{k}^{\eta}\rvert^2\right]\right)\\ 
&\leq 
\eta \mathbb{E}\left[\langle h(\theta_k^{\eta}), V_k^{\eta}\rangle\right] + \frac{L}{2}\eta^2\mathbb{E}\left[\lvert V_k^{\eta}\rvert^2\right]+ \frac{\gamma^2}{4}\bigg(-2 \eta \gamma^{-1} \mathbb{E}\left[\langle\theta_k^\eta+\gamma^{-1} V_k^\eta, h\left(\theta_k^\eta\right)\rangle\right] \\ & \quad + \eta^2 \gamma^{-2}\left(\widetilde{L}_1 \mathbb{E}\left|\theta_k^\eta\right|^2+\widetilde{C}_1\right)+2 \gamma^{-1} \eta \beta^{-1} d\bigg) + \frac{1}{4}(-2 \gamma\eta + \gamma^2\eta^2)\mathbb{E}\left[\lvert V_{k}^{\eta}\rvert^2\right] \\& \quad+\frac{1}{4}\left(- 2\eta(1 - \gamma\eta)\mathbb{E}\left[\langle V_k^{\eta}, h(\theta_k^{\eta})\rangle\right] + \eta^2\left[\widetilde{L}_1\mathbb{E}\left[\lvert \theta_k^{\eta}\rvert^2\right]+ \widetilde{C}_1\right] + 2\gamma\eta\beta^{-1} d\right) \\ & \quad- \frac{\gamma^2\lambda}{4}\left(2\eta\mathbb{E}\left[\langle\theta_k^{\eta}, V_k^{\eta}\rangle\right] + \eta^2\mathbb{E}\left[\lvert V_k^{\eta}\rvert^2\right]\right).
\end{aligned}    
\end{equation*}
Then, we obtain
\begin{equation}\label{Ineq:4}
\begin{aligned}
M_2(k+1) - M_2(k)
&\leq
\frac{\gamma\eta^2}{2} \mathbb{E}\left[\langle h(\theta_k^{\eta}), V_k^{\eta}\rangle\right] - \frac{\gamma\eta}{2}\mathbb{E}\left[\langle h(\theta_k^{\eta}), \theta_k^{\eta}\rangle\right] + \frac{\eta^2\widetilde{L}_1}{2}\mathbb{E}\left[\lvert\theta_k^{\eta}\rvert^2\right] - \frac{\gamma^2\eta\lambda}{2}\mathbb{E}\left[\langle\theta_k^{\eta}, V_k^{\eta}\rangle\right]  \\& \quad + \left(\frac{L\eta^2}{2} -\frac{\eta\gamma}{2} + \frac{\gamma^2\eta^2}{4} - \frac{\lambda\gamma^2\eta^2}{4}\right)\mathbb{E}\left[\lvert V_k^{\eta}\rvert^2\right]+ \frac{\widetilde{C}_1\eta^2}{2} + \gamma\eta\beta^{-1}d
\\ &
\leq - \eta\gamma\lambda \mathbb{E}\left[u(\theta_k^{\eta})\right] - \frac{\lambda\gamma^3\eta}{4}\mathbb{E}\left[\lvert\theta_k^{\eta}\rvert^2\right] + A_c\gamma\eta\beta^{-1} \\ & \quad + \left(\frac{L\eta^2}{2} -\frac{\eta\gamma}{2} + \frac{\gamma^2\eta^2}{4} - \frac{\lambda\gamma^2\eta^2}{4}\right)\mathbb{E}\left[\lvert V_k^{\eta}\rvert^2\right] + \frac{\gamma\eta^2}{2} \mathbb{E}\left[\langle h(\theta_k^{\eta}), V_k^{\eta}\rangle\right] \\ &\quad + \frac{\eta^2\widetilde{L}_1}{2}\mathbb{E}\left[\lvert\theta_k^{\eta}\rvert^2\right] - \frac{\gamma^2\eta\lambda}{2}\mathbb{E}\left[\langle\theta_k^{\eta}, V_k^{\eta}\rangle\right] + \frac{\widetilde{C}_1\eta^2}{2} + \gamma\eta\beta^{-1}d ,\\
\end{aligned}    
\end{equation}
where the last inequality holds due to Remark~\ref{rmk:4}. By using $0 < \lambda \leq \frac{1}{4}$, we obtain  
\begin{equation*}
\begin{aligned}
M_2(k) &= \mathbb{E}\left[u(\theta_k^{\eta}) + \frac{\gamma^2}{4}\left(\lvert\theta_k^{\eta}\rvert^2 + 2\gamma^{-1}\langle \theta_k^{\eta}, V_k^{\eta}\rangle + 2\gamma^{-2}\lvert V_k^{\eta}\rvert^2 - \lambda\lvert\theta_k^{\eta}\rvert^2\right)\right] 
\\ &\leq \mathbb{E}\left[u(\theta_k^{\eta})\right] + \frac{\gamma^2}{4}\mathbb{E}\left[\lvert\theta_k^{\eta}\rvert^2\right] + \frac{\gamma}{2}\mathbb{E}\left[\langle \theta_k^{\eta}, V_k^{\eta}\rangle\right] + \frac{1}{2}\mathbb{E}\left[\lvert V_k^{\eta}\rvert^2\right], 
\end{aligned}
\end{equation*}
which implies, by rearranging the terms, that  
\begin{equation}\label{Ineq:5}
\begin{aligned}  
-\frac{\gamma}{2} \mathbb{E}\left[\langle\theta_k^{\eta}, V_k^{\eta}\rangle\right] \leq -M_2(k) + \mathbb{E}\left[U(\theta_k^{\eta})\right] + \frac{\gamma^2}{4}\mathbb{E}\left[\lvert\theta_k^{\eta}\rvert^2\right] + \frac{1}{2}\mathbb{E}\left[\lvert V_k^{\eta}\rvert^2\right].
\end{aligned}    
\end{equation}
Combining the result in~\eqref{Ineq:4} and~\eqref{Ineq:5}, we obtain 
\begin{equation} \label{func:M}
\begin{aligned}
M_2(k+1)  &\leq (1-\lambda\gamma\eta) M_2(k) + A_c\gamma\eta\beta^{-1} + \frac{\gamma\eta^2}{2} \mathbb{E}\left[\langle h(\theta_k^{\eta}), V_k^{\eta}\rangle\right] + \frac{\eta^2\widetilde{L}_1}{2}\mathbb{E}\left[\lvert\theta_k^{\eta}\rvert^2\right] + \gamma\eta\beta^{-1}d \\&\quad + \frac{\widetilde{C}_1\eta^2}{2}+ \left(\frac{L\eta^2}{2} -\frac{\eta\gamma}{2} + \frac{\gamma^2\eta^2}{4} - \frac{\lambda\gamma^2\eta^2}{4} + \frac{\lambda\eta\gamma}{2}\right)\mathbb{E}\left[\lvert V_k^{\eta}\rvert^2\right] \\ 
&\leq (1-\lambda\gamma\eta) M_2(k) + A_c\gamma\eta\beta^{-1} + \frac{\eta^2}{2}(\gamma L^2 + \widetilde{L}_1)\mathbb{E}\left[\lvert\theta_k^{\eta}\lvert^2\right] +\frac{\gamma\eta^2}{2}\lvert h(0)\rvert^2 + \gamma\eta\beta^{-1}d \\& \quad+ \frac{\widetilde{C}_1\eta^2}{2}+ \left(\frac{L\eta^2}{2} -\frac{\eta\gamma}{2} + \frac{\gamma^2\eta^2}{4} - \frac{\lambda\gamma^2\eta^2}{4} + \frac{\lambda\eta\gamma}{2} + \frac{\gamma\eta^2}{4}\right)\mathbb{E}\left[\lvert V_k^{\eta}\rvert^2\right] \\ 
&\leq  (1-\lambda\gamma\eta) M_2(k) + A_c\gamma\eta\beta^{-1} + \frac{\eta^2}{2}(\gamma L^2 + \widetilde{L}_1)\mathbb{E}\left[\lvert\theta_k^{\eta}\lvert^2\right] +\frac{\gamma\eta^2}{2}\lvert h(0)\rvert^2 + \gamma\eta\beta^{-1}d \\& \quad+ \frac{\widetilde{C}_1\eta^2}{2}+ \eta^2\left(\frac{L}{2} + \frac{\gamma^2}{4} - \frac{\lambda\gamma^2}{4} + \frac{\gamma}{4}\right)\mathbb{E}\left[\lvert V_k^{\eta}\rvert^2\right],\\ 
\end{aligned}    
\end{equation}
where the second inequality holds due to Remark~\ref{rmk:2} and $\langle a, b\rangle \leq \frac{(\lvert a\rvert^2 + \lvert b\rvert^2)}{2}$ with $a = h(\theta_k^{\eta})$ and $b = V_k^{\eta}$, 
and where the last inequality holds due to $0 < \lambda \leq \frac{1}{4}$. By~\eqref{func:V} and the fact that $\max\left\{a, b\right\} \geq \frac{a + b}{2}$, $a, b \in \mathbb{R}$, we have that
\begin{equation}\label{func:M_1}
\begin{aligned}
    M_2(k) &\geq \max\left\{ \frac{\gamma^2}{8}(1 - 2\lambda)\mathbb{E}\left[\lvert\theta_k^{\eta}\rvert^2\right], \frac{1}{4}(1 - 2\lambda)\lvert\mathbb{E}\left[ V_k^{\eta}\rvert^2\right]\right\}\\&
    \geq \frac{\gamma^2}{16}(1 - 2\lambda)\mathbb{E}\left[\lvert\theta_k^{\eta}\rvert^2\right] + \frac{1}{8}(1 - 2\lambda)\lvert \mathbb{E}\left[V_k^{\eta}\rvert^2\right].
\end{aligned}    
\end{equation}
By denoting 
\begin{equation}\label{K_1}
K_1 := \frac{1}{2}\max\left\{\frac{\widetilde{L}_1 + \gamma L^2}{\frac{\gamma^2}{16}(1 - 2\lambda)}, \frac{L + \frac{\gamma^2}{2} - \frac{\gamma^2\lambda}{2} + \frac{\gamma}{2}}{\frac{1}{8}(1 - 2\lambda)}\right\},
\end{equation}
we have that $$2K_1\eta^2M_2(k) \geq \eta^2(\widetilde{L}_1 + \gamma L^2)\mathbb{E}\left[\lvert\theta_k^{\eta}\rvert^2\right] + \eta^2(L + \frac{\gamma^2}{2} - \frac{\gamma^2\lambda}{2} + \frac{\gamma}{2})\mathbb{E}\left[\lvert V_k^{\eta}\rvert^2\right].$$
This, together with~\eqref{func:M}, yields
$$
    M_2(k+1) \leq (1 - \gamma\lambda\eta + K_1\eta^2) M_2(k) + K_2 \eta^2 + K_3\eta,
$$
where 
\begin{equation}\label{K_2}
K_2 := \frac{\gamma\lvert h(0)\rvert^2 + \widetilde{C}_1}{2}
\end{equation} and 
\begin{equation}\label{K_3}
K_3 := \gamma(d + A_c)\beta^{-1}.
\end{equation} 
Therefore, for $0 < \eta \leq \min\left\{\frac{K_3}{K_2}, \frac{\gamma\lambda}{2K_1}\right\}$ and $\eta \leq 2 / \gamma\lambda$ (implied by $\eta \leq 2 / \gamma$ and $\lambda \leq 1 / 4$), we have
\begin{equation}\label{func:M_2}
\begin{aligned}
    M_2(k+1)\leq \left(1 - \frac{\gamma\lambda\eta}{2}\right) M_2(k) + 2K_3\eta
    \leq M_2(0) + 2\eta K_3\frac{1 - (1-\frac{\gamma\lambda\eta}{2})^k}{1 - (1-\frac{\gamma\lambda\eta}{2})}\leq M_2(0) + \frac{4}{\gamma\lambda}K_3.
\end{aligned}    
\end{equation}
By using~\eqref{func:M_1} and~\eqref{func:M_2}, we hence obtain 
\begin{equation*}
\begin{aligned}
\sup _{t \geq 0} \mathbb{E}\left[\left|\theta_k^{\eta}\right|^2\right] \leq C_\theta:=\frac{\int_{\mathbb{R}^{2 d}} \mathcal{V}(\theta, v) \mu_0(\mathrm{~d} \theta, \mathrm{d} v)+\frac{4(d+A_c)}{\lambda}}{\frac{1}{8}(1-2 \lambda) \beta \gamma^2},
\end{aligned}
\end{equation*}
\begin{equation*}
\begin{aligned}
\sup _{t \geq 0} \mathbb{E}\left[\left|V_k^{\eta}\right|^2\right] \leq C_v:=\frac{\int_{\mathbb{R}^{2 d}} \mathcal{V}(\theta, v) \mu_0(\mathrm{~d} \theta, \mathrm{d} v)+\frac{4(d+A_c)}{\lambda}}{\frac{1}{4}(1-2 \lambda) \beta}.
\end{aligned}
\end{equation*}
\end{proof}

\begin{proof}[\textit{Proof of Lemma~\ref{lemma:3.4}}]\label{prf of lemma:3.4}
    Recall the process $\left(\bar{\zeta}_{t}^{\eta, n}\right)_{t \geq nT}$ defined in Definition~\ref{def:1}. According to Lemma~\ref{lemma:3.2}, for any $n \in \mathbb{N}_0$, we have
\begin{equation*}
\begin{aligned}
\sup _{t \in\left[n T,(n+1) T\right]} \mathbb{E}\left[\left|\bar{\zeta}_t^{\eta, n}\right|^2\right] \leq \frac{\mathbb{E}\left[\mathcal{V}(\bar{\theta}_{nT}^{\eta}, \bar{V}_{nT}^{\eta})\right]+\frac{\left(d+A_c\right)}{\lambda}}{\frac{1}{8}(1-2 \lambda) \beta \gamma^2}.
\end{aligned}
\end{equation*}
By using the expression of $M_2(k)$, $k \in \mathbb{N}_0$, defined in~\eqref{M_2(k)},  we have 
\begin{equation}\label{lm2.4: 1}
    \mathbb{E}\left[\mathcal{V}(\bar{\theta}_{nT}^{\eta}, \bar{V}_{nT}^{\eta})\right] = \beta M_2(nT)\leq \beta M_2(0) + \frac{4}{\lambda\gamma}\beta K_3 \leq \beta M_2(0) + \frac{4(A_c + d)}{\lambda} =: \tilde{c}_{10}, 
\end{equation} which implies that,
\begin{equation*}
\begin{aligned}
\sup_{t \in\left[n T,(n+1) T\right]} \mathbb{E}\left[\left|\bar{\zeta}_t^{\eta, n}\right|^2\right] \leq  \frac{\beta M_2(0) + \frac{5(d+A_c)}{\lambda}}{\frac{1}{8}(1-2 \lambda) \beta \gamma^2}=  \frac{\int_{\mathbb{R}^{2 d}} \mathcal{V}(\theta, v) \mu_0(d \theta, d v) + \frac{5(d+A_c)}{\lambda}}{\frac{1}{8}(1-2 \lambda) \beta \gamma^2}.
\end{aligned}
\end{equation*}
Hence, we conclude that 
\begin{equation*}
\begin{aligned}
\sup _{n \in \mathbb{N}_0} \sup _{t \in\left[n T,(n+1) T\right]} \mathbb{E}\left[\left|\bar{\zeta}_t^{\eta, n}\right|^2\right] \leq C_\zeta:=\frac{\int_{\mathbb{R}^{2 d}} \mathcal{V}(\theta, v) \mu_0(d \theta, dv)+\frac{5\left(d+A_c\right)}{\lambda}}{\frac{1}{8}(1-2 \lambda) \beta \gamma^2}.\\
\end{aligned}
\end{equation*}
Similarly, recalling the process $\left(\bar{Z}_{t}^{\eta, n}\right)_{t \geq nT}$ and Lemma~\ref{lemma:3.2}, we obtain
$$
\sup _{t \in[n T,(n+1) T]} \mathbb{E}\left[\left|\bar{Z}_t^{\eta, n}\right|^2\right] \leq \frac{\mathbb{E}\left[\mathcal{V}\left(\bar{\theta}_{n T}^\eta, \bar{V}_{n T}^\eta\right)\right]+\frac{\left(d+A_c\right)}{\lambda}}{\frac{1}{4}(1-2 \lambda) \beta}.
$$
By using~\eqref{lm2.4: 1} yields
$$
\sup _{t \in[n T,(n+1) T]} \mathbb{E}\left[\left|\bar{Z}_t^{\eta, n}\right|^2\right] \leq \frac{\beta M_2(0)+\frac{5\left(d+A_c\right)}{\lambda}}{\frac{1}{4}(1-2 \lambda) \beta}=\frac{\int_{\mathbb{R}^{2 d}} \mathcal{V}(\theta, v) \mu_0(d \theta, d v)+\frac{5\left(d+A_c\right)}{\lambda}}{\frac{1}{4}(1-2 \lambda) \beta}.
$$
Therefore, we have
$$
\sup _{n \in \mathbb{N}_0} \sup _{t \in[n T,(n+1) T]} \mathbb{E}\left[\left|\bar{Z}_t^{\eta, n}\right|^2\right] \leq C_Z:=\frac{\int_{\mathbb{R}^{2 d}} \mathcal{V}(\theta, v) \mu_0(d \theta, d v)+\frac{5\left(d+A_c\right)}{\lambda}}{\frac{1}{4}(1-2 \lambda) \beta}.
$$

\end{proof}

\begin{proof}[\textit{Proof of Lemma~\ref{lemma:3.5}}]\label{prf of lemma:3.5}
In this proof, we follow a similar strategy to the proof of Lemma~\ref{lemma:3.2}. Denote by \begin{equation}\label{delta}\Delta_k^{1} := (1 - \eta\gamma)V_k^{\eta} - \eta H(\theta_k^{\eta}, X_{k+1}) , \quad \Delta_k^{2} := \theta_{k}^{\eta} + \gamma^{-1}V_{k}^{\eta} - \gamma^{-1}\eta H(\theta_{k}^{\eta}, X_{k+1}).\end{equation} Recall SGHMC algorithm~\eqref{SGHMC 1} and~\eqref{SGHMC 2}. By using~\eqref{SGHMC 2}, we obtain
\begin{equation}\label{lm2.5: 1}
\lvert \theta_{k+1}^{\eta}\rvert^2 = \lvert \theta_k^{\eta}\rvert + 2\eta\langle\theta_k^{\eta}, V_k^{\eta}\rangle + \eta^2 \lvert V_k^{\eta}\rvert^2.
\end{equation}
Moreover, by using~\eqref{SGHMC 1}, we have
\begin{equation}\label{lm2.5: 2}
\begin{aligned}
    \lvert V_{k+1}^{\eta} \rvert^2 &=\lvert \Delta_k^{1} \rvert^2 + 2\sqrt{2\eta\gamma / \beta}\langle\Delta_k^{1} , \xi_{k+1}\rangle + \frac{2\eta\gamma}{\beta} \xi_{k+1}^2\\
    &= (1 - \eta\gamma)^2 \lvert V_k^{\eta}\rvert^2 - 2\eta(1 - \eta\gamma)\langle V_k^{\eta}, H(\theta_k^{\eta}, X_{k+1})\rangle + \eta^2\lvert H(\theta_k^{\eta}, X_{k+1})\rvert^2 \\&\quad + 2\sqrt{2\eta\gamma / \beta}\langle\Delta_k^{1} , \xi_{k+1}\rangle + \frac{2\eta\gamma}{\beta} \xi_{k+1}^2.
\end{aligned}
\end{equation}
We further note that
\begin{equation}\label{lm2.5: 3}
\begin{aligned}
\lvert \theta_{k+1}^{\eta} + \gamma^{-1}V_{k+1}^{\eta}\rvert^2 &= \lvert\theta_{k}^{\eta} + \gamma^{-1}V_{k}^{\eta} - \gamma^{-1}\eta  H(\theta_{k}^{\eta}, X_{k+1} )+ \gamma^{-1}\sqrt{2\gamma\eta / \beta}\xi_{k+1}\rvert^2 \\&
= \lvert\theta_{k}^{\eta} + \gamma^{-1} V_{k}^{\eta}\rvert^2 + \frac{2\eta}{\gamma\beta}\lvert\xi_{k+1}\rvert^2  - 2\eta\gamma^{-1}\langle\theta_{k}^{\eta} + \gamma^{-1} V_{k}^{\eta}, H(\theta_k^{\eta}, X_{k+1})\rangle \\& \quad + 2\sqrt{2\gamma^{-1}\beta^{-1}\eta}\langle\Delta_{k}^2, \xi_{k+1}\rangle + \eta^2\gamma^{-2}\lvert H(\theta_k^{\eta}, X_{k+1})\rvert^2.
\end{aligned}
\end{equation}
Using~\eqref{thm_pf:U}, we obtain that
\begin{equation}\label{lm2.5: 4}
u(\theta_{k+1}^{\eta}) - u(\theta_{k}^{\eta}) \leq \eta\langle h(\theta_k^{\eta}), V_k^{\eta}\rangle + \frac{L\eta^2}{2}\lvert V_k^{\eta}\rvert^2.
\end{equation}
Denote by 
\begin{equation}\label{sigma}
\Sigma_k :=  \frac{\gamma^2}{2}\sqrt{2\gamma^{-1}\beta^{-1}\eta}\langle\Delta_{k}^2, \xi_{k+1}\rangle
 + \frac{1}{2}\sqrt{2\eta\gamma / \beta}\langle\Delta_k^{1} , \xi_{k+1}\rangle, 
\end{equation}
and for the ease of notation, we denote by $\mathcal{V}_k := \mathcal{V}\left(\theta_k^\eta, V_k^\eta\right)$. By using~\eqref{lm2.5: 1},~\eqref{lm2.5: 2},~\eqref{lm2.5: 3}, and~\eqref{lm2.5: 4}, we obtain
\begin{equation}\label{lm3.5:v1}
\begin{aligned}
    (\mathcal{V}_{k + 1} - \mathcal{V}_k) / \beta &=  U(\theta_{k+1}^{\eta}) - U(\theta_{k}^{\eta}) + \frac{\gamma^2}{4}\left[ \lvert\theta_{k+1}^{\eta} + \gamma^{-1}V_{k+1}^{\eta}\rvert^2 - \lvert\theta_{k}^{\eta} + \gamma^{-1}V_{k}^{\eta}\rvert^2\right] \\&\quad+ \frac{\gamma^2}{4}\left[\lvert\gamma^{-1}V_{k+1}^{\eta}\rvert^2 - \lvert\gamma^{-1}V_{k}^{\eta}\rvert^2 - \lambda(\lvert\theta_{k+1}^{\eta}\rvert^2 - \lvert\theta_{k}^{\eta}\rvert^2) \right]\\& 
    \leq  \eta\langle h(\theta_k^{\eta}), V_k^{\eta}\rangle + \frac{L\eta^2}{2}\lvert V_k^{\eta}\rvert^2 + \frac{\gamma^2}{4}\bigg[\frac{2\eta}{\gamma\beta}\lvert\xi_{k+1}\rvert^2  - 2\eta\gamma^{-1}\langle\theta_{k}^{\eta} + \gamma^{-1} V_{k}^{\eta}, H(\theta_k^{\eta}, X_{k+1})\rangle \\&\quad + \eta^2\gamma^{-2}\lvert H(\theta_k^{\eta}, X_{k+1})\rvert^2 \bigg] + \frac{1}{4}\bigg[(- 2\eta\gamma + \eta^2\gamma^2) \lvert V_k^{\eta}\rvert^2 - 2\eta(1 - \eta\gamma)\langle V_k^{\eta}, H(\theta_k^{\eta}, X_{k+1})\rangle \\&\quad + \eta^2\lvert H(\theta_k^{\eta}, X_{k+1})\rvert^2 + \frac{2\eta\gamma}{\beta} \lvert\xi_{k+1}\rvert^2\bigg] - \frac{\gamma^2\lambda}{4}\left[2\eta\langle\theta_k^{\eta}, V_k^{\eta}\rangle + \eta^2 \lvert V_k^{\eta}\rvert^2  \right] + \Sigma_k\\
    &\leq -\frac{\eta\gamma}{2}\langle \theta_k^{\eta}, H(\theta_k^{\eta}, X_{k+1})\rangle + \eta\langle h(\theta_k^{\eta}), V_k^{\eta}\rangle + \left(\frac{L\eta^2}{2} - \frac{\eta\gamma}{2} + \frac{\gamma^2\eta^2}{4} - \frac{\eta^2\gamma^2\lambda}{4}\right) \lvert V_k^{\eta}\rvert^2 \\ &\quad + \frac{\eta^2}{2}\lvert H(\theta_k^{\eta}, X_{k+1})\rvert^2 + \gamma\eta\beta^{-1}\lvert \xi_{k+1}\rvert^2 + \left(\frac{\gamma\eta^2}{2} - \eta\right)\langle V_k^{\eta}, H(\theta_k^{\eta}, X_{k+1})\rangle \\ & \quad - \frac{\gamma^2\lambda\eta}{2}\langle\theta_k^{\eta}, V_k^{\eta}\rangle + \Sigma_k.\\
\end{aligned}
\end{equation}
Using the fact that $0 < \lambda \leq \frac{1}{4}$ and the expression of Lyapunov function~\eqref{Lyapunov}, we have that
\begin{equation}\label{lm3.5:v2}
\begin{aligned}  
-\frac{\gamma}{2} \langle\theta_k^{\eta}, V_k^{\eta}\rangle \leq -\mathcal{V}_k / \beta + U(\theta_k^{\eta}) + \frac{\gamma^2}{4}\lvert\theta_k^{\eta}\rvert^2 + \frac{1}{2}\lvert V_k^{\eta}\rvert^2.
\end{aligned}    
\end{equation}
Hence, substituting~\eqref{lm3.5:v2} into~\eqref{lm3.5:v1} yields
\begin{equation*}
\begin{aligned}
(\mathcal{V}_{k + 1} - \mathcal{V}_k) / \beta &\leq  -\frac{\eta\gamma}{2}\langle \theta_k^{\eta}, H(\theta_k^{\eta}, X_{k+1})\rangle + \eta\langle h(\theta_k^{\eta}) - H(\theta_k^{\eta}, X_{k+1}), V_k^{\eta}\rangle \\& \quad+ \left(\frac{L\eta^2}{2} - \frac{\eta\gamma}{2} + \frac{\gamma^2\eta^2}{4} - \frac{\eta^2\gamma^2\lambda}{4} + \frac{\lambda\gamma\eta}{2}\right) \lvert V_k^{\eta}\rvert^2 + \frac{\eta^2}{2}\lvert H(\theta_k^{\eta}, X_{k+1})\rvert^2 + \gamma\eta\beta^{-1}\lvert \xi_{k+1}\rvert^2 \\&\quad + \frac{\gamma\eta^2}{2}\left\langle V_k^{\eta}, H(\theta_k^{\eta}, X_{k+1})\right\rangle - \gamma\lambda\eta \mathcal{V}_k / \beta+ \lambda\eta\gamma U(\theta_k^{\eta}) + \frac{\gamma^3\lambda\eta}{4}\lvert\theta_k^{\eta}\rvert^2 + \Sigma_k.\\
\end{aligned}
\end{equation*}
Then, by rearranging the terms, we arrive at
\begin{align*}
\mathcal{V}_{k + 1}/ \beta & \leq  (1 - \gamma\lambda\eta) \mathcal{V}_k / \beta -\frac{\eta\gamma}{2}\langle \theta_k^{\eta}, H(\theta_k^{\eta}, X_{k+1})\rangle + \eta\left\langle h(\theta_k^{\eta}) - H(\theta_k^{\eta}, X_{k+1}), V_k^{\eta}\right\rangle \\ & \quad +  \left(\frac{L\eta^2}{2} - \frac{\eta\gamma}{2} + \frac{\gamma^2\eta^2}{4} - \frac{\eta^2\gamma^2\lambda}{4} + \frac{\lambda\gamma\eta}{2}\right) \lvert V_k^{\eta}\rvert^2 + \frac{\eta^2}{2}\lvert H(\theta_k^{\eta}, X_{k+1})\rvert^2 + \gamma\eta\beta^{-1}\lvert \xi_{k+1}\rvert^2 \\& \quad + \frac{\gamma\eta^2}{2}\langle V_k^{\eta}, H(\theta_k^{\eta}, X_{k+1})\rangle + \lambda\eta\gamma U(\theta_k^{\eta}) + \frac{\gamma^3\lambda\eta}{4}\lvert\theta_k^{\eta}\rvert^2 + \Sigma_k\\
& = (1 - \gamma\lambda\eta) \mathcal{V}_k / \beta -\frac{\eta\gamma}{2}\langle \theta_k^{\eta}, h(\theta_k^{\eta})\rangle + \frac{\eta\gamma}{2}\langle \theta_k^{\eta}, h(\theta_k^{\eta}) - H(\theta_k^{\eta}, X_{k+1})\rangle + \eta\langle h(\theta_k^{\eta}) - H(\theta_k^{\eta}, X_{k+1}), V_k^{\eta}\rangle \\ & \quad + \left(\frac{L\eta^2}{2} - \frac{\eta\gamma}{2} + \frac{\gamma^2\eta^2}{4} - \frac{\eta^2\gamma^2\lambda}{4} + \frac{\lambda\gamma\eta}{2}\right) \lvert V_k^{\eta}\rvert^2 + \frac{\eta^2}{2}\lvert H(\theta_k^{\eta}, X_{k+1})\rvert^2 + \gamma\eta\beta^{-1}\lvert \xi_{k+1}\rvert^2 \\ & \quad +\frac{\gamma\eta^2}{2}\langle V_k^{\eta}, H(\theta_k^{\eta}, X_{k+1})\rangle + \lambda\eta\gamma U(\theta_k^{\eta}) + \frac{\gamma^3\lambda\eta}{4}\lvert\theta_k^{\eta}\rvert^2  + \Sigma_k\\
& \leq (1 - \gamma\lambda\eta) \mathcal{V}_k / \beta -\lambda\gamma\eta U(\theta_k^{\eta}) - \frac{\gamma^3\lambda\eta}{4}\lvert\theta_k^{\eta}\rvert^2 + \eta\gamma\beta^{-1}A_c -\frac{\eta\gamma}{2}\langle \theta_k^{\eta}, H(\theta_k^{\eta}, X_{k+1}) - h(\theta_k^{\eta})\rangle \\ & \quad+ \eta\left\langle h(\theta_k^{\eta}) + \left(\frac{\eta\gamma}{2}-1\right) H(\theta_k^{\eta}, X_{k+1}), V_k^{\eta}\right\rangle + \left(\frac{L\eta^2}{2} - \frac{\eta\gamma}{2} + \frac{\gamma^2\eta^2}{4} - \frac{\eta^2\gamma^2\lambda}{4} + \frac{\lambda\gamma\eta}{2}\right) \lvert V_k^{\eta}\rvert^2 \\ & \quad + \frac{\eta^2}{2}\lvert H(\theta_k^{\eta}, X_{k+1})\rvert^2 + \gamma\eta\beta^{-1}\lvert \xi_{k+1}\rvert^2  + \lambda\eta\gamma U(\theta_k^{\eta}) + \frac{\gamma^3\lambda\eta}{4}\lvert\theta_k^{\eta}\rvert^2 + \Sigma_k,\\
\end{align*}
where the last inequality holds due to Remark~\ref{rmk:4}. Then, 
\begin{equation}\label{lm2.5 : v3}
\begin{aligned}
\mathcal{V}_{k + 1}/ \beta & \leq (1 - \gamma\lambda\eta) \mathcal{V}_k / \beta + \eta\gamma\beta^{-1}A_c -\frac{\eta\gamma}{2}\langle \theta_k^{\eta}, H(\theta_k^{\eta}, X_{k+1}) - h(\theta_k^{\eta})\rangle \\ & \quad+  \eta\langle h(\theta_k^{\eta}) - H(\theta_k^{\eta}, X_{k+1}), V_k^{\eta}\rangle + \left(\frac{L\eta^2}{2} - \frac{\eta\gamma}{2} + \frac{\gamma^2\eta^2}{4} - \frac{\eta^2\gamma^2\lambda}{4} + \frac{\lambda\gamma\eta}{2}\right) \lvert V_k^{\eta}\rvert^2 \\ & \quad + \frac{\eta^2}{2}\lvert H(\theta_k^{\eta}, X_{k+1})\rvert^2 + \gamma\eta\beta^{-1}\lvert \xi_{k+1}\rvert^2 + \Sigma_k + \frac{\eta^2\gamma}{2}\langle H(\theta_k^{\eta}, X_{k+1}), V_k^{\eta}\rangle \\
& \leq  (1 - \gamma\lambda\eta) \mathcal{V}_k / \beta + \eta\gamma\beta^{-1}A_c -\frac{\eta\gamma}{2}\langle \theta_k^{\eta}, H(\theta_k^{\eta}, X_{k+1}) - h(\theta_k^{\eta})\rangle \\ & \quad +  \eta\langle h(\theta_k^{\eta}) - H(\theta_k^{\eta}, X_{k+1}), V_k^{\eta}\rangle + \left(\frac{L\eta^2}{2} - \frac{\eta\gamma}{2} + \frac{\gamma^2\eta^2}{4} - \frac{\eta^2\gamma^2\lambda}{4} + \frac{\lambda\gamma\eta}{2} + \frac{\eta^2\gamma}{4}\right) \lvert V_k^{\eta}\rvert^2 \\ & \quad+ (\frac{\eta^2}{2} + \frac{\eta^2\gamma}{4})\lvert H(\theta_k^{\eta}, X_{k+1})\rvert^2 + \gamma\eta\beta^{-1}\lvert \xi_{k+1}\rvert^2 + \Sigma_k, \\
\end{aligned}
\end{equation}
where the last inequality holds due to $\langle a, b\rangle \leq \frac{\lvert a\rvert^2 + \lvert b\rvert^2}{2}$, for $a, b \in \mathbb{R}^d$. By using~\eqref{func:V} and the fact that $\max\left\{x, y\right\}\geq (x + y) / 2$, for any $x, y \geq 0$, we obtain
\begin{equation}\label{lm2.5: 5}
\begin{aligned}
    \mathcal{V} (\theta_k^{\eta}, V_k^{\eta}) / \beta
    \geq& \max\left\{ \frac{\gamma^2}{8}(1 - 2\lambda)\lvert\theta_k^{\eta}\rvert^2, \frac{1}{4}(1 - 2\lambda)\lvert V_k^{\eta}\rvert^2\right\},\\
    \geq & \frac{\gamma^2}{16}(1 - 2\lambda)\lvert\theta_k^{\eta}\rvert^2 + \frac{1}{8}(1 - 2\lambda)\lvert V_k^{\eta}\rvert^2.\\
\end{aligned}
\end{equation}
Denote by $\phi := (1 - \gamma\lambda\eta)$ and $$\tilde{K}_1 := \max\left\{\frac{\left((1 + \frac{\gamma}{2})L_1^2 \mathbb{E}\left[(1 + \lvert X_0\rvert)^{2(\rho + 1)}\right]\right)}{\frac{\gamma^2}{16}(1 - 2\lambda)}, \frac{\left(\frac{L\eta^2}{2} + \frac{\gamma^2\eta^2}{4} - \frac{\eta^2\gamma^2\lambda}{4} + \frac{\eta^2\gamma}{4}\right)}{\frac{1}{8}(1 - 2\lambda)}\right\}.$$ 
The result in~\eqref{lm2.5: 5} further implies
\begin{equation}\label{lm2.5: v4}
2\phi\tilde{K}_1\mathcal{V}_k / \beta \geq 2\phi\left(1 + \frac{\gamma}{2}\right)\mathbb{E}\left[(1+ \lvert X_0\rvert)^{2(\rho + 1)}\right]L_1^2\lvert\theta_k^{\eta}\rvert^2 + 2\phi\left(\frac{L\eta^2}{2} + \frac{\gamma^2\eta^2}{4} - \frac{\eta^2\gamma^2\lambda}{4} + \frac{\eta^2\gamma}{4}\right) \lvert V_k^{\eta}\rvert^2.
\end{equation}
Then, by using Remark~\ref{rmk:1},~\eqref{sigma},~\eqref{lm2.5: v4}, $\lambda \leq \frac{1}{4}$, and the fact that $\xi_{k+1}$ is independent of $\Delta_k^{1}$ and $\Delta_k^{2}$ , we have that  
\begin{align*}
&\mathbb{E}\left[\mathcal{V}_{k + 1}^2/ \beta^2\big|\theta_k^{\eta}, V_k^{\eta}\right] \\ & \leq  \phi^2\mathcal{V}_k^2 / \beta^2 + 2\phi \mathcal{V}_k / \beta \bigg(\eta\gamma\beta^{-1}A_c + \gamma\eta\beta^{-1}d + \eta^2\left(1 + \frac{\gamma}{2}\right)\bigg(\mathbb{E}\left[(1+ \lvert X_0\rvert)^{2(\rho + 1)}\right]L_1^2\lvert\theta_k^{\eta}\rvert^2 \\& \quad+ 2\mathbb{E}\left[(1+ \lvert X_0\rvert)^{2(\rho + 1)}\right]L_2^2 + 2\mathbb{E}\left[F^2_{\ast}(X_0)\right]\bigg) + \left(\frac{L\eta^2}{2} + \frac{\gamma^2\eta^2}{4} - \frac{\eta^2\gamma^2\lambda}{4} + \frac{\eta^2\gamma}{4}\right) \lvert V_k^{\eta}\rvert^2\bigg) \\ & \quad + \mathbb{E}\Bigg[\Bigg(\eta\gamma\beta^{-1}A_c + \frac{\eta\gamma}{2}\langle \theta_k^{\eta}, h(\theta_k^{\eta}) - H(\theta_k^{\eta}, X_{k+1}) \rangle +  \eta\langle h(\theta_k^{\eta}) - H(\theta_k^{\eta}, X_{k+1}), V_k^{\eta}\rangle \\ & \quad + \left(\frac{L\eta^2}{2} - \frac{\eta\gamma}{2} + \frac{\gamma^2\eta^2}{4} - \frac{\eta^2\gamma^2\lambda}{4} + \frac{\lambda\gamma\eta}{2} + \frac{\eta^2\gamma}{4}\right) \lvert V_k^{\eta}\rvert^2 + \left(\frac{\eta^2}{2} + \frac{\eta^2\gamma}{4}\right)\lvert H(\theta_k^{\eta}, X_{k+1})\rvert^2\\ & \quad + \gamma\eta\beta^{-1}\lvert \xi_{k+1}\rvert^2 + \Sigma_k\Bigg)^2\Bigg|\theta_k^{\eta}, V_k^{\eta}\Bigg]\\
& \leq \left(\phi^2 + 2\phi\tilde{K}_1\eta^2\right)\mathcal{V}_k^2 / \beta^2 + 2\phi\mathcal{V}_k / \beta \bigg(\eta\gamma\beta^{-1}A_c + \gamma\eta\beta^{-1}d + \eta^2(2 + \gamma)\left(\mathbb{E}\left[(1 + \lvert X_0\rvert)^{2(\rho + 1)}\right] L_2^2 + \mathbb{E}\left[F_{\ast}^2(X_{0})\right]\right)\!\bigg) \\ & \quad + \mathbb{E}\Bigg[\Bigg(\eta\gamma\beta^{-1}A_c + \frac{\eta\gamma}{2}\langle \theta_k^{\eta}, -H(\theta_k^{\eta}, X_{k+1}) + h(\theta_k^{\eta})\rangle +  \eta\langle h(\theta_k^{\eta}) - H(\theta_k^{\eta}, X_{k+1}), V_k^{\eta}\rangle \\ & \quad + \left(\frac{L\eta^2}{2} - \frac{\eta\gamma}{2} + \frac{\gamma^2\eta^2}{4} - \frac{\eta^2\gamma^2\lambda}{4} + \frac{\lambda\gamma\eta}{2} + \frac{\eta^2\gamma}{4}\right) \lvert V_k^{\eta}\rvert^2 +\gamma\eta\beta^{-1}\lvert \xi_{k+1}\rvert^2 + \Sigma_k \\ & \quad + \left(\eta^2 + \frac{\eta^2\gamma}{2}\right)\left((1 + \lvert X_{k+1}\rvert) ^{2(\rho + 1)} L_1^2\lvert\theta_{k}^{\eta}\rvert^2 + 2(1 + \lvert X_{k+1}\rvert) ^{2(\rho + 1)}L^2_2 + 2F^2_{\ast}(X_{k+1})\right)\Bigg)^2\Bigg|\theta_k^{\eta}, V_k^{\eta}\Bigg].\\
\end{align*}
This further implies, by using $\langle a, b\rangle \leq \frac{\lvert a\rvert^2 + \lvert b\rvert^2}{2}$, for any $a, b \in \mathbb{R}^d$, that
\begin{align*}
& \mathbb{E}\left[\mathcal{V}_{k + 1}^2/ \beta^2\big|\theta_k^{\eta}, V_k^{\eta}\right] \\ & \leq\left(\phi^2 + 2\phi\tilde{K}_1\eta^2\right)\mathcal{V}_k^2 / \beta^2 + 2\phi\mathcal{V}_k / \beta \bigg(\eta\gamma\beta^{-1}A_c + \gamma\eta\beta^{-1}d  \\ &\quad + \eta^2(2 + \gamma)\left(\mathbb{E}\left[(1 + \lvert X_0\rvert)^{2(\rho + 1)}\right] L_2^2 + \mathbb{E}\left[F_{\ast}^2(X_{0})\right]\right)\bigg) \\ &\quad  +\mathbb{E}\Bigg[\Bigg(\eta\gamma\beta^{-1}A_c + \frac{\eta\gamma}{4}\lvert\theta_k^{\eta}\rvert^2 + \left(\frac{\eta\gamma}{2} + \eta\right)\left(\lvert h(\theta_k^{\eta})\rvert^2 + \lvert H(\theta_k^{\eta},X_{k+1})\rvert^2\right)
+ \gamma\eta\beta^{-1}\lvert \xi_{k+1}\rvert^2\\ &\quad +  \left(\frac{L\eta^2}{2} - \frac{\eta\gamma}{2} + \frac{\gamma^2\eta^2}{4} - \frac{\eta^2\gamma^2\lambda}{4} + \frac{\lambda\gamma\eta}{2} + \frac{\eta^2\gamma}{2}+ \frac{\eta}{2}\right) \lvert V_k^{\eta}\rvert^2 + \Sigma_k \\ &\quad + \left(\eta^2 + \frac{\eta^2\gamma}{2}\right)\left((1 + \lvert X_{k+1}\rvert) ^{2(\rho + 1)} L_1^2\lvert\theta_{k}^{\eta}\rvert^2 + 2(1 + \lvert X_{k+1}\rvert) ^{2(\rho + 1)}L^2_2 + 2F^2_{\ast}(X_{k+1})\right)\Bigg)^2\Bigg|\theta_k^{\eta}, V_k^{\eta}\Bigg].\\
\end{align*}
By using Remark~\ref{rmk:1} and~\ref{rmk:2}, we have
\begin{align*}
&\mathbb{E}\left[\mathcal{V}_{k + 1}^2/ \beta^2\big|\theta_k^{\eta}, V_k^{\eta}\right] \\  & \leq \left(\phi^2 + 2\phi\tilde{K}_1\eta^2\right)\mathcal{V}_k^2 / \beta^2 + 2\phi\mathcal{V}_k / \beta \bigg(\eta\gamma\beta^{-1}A_c + \gamma\eta\beta^{-1}d  +  \eta^2(2 + \gamma)\left(\mathbb{E}\left[(1 + \lvert X_0\rvert)^{2(\rho + 1)}\right] L_2^2 + \mathbb{E}\left[F_{\ast}^2(X_{0})\right]\right)\!\bigg) \\ &\quad + 
\mathbb{E}\Bigg[\Bigg(\eta\gamma\beta^{-1}A_c + \frac{\eta\gamma}{4}\lvert\theta_k^{\eta}\rvert^2 + \eta\left(1 + \frac{\gamma}{2}\right)\bigg(2 L^2\lvert\theta_k^{\eta}\rvert^2 + 2\lvert h(0)\rvert^2 \\ &\quad  + 2(1 + \lvert X_{k+1}\rvert) ^{2(\rho + 1)} L^2_1\lvert\theta_k^{\eta}\rvert^2 + 4(1 + \lvert X_{k+1}\rvert) ^{2(\rho + 1)}L^2_2 + 4F^2_{\ast}(X_{k+1})\bigg)
+ \gamma\eta\beta^{-1}\lvert \xi_{k+1}\rvert^2\\ &\quad + \left(\frac{L\eta^2}{2} - \frac{\eta\gamma}{2} + \frac{\gamma^2\eta^2}{4} - \frac{\eta^2\gamma^2\lambda}{4} + \frac{\lambda\gamma\eta}{2} + \frac{\eta^2\gamma}{2}+ \frac{\eta}{2}\right) \lvert V_k^{\eta}\rvert^2 + \Sigma_k \\ &\quad  + \eta^2(1 + \frac{\gamma}{2})\bigg((1 + \lvert X_{k+1}\rvert) ^{2(\rho + 1)} L_1^2\lvert\theta_{k}^{\eta}\rvert^2 + 2(1 + \lvert X_{k+1}\rvert) ^{2(\rho + 1)}L^2_2 + 2F^2_{\ast}(X_{k+1})\bigg)\Bigg)^2\Bigg|\theta_k^{\eta}, V_k^{\eta}\Bigg]\\
& =  \left(\phi^2 + 2\phi\tilde{K}_1\eta^2\right)\mathcal{V}_k^2 / \beta^2 + 2\phi\eta(\gamma A_c + \gamma d)\mathcal{V}_k / \beta^2  + 2\phi\eta^2
(2 + \gamma)\left(\mathbb{E}\left[(1 + \lvert X_0\rvert)^{2(\rho + 1)}\right] L_2^2 + \mathbb{E}\left[F_{\ast}^2(X_{0})\right]\right) \mathcal{V}_k / \beta 
\\ &\quad +
\mathbb{E}\Bigg[\Bigg(
\eta / \beta \left(\gamma A_c + \gamma \lvert \xi_{k+1}\rvert^2\right) + \eta \lvert\theta_{k}^{\eta}\rvert^2\left(\frac{\gamma}{4} + 2\left(1 + \frac{\gamma}{2}\right)(L^2 + (1 + \lvert X_{k+1}\rvert)^{2(\rho + 1)}L_1^2) \right) \\ &\quad  + \eta^2 \lvert\theta_{k}^{\eta}\rvert^2 \left(1 + \frac{\gamma}{2}\right)(1 + \lvert X_{k+1}\rvert)^{2(\rho + 1)}L_1^2 + \eta \lvert V_k^{\eta}\rvert^2\left(\frac{1}{2} + \frac{\lambda\gamma}{2} -\frac{\gamma}{2}\right) \\ &\quad  +\eta^2\lvert V_k^{\eta}\rvert^2\left(\frac{L}{2} + \frac{\gamma^2}{4} - \frac{\lambda\gamma^2}{4} + \frac{\gamma}{2}\right) + \Sigma_k + \eta\left(1 + \frac{\gamma}{2}\right)\left(2\lvert h(0)\rvert^2 + 4(1 + \lvert X_{k+1}\rvert)^{2(\rho + 1)}L_2^2 + 4 F_{\ast}^2(X_{k+1})\right) \\ &\quad + \eta^2(2 + \gamma)((1 + \lvert X_{k+1}\rvert)^{2(\rho + 1)}L_2^2 + F_{\ast}^2(X_{k+1}) ) \Bigg)^2\Bigg|\theta_k^{\eta}, V_k^{\eta}\Big].
\end{align*}
By using the fact that, for any integer $k \geq 2$, $ \left(\sum\limits_{i=1}^{k}X_i\right)^2 \leq k\sum\limits_{i=1}^{k} X_i^2$, we obtain
\begin{equation}\label{lm:long}
\begin{aligned}
\mathbb{E}\left[\mathcal{V}_{k + 1}^2/ \beta^2\big|\theta_k^{\eta}, V_k^{\eta}\right] & \leq  \left(\phi^2 + 2\phi\tilde{K}_1\eta^2\right)\mathcal{V}_k^2 / \beta^2 + 2\phi\eta \tilde{c}_1\mathcal{V}_k / \beta^2+ 2\phi\eta^2 \hat{c}_1\mathcal{V}_k/\beta + \eta^2\beta^{-2} \tilde{c}_2 + \eta^2\lvert\theta_k^{\eta}\rvert^4 \tilde{c}_3 \\ &\quad +  \eta^4\lvert\theta_k^{\eta}\rvert^4 \hat{c}_3 + \eta^2\lvert V_k^{\eta}\rvert^4 \tilde{c}_4 +  \eta^4\lvert V_k^{\eta}\rvert^4 \hat{c}_4 + \eta^2\tilde{c}_5 + \eta^4\hat{c}_5 
+ 8\mathbb{E}\left[\Sigma^2_k \Bigg|\theta_k^{\eta}, V_k^{\eta}\right],
\end{aligned}
\end{equation}
where 
\begin{align*}
\tilde{c}_1 &:= \gamma A_c + \gamma d,\quad \hat{c}_1 = \left(2 + \gamma\right)\left(\mathbb{E}\left[(1 + \lvert X_0\rvert)^{2(\rho + 1)}\right] L_2^2 + \mathbb{E}\left[F_{\ast}^2(X_{0})\right]\right),\\
\tilde{c}_2 &:= 16\gamma^2A_c^2 + 48\gamma^2 d^2,\\
\tilde{c}_3 &:= \frac{3}{2}\gamma^2 + 24\left(2 + \gamma\right)^2\left(L^4 + L_1^4\mathbb{E}\left[(1 + \lvert X_0\rvert)^{4(\rho + 1)}\right]\right),\\
\hat{c}_3 &:= 8 \left(1 + \frac{\gamma}{2}\right)^2L_1^4\mathbb{E}\left[(1 + \lvert X_0\rvert)^{4(\rho + 1)}\right],\\
\tilde{c}_4 &:= 2 \left(1 + \lambda\gamma - \gamma\right)^2, \quad \hat{c}_4 := 2\left(L + \frac{\gamma^2}{2} - \frac{\lambda\gamma^2}{2} + \gamma\right)^2,\\
\tilde{c}_5 &:= 96\left(1 + \frac{\gamma}{2}\right)^2\left(\lvert h(0)\rvert^4 + 4L_2^4\mathbb{E}\left[(1 + \lvert X_0\rvert)^{4(\rho + 1)}\right] + 4\mathbb{E}\left[F^4_{\ast}(X_0)\right]\right),\\
\hat{c}_5 &:= 64\left(1 + \frac{\gamma}{2}\right)^2\left(L_2^4\mathbb{E}\left[(1 + \lvert X_0\rvert)^{4(\rho + 1)}\right] + \mathbb{E}\left[F^4_{\ast}(X_0)\right]\right).\\
\end{align*}
Next, recall the expression of $\Delta_k^{1}$, $\Delta_k^{2}$ defined in~\eqref{delta} and of $\Sigma_k$ in~\eqref{sigma}, $k \in \mathbb{N}_0$. We note that 
\begin{align*}
\Sigma_{k}^{2}&  \leq \gamma\beta^{-1}\eta|\Delta_{k}^{1}|^{2}|\xi_{k+1}|^{2} + \gamma^{3}\beta^{-1}\eta|\Delta_{k}^{2}|^{2}|\xi_{k+1}|^{2}\\
& \leq\eta\gamma\beta^{-1}\lvert\xi_{k+1}\rvert^{2}\left(\lvert V_{k}^{n}-\eta\left(\gamma V_{k}^{n}+H(\theta_{k}^{\eta}, X_{k+1})\right)\rvert^{2} + \gamma^{2}\lvert\theta_{k}^{\eta}
+ \gamma^{-1} V_k^{\eta} - \eta\gamma^{-1}H(\theta_k^{\eta}, X_{k+1})\rvert^2\right)\\
& \leq \eta\gamma\beta^{-1}\lvert\xi_{k+1}\rvert^{2}\left(2(1 - \eta\gamma)^2\lvert V_k^{\eta}\rvert^2 + (2+3\eta^2)\lvert H(\theta_k^{\eta}, X_{k+1})\rvert^2 + 3\gamma^2\lvert\theta_k^{\eta}\rvert^2 + 3\lvert V_k^{\eta}\rvert^2 \right)\\
& \leq \eta\gamma\beta^{-1}\lvert\xi_{k+1}\rvert^{2}\bigg(2(1 - \eta\gamma)^2\lvert V_k^{\eta}\rvert^2 + (2+3\eta^2)\bigg(3L_1^2(1 + \lvert X_{k+1}\rvert)^{2(\rho + 1)}\lvert\theta_k^{\eta}\rvert^2 \\& \quad+ 3L_2^2(1 + \lvert X_{k+1}\rvert)^{2(\rho + 1)} + 3F^2_{\ast}(X_{k+1})
\bigg)+ 3\gamma^2\lvert\theta_k^{\eta}\rvert^2 + 3\lvert V_k^{\eta}\rvert^2 \bigg),\\
\end{align*}
where the last inequality holds due to Remark~\ref{rmk:1}. Then by taking expectation, using $\eta \leq 1$ and $\eta \leq  \frac{2}{\gamma}$, and the fact that $\xi_{k+1}$ is independent of $X_{k+1}$, $V_{k}^{\eta}$ and $\theta_k^{\eta}$, we obtain
\begin{equation*}
\begin{aligned}
\mathbb{E}\left[\Sigma_{k}^{2}\big| \theta_k^{\eta}, V_k^{\eta}\right]& \leq  \eta\gamma\beta^{-1}d\bigg(5\lvert V_k^{\eta}\rvert^2 + \left(3\gamma^2 + 15L_1^2\mathbb{E}\left[(1 + \lvert X_{0}\rvert)^{2(\rho + 1)}\right]\right)\lvert\theta_k^{\eta}\rvert^2 \\& \quad+ 15L_2^2\mathbb{E}\left[(1 + \lvert X_{0}\rvert)^{2(\rho + 1)}\right] + 15\mathbb{E}\left[F^2_{\ast}(X_0)\right]\bigg).\\
\end{aligned}
\end{equation*}
Applying~\eqref{lm2.5: 5} to the inequality above yields 
\begin{equation}\label{sigma_1}
\begin{aligned}
\mathbb{E}\left[\Sigma_{k}^{2}\big| \theta_k^{\eta}, V_k^{\eta}\right]\leq & \frac{\eta\mathcal{V}_k}{8\beta^2}\tilde{c}_6 + \frac{\eta}{8\beta}\tilde{c}_7,
\end{aligned}
\end{equation}
where 
$\tilde{c}_6 := 8\max \left\{\frac{5\gamma d}{\frac{1}{8}(1 - 2\lambda)}, \frac{3\gamma d(\gamma^2 + 5L_1^2\mathbb{E}\left[(1 + \lvert X_{0}\rvert)^{2(\rho + 1)}\right])}{\frac{1}{16}(1 - 2\lambda)\gamma^2}\right\}$ and $ \tilde{c}_7 := 120 \gamma dL_2^2\mathbb{E}\left[(1 + \lvert X_{0}\rvert)^{2(\rho + 1)}\right] + 120\gamma d \mathbb{E}\left[F^2_{\ast}(X_0)\right]$.
Substituting~\eqref{sigma_1} into~\eqref{lm:long} and by using $\eta \leq 1$, we obtain
\begin{equation*}
\begin{aligned}
\mathbb{E}\left[\mathcal{V}_{k + 1}^2/ \beta^2\big|\theta_k^{\eta}, V_k^{\eta}\right] 
\leq & \left(\phi^2 + 2\phi\tilde{K}_1\eta^2\right)\frac{\mathcal{V}_k^2}{\beta^2} + (2\phi\tilde{c}_1 + \tilde{c}_6)\frac{\mathcal{V}_k}{\beta^2}\eta+ 2\phi\eta^2 \hat{c}_1\mathcal{V}_k / \beta + \frac{\eta^2}{\beta^2} \tilde{c}_2 \\+& \eta^2\lvert\theta_k^{\eta}\rvert^4 (\tilde{c}_3 + \hat{c}_3) + \eta^2\lvert V_k^{\eta}\rvert^4 (\tilde{c}_4 + \hat{c}_4) + \eta^2(\tilde{c}_5 +\hat{c}_5)
+ \frac{\eta}{\beta}\tilde{c}_7 .\\
\end{aligned}
\end{equation*}
Next, by using~\eqref{lm2.5: 5}, we have
\begin{equation*}
\begin{aligned}
    \mathcal{V}_k^2 / \beta^2
    \geq& \max\left\{ \frac{1}{64}(1 - 2\lambda)^2\gamma^4\lvert\theta_k^{\eta}\rvert^4, \frac{1}{16}(1 - 2\lambda)^2\lvert V_k^{\eta}\rvert^4\right\},\\
    \geq & \frac{1}{128}(1 - 2\lambda)^2\gamma^4\lvert\theta_k^{\eta}\rvert^4 + \frac{1}{32}(1 - 2\lambda)^2\lvert V_k^{\eta}\rvert^4.
\end{aligned}
\end{equation*}
This implies that
\begin{equation*}
\begin{aligned}
\mathbb{E}\left[\mathcal{V}_{k + 1}^2/ \beta^2\big|\theta_k^{\eta}, V_k^{\eta}\right] 
\leq & \left(\phi^2 + (2\phi\tilde{K}_1 + \tilde{c}_8)\eta^2 \right)\mathcal{V}_k^2 / \beta^2 + (2\phi\tilde{c}_1 + \tilde{c}_6)\eta\mathcal{V}_k / \beta^2+ 2\phi\eta^2 \hat{c}_1\mathcal{V}_k / \beta + \frac{\eta^2}{\beta^2} \tilde{c}_2 \\ + &  \eta^2(\tilde{c}_5 +\hat{c}_5)
+ \frac{\eta}{\beta}\tilde{c}_7 ,\\
\end{aligned}
\end{equation*}
where $\tilde{c}_8 := \max \left\{\frac{\tilde{c}_3 + \hat{c}_3}{\frac{1}{128}(1 - 2\lambda)^2\gamma^4}, \frac{\tilde{c}_4 + \hat{c}_4}{\frac{1}{32}(1 - 2\lambda)^2}\right\}$. By taking expectations on both sides of the above inequality and by using the fact that $\phi = 1 - \lambda\gamma\eta $ and $ \phi \leq 1$, we obtain
\begin{equation}
\begin{aligned}
\mathbb{E}\left[\mathcal{V}_{k + 1}^2\right] 
\leq & \left(1 - \lambda\gamma\eta + \tilde{K}\eta^2 \right)\mathbb{E}\left[\mathcal{V}_{k}^2\right]  + \tilde{c}_9\eta\mathbb{E}\left[\mathcal{V}_k\right]+ 2\hat{c}_1\beta\eta^2\mathbb{E}\left[\mathcal{V}_k\right] + \eta^2 \tilde{c}_2 \\+&  \eta^2(\tilde{c}_5 +\hat{c}_5)\beta^2
+ \eta\tilde{c}_7\beta ,\\
\end{aligned}
\end{equation}
where 
\begin{equation}\label{tilde_K}
\tilde{K} := 2\tilde{K}_1 + \tilde{c}_8
\end{equation}
and $\tilde{c}_9 := 2\tilde{c}_1 + \tilde{c}_6$. By using~\eqref{lm2.4: 1} and $\eta \leq 1$, we obtain
$$
\mathbb{E}\left[\mathcal{V}_{k + 1}^2\right] 
\leq \left(1 - \lambda\gamma\eta + \tilde{K}\eta^2 \right)\mathbb{E}\left[\mathcal{V}_{k}^2\right] + \eta D,
$$
where $D := \tilde{c}_9\tilde{c}_{10} + 2\hat{c}_1\tilde{c}_{10}\beta  + \tilde{c}_2  + (\tilde{c}_5 +\hat{c}_5)\beta^2  + \tilde{c}_7\beta$. Since $\eta \leq \frac{\lambda\gamma}{2\tilde{K}}$, we obtain
$$
\mathbb{E}\left[\mathcal{V}_{k + 1}^2\right] 
\leq \left(1 - \frac{\lambda\gamma\eta}{2} \right)\mathbb{E}\left[\mathcal{V}_{k}^2\right] + \eta D, 
$$
which implies 
\begin{equation}\label{C_V}
\mathbb{E}\left[\mathcal{V}_{k}^2\right] 
\leq \mathbb{E}\left[\mathcal{V}_{0}^2\right] + \frac{2D}{\gamma\lambda}:= C_{\mathcal{V}}.
\end{equation}
We claim that, for all $\theta \in \mathbb{R}^d$, we have that
\begin{equation}\label{lemma_A_3}
\frac{a^{\prime}}{3}|\theta|^2-\frac{b^{\prime}}{2} \log 3 \leq u(\theta) \leq u(0)+\frac{L}{2}|\theta|^2+\left|h(0)\right|\lvert\theta\rvert,
\end{equation}
Indeed, by using Remark~\ref{rmk:2}, we obtain
$$
\begin{aligned}
u(\theta) - u(0) & =\int_0^1\langle\theta, h(t \theta)\rangle \mathrm{d} t \\
& \leq \int_0^1|\theta||h(t \theta)| \mathrm{d} t \\
& \leq \int_0^1|\theta|\left(tL|\theta|+\lvert h(0)\rvert\right) \mathrm{d} t .
\end{aligned}
$$
This in turn leads to
$$
u(\theta) \leq u(0) +\frac{L}{2}|\theta|^2+ \lvert h(0)\rvert \lvert \theta\rvert.
$$
Next, we prove the lower bound. To this end, by taking $c \in(0,1)$ and using Remark~\ref{rmk:3}, we write
$$
\begin{aligned}
u(\theta) & =u(c \theta)+\int_c^1\langle\theta, h(t \theta)\rangle \mathrm{d} t \\
& \geq \int_c^1 \frac{1}{t}\langle t \theta, h(t \theta)\rangle \mathrm{d} t \\
& \geq \int_c^1 \frac{1}{t}\left(a^{\prime}|t \theta|^2-b^{\prime}\right) \mathrm{d} t \\
& =\frac{a^{\prime}\left(1-c^2\right)}{2}|\theta|^2+b^{\prime} \log c,
\end{aligned}
$$
which by taking $c=1 / \sqrt{3}$ leads to the bound. By using~\eqref{lemma_A_3}, we have 
\begin{equation}\label{func:V2}
\begin{aligned}
\mathcal{V}(\theta, v) & \leq \beta\left(\frac{L}{2}|\theta|^2+\lvert h(0)\rvert|\theta| + u(0)\right)+\frac{1}{4} \beta \gamma^2\left(\left|\theta+\gamma^{-1} v\right|^2+\left|\gamma^{-1} v\right|^2-\lambda|\theta|^2\right) \\
& \leq \beta\left(\frac{L}{2}\|\theta|^2+\lvert h(0)\rvert|\theta| + u(0)\right)+\frac{1}{4} \beta \gamma^2\left(2|\theta|^2+2 \gamma^{-2}|v|^2+\left|\gamma^{-1} v\right|^2-\lambda|\theta|^2\right) \\
& \leq \beta\left(L|\theta|^2 + u(0) + \frac{\lvert h(0)\rvert^2}{2 L}\right)+\frac{1}{4} \beta \gamma^2\left(2|\theta|^2+2 \gamma^{-2}|v|^2+\left|\gamma^{-1} v\right|^2-\lambda|\theta|^2\right) \\
& \leq\left(\beta L+\frac{1}{2} \beta \gamma^2\right)|\theta|^2+\frac{3}{4} \beta|v|^2+\beta u(0)+\frac{\beta \lvert h(0)\rvert^2}{2 L} .
\end{aligned}
\end{equation}
Taking square on both sides for $(\theta, v) = (\theta_0, V_0)$ and taking expectation yield
\begin{equation}\label{V^2}
\begin{aligned}
\mathbb{E}\left[\mathcal{V}_0^2 \right]
& \leq \mathbb{E}\left[\left(\left(\beta L+\frac{1}{2} \beta \gamma^2\right)|\theta_0|^2+\frac{3}{4} \beta|V_0|^2+\beta u(0)+\frac{\beta \lvert h(0)\rvert^2}{2 L}\right)^2\right] \\ 
& \leq 4\beta^2(L + \frac{\gamma^2}{2})\mathbb{E}\left[\lvert\theta_0\rvert^4\right] + \frac{9}{4}\beta^2\mathbb{E}\left[\lvert V_0\rvert^4\right] + 4\beta^2u(0)^2 + \beta^2\frac{\lvert h(0)\rvert^2}{L^2}\\
& \leq \beta^2\left(\max\{4(L + \frac{\gamma^2}{2}), \frac{9}{4}\}\left(\mathbb{E}\left[\lvert\theta_0\rvert^4\right] + \mathbb{E}\left[\lvert V_0\rvert^4\right]\right) + 4u(0)^2 + \frac{\lvert h(0)\rvert^2}{L^2}\right).
\end{aligned}
\end{equation}
Note that according to Assumption~\ref{asm:A2}, the RHS of~\eqref{V^2} is bounded. 

To obtain the second inequality, for every $n \in \mathbb{N}_0$, we denote by $x \wedge y := \min\left\{x, y\right\}$ for any $x, y \in \mathbb{R}$ and define the stopping time
\begin{equation}\label{stopping time}
\tau_n:= \inf \left\{t \in \left[k T, (k+1)T\right]:\left\lvert\mathcal{V}\left(\bar{\zeta}_t^{\eta, k}, \bar{Z}_t^{\eta, k}\right)\right\rvert > n\right\} \wedge \left(k+1\right) T.
\end{equation}
It is clear that $\tau_n \uparrow (k+1)T$. By using 
$$
\begin{aligned}
&\mathcal{V}\left(\bar{\zeta}_{t \wedge \tau_n}^{\eta, k}, \bar{Z}_{t \wedge \tau_n}^{\eta, k}\right) \\& = \mathcal{V}\left(\bar{\theta}_{k T}^\eta, \bar{V}_{k T}^\eta\right) + \int_{k T}^{t} \eta\mathcal{A} \mathcal{V}\left(\bar{\zeta}_{s \wedge \tau_n}^{\eta, k}, \bar{Z}_{s \wedge \tau_n}^{\eta, k}\right) \mathbbm{1}_{\left\{k T \leq s \leq \tau_n\right\}}\mathrm{d} s  + \sqrt{2\gamma \eta\beta^{-1}}  \int_{k T}^{t} \nabla_v \mathcal{V}\left(\bar{\zeta}_{s \wedge \tau_n}^{\eta, k}, \bar{Z}_{s \wedge \tau_n}^{\eta, k}\right) \mathbbm{1}_{\left\{k T \leq s \leq \tau_n\right\}}\mathrm{d} B_s^{\eta}
\end{aligned}
$$ and applying Itô formula to the stopped process $\mathcal{V}^2\left(\bar{\zeta}_{t \wedge \tau_n}^{\eta, k}, \bar{Z}_{t \wedge \tau_n}^{\eta, k}\right)$, one obtains
\begin{equation}\label{itov}
\begin{aligned}
&\mathcal{V}^2\left(\bar{\zeta}_{t \wedge \tau_n}^{\eta, k}, \bar{Z}_{t \wedge \tau_n}^{\eta, k}\right) \\&= \mathcal{V}^2\left(\bar{\theta}_{k T}^\eta, \bar{V}_{k T}^\eta\right) + 2\eta \int_{k T}^{t}\!\bigg(\mathcal{V}\left(\bar{\zeta}_{s \wedge \tau_n}^{\eta, k}, \bar{Z}_{s \wedge \tau_n}^{\eta, k}\right) \mathcal{A} \mathcal{V}\left(\bar{\zeta}_{s \wedge \tau_n}^{\eta, k}, \bar{Z}_{s \wedge \tau_n}^{\eta, k}\right) + 2\gamma \beta^{-1}\!  \left(\nabla_v \mathcal{V}\left(\bar{\zeta}_{s \wedge \tau_n}^{\eta, k}, \bar{Z}_{s \wedge \tau_n}^{\eta, k}\right)\right)^2 \bigg)\mathbbm{1}_{\left\{k T \leq s \leq \tau_n\right\}}\mathrm{d} s 
\\ & \quad + 2\sqrt{2\gamma \eta\beta^{-1}} \int_{k T}^{t}\mathcal{V}\left(\bar{\zeta}_{s \wedge \tau_n}^{\eta, k}, \bar{Z}_{s \wedge \tau_n}^{\eta, k}\right) \nabla_v \mathcal{V}\left(\bar{\zeta}_{s \wedge \tau_n}^{\eta, k}, \bar{Z}_{s \wedge \tau_n}^{\eta, k}\right) \mathbbm{1}_{\left\{k T \leq s \leq \tau_n\right\}} \mathrm{d} B_s^{\eta}.
\end{aligned}
\end{equation}
By using~\eqref{Lyapunov}, we obtain, for any $\theta, v \in \mathbb{R}^d$, that
$$
\nabla_v \mathcal{V}\left(\theta, v\right)=\beta v+\frac{\beta \gamma}{2} \theta,
$$
which implies
\begin{equation}\label{nable v}
\lvert\nabla_v \mathcal{V}\left(\theta, v\right)\rvert^2 \leq 2\beta^2 v^2 + \frac{\beta^2\gamma^2}{2}\theta^2.
\end{equation}
By taking expectation on both sides of~\eqref{itov}, using~\eqref{lm4.2:2},~\eqref{nable v}, and the fact that
\begin{equation*}
\begin{aligned}
& \mathbb{E}\left[\int_{k T}^{ t }\left\lvert \mathcal{V}\left(\bar{\zeta}_{s \wedge \tau_n}^{\eta, k}, \bar{Z}_{s \wedge \tau_n}^{\eta, k}\right) \nabla_v \mathcal{V}\left(\bar{\zeta}_{s \wedge \tau_n}^{\eta, k}, \bar{Z}_{s \wedge \tau_n}^{\eta, k}\right)\right\rvert^2 \mathbbm{1}_{\left\{k T \leq s \leq \tau_n\right\}} \mathrm{d} s\right]  
 \\ &\leq
2 n^2 \beta^2 \int_{k T}^{ t }  \mathbb{E}\left[\left\lvert\bar{Z}_{s \wedge \tau_n}^{\eta, k}\right\rvert^2\mathbbm{1}_{\left\{k T \leq s \leq \tau_n\right\}}\right] \mathrm{d} s + \frac{n^2 \gamma^2 \beta^2}{2} \int_{k T}^{ t }  \mathbb{E}\left[\left\lvert\bar{\zeta}_{s \wedge \tau_n}^{\eta, k}\right\rvert^2\mathbbm{1}_{\left\{k T \leq s \leq \tau_n\right\}}\right] \mathrm{d} s 
\\ & \leq
2 n^2 \beta^2 \int_{k T}^{ t }  \mathbb{E}\left[\left\lvert\bar{Z}_{s}^{\eta, k}\right\rvert^2\right] \mathrm{d} s + \frac{n^2 \gamma^2 \beta^2}{2} \int_{k T}^{ t }  \mathbb{E}\left[\left\lvert\bar{\zeta}_{s}^{\eta, k}\right\rvert^2\right] \mathrm{d} s  < \infty, 
\end{aligned}
\end{equation*}
where the last inequality holds due to~\eqref{lemma3.4.12} in Lemma~\ref{lemma:3.4}, we obtain that
\begin{equation*}
\begin{aligned}
&\mathbb{E}\left[\mathcal{V}^2\left(\bar{\zeta}_{t \wedge \tau_n}^{\eta, k}, \bar{Z}_{t \wedge \tau_n}^{\eta, k}\right)\right] \\& = \mathbb{E}\left[\mathcal{V}^2\left(\bar{\theta}_{k T}^\eta, \bar{V}_{k T}^\eta\right)\right] + 2\eta \int_{k T}^{t} \mathbb{E}\left[\mathcal{V}\left(\bar{\zeta}_{s \wedge \tau_n}^{\eta, k}, \bar{Z}_{s \wedge \tau_n}^{\eta, k}\right) \mathcal{A} \mathcal{V}\left(\bar{\zeta}_{s \wedge \tau_n}^{\eta, k}, \bar{Z}_{s \wedge \tau_n}^{\eta, k}\right)\mathbbm{1}_{\left\{k T \leq s \leq \tau_n\right\}}\right]\mathrm{d} s 
\\ & \quad + 4\gamma \eta\beta^{-1} \int_{k T}^{t}\mathbb{E}\left[\left(\nabla_v \mathcal{V}\left(\bar{\zeta}_{s \wedge \tau_n}^{\eta, k}, \bar{Z}_{s \wedge \tau_n}^{\eta, k}\right)\right)^2\mathbbm{1}_{\left\{k T \leq s \leq \tau_n\right\}}\right]\mathrm{d} s 
\\ & \leq
\mathbb{E}\left[\mathcal{V}^2\left(\bar{\theta}_{k T}^\eta, \bar{V}_{k T}^\eta\right)\right] + 2\eta\gamma\left(d+ A_c\right) \int_{k T}^{t}\mathbb{E}\left[\left\lvert\mathcal{V}\left(\bar{\zeta}_{s \wedge \tau_n}^{\eta, k}, \bar{Z}_{s \wedge \tau_n}^{\eta, k}\right)\right\rvert\right] \mathrm{d} s + 2\eta\gamma\lambda \int_{k T}^{t} \mathbb{E}\left[\mathcal{V}^2\left(\bar{\zeta}_{s \wedge \tau_n}^{\eta, k}, \bar{Z}_{s \wedge \tau_n}^{\eta, k}\right)\right]\mathrm{d} s \\ & \quad +  8\eta\gamma\beta \int_{k T}^{t}\mathbb{E}\left[\left\lvert \bar{Z}_{s \wedge \tau_n}^{\eta, k}\right\rvert^2\mathbbm{1}_{\left\{k T \leq s \leq \tau_n\right\}}\right] \mathrm{d} s 
+
2\eta\beta\gamma^3 \int_{k T}^{t}\mathbb{E}\left[\left\lvert\bar{\zeta}_{s \wedge \tau_n}^{\eta, k}\right\rvert^2\mathbbm{1}_{\left\{k T \leq s \leq \tau_n\right\}}\right] \mathrm{d} s.
\end{aligned}
\end{equation*}
Note that
\begin{equation}\label{func:V3}
\begin{aligned}
\left\lvert \mathcal{V}(\theta, v)\right\rvert & \leq \beta\left(\frac{L}{2}|\theta|^2+\lvert h(0)\rvert|\theta| + u(0)\right)+\frac{1}{4} \beta \gamma^2\left(\left|\theta+\gamma^{-1} v\right|^2+\left|\gamma^{-1} v\right|^2 
 + \lambda|\theta|^2\right) \\
& \leq \beta\left(\frac{L}{2}\|\theta|^2+\lvert h(0)\rvert|\theta| + u(0)\right)+\frac{1}{4} \beta \gamma^2\left(2|\theta|^2+2 \gamma^{-2}|v|^2+\left|\gamma^{-1} v\right|^2 + \lambda|\theta|^2\right) \\
& \leq \beta\left(L|\theta|^2 + u(0) + \frac{\lvert h(0)\rvert^2}{2 L}\right)+\frac{1}{4} \beta \gamma^2\left(3|\theta|^2+2 \gamma^{-2}|v|^2+\left|\gamma^{-1} v\right|^2 \right) \\
& \leq \beta\left( L+\frac{3}{4} \gamma^2\right)|\theta|^2+\frac{3}{4} \beta|v|^2+\beta u(0) + \frac{\beta \lvert h(0)\rvert^2}{2 L},
\end{aligned}
\end{equation}
where the second inequality holds due to $\lambda \leq 1 / 4$. By using~\eqref{func:V3} and $\eta T \leq 1$, we obtain
\begin{equation*}
\begin{aligned}
&\mathbb{E}\left[\mathcal{V}^2\left(\bar{\zeta}_{t \wedge \tau_n}^{\eta, k}, \bar{Z}_{t \wedge \tau_n}^{\eta, k}\right)\right] 
\\ & \leq
\mathbb{E}\left[\mathcal{V}^2\left(\bar{\theta}_{k T}^\eta, \bar{V}_{k T}^\eta\right)\right] + 2\gamma\left(d+ A_c\right)\left(\beta u(0) + \frac{\beta|h(0)|^2}{2 L}\right)  + 2\eta\gamma\lambda \int_{k T}^{t} \mathbb{E}\left[\mathcal{V}^2\left(\bar{\zeta}_{s \wedge \tau_n}^{\eta, k}, \bar{Z}_{s \wedge \tau_n}^{\eta, k}\right)\right]\mathrm{d} s \\ & \quad +
2 \eta\gamma\beta\left(\gamma^2 + (d+A_c)\left(L + \frac{3}{4}\gamma^2\right)\right)\int_{k T}^{t}\mathbb{E}\left[\left\lvert\bar{\zeta}_{s}^{\eta, k}\right\rvert^2\right] \mathrm{d} s  +  2 \eta\gamma\beta(4+ \frac{3}{4}(d+A_c)) \int_{k T}^{t}\mathbb{E}\left[\left\lvert \bar{Z}_{s}^{\eta, k}\right\rvert^2\right]  \mathrm{d} s 
\\ & \leq C_{\mathcal{V}}^{\prime}, 
\end{aligned}
\end{equation*}
where the last inequality holds due to~\eqref{lemma3.4.12} in Lemma~\ref{lemma:3.4} and Gronwall's lemma, and where
\begin{equation}\label{c_v_prime}
\begin{aligned}
C_{\mathcal{V}}^{\prime}& :=\bigg(C_{\mathcal{V}} + 2\beta\gamma\left(d+A_c\right)\left(u(0) + \frac{|h(0)|^2}{2 L}\right) + 2 \gamma \beta\left(4+\frac{3}{4}\left(d+A_c\right)\right) C_{Z} \\ & \quad + 2\gamma \beta\left(\gamma^2+\left(d+A_c\right)\left(L+\frac{3}{4}\gamma^2\right)\right)C_\zeta \bigg)e^{2\lambda\gamma}.
\end{aligned}
\end{equation}
By letting  $n \rightarrow \infty$ and using Fatou's lemma, one obtains
$$
\begin{aligned}
\mathbb{E}\left[\left\lvert\mathcal{V}\left(\bar{\zeta}_{t}^{\eta, k}, \bar{Z}_{t}^{\eta, k}\right)\right\rvert^2 \right] \leq  \lim_{n \to +\infty} \mathbb{E}\left[\left\lvert\mathcal{V}\left(\bar{\zeta}_{t \wedge \tau_n}^{\eta, k}, \bar{Z}_{t \wedge \tau_n}^{\eta, k}\right)\right\rvert^2 \right] \leq  C_{\mathcal{V}}^{\prime}.
\end{aligned}
$$ 
Then the desired inequality follows by taking the supremum.
\end{proof}

\section{Proofs of Section~\ref{subsec_5.2}}

\begin{lemma}\label{lemmac.1}
Let Assumptions~\ref{asm:A3} and~\ref{asm:A2} hold. For any $0 < \eta \leq \eta_{\max}$ with $\eta_{\max}$ given in~\eqref{eta_max} and any $k=1, \ldots, T$ with $T := \lfloor 1/\eta\rfloor$, we obtain
$$
\sup _{n \in \mathbb{N}} \sup _{t \in[n T,(n+1) T]} \mathbb{E}\left[\left|h\left(\bar{\zeta}_t^{\eta, n}\right)-H\left(\bar{\zeta}_t^{\eta, n}, X_{n T+k}\right)\right|^2\right] \leq \sigma_H,
$$
where $\sigma_H:=8 L_2^2 \bar{\sigma}_z+16 \hat{\sigma}_z, \bar{\sigma}_z:=\mathbb{E}\left[\left(1+\left|X_0\right|+\left|\mathbb{E}\left[X_0\right]\right|\right)^{2 \rho}\left|X_0-\mathbb{E}\left[X_0\right]\right|^2\right]$ and 
$\hat{\sigma}_z:=\mathbb{E}\left[\bar{K}_1^2\left(X_0\right) 
 +\bar{K}_1^2\left(\mathbb{E}\left[X_0\right]\right)\right]$.
\end{lemma}
\begin{proof}
Let $\mathcal{H}_t:=\mathcal{F}_{\infty}^\eta \vee \mathcal{G}_{\lfloor t \rfloor}$. By using Assumptions~\ref{asm:A3},~\ref{asm:A2} and~\cite[Lemma 6.1]{chau2019fixed}, we obtain
$$
\begin{aligned}
\mathbb{E} & {\left[\left|h\left(\bar{\theta}_t^{\eta, n}\right)-H\left(\bar{\theta}_t^{\eta, n}, X_{n T+k}\right)\right|^2\right] } \\
= & \mathbb{E}\left[\mathbb{E}\left[\left|h\left(\bar{\theta}_t^{\eta, n}\right)-H\left(\bar{\theta}_t^{\eta, n}, X_{n T+k}\right)\right|^2 \mid \mathcal{H}_{n T}\right]\right] \\
= & \mathbb{E}\left[\mathbb{E}\left[\left|\mathbb{E}\left[H\left(\bar{\theta}_t^{\eta, n}, X_{n T+k}\right) \mid \mathcal{H}_{n T}\right]-H\left(\bar{\theta}_t^{\eta, n}, X_{n T+k}\right)\right|^2 \mid \mathcal{H}_{n T}\right]\right] \\
\leq & 4 \mathbb{E}\left[\mathbb{E}\left[\left|H\left(\bar{\theta}_t^{\eta, n}, X_{n T+k}\right)-H\left(\bar{\theta}_t^{\eta, n}, \mathbb{E}\left[X_{n T+k} \mid \mathcal{H}_{n T}\right]\right)\right|^2 \mid \mathcal{H}_{n T}\right]\right] \\
\leq & 4 \mathbb{E}\left[\mathbb { E } \left[\left(\left(1+\left|X_{n T+k}\right|+\left|\mathbb{E}\left[X_{n T+k} \mid \mathcal{H}_{n T}\right]\right|\right)^\rho L_2\left|X_{n T+k}-\mathbb{E}\left[X_{n T+k} \mid \mathcal{H}_{n T}\right]\right|\right.\right.\right. \\
& \left.\left.\left.+\bar{K}_1\left(X_{n T+k}\right)+\bar{K}_1\left(\mathbb{E}\left[X_{n T+k} \mid \mathcal{H}_{n T}\right]\right)\right)^2 \mid \mathcal{H}_{n T}\right]\right] \\
\leq & 8 L_2^2 \mathbb{E}\left[\left(1+\left|X_0\right|+\left|\mathbb{E}\left[X_0\right]\right|\right)^{2 \rho}\left|X_0-\mathbb{E}\left[X_0\right]\right|^2\right]+16 \mathbb{E}\left[\bar{K}_1^2\left(X_0\right)+\bar{K}_1^2\left(\mathbb{E}\left[X_0\right]\right)\right] \\
= & 8 L_2^2 \bar{\sigma}_z+16 \hat{\sigma}_z,
\end{aligned}
$$
where $\bar{\sigma}_z:=\mathbb{E}\left[\left(1+\left|X_0\right|+\left|\mathbb{E}\left[X_0\right]\right|\right)^{2 \rho}\left|X_0-\mathbb{E}\left[X_0\right]\right|^2\right]$ and $\hat{\sigma}_z:=\mathbb{E}\left[\bar{K}_1^2\left(X_0\right)+\bar{K}_1^2\left(\mathbb{E}\left[X_0\right]\right)\right]$.
\end{proof}

\begin{proof}[\textit{Proof of Proposition~\ref{lemma:4.1}}]\label{prf of lemma 4.1}
First, we note that
\begin{equation}\label{thm4.1: 1}
\begin{aligned}
W_2\left(\mathcal{L}\left(\bar{\theta}_t^\eta, \bar{V}_t^\eta\right), \mathcal{L}\left(\bar{\zeta}_t^{\eta, n}, \bar{Z}_t^{\eta, n}\right)\right) \leq \mathbb{E}\left[\left|\bar{\theta}_t^\eta-\bar{\zeta}_t^{\eta, n}\right|^2\right]^{1 / 2}+\mathbb{E}\left[\left|\bar{V}_t^\eta-\bar{Z}_t^{\eta, n}\right|^2\right]^{1 / 2}.
\end{aligned}
\end{equation}
To bound the first term on the RHS of~\eqref{thm4.1: 1}, we start by using the definition of $\left(\bar{\theta}_t^\eta, \bar{V}_t^\eta\right)$ in~\eqref{cont.SGHMC} and $\left(\bar{\zeta}_t^{\eta, n}, \bar{Z}_t^{\eta, n}\right)$ in Definition~\ref{def:1}, and by employing the synchronous coupling, to obtain
$$
\lvert\Bar{\theta}_t^{\eta} - \Bar{\zeta}_t^{\eta, n}\rvert \leq \eta \int_{nT}^{t}\lvert\Bar{V}_{\lfloor s\rfloor}^{\eta} - \Bar{Z}_{ s}^{\eta, n}\rvert \mathrm{d}s.
$$
This implies, by using Cauchy-Schwarz inequality, that
$$
\begin{aligned}
\sup _{n T \leq u \leq t} \mathbb{E}\left[\left|\bar{\theta}_u^\eta-\bar{\zeta}_u^{\eta, n}\right|^2\right] & \leq \eta \sup _{n T \leq u \leq t} \int_{n T}^u \mathbb{E}\left[\left|\bar{V}_{\lfloor s\rfloor}^\eta-\bar{Z}_s^{\eta, n}\right|^2\right] \mathrm{d} s, \\
& =\eta \int_{n T}^t \mathbb{E}\left[\left|\bar{V}_{\lfloor s\rfloor}^\eta-\bar{Z}_s^{\eta, n}\right|^2\right] \mathrm{d} s.
\end{aligned}
$$
Next, we set for any $t\in \left[nT, (n+1)T\right]$,
$$
\begin{aligned}
\left|\bar{V}_{\lfloor t\rfloor}^\eta-\bar{Z}_t^{\eta, n}\right| & \leq\left|\bar{V}_{\lfloor t\rfloor}^\eta-\bar{V}_t^\eta\right|+\left|\bar{V}_t^\eta-\bar{Z}_t^{\eta, n}\right| \\
& \leq\left|\bar{V}_{\lfloor t\rfloor}^\eta-\bar{V}_t^\eta\right|+\left|-\gamma \eta \int_{n T}^t\left[\bar{V}_{\lfloor s\rfloor}^\eta-\bar{Z}_s^{\eta, n}\right] \mathrm{d} s-\eta \int_{n T}^t\left[H\left(\bar{\theta}_{\lfloor s\rfloor}^\eta, X_{\lceil s\rceil}\right)-h\left(\bar{\zeta}_s^{\eta, n}\right)\right] \mathrm{d} s\right|\\
& \leq\left|\bar{V}_{\lfloor t\rfloor}^\eta-\bar{V}_t^\eta\right|+\gamma \eta \int_{n T}^t\left|\bar{V}_{\lfloor s\rfloor}^\eta-\bar{Z}_s^{\eta, n}\right| \mathrm{d} s+\eta\left|\int_{n T}^t\left[H\left(\bar{\theta}_{\lfloor s\rfloor}^\eta, X_{\lceil s\rceil}\right)-h\left(\bar{\zeta}_s^{\eta, n}\right)\right] \mathrm{d} s\right| \\
& \leq\left|\bar{V}_{\lfloor t\rfloor}^\eta-\bar{V}_t^\eta\right|+\gamma \eta \int_{n T}^t\left|\bar{V}_{\lfloor s\rfloor}^\eta-\bar{Z}_s^{\eta, n}\right| \mathrm{d} s \\
& +\eta \int_{n T}^t\left|H\left(\bar{\theta}_{\lfloor s\rfloor}^\eta, X_{\lceil s\rceil}\right)-h(\bar{\theta}_{\lfloor s\rfloor}^\eta)\right| \mathrm{d} s + \eta\left|\int_{n T}^t\left[h(\bar{\theta}_{\lfloor s\rfloor}^\eta)-h\left(\bar{\zeta}_s^{\eta, n}\right)\right] \mathrm{d} s\right|.
\end{aligned}
$$
We take the squares on both sides and use $(a + b)^2 \leq 2(a^2 + b^2)$ twice to obtain
$$
\begin{aligned}
\left|\bar{V}_{\lfloor t\rfloor}^\eta-\bar{Z}_t^{\eta, n}\right|^2 &\leq 4\left|\bar{V}_{\lfloor t\rfloor}^\eta-\bar{V}_t^{\eta}\right|^2 + 4\gamma^2\eta^2\left(\int_{nT}^{t}\lvert\bar{V}_{\lfloor s\rfloor}^\eta- \bar{Z}_s^{\eta, n} \rvert \mathrm{d}s \right)^2+4\eta^2\left|\int_{n T}^t\left[h(\bar{\theta}_{\lfloor s\rfloor}^\eta)-h\left(\bar{\zeta}_s^{\eta, n}\right)\right] \mathrm{d} s\right|^2 \\& \quad + 4\eta^2\left(\int_{n T}^t\left|H\left(\bar{\theta}_{\lfloor s\rfloor}^\eta, X_{\lceil s\rceil}\right)-h(\bar{\theta}_{\lfloor s\rfloor}^\eta)\right| \mathrm{d} s\right)^2\\
&\leq 4\left|\bar{V}_{\lfloor t\rfloor}^\eta-\bar{V}_t^{\eta}\right|^2 + 4\gamma^2\eta\left(\int_{nT}^{t}\lvert\bar{V}_{\lfloor s\rfloor}^\eta- \bar{Z}_s^{\eta, n} \rvert^2 \mathrm{d}s \right) + 4\eta\int_{n T}^t\left|h(\bar{\theta}_{\lfloor s\rfloor}^\eta)-h\left(\bar{\zeta}_s^{\eta, n}\right)\right|^2 \mathrm{d} s \\& \quad +4\eta^2\left(\int_{n T}^t\left|H\left(\bar{\theta}_{\lfloor s\rfloor}^\eta, X_{\lceil s\rceil}\right)-h(\bar{\theta}_{\lfloor s\rfloor}^\eta)\right| \mathrm{d} s\right)^2,
\end{aligned}
$$
where the last inequality holds due to Cauchy-Schwarz inequality and $\eta T \leq 1$. Then taking expectations on both sides and by noticing that $X_{\lceil s\rceil}$ is independent of $\bar{\theta}_{\lfloor s\rfloor}^\eta$ and $\bar{\zeta}_s^{\eta, n}$ , we obtain
\begin{equation}\label{lm4.1:V}
\begin{aligned}
\mathbb{E}\left[\left|\bar{V}_{\lfloor t\rfloor}^\eta-\bar{Z}_t^{\eta, n}\right|^2 \right]
&\leq  4\mathbb{E}\left[\left|\bar{V}_{\lfloor t\rfloor}^\eta-\bar{V}_t^{\eta}\right|^2\right] + 4\gamma^2\eta\left(\int_{nT}^{t}\mathbb{E}\left[\lvert\bar{V}_{\lfloor s\rfloor}^\eta- \bar{Z}_s^{\eta, n} \rvert^2\right] \mathrm{d}s \right) \\&\quad + 4\eta^2\mathbb{E}\left[\left(\int_{n T}^t\left|H\left(\bar{\theta}_{\lfloor s\rfloor}^\eta, X_{\lceil s\rceil}\right)-h(\bar{\theta}_{\lfloor s\rfloor}^\eta)\right| \mathrm{d} s\right)^2\right]\\&\quad + 4\eta\int_{n T}^t\mathbb{E}\left[\left|h(\bar{\theta}_{\lfloor s\rfloor}^\eta)-h\left(\bar{\zeta}_s^{\eta, n}\right)\right|^2\right] \mathrm{d} s \\&\leq 4\mathbb{E}\left[\left|\bar{V}_{\lfloor t\rfloor}^\eta-\bar{V}_t^{\eta}\right|^2\right] + 4\gamma^2\eta\left(\int_{nT}^{t}\mathbb{E}\left[\lvert\bar{V}_{\lfloor s\rfloor}^\eta- \bar{Z}_s^{\eta, n} \rvert^2\right] \mathrm{d}s \right) \\&\quad + 4\eta^2\mathbb{E}\left[\left(\int_{n T}^t\left|H\left(\bar{\theta}_{\lfloor s\rfloor}^\eta, X_{\lceil s\rceil}\right)-h(\bar{\theta}_{\lfloor s\rfloor}^\eta)\right| \mathrm{d} s\right)^2\right]\\&\quad + 4\eta L^2\int_{n T}^t\mathbb{E}\left[\left|\bar{\theta}_{\lfloor s\rfloor}^\eta - \bar{\zeta}_s^{\eta, n}\right|^2\right] \mathrm{d} s.
\end{aligned}
\end{equation}
Note that for any $t \geq 0$, we have
$$
\bar{V}_t^\eta=\bar{V}_{\lfloor t\rfloor}^\eta-\eta \gamma \int_{\lfloor t\rfloor}^t \bar{V}_{\lfloor s\rfloor}^\eta \mathrm{d} s-\eta \int_{\lfloor t\rfloor}^t H\left(\theta_{\lfloor s\rfloor}, X_{\lceil s\rceil}\right) \mathrm{d} s+\sqrt{2 \gamma \eta \beta^{-1}}\left(B_t^\eta-B_{\lfloor t\rfloor}^\eta\right) .
$$
By using Remark~\ref{rmk:1} and $\eta \leq 1$, we therefore obtain
$$
\begin{aligned}
\mathbb{E}\left[\left|\bar{V}_{\lfloor t\rfloor}^\eta-\bar{V}_t^\eta\right|^2\right] & =\mathbb{E}\left[\left|\eta \gamma \int_{\lfloor t\rfloor}^t \bar{V}_{\lfloor s\rfloor}^\eta \mathrm{d} s + \eta \int_{\lfloor t\rfloor}^t H\left(\theta_{\lfloor s\rfloor}, X_{\lceil s\rceil}\right) \mathrm{d} s - \sqrt{2 \gamma \eta \beta^{-1}}\left(B_t^\eta-B_{\lfloor t\rfloor}^\eta\right)\right|^2\right] \\
& \leq 3 \eta^2 \gamma^2 C_v + 6\eta^2 L_1^2 \mathbb{E}\left[(1 + \lvert X_0\rvert)^{2(\rho + 1)}\right]C_{\theta} + 12 \eta^2 L_2^2\mathbb{E}\left[(1 + \lvert X_0\rvert)^{2(\rho + 1)}\right] \\&\quad + 12\eta^2\mathbb{E}\left[F^2_{\ast}(X_0)\right] + 6 \gamma \eta \beta^{-1} d \\&
\leq \sigma_V \eta,
\end{aligned}
$$
where $\sigma_V := 3\gamma^2 C_v + 6\mathbb{E}\left[(1 + \lvert X_0\rvert)^{2(\rho + 1)}\right](C_{\theta}L_1^2 + 2L_2^2)+ 12\mathbb{E}\left[F^2_{\ast}(X_0)\right] + 6 \gamma \beta^{-1} d$. By applying Grönwall inequality to~\eqref{lm4.1:V} and using $\eta T \leq 1$, we arrive at
$$
\begin{aligned}
\mathbb{E}\left[\left|\bar{V}_{\lfloor t\rfloor}^\eta-\bar{Z}_t^{\eta, n}\right|^2 \right]
\leq & c_1\bigg(\sigma_{V}\eta + \eta^2\mathbb{E}\left[\left(\int_{n T}^t\left|H\left(\bar{\theta}_{\lfloor s\rfloor}^\eta, X_{\lceil s\rceil}\right)-h(\bar{\theta}_{\lfloor s\rfloor}^\eta)\right| \mathrm{d} s\right)^2\right]\\&\quad + \eta L^2\int_{n T}^t\mathbb{E}\left[\left|\bar{\theta}_{\lfloor s\rfloor}^\eta - \bar{\zeta}_s^{\eta, n}\right|^2\right] \mathrm{d} s\bigg), 
\end{aligned}
$$
where $c_1 = 4 e^{4\gamma^2}$. Next, 
\begin{equation}\label{thm4.1:3}
\begin{aligned}
\sup _{n T \leq u \leq t} \mathbb{E}\left[\left|\bar{\theta}_u^\eta-\bar{\zeta}_u^{\eta, n}\right|^2\right] & \leq \eta \int_{n T}^t \mathbb{E}\left[\left|\bar{V}_{\lfloor s\rfloor}^\eta-\bar{Z}_s^{\eta, n}\right|^2\right] \mathrm{d} s \\
& \leq c_1 \eta L^2 \sup _{n T \leq s \leq t} \int_{n T}^s \mathbb{E}\left[\left|\bar{\theta}_{\left\lfloor s^{\prime}\right\rfloor}^\eta-\bar{\zeta}_{s^{\prime}}^{\eta, n}\right|^2\right] \mathrm{d} s^{\prime} + c_1 \eta \sigma_V \\&\quad + c_1 \eta^3 \int_{n T}^t \mathbb{E}\left[\left|\int_{n T}^s\left[H\left(\bar{\theta}_{\lfloor s^{\prime}\rfloor}^\eta, X_{\left\lceil s^{\prime}\right\rceil}\right)-h\left(\bar{\theta}_{\lfloor s^{\prime}\rfloor}^\eta\right)\right] \mathrm{d} s^{\prime}\right|^2\right] \mathrm{d} s .
\end{aligned}
\end{equation}
To obtain an upper bound for the first term on the RHS of~\eqref{thm4.1:3}, we observe that
\begin{equation}\label{thm4.1:4}
\begin{aligned}
\sup _{n T \leq s \leq t} \int_{n T}^s \mathbb{E}\left[\left|\bar{\theta}_{\left\lfloor s^{\prime}\right\rfloor}^\eta-\bar{\zeta}_{s^{\prime}}^{\eta, n}\right|^2\right] \mathrm{d} s^{\prime} & =\int_{n T}^t \mathbb{E}\left[\left|\bar{\theta}_{\left\lfloor s^{\prime}\right\rfloor}^\eta-\bar{\zeta}_{s^{\prime}}^{\eta, n}\right|^2\right] \mathrm{d} s^{\prime} \\
& \leq \int_{n T}^t \sup_{n T \leq u \leq s^{\prime}} \mathbb{E}\left[\left|\bar{\theta}_{\lfloor u\rfloor}^\eta-\bar{\zeta}_u^{\eta, n}\right|^2\right] \mathrm{d} s^{\prime} \\
& \leq \int_{n T}^t \sup_{n T \leq u \leq s^{\prime}}  \mathbb{E}\left[\left|\bar{\theta}_u^\eta-\bar{\zeta}_u^{\eta, n}\right|^2\right] \mathrm{d} s^{\prime}.
\end{aligned}
\end{equation}
Then, we bound the last term on the RHS of~\eqref{thm4.1:3} by partitioning the integral. For any $s \in \left[nT, t\right]$ and $t \in \left[nT, (n+1)T\right]$, we denote by $K:=\lfloor s - nT\rfloor$ and have 
$$
\left|\int_{n T}^s\left[h\left(\bar{\theta}_{\lfloor s^{\prime}\rfloor}^\eta\right)-H\left(\bar{\theta}_{\lfloor s^{\prime}\rfloor}^\eta, X_{\left\lceil s^{\prime}\right\rceil}\right)\right] \mathrm{d} s^{\prime}\right|=\left|\sum_{k=1}^K I_k+R_K\right|,
$$
where 
$$
I_k=\int_{ n T + k - 1}^{n T+k}\left[h\left(\bar{\theta}_{\lfloor s^{\prime}\rfloor}^\eta\right)-H\left(\bar{\theta}_{\lfloor s^{\prime}\rfloor}^\eta, X_{n T+k}\right)\right] \mathrm{d} s^{\prime} , \quad R_K=\int_{n T+K}^s\left[h\left(\bar{\theta}_{s^{\prime}}^{\eta, n}\right)-H\left(\bar{\theta}_{s^{\prime}}^{\eta, n}, X_{n T+K+1}\right)\right] \mathrm{d} s^{\prime}.
$$
By taking squares on both sides, we obtain
$$
\left|\sum_{k=1}^K I_k+R_K\right|^2=\sum_{k=1}^K\left|I_k\right|^2+2 \sum_{k=2}^K \sum_{j=1}^{k-1}\left\langle I_k, I_j\right\rangle+2 \sum_{k=1}^K\left\langle I_k, R_K\right\rangle+\left|R_K\right|^2.
$$
Define the filtration $\mathcal{H}_s := \mathcal{F}_{\infty}^\eta \vee \mathcal{G}_{\lfloor s\rfloor}, n T \leq s \leq (n+1)T$. We have, for any $k=2, \ldots, K, j=1, \ldots, k-1$, that
$$
\begin{aligned}
& \mathbb{E}\left[\left\langle I_k, I_j\right\rangle\right] \\
& =\mathbb{E}\left[\mathbb{E}\left[\left\langle I_k, I_j\right\rangle \mid \mathcal{H}_{n T+j}\right]\right] \\
& =\mathbb{E}\left[\mathbb{E}\left[\left\langle\int_{n T+(k-1)}^{n T+k}\left[H\left(\bar{\theta}_{s^{\prime}}^{\eta, n}, X_{n T+k}\right)-h\left(\bar{\theta}_{s^{\prime}}^{\eta, n}\right)\right] \mathrm{d} s^{\prime}, \int_{n T+(j-1)}^{n T+j}\left[H\left(\bar{\theta}_{s^{\prime}}^{\eta, n}, X_{n T+j}\right)-h\left(\bar{\theta}_{s^{\prime}}^{\eta, n}\right)\right] \mathrm{d} s^{\prime}\right\rangle \mid \mathcal{H}_{n T+j}\right]\right] \\
& =\mathbb{E}\left[\left\langle\int_{n T+(k-1)}^{n T+k} \mathbb{E}\left[H\left(\bar{\theta}_{s^{\prime}}^{\eta, n}, X_{n T+k}\right)-h\left(\bar{\theta}_{s^{\prime}}^{\eta, n}\right) \mid \mathcal{H}_{n T+j}\right] \mathrm{d} s^{\prime}, \int_{n T+(j-1)}^{n T+j}\left[H\left(\bar{\theta}_{s^{\prime}}^{\eta, n}, X_{n T+j}\right)-h\left(\bar{\theta}_{s^{\prime}}^{\eta, n}\right)\right] \mathrm{d} s^{\prime}\right\rangle\right] \\
& =0.
\end{aligned}
$$
Similarly, we have $\mathbb{E}\left[\left\langle I_k, R_K\right\rangle\right]=0$ for all $1 \leq k \leq K$. Then, 
\begin{equation}\label{thm4.1:5}
\begin{aligned}
&\int_{n T}^t \mathbb{E} {\left[\left|\int_{n T}^s\left[H\left(\bar{\theta}_{s^{\prime}}^{\eta, n}, X_{\left\lceil s^{\prime}\right\rceil}\right)-h\left(\bar{\theta}_{s^{\prime}}^{\eta, n}\right)\right] \mathrm{d} s^{\prime}\right|^2\right] \mathrm{d} s } \\
& =\int_{n T}^t\left[\sum_{k=1}^K \mathbb{E}\left[\left|\int_{n T+(k-1)}^{n T+k}\left[h\left(\bar{\theta}_{s^{\prime}}^{\eta, n}\right)-H\left(\bar{\theta}_{s^{\prime}}^{\eta, n}, X_{n T+k}\right)\right] \mathrm{d} s^{\prime}\right|^2\right]\right] \mathrm{d} s \\
& \quad +\int_{n T}^t \mathbb{E}\left[\left|\int_{n T+K}^s\left[h\left(\bar{\theta}_{s^{\prime}}^{\eta, n}\right)-H\left(\bar{\theta}_{s^{\prime}}^{\eta, n}, X_{n T+K+1}\right)\right] \mathrm{d} s^{\prime}\right|^2\right] \mathrm{d} s \\
& \leq \int_{n T}^t\left[\sum_{k=1}^K \int_{n T+(k-1)}^{n T+k} \mathbb{E}\left[\left|h\left(\bar{\theta}_{s^{\prime}}^{\eta, n}\right)-H\left(\bar{\theta}_{s^{\prime}}^{\eta, n}, X_{n T+k}\right)\right|^2\right] \mathrm{d} s^{\prime}\right] \mathrm{d} s \\
& \quad +\int_{n T}^t \int_{n T+K}^s \mathbb{E}\left[\left|h\left(\bar{\theta}_{s^{\prime}}^{\eta, n}\right)-H\left(\bar{\theta}_{s^{\prime}}^{\eta, n}, X_{n T+K+1}\right)\right|^2\right] \mathrm{d} s^{\prime} \mathrm{d} s \\
& \leq T^2 \sigma_H+T \sigma_H,
\end{aligned}
\end{equation}
where the last inequality holds due to Lemma~\ref{lemmac.1}. Substituting~\eqref{thm4.1:4} and~\eqref{thm4.1:5} into~\eqref{thm4.1:3} yields
\begin{equation}\label{first_term}
\begin{aligned}
\sup _{n T \leq u \leq t} \mathbb{E}\left[\left|\bar{\theta}_u^\eta-\bar{\zeta}_u^{\eta, n}\right|^2\right] & \leq c_1 \eta \sigma_V+  c_1 \eta L^2 \int_{n T}^t \sup _{n T \leq u \leq s^{\prime}} \mathbb{E}\left[\left|\bar{\theta}_u^\eta-\bar{\zeta}_u^{\eta, n}\right|^2\right] \mathrm{d} s^{\prime} \\
&\quad  +c_1 \eta^3\left(T^2 \sigma_H+T \sigma_H\right) \\
& \leq  c_1 \eta \sigma_V+ c_1 \eta L^2  \int_{n T}^t \sup _{n T \leq u \leq s^{\prime}} \mathbb{E}\left[\left|\bar{\theta}_u^\eta-\bar{\zeta}_u^{\eta, n}\right|^2\right] \mathrm{d} s^{\prime} \\
& \quad + c_1 \eta \sigma_H+  c_1 \eta^2 \sigma_H,
\end{aligned}
\end{equation}
where we used $\eta T \leq 1$ in the last inequality. 
Next, we consider the second term of~\eqref{thm4.1: 1}. To this end, we write
$$
\left|\bar{V}_t^\eta-\bar{Z}_t^{\eta, n}\right| \leq\left|\gamma \eta \int_{n T}^t\left[\bar{V}_{\lfloor s\rfloor}^\eta-\bar{Z}_s^{\eta, n}\right] \mathrm{d} s\right|+\eta\left|\int_{n T}^t\left[H\left(\bar{\theta}_{\lfloor s\rfloor}^\eta, X_{\lceil s\rceil}\right)-h\left(\bar{\zeta}_s^{\eta, n}\right)\right] \mathrm{d} s\right|.
$$
Then, by taking the expectation on both sides and by using $\eta T \leq 1$, we obtain

\begin{equation}\label{UB2}
\mathbb{E}\left[\left|\bar{V}_t^\eta-\bar{Z}_t^{\eta, n}\right|^2\right] \leq 2 \gamma^2 \eta \int_{n T}^t \mathbb{E}\left[\left|\bar{V}_{\lfloor s\rfloor}^\eta-\bar{Z}_s^{\eta, n}\right|^2\right]+2 \eta^2 \mathbb{E}\left[\left|\int_{n T}^t\left[H\left(\bar{\theta}_{\lfloor s\rfloor}^\eta, X_{\lceil s\rceil}\right)-h\left(\bar{\zeta}_s^{\eta, n}\right)\right] \mathrm{d} s\right|^2\right] .
\end{equation}
The second term  on the RHS of the above inequality can be bounded as follows:
$$
 \mathbb{E}\left[\left|\int_{n T}^l\left[H\left(\bar{\theta}_{\lfloor s\rfloor}^\eta, X_{\lceil s\rceil}\right)-h\left(\bar{\zeta}_s^{\eta, n}\right)\right] \mathrm{d} s\right|^2\right]\leq T \sigma_H + \sigma_H.
$$
Substituting the result into the second term of~\eqref{UB2} yields
$$
\mathbb{E}\left[\left|\bar{V}_t^\eta-\bar{Z}_t^{\eta, n}\right|^2\right] \leq 2 \gamma^2 \eta \int_{n T}^t \mathbb{E}\left[\left|\bar{V}_{\lfloor s\rfloor}^\eta-\bar{Z}_s^{\eta, n}\right|^2\right] + 2\sigma_H \eta + 2\sigma_H \eta^2 ,
$$
which implies
\begin{equation}\label{snd_term}
\sup _{n T \leq u \leq t} \mathbb{E}\left[\left|\bar{V}_u^\eta-\bar{Z}_u^{\eta, n}\right|^2\right] \leq 2 \gamma^2 \eta \int_{n T}^t \sup _{n T \leq u \leq s} \mathbb{E}\left[\left|\bar{V}_u^\eta-\bar{Z}_u^{\eta, n}\right|^2\right] \mathrm{d} s+2\sigma_H \eta+2\sigma_H \eta^2.
\end{equation}
Then by combining~\eqref{first_term} and~\eqref{snd_term}, we have
\begin{equation*}
\begin{aligned}
&\sup _{n T \leq u \leq t} \left\{\mathbb{E}\left[\left|\bar{\theta}_u^\eta-\bar{\zeta}_u^{\eta, n}\right|^2\right]\right\} +  \sup _{n T \leq u \leq t} \left\{\mathbb{E}\left[\left|\bar{V}_u^\eta-\bar{Z}_u^{\eta, n}\right|^2\right] \right\}
\\& \leq  c_1 \eta L^2  \int_{n T}^t \sup _{n T \leq u \leq s^{\prime}} \mathbb{E}\left[\left|\bar{\theta}_u^\eta-\bar{\zeta}_u^{\eta, n}\right|^2\right] \mathrm{d} s^{\prime}+ 2 \gamma^2 \eta \int_{n T}^t \sup _{n T \leq u \leq s} \mathbb{E}\left[\left|\bar{V}_u^\eta-\bar{Z}_u^{\eta, n}\right|^2\right] \mathrm{d} s \\
& \quad + c_1 \eta \sigma_V+ c_1 \eta \sigma_H+  c_1 \eta^2 \sigma_H + 2\sigma_H \eta+2\sigma_H \eta^2 \\
& \leq (c_1 L^2+  2 \gamma^2)\eta \int_{n T}^t \sup _{n T \leq u \leq s}\left\{ \mathbb{E}\left[\left|\bar{\theta}_u^\eta-\bar{\zeta}_u^{\eta, n}\right|^2\right]\right\} +  \sup _{n T \leq u \leq s}\left\{\mathbb{E}\left[\left|\bar{V}_u^\eta-\bar{Z}_u^{\eta, n}\right|^2\right] \right\}\mathrm{d} s \\
& \quad + \left(c_1 \sigma_V+ 2c_1 \sigma_H+ 4\sigma_H\right)\eta .
\end{aligned}
\end{equation*}
Finally, applying Gronwall's lemma yields
\begin{equation}\label{gron_sup}
\begin{aligned}
\sup _{n T \leq u \leq t} \left\{\mathbb{E}\left[\left|\bar{\theta}_u^\eta-\bar{\zeta}_u^{\eta, n}\right|^2\right]\right\} +  \sup _{n T \leq u \leq t} \left\{\mathbb{E}\left[\left|\bar{V}_u^\eta-\bar{Z}_u^{\eta, n}\right|^2\right] \right\}
\leq 
\left(c_1 \sigma_V+ 2c_1 \sigma_H+ 4\sigma_H\right)e^{c_1 L^2 + 2 \gamma^2}\eta  :={C_1^{\star}}^2 \eta,
\end{aligned}
\end{equation}
where $C_1^{\star} := \sqrt{\left(c_1 \sigma_V+ 2c_1 \sigma_H+ 4\sigma_H\right)}e^{c_1 L^2 / 2+ \gamma^2}.$
Therefore, combining~\eqref{thm4.1: 1} and~\eqref{gron_sup} yields the desired result, i.e.,
$$
W_2\left(\mathcal{L}\left(\bar{\theta}_t^\eta, \bar{V}_t^\eta\right), \mathcal{L}\left(\bar{\zeta}_t^{\eta, n}, \bar{Z}_t^{\eta, n}\right)\right) \leq C_1^{\star} \eta^{1 / 2}.
$$
\end{proof}

\begin{proof}[\textit{Proof of Proposition~\ref{lemma:4.2}}]\label{prf of lemma 4.2}
Let $p = \left\{1, 2\right\}$. By using the triangle inequality, Proposition~\ref{prop:3.6}, and~\cite[Lemma 5.4]{chau2022stochastic}, one obtains
\begin{equation}\label{w_p_app}
\begin{aligned}
W_p & \left(\mathcal{L}\left(\bar{\zeta}_t^{\eta, n}, \bar{Z}_t^{\eta, n}\right), \mathcal{L}\left(\zeta_t^\eta, Z_t^\eta\right)\right) \leq \sum_{k=1}^n W_p\left(\mathcal{L}\left(\bar{\zeta}_t^{\eta, k}, \bar{Z}_t^{\eta, k}\right), \mathcal{L}\left(\bar{\zeta}_t^{\eta, k-1}\right), \bar{Z}_t^{\eta, k-1}\right) \\
& \left.=\sum_{k=1}^n W_p\left(\mathcal{L}\left(\hat{\zeta}_t^{k T, \Bar{\theta}_{k T}^\eta, \Bar{V}_{k T}^\eta, \eta}, \hat{Z}_t^{k T, \Bar{\theta}_{k T}^\eta, \Bar{V}_{k T}^\eta, \eta}\right), \mathcal{L}\left(\hat{\zeta}_t^{(k-1) T, \Bar{\theta}_{(k-1) T}^\eta, \Bar{V}_{(k-1) T}^\eta, \eta}\right), \hat{Z}_t^{(k-1) T, \Bar{\theta}_{(k-1) T}^\eta, \Bar{V}_{(k-1) T}^\eta, \eta}\right)\right) \\
& =\sum_{k=1}^n W_p\left(\mathcal{L}\left(\hat{\zeta}_t^{k T, \Bar{\theta}_{k T}^\eta, \Bar{V}_{k T}^\eta, \eta}, \hat{Z}_t^{k T, \Bar{\theta}_{k T}^\eta, \Bar{V}_{k T}^\eta, \eta}\right), \mathcal{L}\left(\hat{\zeta}_t^{k T, \Bar{\zeta}_{k T}^{\eta, k-1}, \Bar{Z}_{k T}^{\eta, k-1}, \eta}, \hat{Z}_t^{k T, \Bar{\zeta}_{(k-1) T}^\eta, \Bar{Z}_{(k-1) T}^\eta, \eta}\right)\right) \\
& \leq \dot{C} \sum_{k=1}^n e^{-\dot{c} \eta (t-k T)} \mathcal{W}_\rho^{1 / p}\left(\mathcal{L}\left(\Bar{\theta}_{k T}^\eta, \Bar{V}_{k T}^\eta\right), \mathcal{L}\left(\Bar{\zeta}_{k T}^{\eta, k-1}, \Bar{Z}_{k T}^{\eta, k-1}\right)\right)\\
& \leq \left(3 \max \left\{1+\alpha_c, \gamma^{-1}\right\}\right)^{1 / p}\dot{C} \sum_{k=1}^n e^{-\dot{c} \eta (t-k T)}\left(1+\varepsilon_c\mathbb{E}^{1 / 2}\left[\mathcal{V}^2\left(\theta_{k T}^\eta, V_{k T}^\eta\right)\right]+ \varepsilon_c\mathbb{E}^{1 / 2}\left[\mathcal{V}^2\left(\bar{\zeta}_{k T}^\eta, \bar{Z}_{k T}^\eta\right)\right]\right)^{1 / p} \\ 
& \quad \times W_2^{1 / p}\left(\mathcal{L}\left(\theta_{k T}^\eta, V_{k T}^\eta\right), \mathcal{L}\left(\bar{\zeta}_{k T}^{\eta,(k-1) T}, \bar{Z}_{k T}^{\eta,(k-1) T}\right)\right)\\ 
& \leq \left(3 \max \left\{1+\alpha_c, \gamma^{-1}\right\}\right)^{1 / p} \dot{C}\left(1+\varepsilon_c C_{\mathcal{V}} + \varepsilon_c C_{\mathcal{V}}^{\prime}\right)^{1 / p} \left(C_1^{\star}\eta^{1 / 2}\right)^{1 / p}\sum_{k=1}^n e^{-\dot{c} \eta(t-k T)}.
\end{aligned}
\end{equation}
By using the facts that $ 1 - \eta \leq \eta T \leq 1$ and $\lfloor 1 / \eta\rfloor \geq 1$, one obtains
\begin{equation}\label{sum}
\begin{aligned}
\sum_{k=1}^n e^{- \dot{c} \eta (t-k T)} &\leq \sum_{k=0}^{n-1} e^{-\dot{c} \eta k T} \leq \frac{1}{1 - e^{ -\dot{c}\eta T}}  \leq
\left\{
\begin{array}{rcl}
\frac{1}{1 - e^{ -\dot{c}(1-\eta)}}            & {\eta \in (0, \eta_{\max}/ 2)}\\
\frac{1}{1 - e^{ -\dot{c}\eta T} }           & {\eta \in [ \eta_{\max} / 2, \eta_{\max}] }\\
\end{array} \right. \\
& \leq \max\{\frac{1}{1- e^{ -\dot{c}(1-\eta_{\max} / 2)}}, \frac{1}{1 - e^{ -\dot{c}\eta_{\max}\lfloor 1 / \eta\rfloor / 2} } \}\\
& \leq \max\{\frac{1}{1- e^{ -\dot{c} / 2}}, \frac{1}{1 - e^{ -\dot{c}\eta_{\max} / 2} } \}\\
& \leq \frac{1}{1 - e^{ -\dot{c}\eta_{\max} / 2} }.
\end{aligned}
\end{equation}
Substituting~\eqref{sum} into~\eqref{w_p_app} and letting $p = 1$ and $p = 2$ yields
\begin{equation*}
\begin{aligned}
W_1 \left(\mathcal{L}\left(\bar{\zeta}_t^{\eta, n}, \bar{Z}_t^{\eta, n}\right), \mathcal{L}\left(\zeta_t^\eta, Z_t^\eta\right)\right) \leq  C^{\ast} \eta^{1 / 2}, \quad 
W_2 \left(\mathcal{L}\left(\bar{\zeta}_t^{\eta, n}, \bar{Z}_t^{\eta, n}\right), \mathcal{L}\left(\zeta_t^\eta, Z_t^\eta\right)\right) \leq  C_2^{\star} \eta^{1 / 4},
\end{aligned}
\end{equation*}
where 
\begin{equation}\label{C ast}
C^{\ast} := \frac{3 C_1 \max \left\{1+\alpha_c, \gamma^{-1}\right\}}{1 - e^{ -c_1\eta_{\max} / 2 } }C_1^{\star}\left(1+\varepsilon_c \left(C_{\mathcal{V}} + C_{\mathcal{V}}^{\prime}\right)\right)
\end{equation}
and 
$$C_2^{\star} := \frac{ C_2\sqrt{3\max \left\{1+\alpha_c, \gamma^{-1}\right\}}}{1 - e^{ -c_2\eta_{\max} / 2} }\sqrt{C_1^{\star}\left(1+\varepsilon_c \left(C_{\mathcal{V}} + C_{\mathcal{V}}^{\prime}\right)\right)}.$$
\end{proof}

\begin{proof}[\textit{Proof of Proposition~\ref{lemma:4.3}}]\label{prf of lemma 4.3}
We note, by~\cite{eberle2019couplings}, that
$ h(r) \leq \min \left\{r, h\left(R_1\right)\right\} \leq \min \left\{r, R_1\right\}$, for any $r \geq 0$, which implies that
$$
\begin{aligned}
\rho\left((x, v),\left(x^{\prime}, v^{\prime}\right)\right) & \leq \min \left\{r\left((x, v),\left(x^{\prime}, v^{\prime}\right)\right), R_1\right\}\left(1+\varepsilon_c \mathcal{V}(x, v)+\varepsilon_c \mathcal{V}\left(x^{\prime}, v^{\prime}\right)\right) \\
& \leq R_1\left(1+\varepsilon_c \mathcal{V}(x, v)+\varepsilon_c \mathcal{V}\left(x^{\prime}, v^{\prime}\right)\right).
\end{aligned}
$$
Hence, by using~\eqref{func:V2}, we obtain
\begin{equation}\label{rho_bnd}
\begin{aligned}
& \mathcal{W}_\rho\left(\mu_0, \bar{\pi}_\beta\right) \\
& \leq R_1 + R_1 \varepsilon_c\left(\left(\beta L+\frac{1}{2} \beta \gamma^2\right) \int_{\mathbb{R}^{2 d}}\lvert\theta\rvert^2 \mu_0(\mathrm{d} \theta, \mathrm{d} v)+\frac{3}{4} \beta \int_{\mathbb{R}^{2 d}}\lvert v\rvert^2 \mu_0(\mathrm{d} \theta, \mathbb{d} v)+\beta u_0+\frac{\beta \lvert h(0)\rvert^2}{2 L}\right) \\
& \quad+R_1 \varepsilon_c\left(\left(\beta L+\frac{1}{2} \beta \gamma^2\right) \int_{\mathbb{R}^{2 d}}\lvert\theta\rvert^2 \bar{\pi}_\beta(\mathrm{d} \theta, \mathrm{d} v)+\frac{3}{4} \beta \int_{\mathbb{R}^{2 d}}\lvert v\rvert^2 \bar{\pi}_\beta(\mathrm{d} \theta, \mathrm{d} v)+\beta u_0+\frac{\beta \lvert h(0)\rvert^2}{2 L}\right).
\end{aligned}
\end{equation}
Moreover, by using~\cite{raginsky2017non} and the fact that
\begin{equation}\label{pai_ex}
\begin{aligned}
\bar{\pi}_\beta(\mathrm{d} \theta, \mathrm{d} v) =  \frac{\exp \left(-\beta\left(\frac{1}{2}|v|^2+u(\theta)\right)\right) \mathrm{d} \theta \mathrm{d} v}{\int_{\mathbb{R}^{2d}} \exp \left(-\beta\left(\frac{1}{2}|v|^2+u(\theta)\right)\right) \mathrm{d} \theta \mathrm{d} v} = \left(2\pi \beta^{-1}\right)^{-d / 2}\frac{\exp \left(-\beta\left(\frac{1}{2}|v|^2+u(\theta)\right)\right) \mathrm{d} \theta \mathrm{d} v}{\int_{\mathbb{R}^{d}} \exp \left(-\beta u(\theta)\right) \mathrm{d} \theta},
\end{aligned}
\end{equation}
we obtain
$$
\int_{\mathbb{R}^{2 d}}\lvert \theta\rvert^2 \bar{\pi}_{\beta}(\mathbb{d} \theta, \mathbb{d} v) \leq \frac{b^{\prime} + d / \beta}{a^{\prime}}
$$
and
\begin{equation}\label{pi_beta_v_2}
\int_{\mathbb{R}^{2 d}}\lvert v\rvert^2 \bar{\pi}_\beta(\mathrm{d} \theta, \mathrm{d} v)= \left(2 \pi \beta^{-1}\right)^{- d / 2}\int_{\mathbb{R}^d}\lvert v\rvert^2 e^{-\beta \lvert v\rvert^2  /2} \mathrm{d} v=d / \beta.
\end{equation}
By using~\eqref{pai_ex},~\eqref{pi_beta_v_2} and Assumption~\ref{asm:A2}, the RHS of~\eqref{rho_bnd} is bounded. By using Proposition~\ref{prop:3.6} and letting $p=1$, we obtain
$$
\begin{aligned}
W_1\left(\mathcal{L}\left(\zeta_t^\eta, Z_t^\eta\right), \bar{\pi}_\beta\right) \leq& C_1 e^{- c_1 \eta t } \mathcal{W}_\rho(\mu_0, \bar{\pi}_\beta),
\\ \leq& C_2^{\ast} e^{-C_3^{\ast}\eta t},
\end{aligned}
$$
where $C_2^{\ast}:=C_1 \mathcal{W}_\rho\left(\mu_0, \bar{\pi}_\beta\right)$ and $C_3^{\ast}:=c_1$. Similarly, by letting $p=2$, we have 
$$
\begin{aligned}
W_2\left(\mathcal{L}\left(\zeta_t^\eta, Z_t^\eta\right), \bar{\pi}_\beta\right) \leq& C_2 e^{-c_2 \eta t } \sqrt{\mathcal{W}_\rho(\mu_0, \bar{\pi}_\beta)},
\\ \leq& C_3^{\star} e^{-C_4^{\star}\eta t},
\end{aligned}
$$
where $C_3^{\star}:=C_2\sqrt{ \mathcal{W}_\rho\left(\mu_0, \bar{\pi}_\beta\right)}$ and $C_4^{\star}:=c_2$.

\end{proof}

\begin{proof}[\textit{Proof of Proposition~\ref{lemma:4.4}}]\label{prf of lemma 4.4}
We denote by $\pi_{n, \beta}^\eta:=\mathcal{L}\left(\theta_n^\eta, V_n^\eta\right)$ and write
$$
\mathbb{E}\left[u\left(\theta_n^\eta\right)\right]-\mathbb{E}\left[u\left(\theta_{\infty}\right)\right]=\int_{\mathbb{R}^{2 d}} u(\theta) \pi_{n, \beta}^\eta(\mathrm{d} \theta, \mathrm{d} v)-\int_{\mathbb{R}^{2 d}} u(\theta) \pi_\beta(\mathrm{d} \theta, \mathrm{d} v) .
$$
Note that we have
$$
|h(\theta)| \leq L\lvert\theta\rvert+\lvert h(0)\rvert,
$$
hence by using similar arguments to that in~\cite[Proposition 1]{polyanskiy2016wasserstein}, we obtain that
$$
\begin{aligned}
\left\lvert u(\theta_n^{\eta}) - u(\theta_{\infty})\right\rvert  &= \lvert\int_0^1 \left\langle h(\theta_{\infty} + t(\theta_n^{\eta} - \theta_{\infty})), \theta_n^{\eta}-\theta_{\infty}\right\rangle\rvert \mathrm{d} t 
\\ &\leq \int_0^1\lvert h(\theta_{\infty} + t(\theta_n^{\eta} - \theta_{\infty}))\rvert\lvert \theta_n^{\eta}-\theta_{\infty}\rvert \mathrm{d} t 
\\ & \leq L\lvert \theta_n^{\eta}-\theta_{\infty}\rvert\int_0^1 \left( t\lvert\theta_n^{\eta}\rvert + (1 - t)\lvert\theta_{\infty}\rvert\right)\mathrm{d} t + |h(0)| \lvert \theta_n^{\eta}-\theta_{\infty}\rvert
\\ & = L\left(\frac{1}{2}\lvert\theta_n^{\eta}\rvert + \frac{1}{2}\lvert\theta_{\infty}\rvert\right)\lvert \theta_n^{\eta}-\theta_{\infty}\rvert+|h(0)|\lvert \theta_n^{\eta}-\theta_{\infty}\rvert .
\end{aligned}
$$
Then taking expectation and using Hölder inequality yields
$$
\lvert\int_{\mathbb{R}^{2 d}} u(\theta) \pi_{n, \beta}^\eta(\mathrm{d} \theta, \mathrm{d} v)-\int_{\mathbb{R}^{2 d}} u(\theta) \pi_\beta(\mathrm{d} \theta, \mathrm{d} v)\rvert \leq\left(L C_m + \lvert h(0)\rvert\right) W_2\left(\pi_{n, \beta}^\eta, \pi_\beta\right),
$$
where $C_m:=\max \left(\int_{\mathbb{R}^{2 d}}\lvert\theta\rvert^2 \pi_{n, \beta}^\eta(\mathrm{d} \theta, \mathrm{d} v), \int_{\mathbb{R}^{2 d}}\lvert\theta\rvert^2 \pi_\beta(\mathrm{d} \theta, \mathrm{d} v)\right)$.
By using Theorem~\ref{thm:2.1}, we obtain
$$
\mathbb{E}\left[u\left(\theta_n^\eta\right)\right]-\mathbb{E}\left[u\left(\theta_{\infty}\right)\right] \leq\left(L C_m+\lvert h(0)\rvert\right)\left(C_1^{\star} \eta^{1 / 2}+C_2^{\star} \eta^{1 / 4}  + C_3^{\star} e^{-C_4^{\star} n}\right) .
$$
\end{proof}

\begin{proof}[\textit{Proof of Proposition~\ref{lemma:4.5}}]\label{prf of lemma 4.5}
Denote by $p(\theta) := e^{ - \beta u(\theta)} / \Lambda$ for $\theta \in \mathbb{R}^d$ the marginal density of $\bar{\pi}_{\beta}$, where $\Lambda:=\int_{\mathbb{R}^d} e^{-\beta u(\theta)} \mathrm{d} \theta$ is the normalization constant. Let $\theta^*$ be a point that minimizes $u(\theta)$, i.e., $u^*:=\min _{\theta \in \mathbb{R}^d} u(\theta)=u\left(\theta^*\right)$, which exists by Assumption~\ref{asm:A5}, see, e.g.,~\cite[Theorem 2.32]{beck2014introduction}. We note that
\begin{equation}\label{thm4.5:1}
\begin{aligned}
\mathbb{E}\left[u\left(\theta_{\infty}\right)\right]-u^* & =\frac{1}{\beta}\left(-\int_{\mathbb{R}^d} \frac{e^{-\beta u(\theta)}}{\Lambda} \log \frac{e^{-\beta u(\theta)}}{\Lambda} \mathrm{d} \theta-\log \Lambda\right)-u^*,
\end{aligned}
\end{equation}
where the first term on the RHS is the differential entropy of $p(\theta)$. To upper-bound the first term on the RHS of~\eqref{thm4.5:1}, we estimate the second moment of $\theta_{\infty}$ first. 
By using Proposition~\ref{prop:3.6} and letting $p=2$, one obtains
$$
\begin{aligned}
W_2\left(\mathcal{L}\left(\theta_t, V_t\right), \bar{\pi}_\beta\right) \leq   C_2 e^{-c_2 t } \sqrt{\mathcal{W}_\rho(\mathcal{L}\left(\theta_0, V_0\right), \bar{\pi}_\beta)},
\end{aligned}
$$
which implies $W_2\left(\mathcal{L}\left(\theta_t, V_t\right), \bar{\pi}_\beta\right) \stackrel{t \rightarrow \infty}{\longrightarrow} 0$. Note that convergence of probability measures in Wasserstein-2 distance is equivalent to weak convergence plus convergence of second moments (see, e.g.,~\cite[Theorem 7.12]{villani2003topics}). Thus, by using~\eqref{Ineq:2},~\eqref{func:V}, and Fatou's lemma, we obtain
$$
\mathbb{E}\left[\lvert \theta_{\infty}\rvert^2\right]= \lim _{t \rightarrow \infty} \mathbb{E}\left[\lvert \theta_t\rvert^2\right] \leq \frac{8(d / \beta + A_c / \beta)}{ \gamma^2 \lambda(1-2 \lambda) }.
$$
By using the fact that Gaussian distributions maximise the differential entropy over all distributions with the same finite second moment~\cite[Theorem 9.6.5]{cover1999elements}, we obtain
\begin{equation}\label{thm4.5:2}
-\int_{\mathbb{R}^d} \frac{e^{-\beta u(\theta)}}{\Lambda} \log \frac{e^{-\beta u(\theta)}}{\Lambda} \mathrm{d} \theta \leq \frac{d}{2} \log \left(\frac{16 \pi e(d / \beta + A_c/ \beta )}{d \gamma^2 \lambda(1 - 2\lambda)}\right).
\end{equation}
Moreover, since $\theta^{\ast}$ is a minimizer of $u$, i.e., $u\left(\theta^*\right)=\min _{\theta \in \mathbb{R}^d} u(\theta)$, this implies $\nabla u\left(\theta^*\right)=0$. By using Remark~\ref{rmk:2}, we have 
\begin{equation}\label{3.21}
\begin{aligned}
-\beta\left(u\left(\theta^*\right)-u(\theta)\right) \leq & \beta\left|\int_0^1\left\langle h\left(t \theta^*+(1-t) \theta\right)-h\left(\theta^*\right), \theta^*-\theta\right\rangle \mathrm{d} t\right|,\\
\leq & \beta\left(\int_0^1 L(1-t)\lvert \theta - \theta^*\rvert^2 \mathrm{d} t\right), \\
\leq &  \beta \frac{L\lvert\theta - \theta^*\rvert^2}{2}.
\end{aligned}
\end{equation}
Then, we can lower-bound $\log \Lambda $ by using~\eqref{3.21}, one writes
\begin{equation}\label{thm4.5:3}
\begin{aligned}
\log \Lambda & =\log \int_{\mathbb{R}^d} e^{-\beta u(\theta)} \mathrm{d} \theta, \\
& =-\beta u^*+\log \int_{\mathbb{R}^d} e^{\beta\left(u^* - u(\theta)\right)} \mathrm{d} \theta, \\
& \geq-\beta u^*+\log \int_{\mathbb{R}^d} e^{-\frac{\beta L\left\lvert\theta-\theta^*\right\rvert^2}{2}} \mathrm{d} \theta, \\
& =-\beta u^*+\frac{d}{2} \log \left(\frac{2 \pi}{L \beta}\right),
\end{aligned}
\end{equation}
where the last inequality holds due to the fact that $\int_{-\infty}^{+\infty} e^{-a x^2} \mathrm{d} x =  \sqrt{\frac{\pi}{a}}$ for any $a > 0$. Substituting~\eqref{thm4.5:2} and~\eqref{thm4.5:3} into~\eqref{thm4.5:1} yields
$$
\int_{\mathbb{R}^d} u(\theta) \pi(\mathrm{d} \theta)-\min _{\theta \in \mathbb{R}^d} u(\theta) \leq \frac{d}{2 \beta} \log \left(\frac{8 e L}{\gamma^2 \lambda(1 - 2\lambda)}\left(\frac{A_c}{d}+1\right)\right) .
$$
\end{proof}




\section{Proofs of Section~\ref{sec:6}}\label{Appendix: E}
\begin{proof}[\textit{Proof of Proposition~\ref{prop:6.1}}]\label{pf_prop:7.1}
Assumption~\ref{asm:A3} holds with $\rho=0, L_1=2 \lambda_r, L_2=0$ and $\bar{K}_1(x)=2$. Assumption~\ref{asm:A2} is satisfied by construction. Denote by $f_{X}$ the density of $X$ and denote by $\bar{c}_d$ the upper bound of $f_{X}$. Then, Assumption~\ref{asm:A4} holds with $L= 2 \left(\lambda_r + \bar{c}_d \right)$. Indeed, we have, for any $\theta, \theta^{\prime} \in \mathbb{R}^d$, that
$$
\begin{aligned}
\mathbb{E}\left[\left|H\left(\theta, X_0\right)-H\left(\theta^{\prime}, X_0\right)\right|\right] \leq &  2\lambda_r \lvert\theta - \theta^{\prime} \rvert + \mathbb{E}\left[\left\lvert \mathbbm{1}_{\{X_0 < \theta\}} - \mathbbm{1}_{\{X_0 < \theta^{\prime}\}}\right\rvert\right]
\\ \leq & 2\lambda_r \lvert\theta - \theta^{\prime} \rvert + \mathbb{E}\left[\mathbbm{1}_{\{\theta^{\prime} \leq X_0 < \theta\}}\right] + \mathbb{E}\left[\mathbbm{1}_{\{\theta \leq X_0 < \theta^{\prime}\}}\right]
\\ \leq & 2\lambda_r \lvert\theta - \theta^{\prime} \rvert + \left\lvert \int_{\theta^{\prime}}^{\theta} f_{X_0}\left(x\right) d x \right\rvert + \left\lvert \int_{\theta}^{\theta^{\prime}} f_{X_0}\left(x\right) d x \right\rvert
\\ \leq & 2\left(\lambda_r + \bar{c}_d\right) \lvert\theta - \theta^{\prime} \rvert.
\end{aligned}
$$
Furthermore, Assumption~\ref{asm:A5} holds with $A(x)=2\lambda_r \mathbf{I}_d$ and $b(x)=0$, which implies $a=2\lambda_r$ and $b=0$.
\end{proof}

\begin{proof}[\textit{Proof of Proposition~\ref{prop:6.2}}]\label{pf_prop:7.2}
For $N=1, \cdots, d_2$, we denote by $$\sigma_2^{N}:=\sigma_2\left(\sum_{k=1}^{d_1}\left[f(W_1^{N k}) \sigma_1\left(\left\langle W_0^{k \cdot}, Z\right\rangle+f(b_0^k)\right)\right]+b_1^N\right).$$ We claim that $h(\theta):=\nabla u(\theta)=\mathbb{E}\left[H\left(\theta, X_0\right)\right]$. Indeed, for $K=1, \cdots, d_1$ and $ N = 1, \cdots, d_2$, we have
\begin{equation*}
\begin{aligned}
\partial_{W_1^{N K}} u(\theta)&=-2 \mathbb{E}\bigg[\sum_{j=1}^{m_2}\left(Y^j-\mathfrak{N}^j(\theta, Z)\right) W_2^{j N} \sigma_2^N \left(1 - \sigma_2^N\right)f^{\prime}(W_1^{N K}) \sigma_1\left(\left\langle W_0^{K\cdot}, Z\right\rangle+f(b_0^K)\right)\bigg] + 2\lambda_r W_1^{N K},
\end{aligned}
\end{equation*}
\begin{equation*}
\begin{aligned}
\partial_{b_0^K} u(\theta)&= -2 \mathbb{E}\bigg[\sum_{j=1}^{m_2}\left(Y^j-\mathfrak{N}^j(\theta, Z)\right)\sum_{n=1}^{d_2} W_2^{j n} \sigma_2^n \left(1 - \sigma_2^n\right) f(W_1^{n K}) f^{\prime}(b_0^K)\mathbbm{1}_{\{\langle W_0^{K \cdot }, Z\rangle + f(b_0^{K}) > 0\}}\bigg] + 2\lambda_r b_0^K,
\end{aligned}
\end{equation*}
\begin{equation*}
\begin{aligned}
\partial_{b_1^N} u(\theta)= -2 \mathbb{E}\bigg[\sum_{j=1}^{m_2}\left(Y^j-\mathfrak{N}^j(\theta, Z)\right) W_2^{j N} \sigma_2^N \left(1 - \sigma_2^N\right)\bigg] + 2\lambda_r b_1^N.
\end{aligned}
\end{equation*}
We note that the partial derivative of $u$ with respect to $W_1$ and $b_1$ are obtained by using the chain rule. Next, we provide a proof for $\partial_{b_0} u(\theta)$ under the case $m_1=m_2=d_1=d_2 = 1$ for the ease of notation, which could be naturally extended to $m_1, m_2, d_1, d_2 \in \mathbb{R}$. In this case,
\begin{equation*}
\begin{aligned}
u(\theta)& :=\mathbb{E}\left[|Y-\mathfrak{N}(\theta, Z)|^2\right]+\lambda_r|\theta|^2 \\ & =
\mathbb{E}\left[|Y-W_2 \sigma_2 \left(f(W_1)\sigma_1(W_0 Z + f(b_0))+b_1\right)|^2\right]+\lambda_r|\theta|^2 \\ & =
\mathbb{E}\left[Y^2 - 2 Y W_2 \sigma_2 \left(f(W_1)\sigma_1(W_0 Z + f(b_0))+b_1\right) + W_2^2 \sigma_2^2 \left(f(W_1)\sigma_1(W_0 Z + f(b_0))+b_1\right)\right]+\lambda_r|\theta|^2.
\end{aligned}
\end{equation*}
Then, one obtains
\begin{equation}\label{b_0}
\partial_{b_0} u(\theta)=T_{b_0,1}(\theta)+T_{b_0,2}(\theta) + 2\lambda_r b_0,
\end{equation}
where 
\begin{equation*}
\begin{aligned}
T_{b_0,1}(\theta) :=\partial_{b_0}\left(-2 W_2 \int_{-\infty}^{\infty} \int_{-\infty}^{\infty} y\sigma_2\left(f\left(W_1\right) \sigma_1\left(W_0 z+f\left(b_0\right)\right)+b_1\right) f_{Y, Z}(y, z) \mathrm{d} z \mathrm{~d} y\right)
\end{aligned}
\end{equation*}
with $f_{Y, Z}$ denoting the joint density of  $Y, Z$, where $f_Z$ denoting the density function of $Z$, $f_{Z \mid Y}$ denotes the conditional density of $Z$ given $Y$, and where
\begin{equation*}
\begin{aligned}
T_{b_0,2}(\theta):=W_2^2  \partial_{b_0}\left(\int_{-\frac{f(b_0)}{W_0}}^{\infty}\left( \sigma_2^2\left(f\left(W_1\right)\left(W_0 z+f\left(b_0\right)\right)+b_1\right) \right) f_Z(z) \mathrm{d} z + \int_{-\infty}^{-\frac{f(b_0)}{W_0}}\left( \sigma_2^2\left(b_1\right) \right) f_Z(z) \mathrm{d} z \right).
\end{aligned}
\end{equation*}
For any $\theta \in \mathbb{R}^d$, we have
\begin{equation}\label{1b_0, 1}
\begin{aligned}
&T_{b_0,1}(\theta)\\ & =
-2 W_2 \partial_{b_0}\left( \int_{-\infty}^{\infty}y \left(\int_{-\infty}^{\infty} \sigma_2\left(f\left(W_1\right) \sigma_1\left(W_0 z+f\left(b_0\right)\right)+b_1\right) f_{Z|Y}(z) \mathrm{d} z\right) f_{Y}(y) \mathrm{~d} y\right)\\ & =
-2 W_2 \partial_{b_0}\bigg( \int_{-\infty}^{\infty} \!y \bigg(\int_{-\frac{f(b_0)}{W_0}}^{\infty} \!\sigma_2\left(f\left(W_1\right) \left(W_0 z+f\left(b_0\right)\right)+b_1\right) f_{Z|Y}(z) \mathrm{d} z + \int_{-\infty}^{-\frac{f(b_0)}{W_0}} \!\sigma_2\left( b_1\right) f_{Z|Y}(z) \mathrm{d} z \!\bigg) f_{Y}(y) \mathrm{~d} y\!\bigg).
\end{aligned}
\end{equation}
Denote by $\sigma_2:=\sigma_2\left(f\left(W_1\right)\left(W_0 z+f\left(b_0\right)\right)+b_1\right)$. By using $\left\lvert \sigma_2\left(1 - \sigma_2\right)\right\rvert \leq 1$, $\left\lvert f(v)\right\rvert\leq c$, and $\left\lvert f^{\prime}(v)\right\rvert\leq 1$, we obtain, for any $b_0 \in \mathbb{R}$, that
\begin{equation*}
\begin{aligned}
\inf _{\delta \in(0, \infty)} \int_{-\infty}^{\infty} \sup _{\zeta \in[-\delta, \delta]}\left|\int_{-\frac{f(b_0+\zeta)}{W_0}}^{\infty} y \sigma_2\left(1 - \sigma_2\right)f(W_1)f^{\prime}(b_0 + \zeta) f_{Z \mid Y}(z) \mathrm{d} z\right| f_Y(y) \mathrm{d} y \leq c \int_{-\infty}^{\infty}|y| f_Y(y) \mathrm{d} y<\infty.
\end{aligned}
\end{equation*}
Thus, by using~\cite[Theorem A.5.3]{durrett2019probability}, we have
\begin{equation}\label{b_0, 1}
\begin{aligned}
T_{b_0,1}(\theta) & =
-2 W_2 \int_{-\infty}^{\infty} y \bigg(\int_{-\frac{f(b_0)}{W_0}}^{\infty} \sigma_2\left(1 - \sigma_2\right)f(W_1)f^{\prime}(b_0) f_{Z \mid Y}(z) \mathrm{d} z \bigg) f_{Y}(y) \mathrm{~d} y \\ & =
-2 \mathbb{E}\left[Y W_2 \sigma_2\left(1-\sigma_2\right) f\left(W_1\right) f^{\prime}\left(b_0\right) \mathbbm{1}_{\left\{ W_0 Z + f\left(b_0\right)>0\right\}}\right].
\end{aligned}
\end{equation}
Similarly, one obtains
\begin{equation}\label{b_0, 2}
\begin{aligned}
T_{b_0,2}(\theta) & = 2 W_2^2   \int_{-\frac{f(b_0)}{W_0}}^{\infty} \sigma_2^2\left(1 - \sigma_2\right)f(W_1)f^{\prime}(b_0) f_{Z \mid Y}(z) \mathrm{d} z =
2 \mathbb{E}\left[\mathfrak{N}(\theta, Z) W_2 \sigma_2 \left(1 - \sigma_2\right) f(W_1)f^{\prime}(b_0)\right].
\end{aligned}
\end{equation}
Substituting~\eqref{b_0, 1} and~\eqref{b_0, 2} into~\eqref{b_0} yields
$$
\partial_{b_0} u(\theta) = -2 \mathbb{E}\left[\left(Y-\mathfrak{N}(\theta, Z)\right) W_2 \sigma_2\left(1-\sigma_2\right) f\left(W_1\right) f^{\prime}\left(b_0\right) \mathbbm{1}_{\left\{W_0 Z + f\left(b_0\right)>0\right\}}\right] + 2 \lambda_r b_0.
$$
In addition, since $\left(X_n\right)_{n \in \mathbb{N}_0}$ is a sequence of i.i.d.\ random variables with probability law $\mathcal{L}(X)$, by the definitions of $F, G$ given in~\eqref{h=f+g} and as $H:=F+G$, we observe that $h(\theta):=\nabla u(\theta)=\mathbb{E}\left[H\left(\theta, X_0\right)\right]$, for all $\theta \in \mathbb{R}$. Thus, Assumption~\ref{asm:A2} holds. Assumption~\ref{asm:A3} holds with $\rho = 0$, $L_1 = 2\lambda_r$, $L_2 = 0$. Indeed, we have, for any $\theta, \theta^{\prime} \in \mathbb{R}^d$, $x, x^{\prime} \in \mathbb{R}^m$, that
$$
\left\lvert F(\theta, x) - F(\theta^{\prime}, x^{\prime})\right\rvert \leq 2 \lambda_r \left\lvert \theta - \theta^{\prime} \right\rvert.
$$
We then proceed to show that Assumption~\ref{asm:A4} holds. To this end, we note that, for any $j=1, \cdots, m_2$ and $(\theta, z) \in \mathbb{R}^d \times \mathbb{R}^{m_1}$, that
\begin{equation}\label{nn_upper}
\lvert \mathfrak{N}^j(\theta, z)\rvert \leq  \sum_{n=1}^{d_2} \left\lvert W_2^{j n} \right\rvert \left\lvert\sigma_2\left(\sum_{k=1}^{d_1}\left[f(W_1^{n k})\sigma_1\left(\left\langle W_0^{k \cdot}, z\right\rangle+f(b_0^k)\right)\right]+b_1^n\right)\right\rvert
\leq d_2 c_{W_{2}},
\end{equation}
where we recall that $c_{W_2}:=\max\limits_{i, j}\left\{W_2^{i j}\right\}$. By using~\eqref{nn_upper} together with the fact that ${\left\lvert \sigma_2(v)(1 - \sigma_2(v))\right\rvert \leq 1}$, we obtain, for any $(\theta, x)\in \mathbb{R}^d \times \mathbb{R}^m$, $K=1, \cdots, d_1$ and $N=1, \cdots, d_2$,
\begin{equation*}
\begin{aligned}
&\left\lvert G_{b_0^{K}}(\theta, x)\right\rvert \\ & \leq  2 \sum_{j=1}^{m_2}\sum_{n=1}^{d_2}\lvert W_2^{j n}\rvert\left\lvert f(W_1^{n K})\right\rvert \left(\lvert y^j\rvert  \lvert\sigma_2^{n}\left(1 - \sigma_2^{n}\right)\rvert+ \lvert \mathfrak{N}^j(\theta, z)\rvert \lvert \sigma_2^{n}\left(1 - \sigma_2^{n}\right)\rvert \right)\left\lvert f^{\prime}(b_0)\right\rvert \mathbbm{1}_{\left\{\left\langle W_0^{K \cdot}, z\right\rangle+f(b_0^K) > 0\right\}}\\
&\leq  2\sum_{j=1}^{m_2}\sum_{n=1}^{d_2}c \left( c_{W_{2}} \lvert y^K\rvert +  c_{W_{2}} \lvert \mathfrak{N}^j(\theta, z)\rvert \right) \mathbbm{1}_{\left\{\left\langle W_0^{K \cdot}, z\right\rangle+b_0^K > 0\right\}}
\\
& \leq  2 \sum_{j=1}^{m_2}\sum_{n=1}^{d_2}\left(c c_{W_{2}} \lvert y^j\rvert +  d_2 c c_{W_{2}}^2\right) \leq 
C_{G_{b_0}}\left(1 + \left\lvert x \right\rvert\right),
\end{aligned}
\end{equation*}
where $C_{G_{b_0}}:= 2 m_2 d_2 c c_{W_2} \left( 1 + d_2 c_{W_2}\right)$. Similarly, we have that
\begin{equation*}
\begin{aligned}
\left\lvert G_{b_1^N}(\theta, x)\right\rvert &\leq 2 \sum_{j=1}^{m_2}\left(\left\lvert y^j\right\rvert \left\lvert W_2^{j, N} \right\rvert\left\lvert\sigma_2^{N}\left(1 - \sigma_2^{N} \right)\right\rvert + \left\lvert\mathfrak{N}^j(\theta, z)\right\rvert \left\lvert W_2^{j, N} \right\rvert \left\lvert\sigma_2^{N}\left(1 - \sigma_2^{N} \right)\right\rvert\right)\\
& \leq 2 \sum_{j=1}^{m_2}\left( c_{W_{2}} \left\lvert y^j\right\rvert + d_2 c_{W_{2}}^2\right) \leq 
C_{G_{b_1}}\left(1 + \left\lvert x \right\rvert\right),
\end{aligned}
\end{equation*}
where $C_{G_{b_1}}:= 2m_2 c_{W_2} \left(1 + d_2 c_{W_2}\right)$. By using~\eqref{nn_upper} together with $\left\lvert W_0^{K, i}\right\rvert \leq c_{W_0}$, ${\left\lvert W_2^{j, N}\right\rvert\leq c_{W_2}}$ for $K=1, \cdots, d_1$, $i=1,\cdots, m_1$, $N=1, \cdots, d_2$, and $j=1,\cdots, m_2$, and that $\left\lvert f^{\prime}(v)\right\rvert \leq 1$ for any $v \in \mathbb{R}$, we obtain, for any $\left(\theta, x\right) \in \mathbb{R}^d \times \mathbb{R}^m$, that
\begin{equation*}
\begin{aligned}
&\left\lvert G_{W_1^{N K}}(\theta, x)\right\rvert \\ &\leq  2 \sum_{j=1}^{m_2}\left(\left\lvert y^j\right\rvert \left\lvert W_2^{j, N} \right\rvert\left\lvert\sigma_2^{N}\left(1 - \sigma_2^{N} \right)\right\rvert + \left\lvert\mathfrak{N}^j(\theta, z)\right\rvert \left\lvert W_2^{j, N} \right\rvert \left\lvert\sigma_2^{N}\left(1 - \sigma_2^{N} \right)\right\rvert\right) \left\lvert \sigma_1\left(\left\langle W_0^{K\cdot}, z\right\rangle+f(b_0^K)\right) \right\rvert\left\lvert f^{\prime}(W_1^{N K})\right\rvert
\\ &\leq 2\sum_{j=1}^{m_2} \left( c_{W_{2}} \left\lvert y^j\right\rvert + d_2 c_{W_{2}}^2\right) \left\lvert \left(\left\langle W_0^{K\cdot}, z\right\rangle+f(b_0^K)\right) \mathbbm{1}_{\left\{\left\langle W_0^{K \cdot}, z\right\rangle+f(b_0^K) > 0\right\}} \right\rvert\\
& \leq 2 m_2 c_{W_2} \left(1 + d_2 c_{W_2}\right) \left(1 + \left\lvert y\right\rvert\right)\left(m_1 c_{W_0} + c\right)\left(1 + \left\lvert z\right\rvert\right)\leq C_{G_{W_1}}\left(1 + \left\lvert x \right\rvert\right)^2,
\end{aligned}
\end{equation*}
where $C_{G_{W_1}}:= 2 m_2 c_{W_2} \left(1 + d_2 c_{W_2}\right)\left(m_1 c_{W_0} + c\right)$. Thus, for any $\left(\theta, x\right) \in \mathbb{R}^d \times \mathbb{R}^m$, we have that $\left\lvert G(\theta, x) \right\rvert\leq d_1 d_2 C_{G_{W_1}}\left(1 + \left\lvert x \right\rvert\right)^2 + d_1 C_{G_{b_0}}\left(1 + \left\lvert x \right\rvert\right) + d_2 C_{G_{b_1}}\left(1 + \left\lvert x \right\rvert\right)=: \bar{K}_1(x)$, thus Assumption~\ref{asm:A3} is satisfied. Next we show Assumption~\ref{asm:A4} holds. To this end, note that for any $\theta, \bar{\theta}:=\left(\left[\bar{W}_1\right], \bar{b}_0, \bar{b}_1\right) \in \mathbb{R}^d$, we have
\begin{equation}\label{H=F+G}
\begin{aligned}
&\mathbb{E}\left[\lvert H(\theta, X) - H(\bar{\theta}, X)\rvert\right] \\& \leq 2 \lambda_r \lvert \theta - \bar{\theta}\rvert + \mathbb{E}\left[\lvert G_{b_1}(\theta, X) - G_{b_1}(\bar{\theta}, X)\rvert\right] + \mathbb{E}\left[\lvert G_{b_0}(\theta, X) - G_{b_0}(\bar{\theta}, X)\rvert\right] + \mathbb{E}\left[\lvert G_{W_1}(\theta, X) - G_{W_1}(\bar{\theta}, X)\rvert\right].
\end{aligned}
\end{equation}
For each $j=1,\cdots,m_2$, we denote by 
$$
\mathfrak{N}^j(\bar{\theta}, z):=\sum_{n=1}^{d_2} W_2^{j n} \sigma_2\left(\sum_{k=1}^{d_1}\left[f(\bar{W}_1^{n k}) \sigma_1\left(\left\langle W_0^{k \cdot}, z\right\rangle+f(\bar{b}_0^k)\right)\right]+\bar{b}_1^n\right),
$$
and for any $N=1, \cdots, d_2$, we denote $$\bar{\sigma}_2^{N} := \sigma_2\left(\sum_{k=1}^{d_1}\left[f(\bar{W}_1^{N k}) \sigma_1\left(\left\langle W_0^{k \cdot}, z\right\rangle+f(\bar{b}_0^k)\right)\right]+\bar{b}_1^N\right).$$ To upper bound the second term on the RHS of~\eqref{H=F+G},  for any $j = 1, \cdots, m_2$, one writes
\begin{equation}\label{G_b_1_1}
\begin{aligned}
&\mathbb{E}\left[\left|G_{b_1^N}(\theta, X)-G_{b_1^N}(\bar{\theta}, X)\right|\right]  \\ & \leq
2\mathbb{E}\left[\left\lvert \sum_{j=1}^{m_2}\left(Y^j-\mathfrak{N}^j(\theta, Z)\right) W_2^{j N} \sigma_2^{N}\left(1 - \sigma_2^{N}\right) - \sum_{j=1}^{m_2}\left(Y^j-\mathfrak{N}^j(\bar{\theta}, Z)\right) W_2^{j N} \bar{\sigma}_2^{N}\left(1 - \bar{\sigma}_2^{N}\right) \right\rvert\right] \\ 
& \leq 2 \mathbb{E}\left[\sum_{j=1}^{m_2}\left\lvert Y^j\right\rvert \left\lvert W_2^{j N}\right\rvert\left\lvert \sigma_2^{N}\left(1 - \sigma_2^{N}\right) - \bar{\sigma}_2^{N}\left(1 - \bar{\sigma}_2^{N}\right) \right\rvert \right]
\\ & \quad + 2 \mathbb{E}\left[\sum_{j=1}^{m_2} \left\lvert W_2^{j N}\right\rvert\left\lvert\mathfrak{N}^j(\theta, Z) \sigma_2^{N}\left(1 - \sigma_2^{N}\right) - \mathfrak{N}^j(\bar{\theta}, Z)\bar{\sigma}_2^{N}\left(1 - \bar{\sigma}_2^{N}\right) \right\rvert\right]\\ 
& \leq 2 c_{W_2}\sum_{j=1}^{m_2}\bigg(\underbrace{\mathbb{E}\left[\left\lvert Y^j\right\rvert\left\lvert \sigma_2^{N}\left(1 - \sigma_2^{N}\right) - \bar{\sigma}_2^{N}\left(1 - \bar{\sigma}_2^{N}\right) \right\rvert \right]}_{\mathfrak{T}_1}+ \underbrace{\mathbb{E}\left[\left\lvert\mathfrak{N}^j(\theta, Z) \sigma_2^{N}\left(1 - \sigma_2^{N}\right) - \mathfrak{N}^j(\bar{\theta}, z)\bar{\sigma}_2^{N}\left(1 - \bar{\sigma}_2^{N}\right) \right\rvert\right]}_{\mathfrak{T}_2}\bigg).
\end{aligned}
\end{equation}

To obtain an upper bound for $\mathfrak{T}_1$ and $\mathfrak{T}_2$, we proceed by calculating estimates for several quantities that will be useful throughout the proof. To this end, recall that we assume at least one element in each row of the fixed input matrix $W_0 \in \mathbb{R}^{d_1 \times m_1}$ is nonzero. For each $k=1, \cdots, d_1$, denote by $v_k:=\min \left\{i \in\left\{1, \ldots, m_1\right\} \mid W_0^{k i} \neq 0\right\}$, then $W_0^{k v_k}$ denotes the first nonzero element in the $k$-th row of $W_0$. Assume without loss of generality that $W_0^{k v_k}>0$, and $b_0^k \leq \bar{b}_0^k$, which implies that $f(b_0^k)\leq f(\bar{b}_0^k)$ since $f$ is a monotone increasing function. Moreover, for any $i=1, \ldots, m_1, j=1, \ldots, m_2, z \in \mathbb{R}^{m_1}, y \in \mathbb{R}^{m_2}$, denote by 
$$
z_{-i}:=\left(z^1, \ldots, z^{i-1}, z^{i+1}, \ldots, z^{m_1}\right) \in \mathbb{R}^{m_1-1}, \quad y_{-j}:=\left(y^1, \ldots, y^{j-1}, y^{j+1}, \ldots, y^{m_2}\right) \in \mathbb{R}^{m_2-1},
$$
whereas for each $k=1, \cdots,  d_1$, we denote by 
\begin{equation}\label{A_k}
A_k:=\left\{z \in \mathbb{R}^{m_1} \bigg| \left\langle W_0^{k \cdot}, z\right\rangle+f(b_0^k) > 0\right\}, \quad \bar{A}_k:=\left\{z \in \mathbb{R}^{m_1} \bigg| \left\langle W_0^{k \cdot}, z\right\rangle+f(\bar{b}_0^k) > 0\right\}.
\end{equation}
Then we have, for each $k=1, \ldots, d_1 $, $j=1, \cdots, m_2$, that
$$
\begin{aligned}
& \mathbb{E}\left[\left\lvert \mathbbm{1}_{A_k}(Z) - \mathbbm{1}_{\bar{A}_k}(Z) \right\rvert \right]\\ & = 
\mathbb{E}\left[\mathbbm{1}_{\left\{\left(-f(\bar{b}_0^k)-\sum_{i \neq v_k} W_0^{k i} Z^i\right) / W_0^{k v_k} \leq Z^{v_k}<\left(-f(b_0^k)-\sum_{i \neq v_K} W_0^{K i} Z^i\right) / W_0^{k v_k}\right\}}\right] \\ & = 
\int_{\mathbb{R}} \int_{\mathbb{R}^{m_1-1}}\int^{\frac{-f(b_0^k)-\sum_{i \neq v_k} W_0^{k i} z^i}{W_0^{k v_k}}}_{\frac{-f(\bar{b}_0^k)-\sum_{i \neq v_K} W_0^{k i} z^i}{W_0^{k v_k}}} f_{Z^{v_k} \mid Z_{-v_k}, Y^j}\left(z^{v_k} \mid z_{-v_k}, y^j\right) \mathrm{d} z^{v_k} f_{Z_{-v_k}, Y^j}\left(z_{-v_k}, y^j\right) \mathrm{d} z_{-v_k} \mathrm{d} y^j.
\end{aligned}
$$
This implies that 
\begin{equation}\label{mathbb1}
\mathbb{E}\left[\left|\mathbbm{1}_{A_k}\left(Z\right) - \mathbbm{1}_{\bar{A}_k}\left(Z\right)\right|\right] \leq \frac{C_{Z^{v_k}}}{W_0^{k v_k}}\left|f(b_0^k) - f(\bar{b}_0^k)\right| \leq C_{\mathbbm{1}, \max}|\theta-\bar{\theta}|,
\end{equation}
where the last inequality holds due to $\left\lvert f^{\prime}(v)\right\rvert \leq 1$, where we recall $C_{Z^{v_k}}$ is defined in~\eqref{prop:6.2:c} and $ C_{\mathbbm{1}, \max}:= \max\limits_{k} \left\{ \frac{C_{Z^{v_k}}}{ W_0^{k v_k}} \right\}$. Furthermore, we have, for each $k=1, \cdots, d_1$ and $j=1, \cdots, m_2$, that
$$
\begin{aligned}
& \mathbb{E}\left[\left|Z\right|^2\left|\mathbbm{1}_{A_k}(Z)-\mathbbm{1}_{\bar{A}_k}(Z)\right|\right] \\ & = \mathbb{E}\left[\left(\left|Z^{v_k}\right|^2+\left|Z_{-v_k}\right|^2\right) \mathbbm{1}_{\left(-f(\bar{b}_0^k)-\sum_{i \neq v_k} W_0^{k i} Z^i\right) / W_0^{k v_k} \leq Z^{v_k}<\left(-f(b_0^k)-\sum_{i \neq v_k} W_0^{k i} Z^i\right) / W_0^{k v_k}}\right] \\ & = 
\int_{\mathbb{R}} \int_{\mathbb{R}^{m_1-1}}\int^\frac{-f(b_0^k)-\sum_{i \neq v_k} W_0^{k i} z^i}{W_0^{k v_k}}_\frac{-f(\bar{b}_0^k)-\sum_{i \neq v_k} W_0^{k i} z^i}{W_0^{k v_k}} \left\lvert z^{v_k} \right\rvert^2 f_{Z^{v_k} \mid Z_{-v_k}, Y^j}\left(z^{v_k} \mid z_{-v_k}, y^j\right) \mathrm{d} z^{v_k} f_{Z_{-v_k}, Y^j}\left(z_{-v_k}, y^j\right) \mathrm{d} z_{-v_k} \mathrm{d} y^j\\ & \quad + 
\int_{\mathbb{R}} \int_{\mathbb{R}^{m_1-1}}\int^\frac{-f(b_0^k)-\sum_{i \neq v_k} W_0^{k i} z^i}{W_0^{k v_k}}_\frac{-f(\bar{b}_0^k)-\sum_{i \neq v_k} W_0^{k i} z^i}{W_0^{k v_k}} f_{Z^{v_k} \mid Z_{-v_k}, Y^j}\left(z^{v_k} \mid z_{-v_k}, y^j\right) \mathrm{d} z^{v_k}\left\lvert Z_{-v_k} \right\rvert^2 f_{Z_{-v_k}, Y^j}\left(z_{-v_k}, y^j\right) \mathrm{d} z_{-v_k} \mathrm{d} y^j.
\end{aligned}
$$
This implies that
\begin{equation}\label{mathbbz}
\begin{aligned}
\mathbb{E}\left[|Z|^2\left|\mathbbm{1}_{A_k}(Z)-\mathbbm{1}_{\bar{A}_k}(Z)\right|\right] & \leq \frac{\bar{C}_{Z^{v_k}}}{W_0^{k v_k}}\left|f(b_0) - f(\bar{b}_0^k)\right|+\frac{C_{Z^{v_k}}}{W_0^{v_k}}\mathbb{E}\left[\left|Z_{-v_k}\right|^2\right]\left|f(b_0) - f(\bar{b}_0^k)\right|  \\ & 
\leq C_{Z, \max}\left\lvert\theta - \bar{\theta}\right\rvert,
\end{aligned}
\end{equation}
where we recall $C_{Z^{v_k}}$ and $\bar{C}_{Z^{v_k}}$ are defined in~\eqref{prop:6.2:c} and $C_{Z, \max}:= \max\limits_{k}\left\{\frac{1}{W_0^{k v_k}}\left(\bar{C}_{Z^{v_k}} + C_{Z^{v_k}}\right)\right\}\left(1 + \mathbb{E}\left[\left\lvert Z\right\rvert^2\right]\right)$. Similar calculations yield
$$
\begin{aligned}
& \mathbb{E}\left[\left|Y^j\right|\left|\mathbbm{1}_{A_k}(Z)-\mathbbm{1}_{\bar{A}_k}(Z)\right|\right] \\
&= \mathbb{E}\left[\left|Y^j\right| \mathbbm{1}_{\left\{\left(-f(\bar{b}_0^k)-\sum_{i \neq v_k} W_{0}^{k i} Z^k\right) / W_{0}^{k v_k} \leq Z^{v_k}<\left(-f(b_0^k)-\sum_{i \neq v_k} W_{0}^{k i} Z^k\right) / W_{0}^{k v_k}\right\}}\right] \\
&= \int_{\mathbb{R}} \int_{\mathbb{R}^{m_1-1}} \int_{\frac{-f(\bar{b}_0^k)-\sum_{i \neq v_k} W_{0}^{k v_k}}{W_{0}^{k v_k}}}^{\frac{-f(b_0^k)-\sum_{i \neq v_k} W_{1}^{k i} z^i}{W_{0}^{k v_k}}} f_{Z^{v_k} \mid Z_{-v_k}, y^j}\left(z^{v_k} \mid z_{-v_k}, y^j\right) \mathrm{d} z^{v_k} \left|y^j\right| f_{Z_{-v_k}, y^j}\left(z_{-v_k}, y^j\right) \mathrm{d} z_{-v_k} \mathrm{~d} y^j.
\end{aligned}
$$
This implies that
\begin{equation}\label{cont_Y}
\begin{aligned}
\mathbb{E}{\left[\left|Y^j\right|\left|\mathbbm{1}_{A_k}(Z)-\mathbbm{1}_{\bar{A}_k}(Z)\right|\right] }
& \leq \frac{C_{Z^{v_k}}}{W_0^{k v_k}}\left|\bar{b}_0^k-b_0^k\right| \int_{\mathbb{R}} \int_{\mathbb{R}^{m_1-1}}\left|y^j\right| f_{Z_{-v_k}, Y^j}\left(z_{-v_k}, y^j\right) \mathrm{d} z_{-v_k} \mathrm{~d} y^j \\
& \leq C_{Y, \max} \left\lvert \theta-\bar{\theta}\right\rvert, 
\end{aligned}
\end{equation}
where $C_{Y, \max }:=\max\limits_{k}\left\{\frac{C_{Z^{v_k}}}{W_{0}^{k v_k}}\right\}\mathbb{E}\left[\left|Y\right|\right] $. Then one writes, for each $k=1, \cdots, d_1$ and $j=1, \cdots, m_2$, that
\begin{align*}
& \mathbb{E}\left[ \left\lvert Y^j\right\rvert\left\lvert Z\right\rvert^2\left\lvert\mathbbm{1}_{A_k}(Z)-\mathbbm{1}_{\bar{A}_k}(Z)\right\rvert\right] \\ & = \mathbb{E}\left[\left\lvert Y^j\right\rvert\left(\left|Z^{v_k}\right|^2+\left|Z_{-v_k}\right|^2\right) \mathbbm{1}_{\left\{\left(-f(\bar{b}_0^k)-\sum_{i \neq v_k} W_0^{k i} Z^i\right) / W_0^{k v_{k}} \leq Z^{v_k}<\left(-f(b_0^k)-\sum_{i \neq v_k} W_0^{k i} Z^i\right) / W_0^{k v_k}\right\}}\right] \\ & = 
\int_{\mathbb{R}} \int_{\mathbb{R}^{m_1-1}}\int^\frac{-f(b_0^k)-\sum_{i \neq v_k} W_0^{k i} z^i}{W_0^{k v_k}}_\frac{-f(\bar{b}_0^k)-\sum_{i \neq v_k} W_0^{k i} z^i}{W_0^{k v_k}} \left\lvert z^{v_k} \right\rvert^2 f_{Z^{v_k} \mid Z_{-v_k}, Y^j}\left(z^{v_k} \mid z_{-v_k}, y^j\right) \mathrm{d} z^{v_k} \left|y^j\right| f_{Z_{-v_k}, Y^j}\left(z_{-v_k}, y^j\right) \mathrm{d} z_{-v_k} \mathrm{d} y^j\\ & \quad + 
\int_{\mathbb{R}} \int_{\mathbb{R}^{m_1-1}}\int^\frac{-f(b_0^k)-\sum_{i \neq v_k} W_0^{k i} z^i}{W_0^{k v_k}}_\frac{-f(\bar{b}_0^k)-\sum_{i \neq v_k} W_0^{k i} z^i}{W_0^{k v_k}}f_{Z^{v_k} \mid Z_{-v_k}, Y^j}\left(z^{v_k} \mid z_{-v_k}, y^j\right) \mathrm{d} z^{v_k}\left|y^i\right| \left\lvert Z_{-v_k} \right\rvert^2 f_{Z_{-v_k}, Y^j}\left(z_{-v_k}, y^j\right) \mathrm{d} z_{-v_k} \mathrm{d} y^j.
\end{align*}
This implies that
\begin{equation}\label{cont_YZ}
\begin{aligned}
\mathbb{E}\left[\left\lvert Y^j\right\rvert \left\lvert Z\right\rvert^2 \left|\mathbbm{1}_{A_k}(Z)-\mathbbm{1}_{\bar{A}_k}(Z)\right|\right] & \leq \frac{\bar{C}_{Z^{v_k}}}{W_0^{k v_k}}\mathbb{E}\left[\left| Y^j \right|\right]\left|f(b_0^k) - f(\bar{b}_0^k)\right| + \frac{C_{Z^{v_k}}}{W_0^{v_k}}\mathbb{E}\left[\left|Z_{-v_k}\right|^2 \left\lvert Y^j\right\rvert\right]\left|f(b_0^k) - f(\bar{b}_0^k)\right|  \\ & 
\leq C_{ZY, \max}\left\lvert\theta - \bar{\theta}\right\rvert,
\end{aligned}
\end{equation}
where $ C_{ZY, \max}:=\max\limits_{k}\left\{\frac{\bar{C}_{Z^{v_k}}}{W_0^{k v_k}} + \frac{C_{Z^{v_k}}}{W_0^{v_k}}\right\}\left(1 + \mathbb{E}\left[\left| Y \right|\right]\right)\left(1 + \mathbb{E}\left[ \left\lvert Y\right\rvert\left|Z\right|^2\right]\right)$. 

Now we proceed with upper bounding $\mathfrak{T}_1$ in~\eqref{G_b_1_1}. Since $\left\lvert \sigma_2^{\prime\prime}(\theta) \right\rvert \leq 2$ for any $\theta \in \mathbb{R}$, we have $\left\lvert \sigma_2^{\prime}(\theta) - \sigma_2^{\prime}(\bar{\theta})\right\rvert \leq 2 \left\lvert \theta - \bar{\theta} \right\rvert$. This together with the fact that $\left\lvert f(v)\right\rvert \leq c$ for any $v \in \mathbb{R}$ ensures that
\begin{equation}\label{sigma_2}
\begin{aligned}
&\left\lvert\sigma_2^N\left(1-\sigma_2^N\right)-\bar{\sigma}_2^N\left(1-\bar{\sigma}_2^N\right)\right\rvert 
\\ & \leq 
2 \left\lvert \sum_{k=1}^{d_1}\left[f(W_1^{N k}) \sigma_1\left(\left\langle W_0^{k \cdot}, z\right\rangle+f(b_0^k)\right)\right]+b_1^N - \sum_{k=1}^{d_1}\left[f(\bar{W}_1^{N k}) \sigma_1\left(\left\langle W_0^{k \cdot}, z\right\rangle+f(\bar{b}_0^k)\right)\right] - \bar{b}_1^N \right\rvert 
\\ & \leq
2 \left\lvert b_1^N - \bar{b}_1^N\right\rvert + 2 \sum_{k=1}^{d_1} \left\lvert f\left(W_1^{N k}\right) \right\rvert \left\lvert \sigma_1\left(\left\langle W_0^{k \cdot}, z\right\rangle+f\left(b_0^k\right)\right) -  \sigma_1\left(\left\langle W_0^{k \cdot}, z\right\rangle+f\left(\bar{b}_0^k\right)\right) \right\rvert \\ & \quad +  2 \sum_{k=1}^{d_1}\left\lvert \sigma_1\left(\left\langle W_0^{k \cdot}, z\right\rangle+f\left(\bar{b}_0^k\right)\right) \right\rvert \left\lvert f\left(W_1^{j k}\right)  - f\left(\bar{W}_1^{N k}\right)  \right\rvert.
\end{aligned}
\end{equation}
Then by using~\eqref{mathbb1} and~\eqref{mathbbz} together with the fact that $\left\lvert x\right\rvert \leq 1 + \left\lvert x\right\rvert^2$ for any $x \in \mathbb{R}$, we obtain, for each $N=1, \cdots, d_2$, that
\begin{equation}\label{mathbb1_sigma}
\begin{aligned}
& \mathbb{E}\left[\left|\sigma_2^N\left(1-\sigma_2^N\right)-\bar{\sigma}_2^N\left(1-\bar{\sigma}_2^N\right)\right|\right]  \\ & \leq 2\left\lvert\theta - \bar{\theta}\right\rvert + 2c \sum_{k=1}^{d_1} \bigg(\left\lvert W_0^{k \cdot}\right\rvert \mathbb{E}\left[\left\lvert Z\right\rvert\left\lvert \mathbbm{1}_{A_k}(Z) - \mathbbm{1}_{\bar{A}_k}(Z)\right\rvert \right] + c \mathbb{E}\left[\left\lvert \mathbbm{1}_{A_k}(Z)- \mathbbm{1}_{\bar{A}_k}(Z)\right\rvert\right]\bigg) \\ & \quad +  2 c \sum_{k=1}^{d_1} \mathbb{E}\left[\left\lvert\mathbbm{1}_{\bar{A}_k}(Z)\right\rvert\right] \left\lvert f\left(b_0^k\right) - f\left(\bar{b}_0^k\right) \right\rvert  + 2\sum_{k=1}^{d_1}\left(\left\lvert  W_0^{k \cdot}\right\rvert \mathbb{E}\left[\left\lvert Z\right\rvert\right] + c\right) \left\lvert W_1^{N k} - \bar{W}_1^{N k} \right\rvert
\\ & \leq
2\left(1 + c d_1\right)\left\lvert\theta - \bar{\theta}\right\rvert + 2c \sum_{k=1}^{d_1} \bigg(m_1 c_{W_0} \mathbb{E}\left[\left\lvert Z\right\rvert^2\left\lvert \mathbbm{1}_{A_k}(Z) - \mathbbm{1}_{\bar{A}_k}(Z)\right\rvert \right] + \left(c + m_1 c_{W_0}\right) \mathbb{E}\left[\left\lvert \mathbbm{1}_{A_k}(Z)- \mathbbm{1}_{\bar{A}_k}(Z)\right\rvert\right]\bigg)  \\ & \quad  + 2\sum_{k=1}^{d_1}\left(m_1 c_{W_0} \mathbb{E}\left[\left\lvert Z\right\rvert\right] + c\right) \left\lvert W_1^{N k} - \bar{W}_1^{N k} \right\rvert
\\& \leq
2\left(1 + c d_1\right)\left\lvert\theta - \bar{\theta}\right\rvert + 2c d_1 \bigg(m_1 c_{W_0} C_{Z, \max} + \left(c + m_1 c_{W_0}\right) C_{\mathbbm{1}, \max }\bigg)\left\lvert \theta - \bar{\theta}\right\rvert  + 2d_1\left(m_1 c_{W_0} \mathbb{E}\left[\left\lvert Z\right\rvert\right] + c\right) \left\lvert \theta - \bar{\theta} \right\rvert
\\ & :=
C_{\max , 1} \left\lvert\theta - \bar{\theta}\right\rvert,
\end{aligned}
\end{equation}
where $C_{\max , 1} := 2\left(1 + c d_1\right) + 2c d_1 \bigg(m_1 c_{W_0} C_{Z, \max} + \left(c + m_1 c_{W_0}\right) C_{\mathbbm{1}, \max }\bigg) + 2d_1\left(m_1 c_{W_0} \mathbb{E}\left[\left\lvert Z\right\rvert\right] + c\right) $. Moreover, by using the fact that $\left\lvert x\right\rvert \leq 1 + \left\lvert x\right\rvert^2$ holds for any $x \in \mathbb{R}$,~\eqref{cont_Y} and~\eqref{cont_YZ}, for $x \in \mathbb{R}$, we obtain
\begin{equation}\label{max_2}
\begin{aligned}
&\mathbb{E}\left[\left\lvert Y^j\right\rvert\left|\sigma_2^N\left(1-\sigma_2^N\right)-\bar{\sigma}_2^N\left(1-\bar{\sigma}_2^N\right)\right|\right] \\ & \leq 
2\mathbb{E}\left[\lvert Y^j\rvert\right]\left\lvert\theta - \bar{\theta}\right\rvert + 2c \sum_{k=1}^{d_1} \bigg(\left\lvert W_0^{k \cdot}\right\rvert \mathbb{E}\left[\left\lvert Y^j\right\rvert\left\lvert Z\right\rvert\left\lvert \mathbbm{1}_{A_k}(Z) - \mathbbm{1}_{\bar{A}_k}(Z)\right\rvert \right] + c \mathbb{E}\left[\left\lvert Y^j\right\rvert \left\lvert \mathbbm{1}_{A_k}(Z)- \mathbbm{1}_{\bar{A}_k}(Z)\right\rvert\right]\bigg) \\ & \quad +  2 c \sum_{k=1}^{d_1} \mathbb{E}\left[\left\lvert Y^j\right\rvert\left\lvert\mathbbm{1}_{\bar{A}_k}(Z)\right\rvert\right] \left\lvert f\left(b_0^k\right) - f\left(\bar{b}_0^k\right) \right\rvert  + 2\sum_{k=1}^{d_1}\left(\left\lvert  W_0^{k \cdot}\right\rvert \mathbb{E}\left[\left\lvert Y^j\right\rvert\left\lvert Z\right\rvert\right] + c\mathbb{E}\left[\left\lvert Y^j\right\rvert\right]\right) \left\lvert W_1^{N k} - \bar{W}_1^{N k} \right\rvert
\\ & \leq 
2\mathbb{E}\left[\lvert Y\rvert\right]\left\lvert\theta - \bar{\theta}\right\rvert + 2c \sum_{k=1}^{d_1} \bigg( m_1 c_{W_0} \mathbb{E}\left[\left\lvert Y^j\right\rvert\left\lvert Z\right\rvert\left\lvert \mathbbm{1}_{A_k}(Z) - \mathbbm{1}_{\bar{A}_k}(Z)\right\rvert \right] + c \mathbb{E}\left[\left\lvert Y^j\right\rvert \left\lvert \mathbbm{1}_{A_k}(Z) - \mathbbm{1}_{\bar{A}_k}(Z)\right\rvert\right]\bigg) \\ & \quad +  2 c d_1 \mathbb{E}\left[\left\lvert Y\right\rvert\right] \left\lvert \theta - \bar{\theta} \right\rvert  + 2 d_1\left(m_1 c_{W_0} \mathbb{E}\left[\left\lvert Y\right\rvert\left\lvert Z\right\rvert\right] + c\mathbb{E}\left[\left\lvert Y\right\rvert\right]\right) \left\lvert \theta - \bar{\theta} \right\rvert
\\ & \leq
C_{\max , 2} \left\lvert\theta - \bar{\theta}\right\rvert,
\end{aligned}
\end{equation}
where $$C_{\max , 2} := 2\mathbb{E}\left[\lvert Y\rvert\right]  + 2c d_1\! \left(m_1 c_{W_0}C_{Y, \max}  + \left(c + m_1 c_{W_0}\right)C_{ZY, \max}\right)  + 2c d_1 \mathbb{E}\left[\left\lvert Y\right\rvert\right]+ 2d_1\!\left(m_1 c_{W_0} \mathbb{E}\left[\left\lvert Y\right\rvert\left\lvert Z\right\rvert\right] + c\mathbb{E}\left[\left\lvert Y\right\rvert\right]\right).$$ Furthermore, by using~\eqref{mathbbz} and~\eqref{mathbb1}, we have
\begin{equation}\label{max_3}
\begin{aligned}
&\mathbb{E}\left[\left\lvert Z\right\rvert\left|\sigma_2^N\left(1-\sigma_2^N\right)-\bar{\sigma}_2^N\left(1-\bar{\sigma}_2^N\right)\right|\right] 
\\&\leq
2\mathbb{E}\left[\left\lvert Z\right\rvert\right]\left\lvert\theta - \bar{\theta}\right\rvert + 2c \sum_{k=1}^{d_1} \bigg(\left\lvert W_0^{k \cdot}\right\rvert \mathbb{E}\left[\left\lvert Z\right\rvert^2\left\lvert \mathbbm{1}_{A_k}(Z) - \mathbbm{1}_{\bar{A}_k}(Z)\right\rvert \right] + c \mathbb{E}\left[\left\lvert Z\right\rvert\left\lvert \mathbbm{1}_{A_k}(Z)- \mathbbm{1}_{\bar{A}_k}(Z)\right\rvert\right]\bigg) \\ & \quad  +  2 c \sum_{k=1}^{d_1} \mathbb{E}\left[\left\lvert Z\right\rvert\left\lvert\mathbbm{1}_{\bar{A}_k}(Z)\right\rvert\right] \left\lvert f\left(b_0^k\right) - f\left(\bar{b}_0^k\right) \right\rvert + 2\sum_{k=1}^{d_1}\left(\left\lvert  W_0^{k \cdot}\right\rvert \mathbb{E}\left[\left\lvert Z\right\rvert^2\right] + c\mathbb{E}\left[\left\lvert Z\right\rvert\right]\right) \left\lvert W_1^{N k} - \bar{W}_1^{N k} \right\rvert
\\ & \leq
C_{\max, 3} \left\lvert\theta - \bar{\theta}\right\rvert,
\end{aligned}
\end{equation}
where $$C_{\max, 3}:=2\mathbb{E}\left[\left\lvert Z\right\rvert\right] + 2c d_1 \bigg(\left(m_1 c_{W_0} + c\right) C_{Z, \max} + c C_{\mathbbm{1}, \max}\bigg) + 2 c \sum_{k=1}^{d_1} \mathbb{E}\left[\left\lvert Z\right\rvert\right] + 2 d_1\left(m_1 c_{W_0} \mathbb{E}\left[\left\lvert Z\right\rvert^2\right] + c\mathbb{E}\left[\left\lvert Z\right\rvert\right]\right).$$ To obtain upper bound for $\mathfrak{T}_2$, we proceed to calculate the upper estimates of the following quantities. By using the fact that $\left\lvert \sigma_2^{\prime}(v) \right\rvert\leq 1$ and $\left\lvert f^{\prime}(v) \right\rvert\leq 1$ for any $v \in \mathbb{R}$, we have, for each $j=1, \cdots, m_2$, that,
\begin{small}
\begin{align*}
& \left\lvert \mathfrak{N}^j(\theta, z) - \mathfrak{N}^j(\bar{\theta}, z)\right\rvert \\ & \leq
\sum_{n=1}^{d_2} \left\lvert W_2^{j n}\right\rvert \left\lvert \sigma_2\left(\sum_{k=1}^{d_1}\left[f(W_1^{n k}) \sigma_1\left(\left\langle W_0^{k \cdot}, z\right\rangle+f(b_0^k)\right)\right] + b_1^n\right) - \sigma_2\left(\sum_{k=1}^{d_1}\left[f(\bar{W}_1^{n k}) \sigma_1\left(\left\langle W_0^{k \cdot}, z\right\rangle+f(\bar{b}_0^k)\right)\right] + \bar{b}_1^n\right)\right\rvert \\ & \leq
\sum_{n=1}^{d_2} \left\lvert W_2^{j n} \right\rvert \left\lvert \sum_{k=1}^{d_1}\left[f(W_1^{n k}) \sigma_1\left(\left\langle W_0^{k \cdot}, z\right\rangle+f(b_0^k)\right)\right] + b_1^n - \sum_{k=1}^{d_1}\left[f(\bar{W}_1^{n k}) \sigma_1\left(\left\langle W_0^{k \cdot}, z\right\rangle+f(\bar{b}_0^k)\right)\right] + \bar{b}_1^n\right\rvert \\ & \leq
\sum_{n=1}^{d_2} \left\lvert W_2^{j n} \right\rvert \bigg( \left\lvert b_1^n - \bar{b}_1^n\right\rvert + \sum_{k=1}^{d_1}\left\lvert f(W_1^{n k})\right\rvert \left\lvert \sigma_1\left(\left\langle W_0^{k \cdot}, z\right\rangle+f\left(b_0^k\right)\right)- \sigma_1\left(\left\langle W_0^{k \cdot}, z\right\rangle+f\left(\bar{b}_0^k\right)\right)\right\rvert \\ & \quad + \sum_{k=1}^{d_1} \left\lvert \sigma_1\left(\left\langle W_0^{k \cdot}, z\right\rangle+f\left(\bar{b}_0^k\right)\right)\right\rvert \left\lvert f(W_1^{n k}) - f(\bar{W}_1^{n k})\right\rvert
\bigg)\\ & \leq
c_{W_2}\sum_{n=1}^{d_2} \bigg( \left\lvert b_1^n - \bar{b}_1^n \right\rvert + c \sum_{k=1}^{d_1} \left\langle W_0^{k\cdot}, z\right\rangle \left\lvert \mathbbm{1}_{A_k}(z)- \mathbbm{1}_{\bar{A}_k}(z) \right\rvert + c \sum_{k=1}^{d_1} \left\lvert f\left(b_0^k\right) \mathbbm{1}_{A_k}(z) - f\left(\bar{b}_0^k\right) \mathbbm{1}_{\bar{A}_k}(z)\right\rvert \\ & \quad + \sum_{k=1}^{d_1}\left|\left(\left\lvert W_0^{k\cdot}\right\rvert \left\lvert z\right\rvert + c\right)\right|\left|f\left(W_1^{n k}\right)-f\left(\bar{W}_1^{n k}\right)\right|
\bigg)\\ & \leq
c_{W_2} d_2 \left\lvert\theta- \bar{\theta}\right\rvert + c m_1 c_{W_0} c_{W_2}d_2\sum_{k=1}^{d_1}\left\lvert z\right\rvert\left\lvert \mathbbm{1}_{A_k}(z) - \mathbbm{1}_{\bar{A}_k}(z)\right\rvert  + 
c^2 d_2 c_{W_2}\sum_{k=1}^{d_1}\left\lvert \mathbbm{1}_{A_k}(z) - \mathbbm{1}_{\bar{A}_k}(z) \right\rvert
\\ & \quad  + c d_2 c_{W_2}\sum_{k=1}^{d_1}\left\lvert b_0^{k} - \bar{b}_0^{k} \right\rvert +c_{W_2} \sum_{n=1}^{d_2}\sum_{k=1}^{d_1}\left(m_1 c_{W_0}\left\lvert z\right\rvert + c \right) \left\lvert W_1^{n k} - \bar{W}_1^{n k}\right\rvert.
\end{align*}
\end{small}
This implies that
\begin{equation}\label{square_N}
\begin{aligned}
& \left\lvert \mathfrak{N}^j(\theta, z) - \mathfrak{N}^j(\bar{\theta}, z)\right\rvert \\ & \leq
c_{W_2} d_2 \left\lvert\theta- \bar{\theta}\right\rvert + c m_1 c_{W_0} c_{W_2}d_2\sum_{k=1}^{d_1}\left\lvert z\right\rvert\left\lvert \mathbbm{1}_{A_k}(z) - \mathbbm{1}_{\bar{A}_k}(z)\right\rvert  + 
c^2 d_2 c_{W_2}\sum_{k=1}^{d_1}\left\lvert \mathbbm{1}_{A_k}(z) - \mathbbm{1}_{\bar{A}_k}(z) \right\rvert
\\ & \quad  + c d_2 c_{W_2}\sum_{k=1}^{d_1}\left\lvert b_0^{k} - \bar{b}_0^{k} \right\rvert +c_{W_2} \sum_{n=1}^{d_2}\sum_{k=1}^{d_1}\left(m_1 c_{W_0}\left\lvert z\right\rvert + c \right) \left\lvert W_1^{n k} - \bar{W}_1^{n k}\right\rvert.
\end{aligned}
\end{equation}
By using~\eqref{mathbb1},~\eqref{mathbbz} and taking expectation on both sides, we obtain
\begin{equation}\label{mathbbn}
\begin{aligned}
&\mathbb{E}\left[\left\lvert \mathfrak{N}^j(\theta, Z) - \mathfrak{N}^j(\bar{\theta}, Z)\right\rvert\right] \\ & \leq 
c_{W_2} d_2 \left\lvert\theta- \bar{\theta}\right\rvert + c m_1c_{W_0}c_{W_2}d_2\sum_{k=1}^{d_1}\mathbb{E}\left[\left\lvert Z\right\rvert\left\lvert \mathbbm{1}_{A_k}(Z) - \mathbbm{1}_{\bar{A}_k}(Z)\right\rvert \right] + 
c^2 d_2 c_{W_2}\sum_{k=1}^{d_1}\mathbb{E}\left[\left\lvert \mathbbm{1}_{A_k}(Z) - \mathbbm{1}_{\bar{A}_k}(Z) \right\rvert\right]
\\ & \quad  + c d_2 c_{W_2}\sum_{k=1}^{d_1}\left\lvert b_0^{k} - \bar{b}_0^{k} \right\rvert + c_{W_2}\sum_{n=1}^{d_2}\sum_{k=1}^{d_1}\left(m_1 c_{W_0}\mathbb{E}\left[\left\lvert Z\right\rvert\right] + c \right) \left\lvert W_1^{n k} - \bar{W}_1^{n k}\right\rvert\\ & \leq
C_{\max, \mathfrak{N}}\left\lvert \theta - \bar{\theta}\right\rvert,
\end{aligned}
\end{equation}
where $C_{\max, \mathfrak{N}} := c_{W_2}d_1 + c m_1 c_{W_0}c_{W_2} d_1 d_2  C_{Z, \max}+c d_1 d_2 c_{W_2}\left(c + m_1 c_{W_0}\right) C_{\mathbbm{1}, \max} + 2 c d_1 d_2 c_{W_2}+ d_1 d_2 m_1 c_{W_0}c_{W_2} \mathbb{E}\left[\left\lvert Z\right\rvert\right]$. Furthermore, by using the fact that $\left\lvert x\right\rvert \leq 1 + \left\lvert x\right\rvert^2$ holds for any $x \in \mathbb{R}$, together with~\eqref{mathbb1},~\eqref{mathbbz} and~\eqref{square_N}, we obtain
\begin{small}
\begin{equation}\label{cont_zn}
\begin{aligned}
&\mathbb{E}\left[\left\lvert Z\right\rvert\left\lvert \mathfrak{N}^j(\theta, Z) - \mathfrak{N}^j(\bar{\theta}, Z)\right\rvert\right]\\ & \leq 
c_{W_2} d_2 \mathbb{E}\left[\left\lvert Z\right\rvert\right] \left\lvert\theta- \bar{\theta}\right\rvert + c m_1 c_{W_0}c_{W_2}d_2\sum_{k=1}^{d_1}\mathbb{E}\left[\left\lvert Z\right\rvert^2\left\lvert \mathbbm{1}_{A_k}(Z) - \mathbbm{1}_{\bar{A}_k}(Z)\right\rvert \right] + 
c^2 d_2 c_{W_2}\sum_{k=1}^{d_1}\mathbb{E}\left[\left\lvert Z\right\rvert\left\lvert \mathbbm{1}_{A_k}(Z) - \mathbbm{1}_{\bar{A}_k}(Z) \right\rvert\right]
\\ & \quad  + c d_2 c_{W_2}\sum_{k=1}^{d_1}\mathbb{E}\left[\left\lvert Z\right\rvert\right]\left\lvert b_0^{k} - \bar{b}_0^{k} \right\rvert + c_{W_2}\sum_{n=1}^{d_2}\sum_{k=1}^{d_1}\left(m_1 c_{W_0}\mathbb{E}\left[\left\lvert Z\right\rvert^2\right] + c \mathbb{E}\left[\left\lvert Z\right\rvert\right]\right) \left\lvert W_1^{n k} - \bar{W}_1^{n k}\right\rvert\\
\\ & \leq
C_{\max, Z\mathfrak{N}}\left\lvert \theta - \bar{\theta}\right\rvert,
\end{aligned}
\end{equation}
\end{small}
where $C_{\max, Z\mathfrak{N}} :=  c_{W_2}d_2\mathbb{E}\left[\left\lvert Z\right\rvert\right] + c d_1 d_2c_{W_2}\left( m_1 c_{W_0} + c \right)C_{Z, \max}+ c^2 d_1 d_2 c_{W_2}C_{\mathbbm{1}, \max} + 2c d_1 d_2 c_{W_2}\mathbb{E}\left[\left\lvert Z\right\rvert\right]+ d_1 d_2 m_1 c_{W_0}c_{W_2}\mathbb{E}[\left\lvert Z\right\rvert^2] $. By using~\eqref{nn_upper},~\eqref{mathbb1},~\eqref{mathbbn}, and the fact that $\left\lvert \sigma_2(v)\left(1 - \sigma_2(v)\right)\right\rvert\leq 1$, we have, for each $N=1, \cdots, d_2$ and $j=1, \cdots, m_2$, that
\begin{equation}\label{max_4}
\begin{aligned}
&\mathbb{E}\left[\left\lvert\mathfrak{N}^j(\theta, z) \sigma_2^{N}\left(1 - \sigma_2^{N}\right) - \mathfrak{N}^j(\bar{\theta}, z) \bar{\sigma}_2^{N}\left(1 - \bar{\sigma}_2^{N}\right) \right\rvert\right]\\
&\leq \mathbb{E}\left[\left\lvert\mathfrak{N}^j(\theta, z) \sigma_2^{N}\left(1 - \sigma_2^{N}\right) - \mathfrak{N}^j(\theta, z) \bar{\sigma}_2^{N}\left(1 - \bar{\sigma}_2^{N}\right) \right\rvert \right]+ \mathbb{E}\left[\left\lvert\mathfrak{N}^j(\theta, z) \bar{\sigma}_2^{N}\left(1 - \bar{\sigma}_2^{N}\right) - \mathfrak{N}^j(\bar{\theta}, z) \bar{\sigma}_2^{N}\left(1 - \bar{\sigma}_2^{N}\right) \right\rvert\right] \\
& \leq d_2 c_{W_{2}}\mathbb{E}\left[\left\lvert \sigma_2^{N}\left(1 - \sigma_2^{N}\right) - \bar{\sigma}_2^{N}\left(1 - \bar{\sigma}_2^{n}\right) \right\rvert\right] + \mathbb{E}\left[\left\lvert\mathfrak{N}^j(\theta, z) - \mathfrak{N}^j(\bar{\theta}, z) \right\rvert\right] \\
& \leq C_{\max, 4} \left\lvert \theta - \bar{\theta}\right\rvert,
\end{aligned}
\end{equation}
where $C_{\max, 4}:=d_2 c_{W_{2}}C_{\max, 1}  + C_{\max, \mathfrak{N}}$. Substituting~\eqref{max_2} and~\eqref{max_4} into~\eqref{G_b_1_1} yields
\begin{equation}\label{G_b_1}
\begin{aligned}
\mathbb{E}\left[\left|G_{b_1^N}(\theta, X)-G_{b_1^N}(\bar{\theta}, X)\right| \right] \leq C_{\max, b_1}\left\lvert\theta-\bar{\theta}\right\rvert,
\end{aligned}
\end{equation}
where $C_{\max, b_1}:=2 c_{W_2} m_2 C_{\max, 2} + 2c_{W_2} m_2 C_{\max, 4}$. Next, we calculate the upper bound for the second term of~\eqref{H=F+G}. We have that
$$
\begin{aligned}
&\mathbb{E}\left[\left\lvert G_{b_0^K}(\theta, X) - G_{b_0^K}(\bar{\theta}, X)\right\rvert \right] 
\\ & \leq 2 c_{W_2}\sum_{j=1}^{m_2}\sum_{n=1}^{d_2} \bigg(\mathbb{E}\left[\left\lvert Y^j\sigma_2^n(1 - \sigma_2^n) f\left(W_1^{n K}\right)f^{\prime}\left(b_0^{K}\right) \mathbbm{1}_{A_{K}}(Z)
 -  Y^j\bar{\sigma}_2^n(1 - \bar{\sigma}_2^n) f\left(\bar{W}_1^{n K}\right)f^{\prime}\left(\bar{b}_0^{K}\right)\mathbbm{1}_{\bar{A}_{K}}(Z)\right\rvert\right]
\\ & \quad + \mathbb{E}\left[\left\lvert \mathfrak{N}^j(\theta, z)\sigma_2^n(1 - \sigma_2^n) f\left(W_1^{n K}\right)f^{\prime}\left(b_0^{K}\right) \mathbbm{1}_{A_{K}}(Z)
 -  \mathfrak{N}^j(\bar{\theta}, Z)\bar{\sigma}_2^n(1 - \bar{\sigma}_2^n) f\left(\bar{W}_1^{n K}\right)f^{\prime}\left(\bar{b}_0^{K}\right) \mathbbm{1}_{\bar{A}_{K}}(Z)\right\rvert\right] \bigg)\\ & \leq
T_{b_0, 3} + T_{b_0, 4} + T_{b_0, 5} + T_{b_0, 6} + T_{b_0, 7},
\end{aligned}
$$
where
$$
\begin{aligned}
T_{b_0, 3} &:= 2 c_{W_2} \sum_{j=1}^{m_2} \sum_{n=1}^{d_2} \mathbb{E}\left[\left\lvert Y^j  f\left(W_1^{n K}\right) f^{\prime}\left(b_0^K\right)\right\rvert \left\lvert\sigma_2^n\left(1-\sigma_2^n\right) \mathbbm{1}_{A_K}(Z) -  \bar{\sigma}_2^n\left(1-\bar{\sigma}_2^n\right)1_{\bar{A}_K}\right\rvert\right],\\
T_{b_0, 4} &:= 2 c_{W_2} \sum_{j=1}^{m_2} \sum_{n=1}^{d_2} \mathbb{E}\left[\left\lvert Y^j\right\rvert\left\lvert\bar{\sigma}_2^n\left(1-\bar{\sigma}_2^n\right)\mathbbm{1}_{\bar{A}_K}(Z)\right\rvert\left\lvert f\left(W_1^{n K}\right) f^{\prime}\left(b_0^K\right) - f\left(\bar{W}_1^{n K}\right) f^{\prime}\left(\bar{b}_0^K \right) \right\rvert\right],\\
T_{b_0, 5} &:= 2 c_{W_2} \sum_{j=1}^{m_2} \sum_{n=1}^{d_2} \mathbb{E}\left[\left\lvert f\left(W_1^{n K}\right) f^{\prime}\left(b_0^K\right)\mathbbm{1}_{A_K}(Z) \right\rvert \left\lvert \mathfrak{N}^j(\theta, z)\sigma_2^n\left(1-\sigma_2^n\right)  -  \mathfrak{N}^j(\bar{\theta}, z)\bar{\sigma}_2^n\left(1-\bar{\sigma}_2^n\right)\right\rvert\right],
\\
T_{b_0, 6} &:= 2 c_{W_2} \sum_{j=1}^{m_2} \sum_{n=1}^{d_2} \mathbb{E}\left[\left\lvert \mathfrak{N}^j(\bar{\theta}, Z)\bar{\sigma}_2^n\left(1-\bar{\sigma}_2^n\right)\mathbbm{1}_{A_K}(Z) \right\rvert \left\lvert f\left(W_1^{n K}\right) f^{\prime}\left(b_0^K\right)  -  f\left(\bar{W}_1^{n K}\right) f^{\prime}\left(\bar{b}_0^K \right)\right\rvert\right],
\\
T_{b_0, 7} &:= 2 c_{W_2} \sum_{j=1}^{m_2} \sum_{n=1}^{d_2} \mathbb{E}\left[\left\lvert \mathfrak{N}^j(\bar{\theta}, Z)\bar{\sigma}_2^n\left(1-\bar{\sigma}_2^n\right)f\left(\bar{W}_1^{n K}\right) f^{\prime}\left(\bar{b}_0^K\right) \right\rvert \left\lvert \mathbbm{1}_{A_K}(Z)  - \mathbbm{1}_{\bar{A}_K}(Z) \right\rvert\right].
\end{aligned}
$$
To upper bound $T_{b_0, 3}$, we use $\left\lvert f(v)\right\rvert\leq c$ and $\left\lvert \sigma_2(v)\left(1 - \sigma_2(v)\right)\right\rvert\leq 1$ to obtain
$$
\begin{aligned}
T_{b_0, 3} & \leq 
2 c c_{W_2}  \sum_{j=1}^{m_2} \sum_{n=1}^{d_2}\bigg( \mathbb{E}\bigg[\left\lvert Y^j \right\rvert \left\lvert\sigma_2^n\left(1-\sigma_2^n\right) \mathbbm{1}_{A_K}(Z) - \sigma_2^n\left(1-\sigma_2^n\right)\mathbbm{1}_{\bar{A}_K}(Z) \right\rvert \\ & \quad + \left\lvert Y^j\right\rvert \left\lvert \mathbbm{1}_{\bar{A}_K}(Z) \right\rvert\left\lvert \sigma_2^n\left(1-\sigma_2^n\right) -  \bar{\sigma}_2^n\left(1-\bar{\sigma}_2^n\right)\right\rvert\bigg] \bigg)\\ & \leq 
2 c c_{W_2}  \sum_{j=1}^{m_2} \sum_{n=1}^{d_2} \bigg(\mathbb{E}\left[\left\lvert Y^j \right\rvert \left\lvert \mathbbm{1}_{A_K}(Z) - \mathbbm{1}_{\bar{A}_K}(Z) \right\rvert + \left\lvert Y^j\right\rvert\left\lvert \sigma_2^n\left(1 - \sigma_2^n\right) -  \bar{\sigma}_2^n\left(1-\bar{\sigma}_2^n\right)\right\rvert\right]\bigg)\\ & \leq 
2 c c_{W_2} d_2m_2\left(C_{Y, \max} + C_{\max, 2}\right)\left\lvert\theta -\bar{\theta}\right\rvert,
\end{aligned}
$$
where the last inequality holds due to~\eqref{cont_Y} and~\eqref{max_2}. For $T_{b_0, 4}$, by using $\left\lvert \bar{\sigma}_2\left(1 - \bar{\sigma}_2\right)\right\rvert \leq 1$, $\left\lvert f(x)\right\rvert\leq c$, $\left\lvert f^{\prime}(x)\right\rvert\leq 1$ and $\left\lvert f^{\prime \prime}(x)\right\rvert\leq 2 / c$ for any $x \in \mathbb{R}$, one writes
$$
\begin{aligned}
T_{b_0, 4} & \leq 
2 c_{W_2} \sum_{j=1}^{m_2} \sum_{n=1}^{d_2} \mathbb{E}\left[\left\lvert Y^j\right\rvert\left\lvert f\left(W_1^{n K}\right) f^{\prime}\left(b_0^K\right) - f\left(\bar{W}_1^{n K}\right) f^{\prime}\left(\bar{b}_0^K \right) \right\rvert\right]\\ &\leq
2 c_{W_2} \sum_{j=1}^{m_2} \sum_{n=1}^{d_2} \mathbb{E}\left[\left\lvert Y^j\right\rvert\left\lvert f\left(W_1^{n K}\right)\right\rvert\left\lvert  f^{\prime}\left(b_0^K\right) - f^{\prime}\left(\bar{b}_0^K \right)\right\rvert + \left\lvert Y^j\right\rvert\left\lvert f^{\prime}\left(\bar{b}_0^K\right)\right\rvert\left\lvert f\left(W_1^{n K}\right) - f\left(\bar{W}_1^{n K}\right) \right\rvert\right] \\ & \leq
6c_{W_2}m_2d_2\mathbb{E}\left[\left\lvert Y\right\rvert\right]\left\lvert\theta-\bar{\theta}\right\rvert.
\end{aligned}
$$
For $T_{b_0, 5}$, by using $\left\lvert f(v)\right\rvert\leq c$, $\left\lvert f^{\prime}(v)\right\rvert\leq 1$, and~\eqref{max_4}, we have that
$$
\begin{aligned}
T_{b_0, 5} &\leq 2c c_{W_2} \sum_{j=1}^{m_2} \sum_{n=1}^{d_2} \mathbb{E}\left[\left\lvert \mathfrak{N}^j(\theta, Z)\sigma_2^n\left(1-\sigma_2^n\right)  -  \mathfrak{N}^j(\bar{\theta}, Z)\bar{\sigma}_2^n\left(1-\bar{\sigma}_2^n\right)\right\rvert\right]\leq 4cc_{W_2}m_2d_2 C_{\max , 4}\left|\theta-\bar{\theta}\right|.
\end{aligned}
$$
For $T_{b_0, 6}$, by using $\left\lvert \bar{\sigma}_2\left(1 - \bar{\sigma}_2\right)\right\rvert \leq 1$ and~\eqref{nn_upper}, we obtain that
$$
\begin{aligned}
T_{b_0, 6} &\leq 2 d_2 c_{W_2}^2 \sum_{j=1}^{m_2} \sum_{n=1}^{d_2} \mathbb{E}\left[ \left\lvert f\left(W_1^{n K}\right) f^{\prime}\left(b_0^K\right)  -  f\left(\bar{W}_1^{n K}\right) f^{\prime}\left(\bar{b}_0^K\right)\right\rvert\right] \leq 6 m_2 d_2^2 c_{W_2}^2\left\lvert\theta-\bar{\theta}\right\rvert.
\end{aligned}
$$
For $T_{b_0, 7}$, by using $\left\lvert \bar{\sigma}_2\left(1 - \bar{\sigma}_2\right)\right\rvert \leq 1$, $\left\lvert f(v)\right\rvert\leq c$, $\left\lvert f^{\prime}(v)\right\rvert\leq 1$,~\eqref{nn_upper}, and~\eqref{mathbb1}, we have that
$$
\begin{aligned}
T_{b_0, 7} &\leq 2 c_{W_2} \sum_{j=1}^{m_2} \sum_{n=1}^{d_2} \mathbb{E}\left[\left\lvert \mathfrak{N}^j(\bar{\theta}, Z)\bar{\sigma}_2^n\left(1-\bar{\sigma}_2^n\right)f\left(\bar{W}_1^{n K}\right) f^{\prime}\left(\bar{b}_0^K\right) \right\rvert \left\lvert \mathbbm{1}_{A_K}(Z)  - \mathbbm{1}_{\bar{A}_K}(Z) \right\rvert\right] \\ & \leq 2c m_2 d_2^2 c_{W_2}^2 C_{\mathbbm{1}, \max}\left\lvert\theta-\bar{\theta}\right\rvert.
\end{aligned}
$$
By using the results above, we obtain
\begin{equation}\label{G_b_0}
\begin{aligned}
\mathbb{E}\left[\left|G_{b_0^K}(\theta, x)-G_{b_0^K}(\bar{\theta}, x)\right|\right] \leq C_{\max, b_0}\left\lvert\theta -\bar{\theta}\right\rvert,
\end{aligned}
\end{equation}
where $C_{\max, b_0}:= 2 c c_{W_2} d_2m_2\left(C_{Y, \max} + C_{\max, 2}\right) + 6c_{W_2}m_2d_2\mathbb{E}\left[\left\lvert Y^j\right\rvert\right] + 4cc_{W_2}m_2d_2 C_{\max , 4} + 6 m_2 d_2^2 c_{W_2}^2 + 2c m_2 d_2^2 c_{W_2}^2 C_{\mathbbm{1}, \max}$. Then, denoting by $\sigma_1^K:=\sigma_1\left(\left\langle W_0^{K \cdot}, Z\right\rangle+f\left(b_0^K\right)\right)$ and $\bar{\sigma}_1^K:=\sigma_1\left(\left\langle W_0^{K \cdot}, Z\right\rangle+f\left(\bar{b}_0^K\right)\right)$,  by using $\left\lvert f^{\prime}(v)\right\rvert \leq 1$, $\left\lvert f^{\prime \prime}(v)\right\rvert \leq 2 / c$ and $\left\lvert \sigma_1^K\right\rvert \leq m_1 c_{W_0}\left\lvert z\right\rvert+c$, we obtain
\begin{align*}
&\mathbb{E}\left[\left\lvert G_{W_1^{N K}}(\theta, x) - G_{W_1^{N K}}(\bar{\theta}, x)\right\rvert \right] 
\\& \leq 2 \sum_{j=1}^{m_2}\left\lvert W_2^{j N}\right\rvert\bigg( \mathbb{E}\left[\left\lvert Y^j\right\rvert \left\lvert\sigma_1^K\right\rvert \left\lvert \sigma_2^N\left(1 - \sigma_2^N\right)  - \bar{\sigma}_2^N\left(1 - \bar{\sigma}_2^N\right) \right\rvert\right] + \mathbb{E}\left[\left\lvert Y^j\right\rvert\left\lvert\bar{\sigma}_2^N\left(1 - \bar{\sigma}_2^N\right)\right\rvert \left\lvert\sigma_1^K -\bar{\sigma}_1^K\right\rvert\right] \\ & \quad + 
\mathbb{E}\left[\left\lvert \sigma_1^K\right\rvert\left\lvert\mathfrak{N}^i(\theta, Z) \sigma_2^N\left(1 - \sigma_2^N\right)  - \mathfrak{N}^j(\bar{\theta}, Z) \bar{\sigma}_2^N\left(1 - \bar{\sigma}_2^N\right)\right\rvert\right] +\mathbb{E}\left[\left\lvert\mathfrak{N}^j(\bar{\theta}, Z) \bar{\sigma}_2^N\left(1 - \bar{\sigma}_2^N\right)\right\rvert \left\lvert \sigma_1^K - \bar{\sigma}_1^K \right\rvert\right] \\ & \quad 
+ \mathbb{E}\left[\left\lvert Y^j\right\rvert\left\lvert\bar{\sigma}_2^N\left(1 - \bar{\sigma}_2^N\right)\right\rvert \left\lvert\bar{\sigma}_1^K\right\rvert \left\lvert f^{\prime}(W_1^{N K})\right\rvert - \left\lvert f^{\prime}(\bar{W}_1^{N K})\right\rvert\right]  \\ & \quad + \mathbb{E}\left[\left\lvert\mathfrak{N}^j(\bar{\theta}, Z) \bar{\sigma}_2^N\left(1 - \bar{\sigma}_2^N\right)\right\rvert \left\lvert \bar{\sigma}_1^K \right\rvert\left\lvert f^{\prime}(W_1^{N K}) -f^{\prime}(\bar{W}_1^{N K}) \right\rvert\right]\bigg)
\\ & \leq
2c_{W_2} \sum_{j=1}^{m_2} \bigg(m_1 c_{W_0}\underbrace{\mathbb{E}\left[\left\lvert Y^j\right\rvert 
\left\lvert Z\right\rvert\left\lvert \sigma_2^N\left(1 - \sigma_2^N\right)  - \bar{\sigma}_2^N\left(1 - \bar{\sigma}_2^N\right) \right\rvert\right]}_{\mathfrak{T}_3}
+c \underbrace{\mathbb{E}\left[\left\lvert Y^j\right\rvert\ \left\lvert \sigma_2^N\left(1 - \sigma_2^N\right)  - \bar{\sigma}_2^N\left(1 - \bar{\sigma}_2^N\right) \right\rvert\right]}_{\mathfrak{T}_4}
\\ & \quad + m_1 c_{W_0}\underbrace{\mathbb{E}\left[ \left\lvert Z\right\rvert\left|\mathfrak{N}^j(\theta, Z) \sigma_2^N\left(1-\sigma_2^N\right)-\mathfrak{N}^j(\bar{\theta}, Z) \bar{\sigma}_2^N\left(1-\bar{\sigma}_2^N\right)\right| \right]}_{\mathfrak{T}_{5}} \\ & \quad +  c\underbrace{\mathbb{E}\left[\left|\mathfrak{N}^i(\theta, Z) \sigma_2^N\left(1-\sigma_2^N\right)-\mathfrak{N}^i(\bar{\theta}, Z) \bar{\sigma}_2^N\left(1-\bar{\sigma}_2^N\right)\right| \right]}_{\mathfrak{T}_{6}} +\underbrace{\mathbb{E}\left[\left|Y^j\right|\bar{\sigma}_2^N\left(1-\bar{\sigma}_2^N\right)\left|\sigma_1^K-\bar{\sigma}_1^K\right|\right]}_{\mathfrak{T}_{7}} \\ & \quad + \underbrace{\mathbb{E}\left[\left\lvert \mathfrak{N}^j(\bar{\theta}, Z) \bar{\sigma}_2^N\left(1-\bar{\sigma}_2^N\right)\right\rvert\left\lvert \sigma_1^K-\bar{\sigma}_1^K\right\rvert\right]}_{\mathfrak{T}_{8}} \\ & \quad + 
 \frac{2}{c} \left( m_1 c_{W_0} \mathbb{E}\left[\left\lvert Y\right\rvert\left\lvert Z\right\rvert\right]  + c\mathbb{E}\left[\left\lvert Y\right\rvert\right] +  m_1 d_2 c_{W_0}c_{W_2} \mathbb{E}\left[\left\lvert Z\right\rvert\right]+ c d_2 c_{W_2} \right)\left\lvert \theta - \bar{\theta}\right\rvert\bigg).
\end{align*}
To upper bound $\mathfrak{T}_{3}$, by using~\eqref{sigma_2}, $|f(x)| \leq c,\left|f^{\prime}(x)\right| \leq 1$, and the fact that $\lvert x\rvert \leq 1 +\lvert x\rvert^2$ holds for any $x \in \mathbb{R}$, we have that
$$
\begin{aligned}
&\mathbb{E}\left[\left\lvert Y^j\right\rvert \left\lvert Z\right\rvert\left\lvert \sigma_2^N\left(1 - \sigma_2^N\right)  - \bar{\sigma}_2^N\left(1 - \bar{\sigma}_2^N\right) \right\rvert\right] \\ & \leq 2\mathbb{E}\left[\left\lvert Y^j\right\rvert\left\lvert Z\right\rvert\right]\left\lvert\theta - \bar{\theta}\right\rvert + 2 c \sum_{k=1}^{d_1}\bigg( m_1 c_{W_0}\mathbb{E}\left[\left\lvert Y^j\right\rvert \left\lvert Z\right\rvert^2 \left\lvert \mathbbm{1}_{A_k}(Z) - \mathbbm{1}_{\bar{A}_k}(Z)\right\rvert \right]  + c \mathbb{E}\left[\left\lvert Y^j\right\rvert\left\lvert \mathbbm{1}_{A_k}(Z)- \mathbbm{1}_{\bar{A}_k}(Z)\right\rvert\right]\\ & \quad + c \mathbb{E}\left[\left\lvert Y^j\right\rvert\left\lvert Z\right\rvert^2\left\lvert \mathbbm{1}_{A_k}(Z) - \mathbbm{1}_{\bar{A}_k}(Z)\right\rvert\right]\bigg) + 2 c \sum_{k=1}^{d_1} \mathbb{E}\left[\left\lvert Y^j\right\rvert \left\lvert Z\right\rvert \left|\mathbbm{1}_{\bar{A}_k}(Z)\right|\right]\left|f\left(b_0^k\right)-f\left(\bar{b}_0^k\right)\right|\\ & \quad +2 \sum_{k=1}^{d_1}\left(m_1 c_{W_0} \mathbb{E}\left[\left\lvert Y^j\right\rvert \left\lvert Z\right\rvert^2 \right]+c\mathbb{E}\left[\left\lvert Y^j\right\rvert \left\lvert Z\right\rvert \right]\right)\left|W_1^{N k}-\bar{W}_1^{N k}\right|
\\ & \leq 
 C_{T_3, \max}\left\lvert\theta - \bar{\theta}\right\rvert,
\end{aligned}
$$
where $C_{T_3, \max}:= 2\mathbb{E}\left[\left\lvert Y\right\rvert\left\lvert Z\right\rvert\right]+2 c d_1 \left(\left(m_1 c_{W_0} + c\right)C_{ZY, \max} + c C_{Y, \max}\right)+2cd_1\mathbb{E}\left[\left\lvert Y\right\rvert\left\lvert Z\right\rvert\right]$ \\ $+2d_1 \Big(m_1 c_{W_0}\mathbb{E}\left[\left\lvert Y\right\rvert\left\lvert Z\right\rvert^2\right]+c\mathbb{E}\left[\left\lvert Y\right\rvert \left\lvert Z\right\rvert \right]\Big)$. For $\mathfrak{T}_{4}$, by using~\eqref{max_2}, we derive that 
$$
\begin{aligned}
&\mathbb{E}\left[\left\lvert Y^j\right\rvert\ \left\lvert \sigma_2^N\left(1 - \sigma_2^N\right)  - \bar{\sigma}_2^N\left(1 - \bar{\sigma}_2^N\right) \right\rvert\right]\leq C_{\max, 2} \left\lvert \theta - \bar{\theta}\right\rvert.
\end{aligned}
$$
For $\mathfrak{T}_{5}$, by using~\eqref{nn_upper},~\eqref{max_3}, and~\eqref{cont_zn}, we deduce that
$$
\begin{aligned}
&\mathbb{E}\left[ \left\lvert Z\right\rvert\left|\mathfrak{N}^j(\theta, Z) \sigma_2^N\left(1-\sigma_2^N\right)-\mathfrak{N}^j(\bar{\theta}, Z) \bar{\sigma}_2^N\left(1-\bar{\sigma}_2^N\right)\right| \right]\\ &\leq 
\mathbb{E}\left[\left\lvert Z\right\rvert\left\lvert \mathfrak{N}^j(\theta, Z)\right\rvert\left|\sigma_2^N\left(1-\sigma_2^N\right) - \bar{\sigma}_2^N\left(1-\bar{\sigma}_2^N\right)\right|\right]+\mathbb{E}\left[\left\lvert Z\right\rvert\left\lvert \bar{\sigma}_2^N\left(1-\bar{\sigma}_2^N\right)\right\rvert \left|\mathfrak{N}^j(\theta, Z) - \mathfrak{N}^j(\bar{\theta}, Z)\right| \right] \\ & \leq
C_{T_{5}, \max} \left\lvert \theta-\bar{\theta}\right\rvert,
\end{aligned}
$$
where $C_{T_{5}, \max}:= d_2 c_{W_2} C_{\max, 3} + C_{\max , Z\mathfrak{N}}$. 
For $\mathfrak{T}_{6}$, we see that
$$
\begin{aligned}
\mathbb{E}\left[\left|\mathfrak{N}^j(\theta, Z) \sigma_2^N\left(1-\sigma_2^N\right)-\mathfrak{N}^j(\bar{\theta}, Z) \bar{\sigma}_2^N\left(1-\bar{\sigma}_2^N\right)\right| \right] \leq C_{T_{6}, \max}\left\lvert \theta - \bar{\theta}\right\rvert,
\end{aligned}
$$
where $C_{T_{6}, \max} := C_{\max, 4}$. For $\mathfrak{T}_{7}$, by using~\eqref{cont_YZ} and~\eqref{cont_Y}, we get
$$
\begin{aligned}
& \mathbb{E}\left[\left|Y^j\right|\left|\left(\left\langle W_0^{K \cdot}, Z\right\rangle+f\left(b_0^K\right)\right) \mathbbm{1}_{A_K}(Z)-\left(\left\langle W_0^{K \cdot}, Z\right\rangle+f\left(\bar{b}_0^K\right)\right) \mathbbm{1}_{\bar{A}_K}(Z)\right|\right]\\ & \leq 
m_1 c_{W_0}\mathbb{E}\left[\left|Y^j\right|\left\lvert Z\right\rvert\ \left|\mathbbm{1}_{A_K}(Z)- \mathbbm{1}_{\bar{A}_K}(Z)\right|\right] + c\mathbb{E}\left[\left|Y^j\right|\left| \mathbbm{1}_{A_K}(Z) - \mathbbm{1}_{\bar{A}_K}(Z)\right|\right] + \mathbb{E}\left[\left|Y^j\right|\left| f\left(b_0^K\right) - f\left(\bar{b}_0^K\right) \right|\right]\\ & \leq
m_1 c_{W_0}\mathbb{E}\left[\left|Y^j\right|\left\lvert Z\right\rvert^2 \left|\mathbbm{1}_{A_K}(Z)- \mathbbm{1}_{\bar{A}_K}(Z)\right|\right] + (m_1 c_{W_0} + c)\mathbb{E}\left[\left|Y^j\right|\left| \mathbbm{1}_{A_K}(Z) - \mathbbm{1}_{\bar{A}_K}(Z)\right|\right] \\ & \quad + \mathbb{E}\left[\left|Y^j\right|\left| f\left(\bar{b}_0^K\right) - f\left(\bar{b}_0^K\right) \right|\right]\\& \leq
C_{T_{7}, \max }\left\lvert \theta - \bar{\theta}\right\rvert,
\end{aligned}
$$
where $C_{T_{7}, \max } := c_{W_0} C_{YZ, \max} + (c_{W_0} + c)C_{Y, \max} + \mathbb{E}\left[\left|Y\right|\right]$. For $\mathfrak{T}_{8}$, by using $\left\lvert f(v)\right\rvert \leq c$, $\left\lvert f^{\prime}(v)\right\rvert \leq 1$,~\eqref{mathbb1}, and~\eqref{mathbbz}, we notice that 
$$
\begin{aligned}
& \mathbb{E}\left[\left\lvert\mathfrak{N}^j(\bar{\theta}, z) \bar{\sigma}_2^N\left(1-\bar{\sigma}_2^N\right)\right\rvert\left\lvert \sigma_1-\bar{\sigma}_1\right\rvert\right]
\\ & \leq d_2 c_{W_2} \mathbb{E}\left[\left\lvert \left(\left\langle W_0^{K\cdot}, Z\right\rangle+f\left(b_0^K\right)\right) \mathbbm{1}_{A_K}(Z)-\left(\left\langle W_0^{K\cdot}, Z\right\rangle+f\left(\bar{b}_0^K\right)\right) \mathbbm{1}_{\bar{A}_K}(Z)\right\rvert\right] \\ & \leq 
d_2 c_{W_2}\bigg(\!m_1 c_{W_0} \!\left(\mathbb{E}\left[\left\lvert \mathbbm{1}_{A_K}(Z) - \mathbbm{1}_{\bar{A}_K}(Z) \right\rvert \right] +  \mathbb{E}\left[ \left\lvert Z\right\rvert^2\left\lvert \mathbbm{1}_{A_K}(Z) - \mathbbm{1}_{\bar{A}_K}(Z) \right\rvert \right]\right) + \mathbb{E}\left[\left\lvert f(b_0^K)\right\rvert\left\lvert \mathbbm{1}_{A_K}(Z) - \mathbbm{1}_{\bar{A}_K}(Z) \right\rvert \right] \\ & \quad + \mathbb{E}\left[\left\lvert \mathbbm{1}_{\bar{A}_K}(Z) \right\rvert \left\lvert f(b_0^K) - f(\bar{b}_0^K)\right\rvert \right]\bigg)
\\ &\leq 
C_{T_{8}, \max}\left\lvert\theta-\bar{\theta}\right\rvert,
\end{aligned}
$$
where $C_{T_{8}, \max } := d_2 c_{W_2} \left(m_1 c_{W_0}(C_{\mathbbm{1}, \max} + C_{Z, \max}) + cC_{\mathbbm{1}, \max} + 1 \right)$. Finally, we obtain
\begin{equation}\label{G_W_1}
\begin{aligned}
&\mathbb{E}\left[\left\lvert G_{W_1^{N K}}(\theta, x) - G_{W_1^{N K}}(\bar{\theta}, x)\right\rvert \right]  & \leq
C_{\max, W_1} \left\lvert\theta - \bar{\theta}\right\rvert,
\end{aligned}
\end{equation}
where $C_{\max, W_1} := 2c_{W_2} m_2 \bigg(m_1 c_{W_0}C_{T_3, \max}
+c C_{T_4, \max} + m_1c_{W_0}C_{T_{5}, \max} + cC_{T_{6}, \max} + C_{T_{7}, \max} + C_{T_{8}, \max} + C_{T_{9}, \max}\bigg)$ and $C_{T_{9}, \max}:= \frac{2}{c} \left( m_1 c_{W_0} \mathbb{E}\left[\left\lvert Y\right\rvert\left\lvert Z\right\rvert\right]  + c\mathbb{E}\left[\left\lvert Y\right\rvert\right] +  m_1 d_2 c_{W_0}c_{W_2} \mathbb{E}\left[\left\lvert Z\right\rvert\right]+ c d_2 c_{W_2} \right) $. Then, Assumption~\ref{asm:A4} holds by substituting~\eqref{G_b_1},~\eqref{G_b_0}, and~\eqref{G_W_1} into~\eqref{H=F+G}. Furthermore, Assumption~\ref{asm:A5} holds with $A(x)=2 \lambda_r \mathbf{I}_d$ and $B(x)=0$, which implies $a=2 \lambda_r$ and $b=0$. Therefore, the optimization problem~\eqref{op:nn} satisfies Assumptions~\ref{asm:A3}-\ref{asm:A5}. 
\end{proof}
\begin{proof}[\textit{Proof of Corollary~\ref{corollary 7.3}}]
\label{pf_prop:7.3}
Let $\theta=\left(\left[W_1\right], b_0, b_1\right) \in \mathbb{R}^d$, $\bar{\theta}=\left(\left[\bar{W}_1\right], \bar{b}_0, \bar{b}_1\right) \in \mathbb{R}^d$, and let $A_K$ and $\bar{A}_K$ be defined in~\eqref{A_k}. One notes that under the assumptions in Proposition~\ref{prop:6.2}, the following result can be obtained. For each $k=1, \cdots, d_1$ and $j=1, \cdots, m_2$, we have
\begin{align*}
& \mathbb{E}\left[\left|Y^j\right|\left|\mathbbm{1}_{A_k}(Z)-\mathbbm{1}_{\bar{A}_k}(Z)\right|\right] \\
&= \mathbb{E}\left[\left|Y^j\right| \mathbbm{1}_{\left\{\left(-f(\bar{b}_0^k)-\sum_{i \neq v_k} W_{0}^{k i} Z^k\right) / W_{0}^{k v_k} \leq Z^{v_k}<\left(-f(b_0^k)-\sum_{i \neq v_k} W_{0}^{k i} Z^k\right) / W_{0}^{k v_k}\right\}}\right] \\
&= \int_{\mathbb{R}^{m_1-1}} \int_{\frac{-f(\bar{b}_0^k)-\sum_{i \neq v_k} W_{0}^{k v_k}}{W_{0}^{k v_k}}}^{\frac{-f(b_0^k)-\sum_{i \neq v_k} W_{1}^{k i} z^i}{W_{0}^{k v_k}}} \left|y^j(z)\right| f_{Z^{v_k} \mid Z_{-v_k}}\left(z^{v_k} \mid z_{-v_k}\right) \mathrm{d} z^{v_k} f_{Z_{-v_k}}\left(z_{-v_k}\right) \mathrm{d} z_{-v_k} \\ & \leq
\int_{\mathbb{R}^{m_1-1}} \int_{\frac{-f(\bar{b}_0^k)-\sum_{i \neq v_k} W_{0}^{k v_k}}{W_{0}^{k v_k}}}^{\frac{-f(b_0^k)-\sum_{i \neq v_k} W_{1}^{k i} z^i}{W_{0}^{k v_k}}} \left\lvert c_y\left(1+\left\lvert z\right\rvert^\rho\right)\right\rvert f_{Z^{v_k} \mid Z_{-v_k}}\left(z^{v_k} \mid z_{-v_k}\right) \mathrm{d} z^{v_k} f_{Z_{-v_k}}\left(z_{-v_k}\right) \mathrm{d} z_{-v_k}\\ & \leq
\int_{\mathbb{R}^{m_1-1}} \int_{\frac{-f(\bar{b}_0^k)-\sum_{i \neq v_k} W_{0}^{k v_k}}{W_{0}^{k v_k}}}^{\frac{-f(b_0^k)-\sum_{i \neq v_k} W_{1}^{k i} z^i}{W_{0}^{k v_k}}} \left\lvert c_y \right\rvert f_{Z^{v_k} \mid Z_{-v_k}}\left(z^{v_k} \mid z_{-v_k}\right) \mathrm{d} z^{v_k} f_{Z_{-v_k}}\left(z_{-v_k}\right) \mathrm{d} z_{-v_k}
\\ & \quad + 
\int_{\mathbb{R}^{m_1-1}} \int_{\frac{-f(\bar{b}_0^k)-\sum_{i \neq v_k} W_{0}^{k v_k}}{W_{0}^{k v_k}}}^{\frac{-f(b_0^k)-\sum_{i \neq v_k} W_{1}^{k i} z^i}{W_{0}^{k v_k}}}  \left\lvert c_y\right\rvert \left\lvert z\right\rvert^{\rho} f_{Z^{v_k} \mid Z_{-v_k}}\left(z^{v_k} \mid z_{-v_k}\right) \mathrm{d} z^{v_k} f_{Z_{-v_k}}\left(z_{-v_k}\right) \mathrm{d} z_{-v_k}.
\end{align*}
This implies that
\begin{align*}
&\mathbb{E}{\left[\left|Y^j\right|\left|\mathbbm{1}_{A_k}(Z)-\mathbbm{1}_{\bar{A}_k}(Z)\right|\right]}
\\ & \leq \left\lvert c_y\right\rvert\bigg(\frac{ C_{Z^{v_k}}}{W_0^{k v_k}}\left|\bar{b}_0^k-b_0^k\right| + 2^{\rho - 1} \int_{\mathbb{R}^{m_1-1}} \int_{\frac{-f(\bar{b}_0^k)-\sum_{i \neq v_k} W_{0}^{k v_k}}{W_{0}^{k v_k}}}^{\frac{-f(b_0^k)-\sum_{i \neq v_k} W_{1}^{k i} z^i}{W_{0}^{k v_k}}}  \left\lvert z^{v_k}\right\rvert^{\rho} f_{Z^{v_k} \mid Z_{-v_k}}\left(z^{v_k} \mid z_{-v_k}\right) \mathrm{d} z^{v_k} f_{Z_{-v_k}}\left(z_{-v_k}\right) \mathrm{d} z_{-v_k} \\ & \quad + 2^{\rho-1} \int_{\mathbb{R}^{m_1-1}} \int_{\frac{-f(\bar{b}_0^k)-\sum_{i \neq v_k} W_{0}^{k v_k}}{W_{0}^{k v_k}}}^{\frac{-f(b_0^k)-\sum_{i \neq v_k} W_{1}^{k i} z^i}{W_{0}^{k v_k}}} f_{Z^{v_k} \mid Z_{-v_k}}\left(z^{v_k} \mid z_{-v_k}\right) \mathrm{d} z^{v_k}  \left\lvert z_{-v_k}\right\rvert^{\rho} f_{Z_{-v_k}}\left(z_{-v_k}\right) \mathrm{d} z_{-v_k}\bigg)\\ & \leq
\left\lvert c_y\right\rvert\bigg(\frac{ C_{Z^{v_k}}}{W_0^{k v_k}}\left\lvert b_0^k - \bar{b}_0^k\right\rvert + 2^{\rho - 1}\frac{ \bar{C}_{Z^{v_k}}}{W_0^{k v_k}}\left\lvert b_0^k - \bar{b}_0^k\right\rvert + 2^{\rho - 1}\frac{ C_{Z^{v_k}}}{W_0^{k v_k}}\mathbb{E}\left[\left|Z_{-v_k}\right|^{\rho}\right]\left\lvert b_0^k - \bar{b}_0^k\right\rvert \bigg)
\\ & \leq C_{Y_Z, \max} \left\lvert \theta-\bar{\theta}\right\rvert, 
\end{align*}
where the first inequality holds due to $\left\lvert x + y \right\rvert^p \leq 2^{p-1} \left(\left\lvert x\right\rvert^p+\left\lvert y\right\rvert^p\right)$ for any $x, y \in \mathbb{R}$ and $p \geq 1$ and where $C_{Y_Z, \max }:=\max\limits_{k}\left\{\left(1 + 2^{\rho-1}\right)\frac{ C_{Z^{v_k}}}{W_0^{k v_k}} + 2^{\rho - 1}\frac{ \bar{C}_{Z^{v_k}}}{W_0^{k v_k}}\right\}\left\lvert c_y\right\rvert\left(1 + \mathbb{E}\left[\left\lvert Z\right\rvert^{\rho}\right]\right)$. Furthermore, by using~\eqref{mathbbz} and $\left\lvert x + y \right\rvert^p \leq 2^{p-1} \left(\left\lvert x\right\rvert^p+\left\lvert y\right\rvert^p\right)$ for any $x, y \in \mathbb{R}$ and $p \geq 1$, we derive, for each $k=1, \cdots, d_1$ and $j=1, \cdots, m_2$, that
\begin{align*}
& \mathbb{E}\left[\left\lvert Y^j\right\rvert \left\lvert Z\right\rvert^2 \left\lvert\mathbbm{1}_{A_k}(Z)-\mathbbm{1}_{\bar{A}_k}(Z)\right\rvert\right]  
\\ & = \mathbb{E}\left[\left\lvert Y^j\right\rvert\left(\left|Z\right|^2\right) \mathbbm{1}_{\left\{\left(-f(\bar{b}_0^k)-\sum_{i \neq v_k} W_0^{k i} Z^i\right) / W_0^{k v_{k}} \leq Z^{v_k}<\left(-f(b_0^k)-\sum_{i \neq v_k} W_0^{k i} Z^i\right) / W_0^{k v_k}\right\}}\right] 
\\ & \leq
\left\lvert c_y\right\rvert \int_{\mathbb{R}^{m_1-1}}\int^\frac{-f(b_0^k)-\sum_{i \neq v_k} W_0^{k i} z^i}{W_0^{k v_k}}_\frac{-f(\bar{b}_0^k)-\sum_{i \neq v_k} W_0^{k i} z^i}{W_0^{k v_k}} \left\lvert z\right\rvert^2 \left| 1+|z|^\rho\right| f_{Z^{v_k} \mid Z_{-v_k}}\left(z^{v_k} \mid z_{-v_k}\right) \mathrm{d} z^{v_k}  f_{Z_{-v_k}}\left(z_{-v_k}\right) \mathrm{d} z_{-v_k} \\ & \leq 
\left\lvert c_y\right\rvert \int_{\mathbb{R}^{m_1-1}}\int^\frac{-f(b_0^k)-\sum_{i \neq v_k} W_0^{k i} z^i}{W_0^{k v_k}}_\frac{-f(\bar{b}_0)^k-\sum_{i \neq v_k} W_0^{k i} z^i}{W_0^{k v_k}} \left\lvert z\right\rvert^2 f_{Z^{v_k} \mid Z_{-v_k}}\left(z^{v_k} \mid z_{-v_k}\right) \mathrm{d} z^{v_k}  f_{Z_{-v_k}}\left(z_{-v_k}\right) \mathrm{d} z_{-v_k} \\ & \quad +
\left\lvert c_y\right\rvert \int_{\mathbb{R}^{m_1-1}}\int^\frac{-f(b_0^k)-\sum_{i \neq v_k} W_0^{k i} z^i}{W_0^{k v_k}}_\frac{-f(\bar{b}_0^k)-\sum_{i \neq v_k} W_0^{k i} z^i}{W_0^{k v_k}} \left\lvert z\right\rvert^{\rho+2} f_{Z^{v_k} \mid Z_{-v_k}}\left(z^{v_k} \mid z_{-v_k}\right) \mathrm{d} z^{v_k}  f_{Z_{-v_k}}\left(z_{-v_k}\right) \mathrm{d} z_{-v_k}\\ & \leq
\left\lvert c_y\right\rvert \bigg(
\int_{\mathbb{R}^{m_1-1}}\int^\frac{-f(b_0^k)-\sum_{i \neq v_k} W_0^{k i} z^i}{W_0^{k v_k}}_\frac{-f(\bar{b}_0^k)-\sum_{i \neq v_k} W_0^{k i} z^i}{W_0^{k v_k}} \left\lvert z^{v_k}\right\rvert^{2} f_{Z^{v_k} \mid Z_{-v_k}}\left(z^{v_k} \mid z_{-v_k}\right) \mathrm{d} z^{v_k}  f_{Z_{-v_k}}\left(z_{-v_k}\right) \mathrm{d} z_{-v_k} \\&\quad  + 
\int_{\mathbb{R}^{m_1-1}}\int^\frac{-f(b_0^k)-\sum_{i \neq v_k} W_0^{k i} z^i}{W_0^{k v_k}}_\frac{-f(\bar{b}_0^k)-\sum_{i \neq v_k} W_0^{k i} z^i}{W_0^{k v_k}} f_{Z^{v_k} \mid Z_{-v_k}}\left(z^{v_k} \mid z_{-v_k}\right) \mathrm{d} z^{v_k}\left\lvert z_{-v_k}\right\rvert^{2}  f_{Z_{-v_k}}\left(z_{-v_k}\right) \mathrm{d} z_{-v_k} \\ & \quad +
2^{\rho + 1} \int_{\mathbb{R}^{m_1-1}}\int^\frac{-f(b_0^k)-\sum_{i \neq v_k} W_0^{k i} z^i}{W_0^{k v_k}}_\frac{-f(\bar{b}_0^k)-\sum_{i \neq v_k} W_0^{k i} z^i}{W_0^{k v_k}} \left\lvert z^{v_k}\right\rvert^{\rho+2} f_{Z^{v_k} \mid Z_{-v_k}}\left(z^{v_k} \mid z_{-v_k}\right) \mathrm{d} z^{v_k}  f_{Z_{-v_k}}\left(z_{-v_k}\right) \mathrm{d} z_{-v_k} \\ & \quad + 
2^{\rho + 1} \int_{\mathbb{R}^{m_1-1}}\int^\frac{-f(b_0^k)-\sum_{i \neq v_k} W_0^{k i} z^i}{W_0^{k v_k}}_\frac{-f(\bar{b}_0^k)-\sum_{i \neq v_k} W_0^{k i} z^i}{W_0^{k v_k}} f_{Z^{v_k} \mid Z_{-v_k}}\left(z^{v_k} \mid z_{-v_k}\right) \mathrm{d} z^{v_k}\left\lvert z_{-v_k}\right\rvert^{\rho+2}  f_{Z_{-v_k}}\left(z_{-v_k}\right) \mathrm{d} z_{-v_k} \bigg) \\ & \leq
\left\lvert c_y\right\rvert \bigg(\left(1 + 2^{\rho + 1}\right)\frac{\bar{C}_{Z^{v_k}}}{W_0^{k v_k}}+ 
\left(\mathbb{E}\left[\left|Z_{-v_k}\right|^{2}\right]  + 2^{\rho+1}\mathbb{E}\left[\left|Z_{-v_k}\right|^{\rho+2}\right]\right)\frac{C_{Z^{v_k}}}{W_0^{k v_k}} \bigg)\left\lvert \theta - \bar{\theta}\right\rvert \\& \leq
C_{Y_{Z^2}, \max} \left\lvert \theta - \bar{\theta}\right\rvert,
\end{align*}
where $ C_{Y_{Z^2}, \max}:= \left\lvert c_y\right\rvert \max\limits_{k}\left\{ \bigg(\left(1 + 2^{\rho + 1}\right)\frac{\bar{C}_{Z^{v_k}}}{W_0^{k v_k}}+ 
\left(\mathbb{E}\left[\left| Z \right|^{2}\right]  + 2^{\rho+1}\mathbb{E}\left[\left| Z \right|^{\rho+2}\right]\right)\frac{C_{Z^{v_k}}}{W_0^{k v_k}} \bigg)\right\}$. Then, the rest of the proof follows the similar lines as in the proof of Proposition~\ref{prop:6.2}.
\end{proof}
\bibliographystyle{plainnat}

\bibliography{references}

\end{document}